\documentclass[11pt]{amsart}
\usepackage{dynkin-diagrams}
 \usepackage{tikz}
\usepackage{psfrag}

\usepackage{listings}[2013/08/05]

\usepackage{bigints}

\usepackage{multirow}

 \usepackage{oubraces}

\usepackage{tikz}
\usetikzlibrary{matrix}

\tikzset{ 
table/.style={
  matrix of nodes,
  row sep=-\pgflinewidth,
  column sep=-\pgflinewidth,
  nodes={rectangle,text width=3em,align=center},
  text depth=4ex,
  text height=5ex,
  nodes in empty cells
},
row 1/.style={nodes={fill=green!10,text depth=0.4ex,text height=2ex}},
column 1/.style={nodes={fill=green!10, text width=10ex}},
column 2/.style={nodes={text width=20ex}},
column 3/.style={nodes={text width=20ex}},
column 4/.style={nodes={text width=20ex}},
}

\definecolor{darkgreen}{rgb}{.2,.6,.2}
\definecolor{MyDarkBlue}{rgb}{0.1,0,0.55}

\lstdefinelanguage{GAP}{%
  morekeywords={%
    Assert,Info,IsBound,QUIT,%
    TryNextMethod,Unbind,and,break,%
    continue,do,elif,%
    else,end,false,fi,for,%
    function,if,in,local,%
    mod,not,od,or,%
    quit,rec,repeat,return,%
    then,true,until,while%
  },%
  sensitive,%
  morecomment=[l]\#,%
  morestring=[b]",%
  morestring=[b]',%
}[keywords,comments,strings]

\usepackage[T1]{fontenc}
\usepackage[variablett]{lmodern}
\usepackage{xcolor}
\usepackage{lmodern}

\usepackage{cellspace}
\usepackage{slashbox}

\newcolumntype{C}{Sc}
\usepackage{color}
\usepackage[all,cmtip,color,matrix,arrow]{xy}

\usepackage[new]{old-arrows}
\usepackage{geometry}
\usepackage{graphicx}
\usepackage{lipsum}
\newsavebox{\myimage}

\usepackage{adjustbox}

\usepackage{hyperref}

\newcommand{\sk}{\smallskip}
\newcommand{\mk}{\medskip}
\newcommand{\bk}{\bigskip}

\usepackage{shuffle}

\usepackage{bigints}

\usepackage{tikz}

\newcommand{\xdasharrow}[2][->]{
\tikz[baseline=-\the\dimexpr\fontdimen22\textfont2\relax]{
\node[anchor=south,font=\scriptsize, inner ysep=1.5pt,outer xsep=2.2pt](x){#2};
\draw[shorten <=3.4pt,shorten >=3.4pt,dashed,#1](x.south west)--(x.south east);
}
}

\usepackage{amsmath}
\usepackage{amssymb}
\usepackage{amsthm}
\usepackage{latexsym}
\usepackage{pstricks}
\usepackage{amsbsy}
\usepackage{amscd}
\usepackage{mathrsfs}
\usepackage{tipa}
\usepackage{txfonts}
\usepackage{amsfonts}
\usepackage{graphicx}
\usepackage{tabularx}
\usepackage{listings}
\usepackage{tikz}
\usepackage{hyperref}

\usepackage{scalerel,stackengine}
\stackMath
\newcommand\reallywidehat[1]{%
\savestack{\tmpbox}{\stretchto{%
  \scaleto{%
    \scalerel*[\widthof{\ensuremath{#1}}]{\kern-.6pt\bigwedge\kern-.6pt}%
    {\rule[-\textheight/2]{1ex}{\textheight}}
  }{\textheight}%
}{0.5ex}}%
\stackon[1pt]{#1}{\tmpbox}%
}
\newtheorem*{thm*}{Theorem}
\newtheorem{thm}{Theorem}[section]
\newtheorem{cor}[thm]{Corollary}
\newtheorem{lem}[thm]{Lemma}

\newtheorem{prop}[thm]{Proposition}

\newtheorem{rem}[thm]{Remark}
\newtheorem{defn-prop}[thm]{Definition-Proposition}

\newtheorem{question}[thm]{Question}

\usepackage{graphicx,epsfig,amscd,graphics,xypic}

 \newcommand{\eq}[1][r]
   {\ar@<-3pt>@{-}[#1]
    \ar@<-1pt>@{}[#1]|<{}="gauche"
    \ar@<+0pt>@{}[#1]|-{}="milieu"
    \ar@<+1pt>@{}[#1]|>{}="droite"
    \ar@/^2pt/@{-}"gauche";"milieu"
    \ar@/_2pt/@{-}"milieu";"droite"}

  \newcommand{\incl}[1][r]
  {\ar@<-0.2pc>@{^(-}[#1] \ar@<+0.2pc>@{-}[#1]}

\author[L. Pirio]{\href{mailto:luc.pirio@uvsq.fr}{Luc Pirio}\textsuperscript{$\dagger$}}
\thanks{${}^{}$\hspace{-0.4cm}\textsuperscript{$\dagger$}
\href{https://lmv.math.cnrs.fr}{Laboratoire de Math\'ematiques de Versailles}, Univ.\,Paris-Saclay \& CNRS (UMR 8100), 78000 Versailles, France.}

\title[A story of webs]
 {A story of webs: \\ the webs by conics on del Pezzo quartic surfaces \\
 and  Gelfand-MacPherson's web of the spinor tenfold}

\begin{document}

\maketitle

\begin{abstract}
 In a \href{https://arxiv.org/abs/2401.06711}{previous paper},  we studied the web by conics  
$\boldsymbol{\mathcal W}_{{\rm dP}_4}$ on a del Pezzo quartic surface ${\rm dP}_4$ and proved that it enjoys suitable versions of most of the remarkable properties
satisfied by  Bol's web $\boldsymbol{\mathcal B}$.  
 In particular, Bol's web can be seen as the toric quotient of the Gelfand-MacPherson web naturally defined on the $A_4$-grassmannian variety $G_2(\mathbf C^5)$ and  we have
shown  that $\boldsymbol{\mathcal W}_{{\rm dP}_4}$ can be obtained in a  similar way from the 
web $\boldsymbol{\mathcal W}^{GM}_{ \hspace{-0.05cm} \boldsymbol{\mathcal Y}_5}$ which is the quotient  by the Cartan torus of ${\rm Spin}_{10}(\mathbf C)$, of the Gelfand-MacPherson 10-web naturally defined on the tenfold spinor variety $\mathbb S_5$,   a peculiar projective homogenous variety of type $D_5$.
%
%
 %
 In the present paper, by means of direct and explicit computations,  we show that many of the 
remarkable similarities between 
$\boldsymbol{\mathcal B}$ and $\boldsymbol{\mathcal W}_{{\rm dP}_4}$ actually can be extended to, or from an opposite perspective, can be seen as coming from some 
 similarities between Bol's web and 
$\boldsymbol{\mathcal W}^{GM}_{ \hspace{-0.05cm} \boldsymbol{\mathcal Y}_5}$.  The latter web can be seen as a natural uniquely defined rank 5 generalization of Bol's web. In particular, it carries a  peculiar 2-abelian relation, denoted by 
${\bf HLOG}_{ \boldsymbol{\mathcal Y}_5}$, 
which appears as a natural generalization of Abel's five terms relation of the dilogarithm 
and from which one can recover  the weight 3 hyperlogarithmic functional identity of any   quartic del Pezzo surface.
\end{abstract}

\setcounter{secnumdepth}{3}
\setcounter{tocdepth}{1}
\tableofcontents

\vspace{-0.5cm}

\newpage

This paper may be viewed as a continuation of \cite{CP} and \cite{PirioAFST},
to which the reader is referred for a more detailed exposition of the background and motivations underlying the questions addressed here. Nevertheless, the Introduction below should suffice to provide most readers with a general overview of the main themes investigated in this work.  \sk

 Here are a few words about the setting(s) and some general notations we will work with:  
  except in some cases which will be explicitly indicated and for which the setting is real algebraic or real analytic, we work over the field $\mathbf C$ of complex numbers, in 
  an algebraic or analytic framework. 
 We denote indifferently `$\ln$' or `${\rm Log}$' the usual complex logarithm (defined as the primitive vanishing at $1$ of the logarithmic differential $du/u$ on 
 the punctured complex plane $\mathbf C^*$). Given any positive integer $n$, we set  
 $[\hspace{-0.05cm}[ n ]\hspace{-0.05cm}]$ for the set of positive integers less than or equal to $n$: 
 one has  $[\hspace{-0.05cm}[ n ]\hspace{-0.05cm}]=\{1,2,\ldots,n\}$. 
 We will denote by $x,y$  the affine coordinates associated to the affine embedding 
$\mathbf C^2\hookrightarrow \mathbf P^2,\, (x,y)\mapsto [x:y:1]$.

\section{\bf Introduction}
\label{S:Intro}
`{\it Cauchy's identity}' of the logarithm 
$$
\boldsymbol{\big(\mathcal C\big)}
\hspace{5cm}
{\rm Log}(x) \, -\, {\rm Log}(y)\, -\, {\rm Log}\bigg(\frac{x}{y}\bigg) = 0
\hspace{6cm} {}^{}
$$
 admits a `weight 2 ' generalization, the so-called `{\it Abel's identity}' 
$$
\boldsymbol{\big(\mathcal Ab\big)}
\hspace{3cm}
{R}(x)-{R}(y)-{R}\left(\frac{x}{y}\right)-{R}\left(\frac{1-y}{1-x}\right)
+{R}\left(\frac{x(1-y)}{y(1-x)}\right)=0\,, 
\hspace{3cm} {}^{}
$$
satisfied for  any $x,y\in \mathbf R$ such that $0<x<y<1$, by \href{https://mathworld.wolfram.com/RogersL-Function.html}{\it Rogers' dilogarithm} $R$, which is the function defined by $R(x)= {\bf L}{\rm i}_2(x) + {\rm Log}(x){\rm Log}(1 - x)/2 - 
{\pi^2}/{6}$ for $x\in ]0,1[$.\footnote{Here ${\bf L}{\rm i}_2$
 stands for the classical bilogarithm: one has ${\bf L}{\rm i}_2(x)=\sum_{n=1}^{+\infty} \frac{x^n}{n^2}=- \int_0^x \frac{{\rm Log}(1-u)}{u} du$ for any $x$ such that $\lvert x\lvert<1$. 
 Soustracting $\pi^2/6$ in the given definition of Rogers dilogarithm is in order that the RHS of Abel's identity be zero.} 
\mk 

 In \cite{HM} (see also the last paragraph of \cite[\S4.1]{Griffiths}),  it is written (p.\,393) that it has been widely believed that the logarithm and the dilogarithm, together with the two functional identities above that they satisfy, are the first two elements of an infinite sequence of higher logarithms which share analogous properties: in particular each satisfies a peculiar functional identity, which even may be fundamental in a certain sense.  
In the as yet unpublished preprint \cite{GoncharovRudenko}, the authors 
wrote that {\it `writing explicitly functional equations for the classical
$n$-th polylogarithm might not be the ``right" problem}', because {\it `it seems that when $n$ is growing, the functional equations become so complicated that one can not write them down on a piece of paper'}.

 In \cite{CP}, it has been  shown that to work with hyperlogarithms instead of just polylogarithms, a simple and natural geometric construction allows to get a uniform series of functional identities, up to weight 6, which are very concise and whose first two elements precisely are the classical identities  $\boldsymbol{\big(\mathcal C\big)}$ and $\boldsymbol{\big(\mathcal Ab\big)}$ of the logarithm and the dilogarithm. 
Given a del Pezzo surface ${\rm dP}_d$ of degree $d\in \{1,\ldots,6\}$ with canonical class 
$K$, its set of conic classes $\boldsymbol{\mathcal K}=\{ \, \boldsymbol{\mathfrak c} \in {\bf Pic}_{\mathbf Z}\big( {\rm dP}_d\big)\, \lvert 
\, (-K,\boldsymbol{\mathfrak c})=2 \mbox{ and } \boldsymbol{\mathfrak c}^2=0\,\big\}$ is known to be finite and for each 
$\boldsymbol{\mathfrak c}$ in it, if $\phi_{\boldsymbol{\mathfrak c}}: {\rm dP}_d\rightarrow 
\lvert \,\boldsymbol{\mathfrak c}\, \lvert \simeq \mathbf P^1$ stands for the associated fibration in conics, then the set $\Sigma_{\boldsymbol{\mathfrak c}}\subset \mathbf P^1$ of its singular values has exactly  $8-d$ elements.  To $\Sigma_{\boldsymbol{\mathfrak c}}$, 
one can associate the {\it `complete weight $w=7-d$ antisymmetric hyperlogarithm $\boldsymbol{A\hspace{-0.05cm}H}^w_{\boldsymbol{\mathfrak c}}$'}, which is a multivalued function on $\mathbf P^1$, ramified at the points of $\Sigma_{\boldsymbol{\mathfrak c}}$ and which is canonically defined up to sign. 
The main result of \cite{CP} is the following: 
\begin{thm*} Given a point $x\in {\rm dP}_d$ not lying on a line and given 
a determination of the hyperlogarithms $
\boldsymbol{A\hspace{-0.05cm}H}^w_{\boldsymbol{\mathfrak c}}$
 at $\phi_{\boldsymbol{\mathfrak c}}(x)$ for any  ${\boldsymbol{\mathfrak c}} \in \boldsymbol{\mathcal K}$,  
 there exists a $\boldsymbol{\mathcal K}$-tuple   $(\epsilon_{\boldsymbol{\mathfrak c}})
_{ \boldsymbol{\mathfrak c} \in  \boldsymbol{\mathcal K} }
\in \big\{ \pm1 \big\}^ {\boldsymbol{\mathcal K}}$, unique up to a global sign,\footnote{Actually, one can chose the determinations of the $
\boldsymbol{A\hspace{-0.05cm}H}^w_{\boldsymbol{\mathfrak c}}$'s in such a way that one has 
$\epsilon_{\boldsymbol{\mathfrak c}}=1$ for every ${\boldsymbol{\mathfrak c}} \in \boldsymbol{\mathcal K}$, see \cite{CP} (in particular formula (9) in Theorem 3.1 therein).} such that the following identity is satisfied in the (complex analytic) vicinity of $x$: 
$$
\Big({\bf Hlog}^{w}_{ {\rm dP}_d}\Big)
\hspace{4cm}
\sum_{ 
\boldsymbol{\mathfrak c} \in 
\boldsymbol{\mathcal K}
}
\epsilon_{\boldsymbol{\mathfrak c}}\, 
\boldsymbol{A\hspace{-0.05cm}H}^w_{\boldsymbol{\mathfrak c}}
\big( \phi_{\boldsymbol{\mathfrak c}} \big)
=0\,. 
\hspace{6cm} {}^{}
$$
\end{thm*}
For $d = 6$ and $d = 5$, the ramification sets $\Sigma_{\boldsymbol{\mathfrak c}}$ have cardinalities 2 and 3, respectively, and can therefore be assumed to coincide with 
$\{0,\infty\}$ and 
$\{0,1,\infty\}$, respectively.  With this normalization, 
$\boldsymbol{A\hspace{-0.05cm}H}^1_{\boldsymbol{\mathfrak c}}
$ and $\boldsymbol{A\hspace{-0.05cm}H}^2_{\boldsymbol{\mathfrak c}}$ are just the usual logarithm and Rogers dilogarithm $\boldsymbol{R}$ respectively. 
 The del Pezzo surfaces ${\rm dP}_6$ and ${\rm dP}_5$ are unique (there is no modulus), the corresponding hyperlogarithmic functional identities as well and it is easy to check that they respectively coincide with Cauchy's and Abel's identity: one has 
$$
\big({\bf Hlog}^{1}_{ {\rm dP}_6}\big) \simeq  \big( \boldsymbol{\mathcal C} \big)
\qquad 
\mbox{ and } 
\qquad 
\big({\bf Hlog}^{2}_{ {\rm dP}_5}\big)
\simeq 
\big( \boldsymbol{\mathcal Ab}\big)
$$
(where the symbol $\simeq$ here means a coincidence between two identities up to a  (possibly local) change of coordinates).

Motivated by the question of whether the  `{\it del Pezzo identities  ${\bf Hlog}^{w}_{ {\rm dP}_d}$}' for $w=1,\ldots,6$,  genuinely can be considered as the most natural higher weights generalizations of Abel's relation or not, we carried out a thorough comparison of ${\bf Hlog}^{2}_{ {\rm dP}_5}$ and ${\bf Hlog}^{3}_{ {\rm dP}_4}$ in \cite{PirioAFST}. We did that by 
taking a web geometer perspective. For  $X={\rm dP}_d$ with $d=4,5$, let $\boldsymbol{W}$
 be the corresponding Weyl group acting on ${\bf Pic}(X)$ and let us denote by 
 $\boldsymbol{\mathcal W}_{ {\rm dP}_d }$ the web formed by the pencils of conics on $X$.

The main outcome of \cite{PirioAFST}  is that virtually all the remarkable properties of various kinds  satisfied by the pair $\big( \boldsymbol{\mathcal W}_{ {\rm dP}_5 }, {\bf Hlog}^{2}_{ {\rm dP}_5}\simeq  \boldsymbol{\mathcal Ab}\big)$  
admit natural analogues for $\big( \boldsymbol{\mathcal W}_{ {\rm dP}_4 }, {\bf Hlog}^{3}_{ {\rm dP}_4}\big)$, 
which also hold true. Below is a list of some of the remarkable properties shared by both webs 
$\boldsymbol{\mathcal W}_{ {\rm dP}_5}$
and $\boldsymbol{\mathcal W}_{ {\rm dP}_4}$ (see  \S1.1 and \S1.2 of  \cite{PirioAFST} for further details): 
\begin{itemize}
\item geometric definition as the webs formed by the pencils of conics on a del Pezzo surface; 
\mk
\item all their abelian relations are hyperlogarithmic; the Weyl group 
$\boldsymbol{W}$ acts on this space and when viewed as an abelian relation, the identity $\big({\bf Hlog}^{w}_{ {\rm dP}_d}\big)$ transforms according to the  signature under the action of $\boldsymbol{W}$ and spans all the subspaces of antisymetric hyperlogarithmic ARS by residues/monodromy;
dues and monodromy;
\mk 
\item maximality of the rank together with non linearizability (they are exceptional webs);
\mk 
\item combinatorial characterization by means of the hexagonal subwebs;
\mk 
\item modularity (ie.\,definition by means of a web naturally defined on a moduli space of configurations of points in a projective space); 
\mk 
\item cluster character (ie.\,definition by means of some cluster variables of a cluster algebra); 
\mk 
\item geometric construction \`a la Gelfand-MacPherson.
\mk 
\end{itemize} 

The fact that the two pairs $\big( 
\boldsymbol{\mathcal W}_{ {\rm dP}_5}, 
 \boldsymbol{\mathcal Ab}
\big)$ and $\big( 
\boldsymbol{\mathcal W}_{ {\rm dP}_4}, 
{\bf Hlog}^{3}_{ {\rm dP}_4}
\big)$ share so many similarities is quite striking, and may suggest that {\it `the hyperlogarithmic identity ${\bf Hlog}^{3}_{ {\rm dP}_4}$ is, in some sense, the most natural weight-3 generalization of Abel's identity $ \boldsymbol{\mathcal Ab}\simeq {\bf Hlog}^{2}_{ {\rm dP}_5}$'}. 

However, such a strong claim should be approached with caution. First--and quite evidently--the rigorous meaning of `most natural generalization' is far from being clear or universally accepted. Second, what makes Abel's identity for the dilogarithm particularly compelling is its relevance across many seemingly unrelated areas of mathematics. By contrast, to date, we are not aware of significant applications of ${\bf Hlog}^{3}_{ {\rm dP}_4}$  beyond the theory of functional identities itself and the realm of web geometry.
\begin{center}
\vspace{-0.3cm}
$\star$
\end{center}

But there are more substantial considerations that temper the naive assumption of viewing ${\bf Hlog}^{3}_{ {\rm dP}_4}$ as `the' most natural weight-3 analogue of Abel's identity:
\begin{itemize}
\item[$-$]  the (smooth) del Pezzo quintic surface has no moduli 
hence the dilogarithmic identity ${\bf Hlog}^{2}_{ {\rm dP}_5}$ is genuinely unique. This is not the case for del Pezzo surfaces of degree $d=4$ since the moduli space of these surfaces is of dimension 2. Hence 
there is 
a 2-dimensional family of functional identities ${\bf Hlog}^{3}_{ {\rm dP}_4}$, and no unique/well-defined weight 3 hyperlogarithmic identity. 
The same phenomenon occurs for the hyperlogarithmic identities of higher weights: 
 for each 
 $d=1,2,\ldots,5$, 
 the  ${\bf Hlog}^{7-d}_{ {\rm dP}_{d}}$'s form an irreducible complex analytic family of dimension $2(5-d)$ of functional relations and in each family,  no identity appears as being more canonical or particular than the others;
 \mk
\item[$-$] if ${\bf Hlog}^{2}_{ {\rm dP}_5}$ and ${\bf Hlog}^{3}_{ {\rm dP}_4}$ share an important number of nice features, there is a formal difference between these identities which may make  the latter identity appear as less 
fundamental than the former. Indeed,  if Abel's identity 
$\boldsymbol{\mathcal A}{\bf b}\simeq 
{\bf Hlog}^{2}_{ {\rm dP}_5}$ involves only  one function, namely Rogers dilogarithm $R$, this is not the case for ${\bf Hlog}^{3}_{ {\rm dP}_4}$ since for a generic quartic surface $ {\rm dP}_4$, the ten weight 3 hyperlogarithms $\boldsymbol{A\hspace{-0.05cm}H}^3_{\boldsymbol{\mathfrak c}}$ for 
$\boldsymbol{\mathfrak c}\in \boldsymbol{\mathcal K}$ do not coincide, even up to sign and to precomposition by a projective automorphism;\footnote{As it follows easily from the explicit form for 
${\bf Hlog}^{3}_{ {\rm dP}_4}$ given in \cite[\S4]{CP}, the quotient 
$\big\{ \boldsymbol{A\hspace{-0.05cm}H}^3_{\boldsymbol{\mathfrak c}}
 \, \lvert 
 \,  \boldsymbol{\mathfrak c}\in \boldsymbol{\mathcal K} \, 
 \big\}   \big/ {\sim}$ 
is of cardinal 5 if $\sim$ stands for the following equivalence relation: given  two (germs of) functions $F,G$, one has $F\sim G$ if and only if 
 $\pm F= G\circ \gamma$ for some  projective automorphism $\gamma \in {\rm Aut}(\mathbf P^1)={\bf PGL}_2(\mathbf C)$.}
 \mk 
\item[$-$] the Weyl group ${W}_{ {\rm dP}_d } $ of a del Pezzo surface ${\rm dP}_d$ acts on the space of ARs of  $\boldsymbol{\mathcal W}_{ {\rm dP}_d }$ but there is a major difference between the case $d=5$ and 
$d\in \{1,\ldots,4\}$. 
 The natural group embedding ${\rm Aut}\big({\rm dP}_d\big)\hookrightarrow {W}_{ {\rm dP}_d }$ is an  isomorphism only for $d=5$ hence the Weyl group action on $\boldsymbol{AR}\big( \boldsymbol{\mathcal W}_{ {\rm dP}_d }
\big)$ is not geometric (that is is not induced by automorphisms of the considered del Pezzo surface)  
for any $d\in \{1,\ldots,4\}$;
 \mk 
\item[$-$]  if the identities ${\bf Hlog}^{3}_{ {\rm dP}_4}$ were truly the `right' weight 3 generalizations of Abel's identity, it would be natural to expect the same for  the  identities ${\bf Hlog}^{w}_{ {\rm dP}_d}$'s in higher weights 
 $w=8-d$. It turns out that it does not seem to be the case since some of the most striking  remarkable properties shared by the webs $\boldsymbol{\mathcal W}_{ {\rm dP}_5}$ and 
$\boldsymbol{\mathcal W}_{ {\rm dP}_4 }$ 
are no longer satisfied by the webs $\boldsymbol{\mathcal W}_{ {\rm dP}_d}$ for $d\leq 3$. For instance, by a direct computation, we have verified that 
the curvature\footnote{By definition, this is the sum of the Blaschke-Dubourdieu's curvatures of all of 
the   
$3$-subwebs 
of $\boldsymbol{\mathcal W}_{ {\rm dP}_3}$ (see \cite{CL}).} 
of a 27-web $\boldsymbol{\mathcal W}_{ {\rm dP}_3}$ is non zero, which  implies that this web is not of maximal rank, contrarily to the webs $\boldsymbol{\mathcal W}_{ {\rm dP}_d}$ for $d=4,5$.  
\end{itemize}

\newpage
The above considerations naturally led us to seek a generalization of the relation 
${\mathcal A}{\bf b}\simeq 
{\bf Hlog}^{2}_{ {\rm dP}_5}$ 
 that satisfies the following desirable properties:
\begin{itemize}
  \item {\it It is defined on a space equipped with a natural action of the Weyl group of type $D_5$, and this action operates via automorphisms. Moreover, the looked for generalization behaves coherently and compatibly with respect to this group action.}
  \mk 
  \item {\it It consists of a single, well-defined identity, without any  module entering into the picture.}
  \mk 
  \vspace{-0.4cm}
  \item {\it This identity involves a unique element whose sum of pullbacks under a specific family of maps vanishes identically.}
  \mk 
  \item {\it Moreover, every hyperlogarithmic functional identity ${\bf Hlog}^{3}_{ {\rm dP}_4}$ can be recovered from this single identity, and in a natural, canonical way.}
\end{itemize}

In this paper, we provide a fully explicit formulation of such an identity and investigate several of its noteworthy properties. In particular, we show that it is unique (up to multiplication by a nonzero scalar) and that it satisfies all four of the criteria listed above.

An additional remarkable feature of $
\boldsymbol{\mathcal W}_{ {\rm dP}_5}$ is that it is a `cluster web'.  We further demonstrate that a certain lift of the generalization of $ \boldsymbol{\mathcal W}_{ {\rm dP}_5} $
discussed here also possesses a cluster structur-- albeit with respect to a generalized notion of cluster variety.

In the following sections, we present our main results in greater detail.
\begin{center}
\vspace{-0.3cm}
$\star$
\end{center}

In their interesting paper \cite{GM}, Gelfand and MacPherson describe  a geometric setting which allows them to give a cohomological-analytic construction of Abel's identity. On the purely geometric side,  they show that the del Pezzo quintic surface 
${\rm dP}_5\simeq \overline{\mathcal M}_{0,5}$ together with the five 
fibrations in conics, which correspond to the 
five forgetful morphisms $\overline{\mathcal M}_{0,5}\rightarrow \overline{\mathcal M}_{0,4} \simeq \mathbf P^1$,  are equivariant quotients, under the  action 
of the Cartan torus $H_{A_4}$ of ${\rm GL}_5(\mathbf C)$, 
of the grassmannian $G_2(\mathbf C^5)$ and of the five natural rational maps 
$G_2(\mathbf C^5)\dashrightarrow G_2\big(\mathbf C^5/\langle e_i\rangle\big)\simeq G_2(\mathbf C^4)$,  induced for each $i\in \{1,\ldots,5\}$ by the linear projection $\mathbf C^5\rightarrow 
\mathbf C^5/\langle e_i\rangle$ onto the quotient of  $\mathbf C^5$ by the line
spanned by the $i$-th element of the  canonical basis $(e_k)_{k=1}^5$ of $\mathbf C^5$.

In view of generalizing Gelfand and MacPherson approach 
to the webs $\boldsymbol{\mathcal W}_{{\rm dP}_d}$ for any $d\in \{1,2,\ldots,4\}$, we remarked in \cite{PirioAFST} that there is an intrinsic way to recover $G_2(\mathbf C^5)$ from  $X_4={\rm dP}_5\simeq \overline{\mathcal M}_{0,5}$: the former variety is the (projective) Cox variety of the latter.  

Let ${\rm dP}_d$ be a fixed smooth del Pezzo surface  of degree $d\in \{2,\ldots,5\}$. We will also denote it by $X_r$ to emphasize that it can be obtained as the blow-up of $\mathbf P^2$ at $r=9-d$ points in general position.  
We denote by $\boldsymbol{\mathcal L}$ the set of lines contained in $X_r$
 and by $E_r$ the Dynkin type of the considered del Pezzo surface. The associated Weyl group acting transitively by permutations on $\boldsymbol{\mathcal L}$ will be denoted by $\boldsymbol{W}_r$. 
 
 Then we have the following  (see \cite[\S4.5.2]{PirioAFST} for more details and references): \sk
 \begin{itemize}
  \item there is a projective Cox variety $\mathbf P(X_r)$ which is acted upon by a rank $r$ torus $\mathcal T_{\rm NS}= {\rm T}_{\rm NS}/\mathbf C^*\simeq \big(  \mathbf C^*
\big)^r  
  $ which is the quotient of the Neron-Severi torus 
  ${\rm T}_{\rm NS}={\rm Hom}_{\mathbf Z}\big( {\bf Pic}(X_r),\mathbf C^* \big)$  by a one-parameter subgroup associated to the (anti)canonical class 
 of $X_r$;
 \mk 
    \item the Cox ring of $X_r$ is generated by the lines contained in it.\footnote{More rigorously, the Cox ring of $X_r$ 
is generated by any set  $\{ \sigma_\ell\}_{ \ell\in \boldsymbol{\mathcal L} } $ where for any line $\ell \subset X_r$,  $\sigma_\ell$ is a non zero element of
${\bf H}^0\big(X_r, \mathcal O_{ X_r } (\ell) 
\big)\simeq \mathbf C$.} Consequently,  the projective Cox variety ${\mathbf P}(X_r)$ embeds into the projective space $\mathbf P\big(\mathbf C^{\boldsymbol{\mathcal L}} \big)$; 
\mk   
 \item from some results by 
several authors (Popov, Batyrev-Popov, Derenthal, Serganova-Skoro\- bogatov)  it follows that  $\mathbf C^{\boldsymbol{\mathcal L}}$ naturally is a minuscule representation of a simple complex Lie group $\boldsymbol{G}_r$ of type $E_r$. 
The weights of this representation identify with the lines included in $X_r$ and the 
Weyl group $\boldsymbol{W}_r$ acts transitively on their set;
\mk   
 \item viewed as contained in $\mathbf P\big(\mathbf C^{\boldsymbol{\mathcal L}} \big)$, the Cox variety ${\mathbf P}(X_r)$ actually is a subvariety of 
the associated minuscule homogeneous space $\boldsymbol{\mathcal G}_r=\boldsymbol{G}_r/\boldsymbol{P}_r \subset \mathbf P\big(\mathbf C^{\boldsymbol{\mathcal L}} 
\big) $, where $\boldsymbol{P}_r$ is a suitable maximal parabolic subgroup of $\boldsymbol{P}_r$; 
\mk 
 \item moreover, up to a natural isomorphism $\mathcal T_{\rm NS}\simeq \boldsymbol{H}_r$ between the torus acting on  ${\mathbf P}(X_r)$ and the Cartan torus $\boldsymbol{H}_r$ of $\boldsymbol{G}_r$, the 
map $\mathbf P(X_r)\hookrightarrow \boldsymbol{\mathcal G}_r \subset \mathbf P\big(\mathbf C^{\boldsymbol{\mathcal L}} \big)$ turns out to be a torus-equivariant embedding; 
\mk 
 \item  as shown by Skorobogatov in \cite{Skorobogatov}, there exists a Zariski open subset $\boldsymbol{\mathcal G}_r^{sf}\subset \boldsymbol{\mathcal G}_r$ whose complement has codimension at least $2$, on which the action of $H_r$ is sufficiently well-behaved to define a geometric quotient  $\boldsymbol{\mathcal Y}_r=\boldsymbol{\mathcal G}_r^{sf}/\boldsymbol{H}_r$. This quotient is a quasi-projective variety. 
Moreover: \sk 
 \begin{enumerate}
 \item[$-$] the linear action of the Weyl group $\boldsymbol{W}_r$ on $\mathbf C^{\boldsymbol{\mathcal L}}$ gives rise to an isomorphism $\boldsymbol{W}_r \simeq 
 {\rm Aut}\big( \boldsymbol{\mathcal Y}_r\big)$ (see  \cite[Theorem 2.2]{Skorobogatov});
 \sk 
 \item[$-$] there is a natural embedding $F_{\hspace{-0.03cm}S\hspace{-0.04cm}S} : X_r \hookrightarrow 
 \boldsymbol{\mathcal Y}_r$  inducing an isomorphism of the Picard lattices $F_{\hspace{-0.03cm}S\hspace{-0.04cm}S}^* : {\bf Pic}_{\mathbf Z}( \boldsymbol{\mathcal Y}_r)\simeq {\bf Pic}_{\mathbf Z}( X_r)$
 and making commutative the following diagram: 
\begin{equation}
\label{Eq:Embedding-SS}
  \xymatrix@R=1cm@C=0.4cm{ 
\mathbf P(X_r)  \,  
\ar@{^{(}->}[rrrrr]
\ar@{->}[d]
   &  &&&    &
\ar@{->}[d]   \boldsymbol{\mathcal G}_r  &
\hspace{-0.7cm}
 \subset \mathbf P \big( \mathbf C^{\boldsymbol{\mathcal L} } \big)
 \\
 X_r   \ar@{->}[rrrrr]^{F_{\hspace{-0.03cm}S\hspace{-0.03cm}S} \hspace{0.4cm} }
  &  &&& & \boldsymbol{\mathcal Y}_r \, .
  }
\end{equation}
 \end{enumerate}
 \end{itemize}
 
 The interest of the material above is that it allows to construct the web $\boldsymbol{\mathcal W}_{ {\rm dP}_d}$, for an arbitrary del Pezzo surface,  from a unique web on  $\boldsymbol{\mathcal Y}_r $  obtained as the quotient of a $\boldsymbol{H}_r$-equivariant web naturally defined  on 
 $ \boldsymbol{\mathcal G}_r $ and induced by linear projection on $\mathbf C^{\boldsymbol{\mathcal L}}$.  Indeed, denoting by $\mathfrak h_{\mathbf R}$ the Lie algebra of the real part of $\boldsymbol{H}_r$ and by $\boldsymbol{H}_r^{+}\simeq ( \mathbf R_{>0})^r$ the `positive part' of the latter, we have the following : 
\begin{itemize}
 \item there is a (real analytic) moment map $\mu :  \boldsymbol{\mathcal G}_r  \rightarrow 
 \mathfrak h_{\mathbf R}^*$ whose image is the associated `minuscule weight polytope' $\Delta_{r}=\Delta_{\boldsymbol{G}_r, \boldsymbol{P}_r}$, that is the convex envelope of the weights of the minuscule representation $\mathbf C^{\boldsymbol{\mathcal L}}$; 
 \sk 
 \item  for $\zeta$  generic, that is in a certain dense Zariski open subset $\boldsymbol{\mathcal G}_r^\circ $, the moment map induces an isomorphism between the positive orbit $\boldsymbol{H}_r^{+}\cdot \zeta$ and the interior 
 $\mathring{\Delta}_{r}$ of the weight polytope, which extends to an isomorphism of real analytic manifolds with corners $\mu : \overline{\boldsymbol{H}_r^{+}\cdot \zeta}\simeq \Delta_{r}$; 
 \sk 
 \item
 \vspace{-0.35cm}
   for each facet (that is a face of codimension 1) $F$ of $\Delta_{r}$, let 
 $\boldsymbol{\mathcal L}_F$ be the set of lines belonging to $F$ (in other terms 
$ \boldsymbol{\mathcal L}_F$ is the set of vertices of $F$) and let $\Pi_F: \mathbf C^
{\boldsymbol{\mathcal L}} \rightarrow \mathbf C^{\boldsymbol{\mathcal L}_F}$ be the linear projection associated to the inclusion $\boldsymbol{\mathcal L}_F\subset \boldsymbol{\mathcal L}$ (i.e. ${\rm Ker}\big(\Pi_F \big)=  
\mathbf C^
{\boldsymbol{\mathcal L} \setminus {\boldsymbol{\mathcal L}_F}}$);
 \mk 
 \item  there exists a quotient $\boldsymbol{G}_F$ of the subgroup of $G$ letting invariant 
 the decomposition in direct sum $\mathbf C^
{\boldsymbol{\mathcal L}} =   \mathbf C^{\boldsymbol{\mathcal L}_F}
\oplus
\mathbf C^
{\boldsymbol{\mathcal L} \setminus {\boldsymbol{\mathcal L}_F}}
$ and a homogeneous projective space $\boldsymbol{\mathcal G}_F= \boldsymbol{G}_F/\boldsymbol{P}_F\subset \mathbf P\big( 
\mathbf C^{\boldsymbol{\mathcal L}_F}
\big)$ (for a certain parabolic subgroup $\boldsymbol{P}_F$ constructed from $P$) 
such that $(i).$ one has $ \boldsymbol{\mathcal G}_F=\boldsymbol{\mathcal G}_r \cap 
\mathbf P\big(\mathbf C^
{\boldsymbol{\mathcal L}}\big)$; and 
 $(ii).$  the restriction to $\boldsymbol{\mathcal G}_r$ of the (projectivization of the) linear projection $\Pi_F$ gives rise to a dominant rational map $
\boldsymbol{\mathcal G}_r\dashrightarrow \boldsymbol{\mathcal G}_F$, again denoted by $\Pi_F$; 
 \sk 
 \item  let $\boldsymbol{H}_F$ be the Cartan torus of  $\boldsymbol{G}_F$. 
 There is a natural epimorphism of tori $\boldsymbol{H}\rightarrow \boldsymbol{H}_F$ 
 with respect to which $\Pi_F: \boldsymbol{\mathcal G}_r\dashrightarrow \boldsymbol{\mathcal G}_F$ is equivariant. It follows that there exists a dominant rational map $\pi_F$  from  
$ \boldsymbol{\mathcal Y}_r $ onto the quotient $
  \boldsymbol{\mathcal Y}_F=\boldsymbol{\mathcal G}_F^{sf} /\boldsymbol{H}_F$ such that the following diagram commutes: 
 \begin{equation}
\label{Eq:pi-F}
\xymatrix@R=0.5cm@C=1.3cm{
\boldsymbol{\mathcal G}_r  \ar@{->}[d]
 \ar@{-->}[r]^{\Pi_F }   
& \ar@{->}[d]
 \boldsymbol{\mathcal G}_F
&
\hspace{-1.5cm}
 \subset \mathbf P \Big( \mathbf C^{\boldsymbol{\mathcal L}_F } \Big) 
  \\
\boldsymbol{\mathcal Y}  \ar@{-->}[r]^{\pi_F }   
& \boldsymbol{\mathcal Y}_F\, ; }
\end{equation}
 \sk 
 \item  for any conic class $\boldsymbol{\mathfrak c}\in \boldsymbol{\mathcal K}$ on $X_r$, there exists a uniquely determined facet $F_{\boldsymbol{\mathfrak c}}$ of 
 $\Delta_r$ such that gluing the two  diagrams \eqref{Eq:Embedding-SS} and \eqref{Eq:pi-F} gives the following commutative one
\begin{equation*}
  \xymatrix@R=1cm@C=0.2cm{ 
\mathbf P(X_r)  \,  
\ar@{^{(}->}[rrrrr]
\ar@{->}[d]
   &  &&&    &
\ar@{->}[d]   \boldsymbol{\mathcal G}_r 
\ar@{-->}[rrr]^{\Pi_{F_{\boldsymbol{\mathfrak c}}}}
 & & & \boldsymbol{\mathcal G}_{F_{\boldsymbol{\mathfrak c}}}\ar@{->}[d] 
 \\
 X_r   \ar@{->}[rrrrr]^{F_{S\hspace{-0.02cm}S}  \hspace{0.4cm} }
  &  &&& & \boldsymbol{\mathcal Y}_r 
 \ar@{-->}[rrr]^{\pi_{F_{\boldsymbol{\mathfrak c}}}}
 & & & \boldsymbol{\mathcal Y}_{F_{\boldsymbol{\mathfrak c}}} 
  \, .
  }
\end{equation*}
which is such that the composition of the rational maps of the bottom line  coincides with the conic fibration $\phi_{ \boldsymbol{\mathfrak c} }$ associated to $\boldsymbol{\mathfrak c}$: as rational maps on $X_r$, one has $\phi_{ \boldsymbol{\mathfrak c} } = \pi_{ F_{\boldsymbol{\mathfrak c}} }\circ 
 f_{S\hspace{-0.04cm}S}$.
 \sk 
\end{itemize}

At this point, one can define the {\it `Gelfand-MacPherson's webs} 
$\boldsymbol{\mathcal W}^{GM}_{{}^{} \hspace{-0.05cm} \boldsymbol{\mathcal G}_r  }$ and $\boldsymbol{\mathcal W}^{GM}_{{}^{} \hspace{-0.05cm} \boldsymbol{\mathcal Y}_r  }$, which are respectively the (generalized) web on 
$ \boldsymbol{\mathcal G}_r $ and $ \boldsymbol{\mathcal Y}_r $,  induced by the rational maps ${\Pi_{F_{\boldsymbol{\mathfrak c}}}}$ and ${\pi_{F_{\boldsymbol{\mathfrak c}}}}$  for $\boldsymbol{\mathfrak c}$ ranging in the set 
$\boldsymbol{\mathcal K}$ of all conic classes of $X_r$: one has 
$$
\boldsymbol{\mathcal W}^{GM}_{{}^{} \hspace{-0.05cm} \boldsymbol{\mathcal G}_r  }
= \boldsymbol{\mathcal W}\Big( \, {\Pi_{F_{\boldsymbol{\mathfrak c}}}} \, \big\lvert \, 
\boldsymbol{\mathfrak c} \in \boldsymbol{\mathcal K}\, 
\Big)
\qquad \mbox{ and } 
\qquad 
\boldsymbol{\mathcal W}^{GM}_{{}^{} \hspace{-0.05cm} \boldsymbol{\mathcal Y}_r  }
 = \boldsymbol{\mathcal W}\Big( \, {\pi_{F_{\boldsymbol{\mathfrak c}}}} \, \big\lvert \, 
\boldsymbol{\mathfrak c} \in \boldsymbol{\mathcal K}\, 
\Big)\; . 
$$
From the last two statements in the list above, we deduce that $\boldsymbol{\mathcal W}^{GM}_{{}^{} \hspace{-0.05cm} \boldsymbol{\mathcal Y}_r  }$ can be seen as the quotient of the web $\boldsymbol{\mathcal W}^{GM}_{{}^{} \hspace{-0.05cm} \boldsymbol{\mathcal G}_r  }$ which is $\boldsymbol{H}_r$-equivariant, and also that 
del Pezzo's web $\boldsymbol{\mathcal W}_{ {\rm dP}_d }$ is the pull-back 
 of the web $\boldsymbol{\mathcal W}^{GM}_{{}^{} \hspace{-0.05cm} \boldsymbol{\mathcal Y}_r  }$ under Skorobogatov-Serganova's embedding $ F_{S\hspace{-0.04cm}S}
 : {\rm dP}_d=X_r\rightarrow \boldsymbol{\mathcal Y}_r$: 
one has 
\begin{equation}
\boldsymbol{\mathcal W}_{ {\rm dP}_d }= 
F_{\hspace{-0.02cm}S\hspace{-0.04cm}S}^*
\Big( 
\boldsymbol{\mathcal W}^{GM}_{{}^{} \hspace{-0.05cm} \boldsymbol{\mathcal Y}_r  }
\Big)\, . \footnotemark
\end{equation}
\footnotetext{This is proved by direct computations for $d\in \{2,\ldots,4\}$ (Maple worksheets of these computations are available under request). For more details in the case when $d=4$ (which is equivalent to $r=5$), see \cite[Prop.\,4.16]{PirioAFST}.}

In this paper, we focus on the webs by conics of del Pezzo quartic surfaces and their relations to the Gelfand-MacPherson web $\boldsymbol{\mathcal W}^{GM}_{{}^{} \hspace{-0.05cm} \boldsymbol{\mathcal Y}_5  }$. More precisely, 
since the  $\boldsymbol{\mathcal W}_{  \hspace{-0.05cm} {\rm dP}_4}$'s can all be obtained from 
Gelfand-MacPherson web 
 $ \boldsymbol{\mathcal W}^{GM}_{
\boldsymbol{\mathcal Y}_5
}$, 
  it is not unreasonable to ask whether this latter web cannot be seen as a  more natural generalization of $\boldsymbol{\mathcal W}_{  \hspace{-0.05cm}{\rm dP}_5}\simeq 
 \boldsymbol{\mathcal W}^{GM}_{ \hspace{-0.05cm}
\boldsymbol{\mathcal Y}_4}$ than the $\boldsymbol{\mathcal W}_{  \hspace{-0.05cm}{\rm dP}_4}$'s. 
\begin{center}
\vspace{-0.15cm}
$\star$
\end{center}

  The web $ \boldsymbol{\mathcal W}^{GM}_{ \hspace{-0.05cm}
\boldsymbol{\mathcal Y}_5}$ is defined on the 
variety $\boldsymbol{\mathcal Y}_5$, which is rational and of dimension 5, by means of 10 rational first integrals 
$\psi_i^{\epsilon}: \boldsymbol{\mathcal Y}_5\dashrightarrow  \mathbf P^2$ with $i=1,\ldots,5$ and $\epsilon=\pm $ (here we use notations similar to those of \cite{PirioAFST}).  The main theme of the present paper is the study of the $k$-abelian relations of $ \boldsymbol{\mathcal W}^{GM}_{ \hspace{-0.05cm}
\boldsymbol{\mathcal Y}_5}$ for $k=0,1,2$ and how these are related to the abelian relations of a given $\boldsymbol{\mathcal W}_{  \hspace{-0.05cm}{\rm dP}_4}$.  
 Recall that a $k$-abelian relation (ab.\,$k$-AR) for $ \boldsymbol{\mathcal W}^{GM}_{ \hspace{-0.05cm} \boldsymbol{\mathcal Y}_5}$, is a 10-tuple of $k$-forms $\big( \eta_{i}^\epsilon\big)_{i,\epsilon}$ such that 
$\sum_{i,\epsilon} \big(\psi_i^{\epsilon}\big)^* \big( 
\eta_{i}^\epsilon \big)=0$  (possibly just locally) on $\boldsymbol{\mathcal Y}_5$.
They form a vector space which we denote by $\boldsymbol{AR}^{k}\big( 
 \boldsymbol{\mathcal W}^{GM}_{ \hspace{-0.05cm}
\boldsymbol{\mathcal Y}_5}\big)$, and whose dimension ${\rm rk}^k\big( 
 \boldsymbol{\mathcal W}^{GM}_{ \hspace{-0.05cm}
\boldsymbol{\mathcal Y}_5}\big)$ is the {\it `$k$-rank'} of 
$\boldsymbol{\mathcal W}^{GM}_{ \hspace{-0.05cm}
\boldsymbol{\mathcal Y}_5}$.  
Although $ \boldsymbol{\mathcal W}^{GM}_{ \hspace{-0.05cm}
\boldsymbol{\mathcal Y}_5}$ is a `web' only in a generalized sense, 
a similar approach to the one described in \cite[\S1.3.4]{ClusterWebs} can be applied to it and  one can define the {\it `virtual $k$-rank'} $\rho^{k}(\boldsymbol{\mathcal W})$ of any subweb $\boldsymbol{\mathcal W}$ of $ \boldsymbol{\mathcal W}^{GM}_{ \hspace{-0.05cm}
\boldsymbol{\mathcal Y}_5}$, this for $k\in \{0,1,2\}$. An interesting fact is that all the virtual ranks of $ \boldsymbol{\mathcal W}^{GM}_{ \hspace{-0.05cm}
\boldsymbol{\mathcal Y}_5}$ are finite (see Proposition \ref{Prop:Virtual-Ranks} further) and in this paper we will describe the spaces of $k$-ARs  of the web $ \boldsymbol{\mathcal W}^{GM}_{ \hspace{-0.05cm} \boldsymbol{\mathcal Y}_5}$ for $k=0,1, 2$, see from \S3.2 to \S3.4 (the two tables in \S\ref{SS:pictural} provide a concise summary of the results obtained in these three subsections). 
\mk

The most interesting case is that of 2-ARs of $\boldsymbol{\mathcal W}^{GM}_{ \hspace{-0.05cm}
\boldsymbol{\mathcal Y}_5}$ and how they give rise to 1-ARs for any del Pezzo web 
$\boldsymbol{\mathcal W}_{  \hspace{-0.05cm}{\rm dP}_4}$. 
In order to state our main result about the 2-ARs, introducing some terminology will be useful. For  $\psi$ standing for one of the first integrals  
$\psi_i^\epsilon$ of $ \boldsymbol{\mathcal W}^{GM}_{ \hspace{-0.05cm}
\boldsymbol{\mathcal Y}_5}$, we denote by 
\begin{itemize}
\item 
$\mathbf C(\psi)$ the subalgebra of 
$\mathbf C(\boldsymbol{\mathcal Y}_5)$ formed by compositions $f\circ \psi$ with $f \in\mathbf C( \mathbf P^2)$;\sk
\item ${\rm Log}\mathbf C(\psi)$ the family of multivalued functions on $\boldsymbol{\mathcal Y}_5$ of the form  ${\rm Log}( \phi)$ with $\phi \in \mathbf C( \psi)$;
\sk
\item $d{\rm Log}\mathbf C(\psi)$ the space of $\psi$-logarithmic differential 1-forms, that is of rational 1-forms on 
$\boldsymbol{\mathcal Y}_5$ of the form  $d{\rm Log}( \phi)=
d\phi/ \phi$ with $\phi  \in \mathbf C( \psi)$.
\end{itemize}

\label{Page:tables}

The following theorem gathers some of the most interesting results obtained in this paper: 
\begin{thm} 
\label{THM:WdP5-WGMS5-similarities}
%
%
\begin{enumerate}
\item[]  \hspace{-1.2cm} {\rm 1.} One has  $\rho^{2}\Big(
\boldsymbol{\mathcal W}^{GM}_{ \hspace{-0.05cm}
\boldsymbol{\mathcal Y}_5}
\Big)= 11$ 
and $\rho^{2}\big(
\boldsymbol{\mathcal W}
\big)\leq 1$  for every 5-subweb  $\boldsymbol{\mathcal W}$ of $\boldsymbol{\mathcal W}^{GM}_{ \hspace{-0.05cm}
\boldsymbol{\mathcal Y}_5}$. 
\mk 
\item[2.] Among all the 
 5-subwebs of $\boldsymbol{\mathcal W}^{GM}_{ \hspace{-0.05cm}
\boldsymbol{\mathcal Y}_5}$, exactly 16 have virtual 2-rank equal to 1. These are the subwebs $\boldsymbol{\mathcal W}^{\underline{\epsilon}}=\boldsymbol{\mathcal W}\big(\, \psi_1^{\epsilon_1}\, , \, \ldots,  
\psi_5^{\epsilon_5}\,\big)$ for the sixteen 5-tuples $\underline{\epsilon}=(\epsilon_i)_{i=1}^5\in \{\pm 1\}$
such that $p(\underline{\epsilon})=\#\,\{ i
 \, \lvert \, \epsilon_i=1\,\}$ is odd.\footnote{This has to be compared with the description of Bol's subwebs of 
$\boldsymbol{\mathcal W}_{  \hspace{-0.05cm} {\rm dP}_4}$ given in 
\cite[\S4.3]{PirioAFST}.} Each such subweb $\boldsymbol{\mathcal W}^{\underline{\epsilon}}$ actually  has maximal rank 1, with 
$\boldsymbol{AR}^{2}\big( \boldsymbol{\mathcal W}^{\underline{\epsilon}}\big)$
spanned by 
a
 2-AR 
${\bf LogAR}^{\underline{\epsilon}}$ which is complete, irreducible and logarithmic, in the sense  that the $\psi_i^{\epsilon_i}$-th component of 
 ${\bf LogAR}^{\underline{\epsilon}}$
belongs to $d{\rm Log}\mathbf C(\psi_i^{\epsilon_i})\wedge d{\rm Log}\mathbf C(\psi_i^{\epsilon_i})$  for every $i=1,\ldots,5$. 
\mk
\item[3.]  The ${\bf LogAR}^{\underline{\epsilon}}$'s for all odd 5-tuples 
$\underline{\epsilon}$'s span 
a vector space denoted by 
$\boldsymbol{AR}^{2}_C\Big( 
 \boldsymbol{\mathcal W}^{GM}_{ \hspace{-0.05cm}
\boldsymbol{\mathcal Y}_5}\Big)$ and called the space of `combinatorial ARs' of 
$\boldsymbol{\mathcal W}^{GM}_{ \hspace{-0.05cm}
\boldsymbol{\mathcal Y}_5}$.  Moreover, one has 
$$
{\rm rk}^{2}_C
\Big( 
 \boldsymbol{\mathcal W}^{GM}_{ \hspace{-0.05cm}
\boldsymbol{\mathcal Y}_5}\Big)=\dim\, 
\boldsymbol{AR}^{2}_C\Big( 
 \boldsymbol{\mathcal W}^{GM}_{ \hspace{-0.05cm}
\boldsymbol{\mathcal Y}_5}\Big)
=\rho^{2}\Big(
\boldsymbol{\mathcal W}^{GM}_{ \hspace{-0.05cm}
\boldsymbol{\mathcal Y}_5}
\Big)-1=10
\,.
$$
\item[4.] There exists a 2-AR of $\boldsymbol{\mathcal W}^{GM}_{ \hspace{-0.05cm}
\boldsymbol{\mathcal Y}_5}$, denoted by ${\bf HLOG}_{ 
\boldsymbol{\mathcal Y}_5}$, which is 
complete, irreducible and 
 whose components are `dilogarithmic' in the sense that  
  for any 
 first integral $\psi_i^\epsilon$ of $\boldsymbol{\mathcal W}^{GM}_{ \hspace{-0.05cm}
\boldsymbol{\mathcal Y}_5}$, 
 the $\psi_i^\epsilon$-th component of 
 ${\bf HLOG}_{ \boldsymbol{\mathcal Y}_5 }$
  belongs to 
 ${\rm Log}\mathbf C(\psi_i^\epsilon)\,d{\rm Log}\mathbf C(\psi_i^\epsilon)\wedge  d{\rm Log}\mathbf C(\psi_i^\epsilon)$.\sk 
 
  Moreover, ${\bf HLOG}_{ \boldsymbol{\mathcal Y}_5 }$ is unique up to multiplication 
by a non-zero scalar and   
  one has 
\begin{equation}
\label{Eq:Decomp}
\boldsymbol{AR}^{2}\Big( 
 \boldsymbol{\mathcal W}^{GM}_{ \hspace{-0.05cm}
\boldsymbol{\mathcal Y}_5}\Big)=
\boldsymbol{AR}^{2}_C\Big( 
 \boldsymbol{\mathcal W}^{GM}_{ \hspace{-0.05cm}
\boldsymbol{\mathcal Y}_5}\Big)\oplus 
\Big\langle\,
{\bf HLOG}_{ \boldsymbol{\mathcal Y}_5 }
\, 
\Big\rangle\, , 
\end{equation}
which implies $
{\rm rk}^{2}
\Big( 
 \boldsymbol{\mathcal W}^{GM}_{ \hspace{-0.05cm}
\boldsymbol{\mathcal Y}_5}\Big)
=\rho^{2}\Big(
\boldsymbol{\mathcal W}^{GM}_{ \hspace{-0.05cm}
\boldsymbol{\mathcal Y}_5}
\Big)=11
$: the web $\boldsymbol{\mathcal W}^{GM}_{ \hspace{-0.05cm}
\boldsymbol{\mathcal Y}_5}$  has maximal 2-rank.\footnote{It would be more rigorous to  state this as {\it `the 2-rank of $\boldsymbol{\mathcal W}^{GM}_{ \hspace{-0.05cm}
\boldsymbol{\mathcal Y}_5}$ is AMP'}, using the terminology introduced in \cite[\S1.3.5]{ClusterWebs}.}
\mk 
\item[5.] One can associate a divisor $\boldsymbol{\mathcal D}_w$ of 
$\boldsymbol{\mathcal Y}_5$ to each weight $w$ of the minuscule half-spin representation $S_5^+$. Then  the 2-ARs of $\boldsymbol{\mathcal W}^{GM}_{ \hspace{-0.05cm}
\boldsymbol{\mathcal Y}_5}$ are regular on 
 the complement in $\boldsymbol{\mathcal Y}_5$ of the  union of all the $\boldsymbol{\mathcal D}_w$'s which coincides with $\boldsymbol{\mathcal Y}_5^\times $. Moreover, the sixteen ARs ${\bf LogAR}^{\underline{\epsilon}}$'s are exactly the logarithmic ARs obtained by considering the residues 
of ${\bf HLOG}_{ \boldsymbol{\mathcal Y}_5 }$ along the $\boldsymbol{\mathcal D}_w$'s: for any 
odd 5-tuple $\underline{\epsilon}$, there exists a uniquely defined weight $w(\underline{\epsilon})$ such that,  up to multiplication by a non-zero constant, one has 
$${\rm Res}_{ \boldsymbol{\mathcal D}_{w(\underline{\epsilon})} }\Big( 
{\bf HLOG}_{ \boldsymbol{\mathcal Y}_5 }
\Big) = {\bf LogAR}^{\underline{\epsilon}}\, .$$
\item[6.] The action of the Weyl group $W_{D_5}$ 
by automorphisms on $\boldsymbol{\mathcal Y}_5$ ({\it cf.}\,\cite[Theorem 2.2]{Skorobogatov}{\rm )} gives rise to a linear 
$W_{D_5}$-action  on $
\boldsymbol{AR}^{2}\big( 
 \boldsymbol{\mathcal W}^{GM}_{ \hspace{-0.05cm}
\boldsymbol{\mathcal Y}_5}\big)$ which lets invariant the decomposition in direct sum \eqref{Eq:Decomp}. This decomposition actually is the one into $W_{D_5}$-irreducibles: the irrep associated to the 1-dimensional component 
$\big\langle\,
{\bf HLOG}_{ \boldsymbol{\mathcal Y}_5 }
\, 
\big\rangle$ is the signature representation whereas $\boldsymbol{AR}^{2}_C\Big( 
 \boldsymbol{\mathcal W}^{GM}_{ \hspace{-0.05cm}
\boldsymbol{\mathcal Y}_5}\Big)$ is the $W_{D_5}$-irreducible module $V^{10}_{[11,111]}$.\footnote{Seeing it as  a decomposition in $W_{D_5}$-irreducibles,  
\eqref{Eq:Decomp},
must be compared with some results given in \S\ref{SS:WGM-+-Bol's-Web}: in some sense which could be made precise (cf. Proposition \ref{Prop:W+-B}), 
$\boldsymbol{AR}^{2}_C\big( 
 \boldsymbol{\mathcal W}^{GM}_{ \hspace{-0.05cm}
\boldsymbol{\mathcal Y}_5}\big)$ corresponds to ${\bf HLogAR}^2_{\rm asym}$ 
and ${\bf HLOG}^2$ to ${\bf HLog}^3$.}
\mk
\item[7.] For any del Pezzo surface ${\rm dP}_4$, using Serganova-Skorobogatov embedding $f_{S\hspace{-0.05cm} S}: {\rm dP}_4 \hookrightarrow \boldsymbol{\mathcal Y}_5$ (cf.\,\eqref{Eq:Embedding-SS} above), 
the weight 3 hyperlogarithmic abelian relation ${\bf HLog}^3_{  \hspace{0.05cm} {\rm dP}_4}$ 
of $\boldsymbol{\mathcal W}_{  \hspace{-0.05cm} {\rm dP}_4}$ 
(resp.\,the sixteen elements of ${\bf HLogAR}^2_{\rm asym}$
 equivalent to ${\bf HLog}^2\simeq  \boldsymbol{\mathcal Ab}$)  
can be obtained in a natural way from 
${\bf HLOG}_{ \boldsymbol{\mathcal Y}_5 } $ (resp.\,from the sixteen logarithmic 2-abelian relations ${\bf LogAR}^{\underline{\epsilon}}\in \boldsymbol{AR}^{2}_C\big( 
 \boldsymbol{\mathcal W}^{GM}_{ \hspace{-0.05cm}
\boldsymbol{\mathcal Y}_5}\big)${\rm )} by means of residues. 
\end{enumerate}
\end{thm}

An interesting feature of $\boldsymbol{\mathcal W}_{  \hspace{-0.05cm} {\rm dP}_5}$ is that it is a cluster web. More precisely, it admits a birational model which admits as first integrals the $\mathcal X$-cluster variables of type $A_2$. In \cite{PirioAFST}, we established that each del Pezzo web  $\boldsymbol{\mathcal W}_{  \hspace{-0.05cm} {\rm dP}_4}$ can also be defined by 
$\mathcal X$-cluster variables (of type $D_4$). It is natural to wonder whether Gelfand-MacPherson's web  $\boldsymbol{\mathcal W}^{GM}_{ \hspace{-0.05cm}
\boldsymbol{\mathcal Y}_5}$ is cluster as well.  Unfortunately we do not have a complete answer for this web yet, but we have one for its lift $\boldsymbol{\mathcal W}^{GM}_{ \hspace{-0.05cm}
{\mathbb S}_5}$.  
In \cite{Ducat},  Ducat gave the construction of a generalized cluster structure on $\mathbb S_5$ which is not a classical cluster one, but enjoys the nice property of being of finite type. 
We prove the following result: 
\begin{prop}
Gelfand-MacPherson's web 
$\boldsymbol{\mathcal W}^{GM}_{ \hspace{-0.05cm}
{\mathbb S}_5}$ is cluster with respect to Ducat's generalized cluster structure on $\mathbb S_5$.
\end{prop}

%
%
%

The similarities between the statements above to the corresponding ones  for Bol's web $\boldsymbol{\mathcal B}\simeq \boldsymbol{\mathcal W}_{  \hspace{-0.05cm} {\rm dP}_5}$ in \cite[\S1.1]{PirioAFST}
are even more striking than those  in \cite[\S1.2]{PirioAFST} about $\boldsymbol{\mathcal W}_{  \hspace{-0.05cm} {\rm dP}_4}$.  For this reason, and also because the theorem above can be generalized to all the Gelfand-MacPherson webs 
$ \boldsymbol{\mathcal W}^{GM}_{ \hspace{-0.05cm}
\boldsymbol{\mathcal Y}_r}$ for $r=4,\ldots,8$ (see 
 \S\ref{SS:r=6,7} further), 
we believe that these latter webs are those which must really be considered as the most direct/fundamental generalizations of $\boldsymbol{\mathcal B}\simeq
\boldsymbol{\mathcal W}_{  \hspace{-0.05cm} {\rm dP}_5}\simeq 
 \boldsymbol{\mathcal W}^{GM}_{ \hspace{-0.05cm}
\boldsymbol{\mathcal Y}_4}$, and not really the del Pezzo's webs $\boldsymbol{\mathcal W}_{  \hspace{-0.05cm} {\rm dP}_{9-r}}$ which actually are 2-dimensional slices of the corresponding Gelfand-MacPherson webs.  All these considerations 
make it natural to ask the following 
\begin{question}
\label{Question:DilogAR-a-la-GM}
Can the 2-abelian relation ${\bf HLOG}_{ \boldsymbol{\mathcal Y}_5 }$ 
of 
 $ \boldsymbol{\mathcal W}^{GM}_{ \hspace{-0.05cm}
\boldsymbol{\mathcal Y}_5}$
be obtained following the geometric approach of  Gelfand-MacPherson,  that is 
by integrating an invariant differential form $\Omega_P$ on a real form ${\bf S}_5$ of the spinor tenfold $\mathbb S_5$ which represents a certain characteristic class $P\in 
{\bf H}^*\big( {\bf S}_5,\mathbf R\big)$ along the orbits of the action on ${\bf S}_5$ of the Cartan torus 
of a split real form of 
 the 
  group 
 ${\rm Spin}(\mathbf C^{10})$?  
\end{question}

This question is the subject of ongoing research by the author at the time of writing.
\sk 

The rest of the article is organised as follows:  Section 2 consists in some preliminaries. After recalling some elements of web geometry in \S\ref{SS:Elements-web-geometry}, we succinctly review  the notion 
of Gelfand-Mapherson web in \S\ref{SS:Gelfand-MacPherson webs}. Then we  discuss the geometry around the spinor tenfold in \S\ref{SS:around-S5}. The third section is devoted to the study of  $ \boldsymbol{\mathcal W}^{GM}_{ \hspace{-0.05cm}
\boldsymbol{\mathcal Y}_5}$ especially from the point of view of its ranks and abelian relations.  First, in \S\ref{SS:WGMS5-in-coordinates}, 
working with some adapted rational coordinates previously introduced, 
we give a list of explicit rational first integrals $U_i$ (see \eqref{Eq:Formules-Ui}) for a birational model of $ \boldsymbol{\mathcal W}^{GM}_{ \hspace{-0.05cm}\boldsymbol{\mathcal Y}_5}$, denoted by $ \boldsymbol{\mathcal W}^{GM}_{ \boldsymbol{Y}_5}$. 
After having computed the virtual and ordinary ranks of $ \boldsymbol{\mathcal W}^{GM}_{ \boldsymbol{Y}_5}$, we start to 
study the abelian relations of this web in the following subsections. 
The most important of the ARs is the 2-AR with logarithmic coefficient ${\bf HLOG}_{ \boldsymbol{Y}_5}$ which corresponds to the differential identity denoted the same which is given in Proposition \ref{Eq:AR-HLOG-Y5}.  Then the space of abelian relations $\boldsymbol{AR}^k\big( \boldsymbol{\mathcal W}^{GM}_{ \boldsymbol{Y}_5} \big)$ for $k=2,1,0$ 
are successively studied in the subsections 
\S\ref{SS:2-AR-of-WGMY5}, 
\S\ref{SS:1-AR-of-WGMY5} and 
\S\ref{SS:0-AR-of-WGMY5} respectively. In particular, we describe the structures of these spaces relatively to the action of the Weyl group $W_{D_5}$ and indicate how these spaces (or rather some subspaces of them) are related with respect to the total derivatives or to taking residues of ARs. Many of the results concerning the ARs of $ \boldsymbol{\mathcal W}^{GM}_{ \boldsymbol{Y}_5} $ are brought together succinctly in the tables page \ref{Page:tables}. 

 In Section 4, we investigate the cluster nature of Gelfand--MacPherson's web $ \boldsymbol{\mathcal W}^{GM}_{ {\mathbb S}_5} $ on the spinor tenfold ${\mathbb S}_5$. We show that, in suitable coordinates, the components of the face maps defining this web are, up to minor modifications, given by the cluster variables of the finite-type LPA structure on ${\mathbb S}_5$, constructed by Ducat in \cite{Ducat}.

Section~5 is devoted to the study of a particular $5$-subweb $ \boldsymbol{W}^{+}_{ \boldsymbol{Y}_5} $ of the Gelfand--MacPherson web $ \boldsymbol{\mathcal W}^{GM}_{ \boldsymbol{Y}_5} $, defined by first integrals with simple monomial components. We investigate its $k$-abelian relations for $k = 0,1,2$. In particular, we show that 
$ \boldsymbol{W}^{+}_{ \boldsymbol{Y}_5} $ carries a distinguished logarithmic $1$-abelian relation ${\bf AR}^1_\delta$, from which Abel's classical five-term identity for the dilogarithm can be recovered in a direct and natural way.

One of the main results of this paper, established in Section~6, is that for any smooth del Pezzo quartic surface ${\rm dP}_4$, the hyperlogarithmic weight-$3$ identity  ${\bf HLog}_{{\rm dP}_4}$  can be deduced from the $2$-abelian relation ${\bf HLOG}_{\boldsymbol{\mathcal Y}_5}$  of $ \boldsymbol{\mathcal W}^{GM}_{ \boldsymbol{\mathcal Y}_5} $. This is first shown at the symbolic level in \S\ref{SS:Arguing-symbolically}, and then concretely by manipulating explicit abelian relations in \S\ref{SS:manipulating-ARs}.

Section~7 explains how most of the explicit and computational methods developed throughout this paper for the Gelfand--MacPherson web 
$ \boldsymbol{\mathcal W}^{GM}_{ \boldsymbol{\mathcal Y}_5} $
 can be extended to the entire family of webs $ \boldsymbol{\mathcal W}^{GM}_{ \boldsymbol{\mathcal Y}_r} $ for $r = 4, 5, 6, 7$. We first describe how to make these webs explicit in the cases $r = 6,7$ 
  (see \S\ref{S:7-1}), 
 and then show that the main results regarding the top-degree abelian relations of 
$ \boldsymbol{\mathcal W}^{GM}_{ \boldsymbol{\mathcal Y}_5} $ 
 extend naturally to the webs $ \boldsymbol{\mathcal W}^{GM}_{ \boldsymbol{\mathcal Y}_r} $ for $r = 6,7$ (see Theorem \ref{THM:Main-r=4,5,6,7}). In particular, we prove that each web $ \boldsymbol{\mathcal W}^{GM}_{ \boldsymbol{\mathcal Y}_r} $ carries an essentially unique `master' $(r{-}3)$-abelian relation ${\bf HLOG}_{\boldsymbol{\mathcal Y}_r }$, from which all other $(r-3)$-abelian relations of the web under consideration can be recovered via residue or monodromy.

The final section, Section~8, outlines several perspectives and open questions inspired by our results. In \S\ref{SS:Comp}, we reflect on the many striking properties shared by the webs $ \boldsymbol{\mathcal W}_{ {\rm  dP}_4} 
\simeq 
\boldsymbol{\mathcal W}^{GM}_{ \boldsymbol{\mathcal Y}_4} $  and 
$ \boldsymbol{\mathcal W}^{GM}_{ \boldsymbol{\mathcal Y}_5} $. 
We conclude in \S\ref{SS:HLOG-Yr-in-terms-scattering-diagram}  with a speculative discussion on a possible interpretation of the differential identity ${\bf HLOG}_{\boldsymbol{\mathcal Y}_5}$ as the manifestation of an as-yet-undetermined property of a scattering diagram conjecturally associated with $\boldsymbol{\mathcal Y}_5$.


\section{\bf Preliminaries: elements of web geometry and birational geometry around 
$\boldsymbol{{\mathbb S}_5}$}
\label{S:Prelim}
We start be recalling/introducing general notions of web geometry in a generalized setup (namely, for `generalized webs'). Then we quickly review the notion of `Gelfand-MacPherson web' introduced in \cite{PirioAFST} in \S\ref{SS:Gelfand-MacPherson webs} before discussing  in \S\ref{SS:around-S5} several properties of the most important space considered in this paper, namely 
  the spinor tenfold $\mathbb S_5$.

\subsection{\bf Elements of web geometry}
\label{SS:Elements-web-geometry}
We introduce here basic notions of web geometry. The webs we are considering here are quite general, and in particular more general than the webs encountered in the classical literature on the subject.  For a more detailed overview of web geometry as we need it here, see \cite[\S1]{ClusterWebs}.
\sk 

Let $M$ be an irreducible analytic (hence possibly singular) variety. In this text, a {\it `$d$-web'} on $M$ (for a positive integer $d$) is a finite collection $\boldsymbol{\mathcal W}=\big( \mathcal F_1,\ldots, \mathcal F_d\big)$ of $d$ pairwise distinct foliations on $M$, all of the same codimension.\footnote{Classically, one requires that the tangent spaces of the leaves of a web satisfy a certain general position assumption (everywhere or only at the generic point of $M$), but it is not the case here.} We will also assume that for $m\in M$ generic, then the tangent spaces of the foliations of $\boldsymbol{\mathcal W}$ at $m$ span the whole tangent space of $M$ at this point, ie. $T_m M= 
\big\langle \, T_m \mathcal F_i \, \lvert \, {i \in [\hspace{-0.05cm}[ d
]\hspace{-0.05cm}]}
\, \big\rangle
$, and that for $i,j$ distinct, $T_m \mathcal F_i$ and $T_m \mathcal F_j$ intersect transversally in $T_m M$. When these two conditions are fulfilled, $m$ is said to be a {\it `smooth or a regular point'} of $\boldsymbol{\mathcal W}$. 

 Another $d$-web $\boldsymbol{\mathcal W}'=\big( \mathcal F_i'\big)_ {i \in [\hspace{-0.05cm}[ d
]\hspace{-0.05cm}]}$ defined on another manifold  $M'$ is  {\it `equivalent'} to $\boldsymbol{\mathcal W}$ if there exist $m\in M$ and $m'\in M'$, regular points for 
$\boldsymbol{\mathcal W}$ and $\boldsymbol{\mathcal W}'$ respectively, as well as a local biholomorphism $\varphi: (M,m)\rightarrow (M',m')$ such that $\varphi^*\big(
 \boldsymbol{\mathcal W}'\big)=\big(\varphi^* (\mathcal F_i')\big)_ {i \in [\hspace{-0.05cm}[ d
]\hspace{-0.05cm}]}$ coincides with the germ of $\boldsymbol{\mathcal W}$ at $m$, possibly up to reindexing the foliations $\mathcal F_i$ of this web.  
One writes $\boldsymbol{\mathcal W}\simeq \boldsymbol{\mathcal W}'$ when this is occuring.

For simplicity, assume that $M$ is a connected open subset in $\mathbf C^n$ and that each $\mathcal F_i$ is defined by a global holomorphic submersion $U_i\rightarrow \mathbf C^c$, where $c\in \{1,\ldots,n-1\}$ stands for the codimension of the web. For any $
i \in [\hspace{-0.05cm}[ d
]\hspace{-0.05cm}]$, any 
$k\leq c$ and any subset $I=\{i_1,\ldots,i_k\}$
with $1\leq i_1<i_2<\ldots < i_k\leq c$, one sets 
$\Omega_{i,I}^k=dU_{i,i_1}\wedge \ldots \wedge dU_{i,i_k}$. Then one defines 
$\Omega^k_{\mathcal F_i}$ as the locally-free sheaf of $\mathcal O_M$-modules  freely spanned   by the $\Omega_{i,I}^k$'s,  for all  $I\subset \{1,\ldots,c\}$ of cardinality $k$. Given an open subset $O\subset M$ and $k
\in \{0,\ldots,c\}$, one defines a {\it $k$-abelian relation (ab.\,$k$-ARs) of $\boldsymbol{\mathcal W}$
 on $O$} as a $d$-tuple $(\eta_i)_{i=1}^d \in \prod_{i=1}^d \Omega^k_{\mathcal F_i}$ such that $\sum_{i=1}^d \eta_i=0$ in $\Omega^k(O)$.  The abelian relations form a vector subspace of $\Omega^k(O)^{\oplus d}$ and letting $O$ vary among the open subsets of $M$, one defines a local system which we will denote by 
 $\boldsymbol{AR}^k\big(\boldsymbol{\mathcal W} \big)$ or just 
$\boldsymbol{AR}^k$ whenever there is no ambiguity about the web under consideration.  When working on a fixed domain $D$, we will allow ourselves to use the same notation $\boldsymbol{AR}^k\big(\boldsymbol{\mathcal W} \big)$ for the
vector space of $k$-ARs on $D$, which, rigorously, should be denoted by 
$\boldsymbol{AR}^k\big(\boldsymbol{\mathcal W} \big)(D)$. 

By definition, the $k$-rank $r^k\big( \boldsymbol{\mathcal W} \big)$ is the rank of 
$\boldsymbol{AR}^k\big(\boldsymbol{\mathcal W} \big)$. It is an element of $\overline{\mathbf N}=\mathbf N\cup \{ \infty\}$ which is invariantly attached to $\boldsymbol{\mathcal W}$: two equivalent webs have the same $k$-rank, for any $k$ less than to equal to the codimension of the webs. Given a point $m\in M$ smooth for $\boldsymbol{\mathcal W}$, let $T_m \boldsymbol{\mathcal W}$ be the constant web on $T_m M\simeq \mathbf C^n$ with first integrals the linear maps 
$dU_i(m): T_m M\rightarrow T_{U_i(m)} \mathbf C^c\simeq \mathbf C^c$. 
For $k\in \{0,\ldots,c\}$ and any $\sigma\geq 0$ (with $\sigma>0$ if $k=0$), let $\boldsymbol{AR}^k_\sigma(T_m \boldsymbol{\mathcal W})$ be the space of $k$-ARs $(\eta_i)_{i=1}^d$ 
of $T_m \boldsymbol{\mathcal W}$ with each $\eta_i$ being a $k$-differential form 
on $T_m M$ with coefficients in ${\rm Sym}^\sigma\big( T_m^* M\big)$ (that is, polynomials of degree $\sigma$ on $T_m M$).  Then one defines the {\it `$\sigma$-th virtual $k$-rank of $\boldsymbol{\mathcal W}$ at $m$'}, denoted by 
$\rho^k_{m,\sigma}\big(\boldsymbol{\mathcal W}\big)$, as the dimension of 
$\boldsymbol{AR}^k_\sigma(T_m \boldsymbol{\mathcal W})$.  
One verifies easily that the dimension of the space of germs of $k$-ARs at $m$ is less than or equal to the {\it `total virtual $k$-rank at $m$'} $\rho^k_m\big( 
\boldsymbol{\mathcal W}\big):= \sum_{\sigma} \rho^k_{m,\sigma}\big(\boldsymbol{\mathcal W}\big)$.  We thus define the {\it `$\sigma$-th (resp. the total) virtual $k$-rank'} 
 $\rho^k_{\sigma}\big(\boldsymbol{\mathcal W}\big)$ resp. 
$\rho^k\big(  \boldsymbol{\mathcal W}\big)$) of the considered web as the value 
$\rho^k_{m,\sigma}\big(\boldsymbol{\mathcal W}\big)$ (resp.  $\rho^k_m\big( 
\boldsymbol{\mathcal W}\big)$) for $m$ generic. 

The web $\boldsymbol{\mathcal W}$ is said to have {\it `as maximal as possible'} (ab.\,{\it AMP}) $k$-rank when $r^k\big( \boldsymbol{\mathcal W} \big)=\rho^k\big( \boldsymbol{\mathcal W} \big)$, with this integer being positive and finite.  In this case, we will also say that $\boldsymbol{\mathcal W}$ is {\it `$k$-AMP'}. Being ($k$-)AMP has to be view as a strong feature of a web, a wide generalization of the classical notion in web geometry of being of `maximal rank' (for more perspective on this notion, see \cite[\S1.3]{ClusterWebs}). 
\mk 

For any nonnegative $k$ such that $k\leq c$, the total derivative gives rise to a linear map 
\begin{align}
\label{Eq:derivative-d1}
d^k:  
\boldsymbol{AR}^k\big(\boldsymbol{\mathcal W}\big) & \longrightarrow 
\boldsymbol{AR}^{k+1}\big(\boldsymbol{\mathcal W}\big)\, . \\
 \big( \eta_i \big)_{i=1}^{d}
& \longmapsto  \big( d\eta_i \big)_{i=1}^{d}
\nonumber 
\end{align}
Accordingly to the usual terminology, abelian relations in the kernel of $d^k$ will be said to be {\it `closed'}, and those in its image, {\it `exact ARs'}.

Of course, one always have an inclusion $d^k\big( 
\boldsymbol{AR}^k\big(\boldsymbol{\mathcal W}\big)
\big) \subset {\rm Ker}\big(d^{k+1}\big)$ in 
$\boldsymbol{AR}^{k+1}\big(\boldsymbol{\mathcal W}\big)$
 but, even if one is considering the ARs locally, on a simply connected domain or even at the level of germs,  in general it is not true that this inclusion is an equality. Indeed, let 
 $\boldsymbol{\chi}=(\chi_i)_{i=1}^d$ be a germ of $(k+1)$-AR for $\boldsymbol{\mathcal W}$, at a smooth point $m\in M$ of this web. 
  This means that $\chi_i\in \Omega^{k+1}(M,m)$ is $\mathcal F_i$ basic for each ${i \in [\hspace{-0.05cm}[ d
]\hspace{-0.05cm}]}$ and that $\sum_{i=1}^d \chi_i=0$. Let us assume moreover that $\boldsymbol{\chi}$ is closed, i.e. one has $d\chi_i=0$ for any $i$.  Hence any $\chi_i$ is locally exact thus one may first think that it is possible to integrate $\boldsymbol{\chi}$ and construct a (germ of) $k$-AR $\boldsymbol{\eta}=(\eta_i)_{i=1}^d$ such that $d^k\big( \boldsymbol{\eta}\big)=\boldsymbol{\chi}$.  

But this is not always possible. Indeed, the $\eta_i$'s, in addition to being primitives of the corresponding $\chi_i$'s, must satisfy two additional conditions: (1) each $\eta_i$ must be $\mathcal F_i$-basic, and (2) $\sum_{i=1}^d \eta_i=0$ in $\Omega^k(M,m)$. These two conditions cannot always be simultaneously fulfilled. In particular, the 2-abelian relation ${\bf HLOG}_{ \boldsymbol{\mathcal Y}_5}$ considered in this paper, although closed, is not exact as an abelian relation (see the decompositions in direct sums 
 \eqref{Eq:Decomp-AR2-WGM}
 and 
 \eqref{Eq:Decomp-AR1-WGM} 
 in 
Coro/Theorem(?)  \ref{Cor:2-RA-WGM} and Theorem \ref{Thm:1-AR-} respectively).\footnote{Another place to look is the last line of Table 1.} 
%
%

\subsection{\bf Gelfand-MacPherson webs}
\label{SS:Gelfand-MacPherson webs}
We review very quickly some  material introduced in \cite[\S4.5.1]{PirioAFST} to which we refer for further details. \sk

Let $G$ be a simple complex Lie group of Dynkin type $D$, and let $H\subset G$ be a Cartan torus, with associated set of simple roots $\Phi\subset \boldsymbol{\mathfrak h}^*_{\mathbf R}$, where $ \boldsymbol{\mathfrak h}_{\mathbf R}$ stands for the real part of the Lie algebra of $H$. Let $P$ be a standard parabolic subgroup, assumed to be maximal (to simplify), and let $\omega_P$ be the vertex of $D$ corresponding to it.
If $V=V_P$ stands for the $G$-representation of highest weight $\omega_P$, there exists a `highest-weight vector' $v_P\in V$ wich is of weight $\omega_P$ and such that $P$ coincides with the stabilizer of the line $[v_P]\in \mathbf P(V)$ and the orbit $X=G\cdot 
[v_P]$ is closed in $ \mathbf P(V)$. It follows that $X$ is a homogeneous algebraic submanifold of $ \mathbf P(V)$ which naturally identifies with $G/P$. Let $\mathfrak W_P$ be the set of weights of the considered representation, namely the set of 
weights  $\mathfrak w \in \boldsymbol{\mathfrak h}^*_{\mathbf R}$ 
for which the corresponding weight subspace $V_{\mathfrak w}\subset V$ is non-trivial. 
The `weight polytope' $\Delta=\Delta_{D,\omega_P}$ of the representation $V$ is the convex envelope of $\mathfrak W_P$ in $\boldsymbol{\mathfrak h}^*_{\mathbf R}$. The Weyl group $W_D=N_G(H)/H$ acts on  $\boldsymbol{\mathfrak h}_{\mathbf R}^*$ and the set of vertices of $\Delta$ can be proved to coincide with the Weyl orbit $W\!\cdot \!\omega_P$ of the highest weight.   \sk 

To a  facet $F$  of $\Delta$ (that is, a face of codimension 1), one can associate a type $(D_F,\omega_F)$ which is a marked Dynkin diagram obtained from $(D,\omega_P)$ by removing an extremal vertex.  Then $V$ can be seen as a representation
 for the  simple complex Lie group of type $D_F$, noted by $G_F$, and there exists a decomposition of
 $G_F$-subrepresentations $V=V^F\oplus V_F$ where $V_F$  is the sub-representation with highest weight $\omega_F$. To the facet $ F$ is associated a one-parameter subgroup $\mathcal H_F$ of $H$ and denoting by $H_F$ the Cartan torus of $G_F$,  there is a surjective map of Cartan tori $H\rightarrow H_F$  with kernel $\mathcal H_F$ which induces an identification $
 H/_{\mathcal H_F}\stackrel{\sim}{\rightarrow} H_F$
 Then if $P_F$ stands for the maximal standard parabolic subgroup 
of $G_F$ associated to $\omega_F$, then the two following facts occur: 
\begin{itemize}
\item[$(i).$]
setting $X_F=G_F/P_F\subset \mathbf PV_F$, one has $X_F=X\cap \mathbf PV_F$; \sk
\item[$(ii).$]
the linear  projection  from $\mathbf PV^F$ onto $\mathbf PV_F$ gives rise to a surjective rational map $\Pi_F : X\dashrightarrow X_F$. Moreover, this map is equivariant with respect to the epimorphism of tori $H\rightarrow H_F$.
\end{itemize}

Now we restrict the discussion to some particularly nice cases\footnote{See \cite{Skorobogatov}, and in particular the list of cases to be excluded given in Proposition 2.1 therein.} which contain those of the minuscule homogeneous spaces hence the specific case we will be interested in.  In \cite{Skorobogatov} (see also \cite{SerganovaSkorobogatov}), the author(s) consider a Zariski-open subset $X^{sf}\subset X$ on which $H$ acts nicely, and such that the quotient $\boldsymbol{\mathcal Y}=X^{sf}/H$ is a rational quasi-projective variety.  From the point $(ii).$ above, it follows that $\Pi_F$ descends to a rational map $\pi_F: \boldsymbol{\mathcal Y} \dashrightarrow X_F/H_F$ where the target space has to be understood as a quotient 
with respect to the action of $H_F$ on $X_F$ viewed as a rational action.  A facet is said to be {\it `${\mathcal W}$-relevant'} whenever $X_F/H_F$ has positive dimension. We define the {\it `Gelfand-MacPherson webs'} $\boldsymbol{\mathcal{W}}^{GM}_{\hspace{-0.05cm}X}$ and $\boldsymbol{\mathcal{W}}^{GM}_{\hspace{-0.05cm}\boldsymbol{\mathcal Y}}$ as the webs  respectively defined by the face maps $\Pi_F$ and $\pi_F$, for all ${\mathcal W}$-relevant facets $F$ of the weight polytope $\Delta_{G,P}$. From the equivariance property stated in $(ii).$ above, it follows that $\boldsymbol{\mathcal{W}}^{GM}_{\hspace{-0.05cm}X}$ is $H$-equivariant and that its quotien by $H$ is precisely $\boldsymbol{\mathcal{W}}^{GM}_{\hspace{-0.05cm}\boldsymbol{\mathcal Y}}$. 
\sk

The most basic example to have in mind is the one when $X=G_2(\mathbf C^5)$, which corresponds to the case of type $(A_4,\omega_2)$.\footnote{The best reference on this case is the great paper \cite{GM} by Gelfand and MacPherson.} The moment/weight polytope is the hypersimplex $\Delta_{2,5}=\{ (t_i)_{i=1}^5\in [0,1]^5\, \lvert \, \sum_{i=1}^5 t_i=2\,\}$. This polytope has 10 facets which are obtained by intersecting it with the 
 the affine hyperplanes cut out by $t_i=\tau$ for $i=1,\ldots,5$ and $\tau\in \{0,1\}$. All the facets when $\tau=1$ are 3-simplices, and the associated type is $(A_3,\omega_3)$, hence these facets are not ${\mathcal W}$-relevant. Any facet $F_i=\Delta_{2,5} \cap \{ t_i=0\}$ is a hypersimplex of type  $(A_3,\omega_2)$, with $X_{F_i}\simeq G_2(\mathbf C^4)$. 
Denoting by $(e_i)_{i=1}^5 $ the canonical basis of $\mathbf C^5$, 
 the face map $\Pi_{F_i}$ associated to $F_i$ identifies with the rational map
 $G_2(\mathbf C^5)\dashrightarrow G_2\big(\mathbf C^5/{\langle e_i\rangle}\big)
 \simeq G_2(\mathbf C^4) $ induced by the 
 linear quotient map $\mathbf C^5 
 \rightarrow 
 \mathbf C^5/{\langle e_i\rangle}$. 
 If one denotes by $H$ and $H_{F_i}$ the Cartan tori of the linear groups acting on 
 $X$ and $X_{F_i}$, then the corresponding GIT quotients respectively are 
 $X/\hspace{-0.06cm} / H= G_2(\mathbf C^5)/\hspace{-0.06cm} / H\simeq \overline{\boldsymbol{\mathcal M}}_{0,5}$ and $X_{F_i}/\hspace{-0.05cm} / H_{F_i} \simeq 
 \overline{\boldsymbol{\mathcal M}}_{0,4}\simeq 
 \mathbf P^1$. And the $H$-equivariant quotient of $\Pi_{F_i}$ is the $i$-th forgetful map  
 $
\pi_{F_i} :  \overline{\boldsymbol{\mathcal M}}_{0,5}\dashrightarrow \overline{\boldsymbol{\mathcal M}}_{0,4} = \mathbf P^1$ (which actually is a morphism). 
 It follows that in this case, Gelfand-MacPherson web on $\boldsymbol{\mathcal Y}=
 \overline{\boldsymbol{\mathcal M}}_{0,5}$ is the 5-web with the five forgetful maps as first integrals, hence is a model of the dilogarithmic Bol's web $\boldsymbol{\mathcal B}$: one has 
 $$\boldsymbol{\mathcal{W}}^{GM}_{\hspace{-0.05cm}\overline{\boldsymbol{\mathcal M}}_{0,5}}
 =\boldsymbol{\mathcal{W}}\Big( \, \pi_{F_i}: 
 \overline{\boldsymbol{\mathcal M}}_{0,5}\longrightarrow \overline{\boldsymbol{\mathcal M}}_{0,4} \simeq \mathbf P^1
 \,\Big)_{i=1}^5 \simeq \boldsymbol{\mathcal B}\, . $$

The main object of study in this paper is Gelfand-MacPherson's web of the quotient, by the rank 5 Cartan torus, of the spinor 10-fold $\mathbb S_5$ which is a homogenous projective variety of type $D_5$.

\subsection{\bf Birational geometry around the spinor 10-fold $\boldsymbol{\mathbb S_5}$}
\label{SS:around-S5}
We recall some facts we will need further in the paper. 
For details, see our previous article \cite{PirioAFST} and the references therein. 

\subsubsection{\bf Groups of Lie type $\boldsymbol{D_5}$}
Let $q\in {\rm Sym}^2(V^\vee)$ be the non-degenerate complex quadratic form on $V=\mathbf C^{10}$ given by 
$$q(x)=\sum_{i=1}^5 x_ix_{i+5}$$
 in the standard coordinate system $(x_i)_{i=1}^{10}$ on $V$. The associated orthogonal group $O(V,q)={O}_{10}(\mathbf C)$ is a simple complex Lie group of type $D_5$. Its universal covering is the so-called {\it `spin group'} ${\rm Spin}_{10}(\mathbf C)$ and the canonical covering 
${\rm Spin}_{10}(\mathbf C)\rightarrow 
{O}_{10}(\mathbf C)$ is known to be  2-to-1.  The corresponding Dynking diagram of type $D_5$ we are going to work with is labeled as follows: 
 \begin{figure}[h!]
\begin{center}
\scalebox{1}{
 \includegraphics{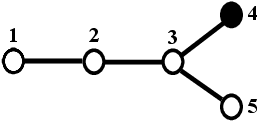}}
 \vspace{0.17cm}
\caption{The marked Dynkin diagram of type $(D_5,\omega_4)$. (The marking of the 
fourth vertice corresponds to the choice of the standard maximal parabolic subgroup 
used to construct the spinor tenfold, see below).}
\label{Fig:Dynkin-diagram-D5}
\end{center}
\end{figure}

Coxeter matrix associated to the Dynkin $D$ diagram of type $D_5$: it is the symmetric $5\times 5$ matrix $(m_{ij})_{i,j=1}^5$  where for all $i,j$ such that $1\leq i\leq j \leq  5$, one has $m_{ii}=1$ and when $i<j$, one has $m_{ij}=2$ if there is no arrow between the $i$-th and  $j$-th vertices of the diagram, and $m_{ij}=3$ otherwise.  
The associated {\it `Weyl group of type $D_5$'} is the group with generators $s_i$ for $i=1,\ldots,5$, with relations $(s_is_j)^{m_{ij}}$ for any $i,j\in [\hspace{-0.05cm}[5 
]\hspace{-0.05cm}]$. In the case under scrutiny, the generators $s_i$'s of $W_{D_5}$ are  involutions, which pairwise commute, i.e. $(s_is_j)^2=1$ except if the $i$-th and $j$-th vertices of $D_5$ are linked by one of its arrows, in which case one has 
$(s_is_j)^3=1$.  To summarise, setting $J=\{ (j,j+1)\, \lvert \, j=1,\ldots,4\,\} \cup \{ (3,5)  \}$, then for all  distinct $i,j\in  [\hspace{-0.05cm}[5 
]\hspace{-0.05cm}]$, one has 
\begin{equation}
\label{Eq:relations-W(D5)}
\begin{tabular}{l}
$-$  $s_i^2=1$; 
\sk\\
$-$ $(s_is_j)^2=1$ (that is $s_i$ and $s_j$ commute) if $(i,j) \not \in J$; \sk 
\\
$-$ $(s_is_j)^3=1$ if the pair $(i,j)$ belongs to $ J$. 
\end{tabular}
\end{equation}
The explicit form of these relations will be used further in \S\ref{SS:Birational-action-w(D5)}. 

\subsubsection{\bf The orthogonal grassmannian $\boldsymbol{OG_5(\mathbf C^{10})}$, the spinor 10-fold $\boldsymbol{\mathbb S}_{5}$
and the  half-spin representations.}
By definition, the {\it `orthogonal grassmannian $OG_5(V)$'} is the subvariety of the standard grassmannian $G_5(V)$ whose points are the 5-dimensional subspaces $\Pi\subset V$ which are `totally isotropic', that is such that $q\lvert_\Pi\equiv 0$. It is known that $OG_5(V)$ is the disjoint union of two indistinguishable isomorphic irreducible components which will be denoted by $OG_5^\epsilon(V)$ with $\epsilon=\pm$. 

The component $OG^+_5(V)$, which will be the privileged one for us, is homogeneous under the natural action of the `{\it spin group}' ${\rm Spin}_{10}(\mathbf C)$. In fact, if $P_4$ stands for the standard maximal parabolic subgroup of the spin group associated to the fourth vertex of the Dynkin diagram of Figure \ref{Fig:Dynkin-diagram-D5}, then 
one defines the  {\it `spinor 10-fold'} as the homogeneous variety 
$\mathbb S_5={\rm Spin}_{10}(\mathbf C)/P_4$ which can be proved to be isomorphic to 
$OG^+_5(V)$. Indeed, the parabolic subgroup $P_4$ being maximal, $\mathbb S_5$ has Picard number 1 with its Picard group spanned by an ample class denoted by $H$. Moreover, 
$\mathbb S_5$ is Fano with $-K_{\mathbb S_5}=h_{D_5} H$ where $h_{D_5}$ stands for the Coxeter number in type $D_5$ (namely $h_{D_5}=8$). Actually, $-K_{\mathbb S_5}$ is very ample and the anticanonical linear system 
gives rise to an equivariant embedding  
$\varphi : \mathbb S_5 \hookrightarrow \mathbf P(S^+)$ where $S^+\simeq {\bf H}^0\big( 
\mathbb S_5, -K_{\mathbb S_5}\big)^\vee$ is a 16-dimensional ${\rm Spin}_{10}(\mathbf C)$-representation, one of the so-called {\it half-spin representations $S^{\pm}$}.\footnote{The two half-spin representations $S^+$ and $S^-$ are isomorphic, but not in a canonical way.} Then it can be proved that the second Veronese embedding of $
{\mathbb S_5}\subset \mathbf P(S^+)\simeq \mathbf P^{15}$ coincides with the image of $OG_5^+(\mathbf C^5)$ in $\mathbf P(\wedge^5 V)$ obtained by taking the image of its inclusion into the ordinary grassmannian of 5-planes in $V$ 
 post-composed with the standard Pl\"ucker embedding $G_5(V)\hookrightarrow \mathbf P(\wedge^5 V)$. 
\sk 

The half-spin representations are know to be minuscule representations:  for $\epsilon =\pm$, 
the Weyl group $W_{D_5}$ acts transitively on the set of weights $\mathfrak W^\epsilon$ which is of cardinality 16. Each weight space is a complex line and $S^\epsilon$ admits a basis $(v_\varpi)_{\varpi \in \mathfrak W^\epsilon}$ with $v_\varpi$ being a weight vector of weight $\varpi$ for any $\varpi\in \mathfrak W^\epsilon$.  In other terms, there is a decomposition as a direct sum in which each weight subspace is of dimension 1:
\begin{equation}
\label{Eq:S+-direct-sum}
S^\epsilon=\oplus_{\varpi \in \mathfrak W^\epsilon}  \mathbf C v_\varpi
\, .
\end{equation}

The set of weights in $\mathfrak W^\epsilon$ can be described quite explicitly as the set 
of 5-tuples $\underline{\varepsilon}/2$ for all $\underline{\varepsilon}=(\varepsilon_i)_{i=1}^{5}\in \{ \pm 1\}^5$ whose parity $p(\underline{\varepsilon})=\varepsilon_1\cdots
\varepsilon_5\in \{\pm1\}$ coincides with $\epsilon$.  More explicitly, denoting by $(e_1,\ldots,e_5)$ the standard basis of $\mathbf R^5\simeq \mathfrak h_{D_5}^{\mathbf R}$
 and setting $e_K=\sum_{k \in K} e_k$ for any $K\subset  [\hspace{-0,05cm}[5]\hspace{-0,05cm}]$, the elements of $\mathfrak W^+$ (resp.\,of $\mathfrak W^-$) are the vectors 
\begin{equation}
\label{Eq:weight-vectors}
\mathfrak w_L=\frac12 \big(  e_{[\hspace{-0,05cm}[5]\hspace{-0,05cm}]\setminus L}-e_L\big) 
\end{equation}
for all subsets $L\subset  [\hspace{-0,05cm}[5]\hspace{-0,05cm}]$ of even (resp.\,of odd) cardinal.

\subsubsection{\bf Wick's embedding} 
As it is well known, a generic 5-plane in $V$ can be represented by a  matrix $M_A=\big[ {\bf Id}_5,A\big]\in {\rm Mat}_{5\times 10}(\mathbf C)$ where ${\bf Id}_5$ stands for the identity $5\times 5$ matrix and $A$ a square matrix of the same size but arbitrary otherwise, the associated 5-plane to $M_A$, denoted by $\zeta_A=\big\langle [ {\bf Id}_5,A]\big\rangle$,  being the one spanned by the vectors 
whose coordinates (in the standard coordinate system) are the given by the five lines of $M$. Then one verifies easily that the 5-plane $\zeta_M$ is $q$-isotropic if and only if the matrix $A$ is antisymmetric.  Moreover, it can be proved that for $A,B\in {\rm Asym}_5(\mathbf C)$ the two following facts are satisfied : (1) $\zeta_A$ and $\zeta_B$ belong to the same components of $OG_5(V)$, that we choose to be $OG_5^+(V)$; (2) $\zeta_A=\zeta_B$ if and only if $A=B$. We then have that the affine map
\begin{align*}
\widetilde W : {\rm Asym}_5(\mathbf C) & \longrightarrow OG_5^+(V)
\end{align*}
is an embedding which induces a birational equivalence 
\begin{equation}
\label{Eq:Asym5-bir-S5}
 {\rm Asym}_5(\mathbf C) \simeq OG_5^+(V)\simeq \mathbb S_5\,.
 \end{equation}
  One can deduce from it a birational parametrization of the image of $\mathbb S_5$ in $\mathbf P(S^+)\simeq P^{15}$, known as the Wick's parametrization, which is easily described using the following notations: 
 \begin{itemize}
 \item[$-$]  for $X \in {\rm Asym}_5(\mathbf C)$ and $i\in  [\hspace{-0,05cm}[5]\hspace{-0,05cm}]$, we denote by $X_{\hat \imath}$ the symmetric $4\times 4$ matrix obtained by deleting the $i$-th line and the $i$-th column from $Y$;
 \sk
  \item[$-$] ${\rm Pf}$ stands for the pfaffian of an antisymmetric matrix;
  \sk
   \item[$-$] for any $i,j\in  [\hspace{-0,05cm}[5]\hspace{-0,05cm}]$, $x_{ij}$ denotes the rational function on ${\rm Asym}_5(\mathbf C)$ associating the $(i,j)$-th coefficient of an  antisymmetric $5\times 5$ matrix; 
   \sk
     \item[$-$] 
       let 
     $ [\hspace{-0,05cm}[5]\hspace{-0,05cm}]_<^2$  be the ordered set of pairs $(i,j)$ such that 
     $1\leq i<j\leq 5$, with the lexicographic order. 
      Then for $k=1,\ldots,10$, one denotes by $X_k$ the coordinate $x_{ij}$ 
     if $(i,j)$ is the $k$-th element of $ [\hspace{-0,05cm}[5]\hspace{-0,05cm}]_<^2$: one has 
    $X_1=x_{12}$, $X_2=x_{13}$, $\ldots$, 
    $X_9=x_{35}$ and    $X_{10}=x_{45}$. 
     \end{itemize}  
  The projectivization 
\begin{align}
\label{Eq:map-W-proj} W =\Big[ \widehat{W} \Big] : {\rm Asym}_5(\mathbf C)  \longrightarrow \mathbf P^{15}
\end{align}
 of the affine map 
\begin{align}
\label{Eq:map-W}
\widehat{W} \, : \, {\rm Asym}_5(\mathbf C) & \longrightarrow \; {\mathbf C}^{16}\\
X=
\big(x_{ij}\big)_{i,j=1}^5 & \longmapsto \Big( \, 1 \, , \, X_1\,  , \, \ldots \, , \,  X_{10}
\, ,\, 
{\rm Pf}\big( X_{\hat 1}\big) 
\,  , \, \ldots \, , \, {\rm Pf}\big( X_{\hat 5}\big)\,  \Big)
\nonumber
\end{align}
is an affine parametrization of $\mathbb S_5\subset \mathbf P^{15}$, known as {\it Wick's parametrization} of the spinor variety.  It enjoys the nice property of having its components (as given in \eqref{Eq:map-W}) corresponding to the direct sum decomposition 
of $S^+$ into weight subspaces:  for each component $C$ of $\widehat{W}$, there is a 
well-defined weight $\mathfrak w(C)\in \mathfrak W^+$ such that $C$ is 
the composition of $\widehat{W}$ with the linear projection onto the 
weight subspace $\mathbf C v_{\mathfrak w(C)}$ 
(with respect to the decomposition in direct sum \eqref{Eq:S+-direct-sum}). 
The weights $\mathfrak w(C)$ are given by the following formulas where we use the notation \eqref{Eq:weight-vectors}: 
 one has 
$$
\mathfrak w(1)=\mathfrak w_{\emptyset}\, , \qquad 
\mathfrak w(x_{ij})=\mathfrak w_{ \{i,j\}}
\qquad \mbox{ and } \qquad 
\mathfrak w\Big( {\rm Pf}\big( X_{\hat k}\big)\Big)=\mathfrak w_{  [\hspace{-0,05cm}[5]\hspace{-0,05cm}] \setminus \{k\}}
$$
for all $(i,j)\in  [\hspace{-0,05cm}[5]\hspace{-0,05cm}]_<^2$ and all $k=1,\ldots,5$.
%
%
%
%

\subsubsection{\bf The action of the Cartan torus on $\boldsymbol{\mathbb S_5}$ and a birational model of the associated torus quotient}
We now recall some results of Serganova and Skorobogatov about the action of the Cartan torus $H_{D_5}$ of ${\rm Spin}_{10}(\mathbf C)$ on $\boldsymbol{\mathbb S_5}$.  
\sk 

To simplify the notation, we write $H$ instead of $H_5$ below. 
Following Serganova and Skorobogatove, we denote by $\boldsymbol{\mathbb S}^{sf}_5$ the subset of points $x\in \mathbb S_5$ which are stable under the action of $H$ with stabilizer ${\rm Stab}_{H}(x)=Z\big( \big)=\{ \pm 1\}$. It is open in $\mathbb S_{5}$, and of codimension $\geq 2$. One can prove that it contains the points of $\mathbb S_5$ whose at most one of the 16 Wick coordinates vanish. In particular, if $H_{\mathfrak w}$ stands for the coordinate hyperplane in $\mathbf P(S^+)$ given by the vanishing of the $\mathfrak w$-coordinate, 
$\boldsymbol{\mathbb S}^{sf}_5$ contains the generic point of the coordinate hyperplane section $\mathbb S_5\cap H_{\mathfrak w}$. \sk 

From a direct application of GIT, one obtains that the following points hold:
\begin{itemize}
\item the quotient $\boldsymbol{\mathcal Y}_5=\boldsymbol{\mathbb S}^{sf}_5/H$ is a 5-dimensional quasi-projective variety which is a Zariski-open subset of the GIT quotient $\overline{\boldsymbol{\mathcal Y}}_5=\boldsymbol{\mathbb S}_5/\hspace{-0.07cm}/H$.
Moreover, the  canonical map $\nu : \boldsymbol{\mathbb S}^{sf}_5\rightarrow \boldsymbol{\mathcal Y}_5$ is a 
geometric quotient: the preimages of $\nu$ are the $H$-orbits in $\boldsymbol{\mathbb S}^{sf}_5$;
\mk 
\item  for any weight $\mathfrak w\in \mathfrak W^+$,  
\begin{equation}
\label{Eq:Dw}
D_\mathfrak w=\nu\Big( \mathbb S_5^{sf} \cap H_{\mathfrak w}
\Big)
\end{equation}
is an irreducible divisor in $\boldsymbol{\mathcal Y}_5$, that we will call {\it `the weight divisor of weight $\mathfrak w$}; 
\mk 
\item one has $\nu(  \mathbb S_5^*)=\boldsymbol{\mathcal Y}_5^*$, where we have set
$$ \mathbb S_5^*= \mathbb S_5\setminus \Big( 
\bigcup_{ \mathfrak w \in \mathfrak W^+} H_{\mathfrak w}
\Big)
\qquad 
\mbox{ and } 
\qquad  
\boldsymbol{\mathcal Y}_5^*=
\boldsymbol{\mathcal Y}_5
\setminus \Big( 
\bigcup_{ \mathfrak w \in \mathfrak W^+} D_{\mathfrak w}
\Big)\, .$$ 
\end{itemize}
Moreover, Serganova and Skorobogatov proved very nice results about the quasi-projective variety $\boldsymbol{\mathcal Y}_5$: 
\begin{itemize}
\item let $\widehat T$ be the character lattice of the diagonal  subtorus $T$ of $ {\rm GL}(S^+)$ generated by  the 1-parameter 
 subgroup of scalar matrices  
 $ \mathbf C^* {\bf I}{\rm d}_{S^+}$
 and 
the  image of $H$ by the natural embedding ${\rm Spin}_{10}(\mathbf C)\subset  {\rm GL}(S^+)$. 
Then there is a natural 
 isomorphism  of lattices $\widehat T\simeq 
{\bf Pic}_{\mathbf Z}\big( \boldsymbol{\mathcal Y}_5\big)$ (see the very end of \cite{SerganovaSkorobogatov}); 
\sk
\item 
the set $ \big\{ \,D_{\mathfrak w}\,\big\}_{ \mathfrak w \in \mathfrak W^+ }$ is the unique minimal set of $\mathbf Z_{>0}$-generators of the 
the semi-group of effective divisor classes in ${\bf Pic}_{\mathbf Z}\big( 
\boldsymbol{\mathcal Y}_5
\big)$ (see \cite[Theorem 1.6]{Skorobogatov}); 
\sk
\item  it follows from  \cite[Theorem 2.2.2]{Skorobogatov} that 
there is a canonical isomorphism of group 
\begin{equation}
\label{Eq:WD5=Aut(Y5)}
W_{D_5}\simeq {\rm Aut}
\big( 
\boldsymbol{\mathcal Y}_5
\big)\, . 
\end{equation}
\end{itemize}

The variety $\boldsymbol{\mathcal Y}_5$ is rational.  We recall below the construction of 
 the birational equivalence $\mathbf C^5\simeq \boldsymbol{\mathcal Y}_5$  considered in \cite{PirioAFST}. We will work with it further to get  in explicit form the 
 birational realization of 
the Weyl group $W_{D_5}={\rm Aut}
\big( 
\boldsymbol{\mathcal Y}_5
\big)$ induced by this birational identification. \sk 

The quotient of $\mathbb S_5$ by $H_{D_5}$ can be birationally identified with that of $OG_5^+(\mathbf C^{10})$ by the Cartan torus $H'_{D_5}$ of ${\rm SO}_5(\mathbf C^{10})$, hence to the one of $\mathcal H=(\mathbf C^*)^5$ on ${\rm Asym}_5(\mathbf C)$, where the action of the latter torus is given by 
$$
A\cdot h = \Big(h_ih_j A_{ij}\Big)_{i,j=1}^5
\footnotemark
$$
\footnotetext{From a matricial point of view, it is more natural to see the 
action of 
$(\mathbf C^*)^5$ on ${\rm Asym}_5(\mathbf C)$ as a right-action,  but this is just a matter of notation.}for any $h=(h_i)_{i=1}^5\in \mathcal H=(\mathbf C^*)^5$ and any $A=\big(A_{ij}\big)_{i,j=1}^5\in 
{\rm Asym}_5(\mathbf C)$.  Identifying $\mathbf C\big( 
{\rm Asym}_5(\mathbf C) \big) $ with $\mathbf C( x_{i,j}\,\lvert \,1\leq i<j\leq 10)$ (where $x_{ij}$ stands for the map associating the $(i,j)$-th coefficient), one can prove the
\begin{prop}
\label{Prop:maps-P5-Y}
{\bf 1.} The algebra 
 of rational functions on ${\rm Asym}_5(\mathbf C)$ invariant by the action of $\mathcal H$ is free and generated by the components of the rational map 
 $\mathcal P_5 : 
 {\rm Asym}_5 (\mathbf C) \dashrightarrow  \mathbf C^5$ 
given by 
\begin{align*}
 \scalebox{0.79}{$\begin{bmatrix}
0 & x_{1,2} & x_{1,3}& x_{1,4}  &  x_{1,5}\\
 -x_{1,2}&  0& x_{2,3} & x_{2,4} & x_{2,5}\\
 -x_{1,3}&  -x_{2,3} & 0& x_{3,4}  & x_{3,5} \\
-x_{1,4} & -x_{2,4}  & -x_{3,4}& 0 &  x_{4,5}\\
-x_{1,5} & -x_{2,5}  & -x_{3,5}&  -x_{4,5}&  0
 \end{bmatrix} $} \longmapsto 
\left( \, \frac{x_{4,5}\, x_{1,3}}{x_{1,5} \,x_{3,4}}
\, , \, 
\frac{x_{1,5} \,x_{2,4}}{x_{1,2} \, x_{4,5}}
\, , \, 
\frac{x_{1,2} \,x_{3,5}}{x_{1,5}\, x_{2,3}}
\, , \, 
\frac{x_{2,3} \, x_{1,4}}{x_{3,4} \,x_{1,2}}
\, , \, 
\frac{x_{3,4} \,x_{2,5}}{x_{2,3} 
\, x_{4,5}} \right)\, .
\end{align*}
Consequently,  $\mathcal P_5$ is a  birational model of the quotient map $\nu : {\mathbb S}_5\dashrightarrow \boldsymbol{\mathcal Y}_5$. \sk 

\noindent{\bf 2.} Moreover, the map $\mathcal P_5$ admits the following map as a rational section: $$
Y: y=\big(\,y_{13},y_{14},y_{24},y_{25},y_{35}
\big) \longmapsto 
Y(y)= \left[\begin{array}{ccccc}
0 & 1& {y_{13}}& {y_{14}} & 1
\\
-1 & 0 & 1 & {y_{24}} &  {y_{25}} 
\\
-y_{13} & -1 & 0 & 1& {y_{35}}  
\\
- {y_{14}}  &- {y_{24}}  & -1& 0 &  1
\\
- 1  & -{y_{25}} & -{y_{35}} & -1& 0 
\end{array}\right]\, , 
$$
i.e. setting $\mathbf C(y)= \mathbf C( y_{13},y_{14},y_{24},y_{25}, y_{35})$, one has $\mathcal P_5\circ Y={\bf I}{\rm d}_{\mathbf C(y)}$ (equality as rational maps). 
\end{prop}

Considering this proposition, let $\boldsymbol{Y}_5$ be the affine space $\mathbf C^5$ 
 with the rational coordinates $y_1,\ldots,y_5$ related to the $y_{ij}$'s appearing in the definition of $Y$ above, via the relations
$y_1=y_{13}$, $y_2=y_{14}$, $y_3=y_{24}$, $y=y_{25}$, $y=y_{35}$. 
From the relation $\mathcal P_5\circ Y={\bf I}{\rm d}_{\mathbf C(y)}$, one deduces that 
the map 
$\Theta$ 
defined by requiring that the following diagram of rational maps commutes 
\begin{equation}
\label{Eq:birat-yy}
  \xymatrix@R=1cm@C=0.4cm{ 
  {\rm Asym}_5(\mathbf C) 
\ar@{->}[d]^{ \mathcal P_5 }  
\ar@{^{(}->}[rrr]^{W}
   &  &   &
 \ar@{->}[d]^\nu   \mathbb S_5  &
\hspace{-0.7cm}
 \subset \mathbf P \big(S^+ \big)
 \\
 \boldsymbol{Y}_5
  \ar@/^1pc/[u]^{  Y} 
    \ar@{->}[rrr]^{\Theta} 
  &  & &\,  \boldsymbol{\mathcal Y}_5 \,,
  }
\end{equation}
is birational.  We thus have a birational identification 
\begin{equation}
\label{Eq:Theta} 
\Theta : \mathbf C^5=
\boldsymbol{Y}_5\stackrel{\sim}{\dashrightarrow}  \boldsymbol{\mathcal Y}_5
\end{equation}
 and it is the one we will always work with in what follows. 

The weight divisors $D_{\mathfrak w}$ in $\boldsymbol{\mathcal Y}_5$ will be important for our purpose hence it is interesting to investigate how they appear (or do not appear) on the birational model  $\boldsymbol{\mathcal Y}_5$. Remark first that $Y$ is polynomial hence so is the composition $W\circ Y: \boldsymbol{Y}_5\rightarrow \mathbb S_5\subset \mathbf P(S^+)\simeq \mathbf P^{15}$. Viewed the definition of the weight divisors in 
$\boldsymbol{\mathcal Y}_5$, we get that their pull-backs under $\Theta$ are cut out by the components of $\widehat W\circ Y$. Since some of the entries of $Y(y)$ are equal to 1, we see that some of the weight divisors in $\boldsymbol{\mathcal Y}_5$ have  no pull-back as divisors in $\boldsymbol{Y}_5$ under $\Theta$.  Clearly, there are 10 divisors in $\boldsymbol{\mathcal Y}_5$ corresponding to some $D_{\mathfrak w}$'s, and these are 
the five coordinate hyperplanes $\{y_i=0\}$ with $i=1,\ldots,5$, and the other five  are given by the vanishing of the equations obtained by taking the pfaffians of the $4\times 4$ principal antisymmetric submatrices $Y_{\hat 1},\ldots,Y_{\hat 5} $ of $Y(y)$. These subpfaffians, denoted by  $P_i={\rm Pf}\big(Y_{\hat{\imath}}\big)$ for $i=1,\ldots,5$, are given by:  
\begin{align}
\label{Eq:Pi}
P_1
= &
 1+\mathit{y_{25}}-\mathit{y_{24}} \mathit{y_{35}} 
\,, \quad 
P_2
=  
 1+\mathit{y_{13}} -\mathit{y_{35}} \mathit{y_{14}}
\,, \quad 
P_3
 =  
 1+\mathit{y_{24}} -\mathit{y_{14}} \mathit{y_{25}} 
\\ 
&\qquad P_4
 =  
 1+\mathit{y_{35}}-\mathit{y_{13}} \mathit{y_{25}}
\,, \quad  
P_5
 =  
 1+\mathit{y_{14}} -\mathit{y_{13}} \mathit{y_{24}} \, . 
 \nonumber
\end{align}

Then setting  
\begin{equation}
\label{Eq:zeta-i}
\zeta_1=y_{13}
\,, \quad 
\zeta_2=y_{14}
\,, \quad
\zeta_3=y_{24}
\,, \quad
\zeta_4=y_{25}\,, \quad
\zeta_5=y_{35} 
\quad \mbox{ and } \quad 
\zeta_{i+5}=P_{i} \quad \mbox{for }\, 
i \in [\hspace{-0.05cm}[5 
]\hspace{-0.05cm}]\, , 
\end{equation}
we define 
$ Z_k=\big\{ \,\zeta_k=0\,\big\}\subset \boldsymbol{Y}_5$ for $k=1,\ldots,10$ and one sets
\begin{equation}
\label{Eq:Def-Z}
Z=\cup_{k=1}^{10} Z_k= \Big\{ \,\zeta_1\cdot \cdots \cdot \zeta_{10}=0\,\Big\}
\qquad \mbox{ and } \qquad \boldsymbol{Y}_5^*=\mathbf C^5\setminus Z\, .
\end{equation}

One verifies easily the 
\begin{prop}
\label{Prop:Map-Theta}
The map $\Theta$ is defined at every point of $\boldsymbol{Y}_5^*$ and 
sends it into $\boldsymbol{\mathcal Y}_5^*$. Actually, one has 
$\Theta(\boldsymbol{Y}_5^*)= \boldsymbol{\mathcal Y}_5^*$ and 
 the restriction $\Theta\lvert_{\boldsymbol{Y}_5^*}$ induces an isomorphism from $\boldsymbol{Y}_5^*$ onto 
$\boldsymbol{\mathcal Y}_5^*$.
\end{prop}
%
%
%
%
%
%
%

\subsubsection{\bf Del Pezzo quartic surfaces and Serganova-Skorobogatov embeddings}
\label{SS:dP4-SS-Embedding}
The geometry of  a del Pezzo quartic surface (and in particular the geometry of the lines contained in it)  is related in a nice way to that of the spinor tenfold $\mathbb S_5$ and of its torus quotient $\boldsymbol{\mathcal Y}_5$. This is important for our purpose and because it also allows a convenient way to work with the weights of the minuscule representation $S^+$, we recall it now. \sk 

Let ${\rm dP}_4$ be a fixed smooth del Pezzo surface. In finitely many ways (none of which more canonical than the others), it is isomorphic to the total space 
$X={\bf Bl}_P(\mathbf P^2)$ of the projective plane along a subset $P=\{p_1,\ldots,p_5\}$ of five points $p_i$ in general position in $\mathbf P^2$.  There are 16 lines in $X$ and each can be identified with its class in ${\bf Pic}_{\mathbb Z}(X)$. If $e_i$ stands for the class of the exceptional divisor over $p_i$ in the blow-up $\beta=\beta_P : X={\bf Bl}_P(\mathbf P^2)\rightarrow \mathbf P^2$ along $P$ (for any $i\in [\hspace{-0.05cm}[5]\hspace{-0.05cm}]$, and if $h$ denotes the class of $\beta^{-1}(\ell)$ for any line $\ell\subset \mathbf P^2\setminus P$, then the Picard lattice ${\bf Pic}_{\mathbb Z}(X)$ is freely spanned over the integers by the $e_i$'s and $h$ and the  
the (classes of) lines in $X$ are the following: 
$$
e_i\, , \qquad h-e_i-e_j\, \qquad \mbox{ and } \qquad 2h-e_{tot}\, , 
$$
with $i$ and $j$ distinct and ranging in
$[\hspace{-0.05cm}[5]\hspace{-0.05cm}]$, where we use the notation $e_{tot}=\sum_{k=1}^5 e_k$.  We will denote by $\mathcal L_X$ the set of lines contained in $X$, viewed as a subset of the Picard lattice of $X$.\sk

As  is well-known, the lines of $X$ can be put in a 1-1 relation with the weights in $\mathfrak W^+$. Let us recall how it goes: let $R$ be the orthogonal, for the intersection product, of the canonical class $K_X=-3h+e_{tot}$ in the  Picard group 
${\bf Pic}(X)={\bf Pic}_{\mathbb Z}(X)\otimes \mathbf R$. Endowed with the opposite of the (restriction of the) intersection product, $R$ becomes a `root space of type $D_5$': it is Euclidean and admits as a basis the following five `fundamental roots' 
$$
\rho_i=e_{i+1}-e_i\quad \big(i=1,\ldots,4\big) \qquad \mbox{ and }\qquad 
\rho_5=h-e_1-e_2-e_3\, . 
$$

Denote by $\Pi: {\bf Pic}(X)\rightarrow R$  the orthogonal projection (whose kernel is $\langle K\rangle =R^\perp$.  Let $(f_i)_{i=1}^5$ be the other basis of $R$ determined by the relations 
$f_i-f_{i+1}=-\rho_i=e_i-e_{i+1}$ for $i=1,\ldots,4$ and $f_4+f_5=\rho_5$, and let 
\begin{equation}
\label{Eq:psi}
 \psi : R\longrightarrow \mathbf R^5,\, 
 \sum_{i=1}^5 u_i\,f_i\longmapsto 
 (u_i)_{i=1}^5
\end{equation}
be the associated isomorphism. 
Then the linear map $ \mu=\psi\circ \Pi :  {\bf Pic}(X)\rightarrow  \mathbf R^5$
 induces a bijection 
 \begin{equation}
 \label{Eq:bijection-mu}
 \mu : \mathcal L_X \simeq  \mathfrak W^+
 \end{equation}
  from the set of lines in $X$ 
onto the set of weights of $S^+$ which is explicitly given  by
  $$
 \mu(e_i)=\mathfrak w_{[\hspace{-0.05cm}[5]\hspace{-0.05cm}] \setminus \{i\}}  
 \, , \quad 
  \mu(h-e_i-e_j)=\mathfrak w_{ \{i,j\}} 
 \qquad \mbox{and} \qquad 
 \mu(2h-e_{tot})=\mathfrak w_{\emptyset}\, , 
 $$
for all $i,j\in [\hspace{-0.05cm}[5]\hspace{-0.05cm}]$ distinct.

The previous bijection $\mathcal L_X\simeq \mathfrak W^+$ can be used to compute the action of the Weyl group $W_{D_5}$ on the weights of $S^+$. Indeed, building a graph with the $i$-th vertice corresponding to $\rho_i$ and an arrow between two vertices with labels $i$ and $j$ if and only if the two corresponding roots satisfy $(\rho_i,\rho_j)=1$, one exactly gets the labeled Dynkin diagram of Figure \ref{Fig:Dynkin-diagram-D5}.  Then for any $i = 1, \ldots,5$, one considers the map
\begin{equation}
\label{Eq:si}
s_i=s_{\rho_i} :  d\mapsto d+(\rho_i,d)\,\rho_i \, .
\end{equation}
One verifies that it is an involutive automorphism of $\big({\bf Pic}(X), (\cdot,\cdot)\big)$ letting 
$K_X$ invariant, hence it gives rise to an orthogonal involution of $R$. The $s_i$'s satisfy the relations \eqref{Eq:relations-W(D5)} so what we have obtained is a concrete realization of the Weyl group $W_{D_5}$, 
previously abstractly defined as a subgroup of $O(R)$ (in particular, this justifies  having denoted 
by $s_i$ the map \eqref{Eq:si}). 
 This will allow us to compute in an explicit way the action of $W_{D_5}$ on the set $\mathfrak W^+$ of weights of $S^+$. As an example, we treat the case of $s_1$. 
 Via \eqref{Eq:si}, it acts linearly on ${\bf Pic}(X)$ and is fully determined by indicating that it exchanges $e_1$ and $e_2$ and lets all the other generators $e_3,e_4,e_5$ and $h$ unchanged. We deduce that its action on $\mathcal L_X
\simeq  \mathfrak W^+$ is given by 
\begin{equation}
\label{Eq:echanges-si}
\mathfrak w_{[\hspace{-0.05cm}[5]\hspace{-0.05cm}] \setminus \{1\}}  
\leftrightarrow 
 \mathfrak w_{[\hspace{-0.05cm}[5]\hspace{-0.05cm}] \setminus \{2\}} 
\qquad \mbox{ and } \qquad 
\mathfrak w_{ \{1,k\}} 
\leftrightarrow 
\mathfrak w_{  \{2,k\}}
\end{equation}
 with $k=3,4,5$, with the convention that the exchanges above (indicated by the double arrow $\leftrightarrow $) are the only non-trivial actions of $s_i$ on the set of weights. 
\begin{center}
$\star$
\end{center}

The 1-1 correspondence between $\mathcal L_X$ and $\mathfrak W^+$ actually can be given a geometric origin.  We denote by $L=\cup_{\ell \in \mathcal L} \ell $ with the lines viewed as subsets of $X$, and we set  $X^*=X\setminus L$.

For any line $\ell\subset X$, let $\sigma_\ell$ be a generator of the 1-dimensional 
space of global sections ${\bf H}^0(X, \mathcal O_X(\ell))$ and let us consider the map $\hat \sigma : X\dashrightarrow  S^+$ the $\mathfrak w$-component of which, for any weight $\mathfrak w\in \mathfrak W^+$, is given by the section $\sigma_\ell$ for the line $\ell$ which corresponds to $\mathfrak w$ up to the bijection $\mu : \mathcal L_X\simeq \mathfrak W^+$ discussed above: $\hat \sigma_{ \hspace{-0.06cm} \boldsymbol{\mathcal L} } =\big( \sigma_{\mu^{-1}(\mathfrak w)}\big)_{ \mathfrak w \in \mathfrak W^+}$. This map becomes a morphism when restricted to $X^*$ and is well-defined up to post-composition by an element of the diagonal torus of ${\rm GL}( S^+)$. We claim that there is a way to chose the components $\sigma_\ell$ such that the projectivisation 
\begin{equation}
\label{Eq:F-L}
F_{ \hspace{-0.06cm} \boldsymbol{\mathcal L} }=\big[ \hat \sigma_{ \hspace{-0.06cm} \boldsymbol{\mathcal L} } \big] : X^*\rightarrow \mathbf P(S^+)
\end{equation}
 has values in $\mathbb S_5^*$. Moreover, this map extends to a well-defined morphism from $X$ to $\mathbb S_5^{sf}$ denoted the same.  For any weight $\mathfrak w\in \mathfrak W^+$, let $H_w$ be the coordinate hyperplane in $\mathbf P(S^+)$ corresponding to the vanishing of the affine $\mathfrak w$-coordinate 
(with respect to the decomposition in direct sum \eqref{Eq:S+-direct-sum}). Then one has $F_{ \hspace{-0.06cm} \boldsymbol{\mathcal L} }^{-1}\big( H_{\mathfrak w(\ell)}
 \big)=\ell$ for any line $\ell\in \mathcal L$, which gives a geometric explanation to the bijection 
\eqref{Eq:bijection-mu}. 
\sk

 It can be proven that  the composition 
 $ \nu\circ F_{ \hspace{-0.06cm} \boldsymbol{\mathcal L} } $
 makes commutative the following diagram
\begin{equation}
\label{Eq:Embedding-SS}
  \xymatrix@R=1cm@C=0.4cm{ 
   &  &&&   &
 \ar@{->}[d]^\nu   \mathbb S_5^{sf}  &
\hspace{-0.7cm}
 \subset \mathbf P \big(S^+ \big)
 \\
 X
  \ar@/^1pc/[rrrrru]^{ F_{ \hspace{-0.06cm} \boldsymbol{\mathcal L} }} 
    \ar@{->}[rrrrr]^{f_{S\hspace{-0.03cm}S}  \hspace{0.3cm} } 
  &  &&& &\,  \boldsymbol{\mathcal Y}_5 \,
  }
\end{equation}
where the map 
$f_{S\hspace{-0.05cm}S}   : X\rightarrow \boldsymbol{\mathcal Y}_5$ is an embedding which has been introduced (in a non-constructive way) by Serganova and Skorobogatov in \cite{SerganovaSkorobogatov} (hence the two $S$'s in the notation $f_{S\hspace{-0.05cm}S} $). In addition to being an embedding, $f_{S\hspace{-0.05cm}S}$ satisfies several nice properties, such as the fact that the preimage of a 
weight divisor of $\boldsymbol{\mathcal Y}_5 $ is the line in $X$ corresponding to the weight\footnote{In mathematical terms:  for any weight $\mathfrak w\in \mathfrak W^+$, one has $f_{S\hspace{-0.05cm}S}^{-1}( D_{\mathfrak w})=\ell \subset X$, where $\ell$ stands for the line such that $\mu(\ell)=\mathfrak w$.} or that it induces an isomorphism of Picard lattices $
f_{S\hspace{-0.05cm}S} ^* : 
{\bf Pic}_{\mathbf Z}\big( 
\boldsymbol{\mathcal Y}_5 \big)\simeq 
{\bf Pic}_{\mathbf Z}( X )\simeq \mathbf Z^6$. 

For our purpose, we  need to make explicit the rational map  $ F_{ \hspace{-0.06cm} \boldsymbol{\mathcal L} }\circ \beta^{-1} : \mathbf P^2\dashrightarrow \mathbb S_5$, or rather its birational model 
$$ G_{ \hspace{-0.06cm} \boldsymbol{\mathcal L} }= 
W^{-1}\circ 
 F_{ \hspace{-0.06cm} \boldsymbol{\mathcal L} }\circ \beta^{-1} : \mathbf P^2\dashrightarrow 
{\rm Asym}_5(\mathbf C)  \, .$$   Assume that $X$ is the blow-up of $\mathbf P^2$ in five points $p_i$
which, with respect to the standard homogeneous coordinates, are the vertices of the standard simplex for $p_1,\ldots,p_4$, and with $p_5=[a: b:1]$ where $a,b$ are two  complex numbers  only assumed to be such that $ab(a-1)(b-1)(a-b)\neq 0$ (this to ensure that the $p_i$'s are in general position in $\mathbf P^2$). In \cite{PirioAFST}, we have established that 
the expression of 
$ G_{ \hspace{-0.06cm} \boldsymbol{\mathcal L} }$ 
in the affine coordinates $x,y$ is given by 
$$
G_{ \hspace{-0.06cm} \boldsymbol{\mathcal L} }
(x,y)=
\frac{1}{ \mathcal C_{ab} }
\scalebox{0.9}{$
\left[\begin{array}{ccccc}
0 & -b -a  & -\left(a +1\right) y  & \left(1-y \right) a  & y -b
\\
 b +a  & 0 & -x \left(1+b \right) & \left(1-x \right) b  & x-a  
\\
 \left(a +1\right) y  & x \left(1+b \right) & 0 & y -x  & bx-a y  
\\
 \left(y -1\right) a  & \left(x-1 \right) b  & x-y  & 0 & 
 (b-1)x + (1-a)y - b + a
\\
 b -y  & a -x  & a y -b x  &  (1-b)x + (a-1)y + b - a  & 0 
\end{array}\right]
$}
$$
%
%
where $C_{ab}=\left(a-b  \right) x y +b\left(1-a \right) x +a\left(b -1 \right) y$ is an affine 
quadratic polynomial which cuts out the (affine part of the) conic passing through the points $p_1,\ldots,p_5$. \mk

In the considerations  above, $a$ and $b$ are fixed parameters, but it makes sense and it is interesting to let them vary. Doing this,  $G_{ \hspace{-0.06cm} \boldsymbol{\mathcal L} }$ is seen as a rational map from $\mathbf C^4$ (with affine coordinates $x,y,a,b$) onto 
${\rm Asym}_5(\mathbf C)$, that we will denote simply by $G$. This map can be considered  as a birational model of a putative rational map $\mathcal G: d\mathcal P_4\dashrightarrow \mathbb S_5\subset \mathbf P(S^+)$, 
where the source space is the `universal surface $d\mathcal P_4$ over the moduli space $\mathcal M_{ {\rm dP}_4}$ (birational to $\mathbf P^2$) of quartic del Pezzo surfaces', such that the restriction of 
$\mathcal G$ along a fiber of $d\mathcal P_4 \rightarrow \mathcal M_{ {\rm dP}_4}$ birationally  coincides with the map \eqref{Eq:F-L} of the corresponding del Pezzo quartic surface. 

In the putative setting considered above, post-composing 
 $\mathcal G$ with the quotient mapping $\mathbb S_{5}\dashrightarrow \boldsymbol{\mathcal Y}_5$ would give a map 
$\mathcal F : d\mathcal P_4\dashrightarrow \boldsymbol{\mathcal Y}_5 $ whose restriction along the fibers of the universal del Pezzo quartic surface would coincide with Serganova-Skorobogatov's embedding of the surface. The well-defined rational map 
\begin{equation}
\label{Eq:map-F}
F=(F_i)_{i=1}^5= \mathcal P_5\circ G: \mathbf C^4 \dashrightarrow \mathbf C^5
\end{equation}
 can be thought of as a birational model of $\mathcal F$. 
 Its components $F_i\in \mathbf C(x,y,a,b)$ are the following ones: 
\begin{align*}
F_1= &\, 
\frac{\left(a y -b x -a +b +x -y \right) \left(a +1\right) y}{\left(b -y \right) \left(-y +x \right)}
\\
F_2= &\, 
\frac{x \left(b +1\right) \left(y -1\right) a}{\left(-y +x \right) \left(b +a \right)}
\\
F_3= &\, 
\frac{\left(b -y \right) \left(-1+x \right) b}{\left(b +a \right) \left(\left(1-x \right) b +x +\left(y -1\right) a -y \right)}
\\
F_4= &\, 
\frac{\left(-y +x \right) \left(a -x \right)}{x \left(\left(1-x \right) b +x +\left(y -1\right) a -y \right) \left(b +1\right)}
\\
\mbox{and } \quad  F_5= &\, 
\frac{\left(b +a \right) \left(a y -b x \right)}{\left(b -y \right) x \left(b +1\right)}
\, .
\end{align*}
%
%
%
%

Since for any del Pezzo quartic surface, the preimages by the associated map $f_{S\hspace{-0.05cm}S}$ of the 16 weight divisors $D_{\mathfrak w}=\nu\big( H_{\mathfrak w} \cap \mathbb S_5^{sf}\big)$ are the 16 lines of the del Pezzo surface, one can expect something similar for the map $F$.  And this does occur in a way. Indeed, the divisor $Z$ in $Y_5$  has only 10 components, but it can be verifies that the preimage of $Z=\{ 
\zeta_F=\prod_{i=1}^{10} \zeta_i=0\}$ by $F$ has 16 hypersurface components somehow. More precisely,  one verifies that the support of the divisor cut out  by 
$\zeta_Z(F)=\prod_{i=1}^{10}\zeta_i(F)=0$  in $\mathbf C^4$  has 15 irreducible components, namely the irreducible affine hypersurfaces cut out by the following polynomials in the variables $x,y,a$ and $b$:
\begin{align*}
x  \,\, , \,\,
   y
     \,\, , \,\,
      x-1  \,\, , \,\,
 x-a
  \,\, , \,\, 
 y-1   
  \,\, , \,\,
 y-b
 &
  \,\, , \,\,
 x-y 
  \,\, , \,\, 
 b x- a y
  \,\, , \,\, 
 a
   \,\, , \,\,
 a +1
   \,\, , \,\,
b  \,\, , \,\, 
b +1  \,\, , \,\, b +a   \,\, , 
  \\
  &\; 
ab (x - y)  -(a-b) xy +a y -b x
  \,\, , \,\,
(a-1) y -(b-1) x -a +b  
     \, .
\end{align*}
Together with the hyperplane at infinity in a compactification $\mathbf C^4\subset \mathbf P^4$, this makes 16 irreducible divisors in total, which correspond to the 16 weight divisors $D_{\mathfrak{w}}$ (with $\mathfrak w\in \mathfrak W^+$) of $\boldsymbol{\mathcal Y}_5$. We will not really use this below, but we wanted to stress this fact which indicates that  the coordinates $x,y,a,b$ and the map $F$  are well-adapted to study $\boldsymbol{\mathcal Y}_5$ from a birational point of view.

\subsubsection{\bf The birational action of $\boldsymbol{W_{D_5}}$}
\label{SS:Birational-action-w(D5)}
Conjugating the isomorphism \eqref{Eq:WD5=Aut(Y5)}
by the birational identification \eqref{Eq:Theta} gives rise to a birational representation 
of the Weyl group $W_{D_5}$ into the Cremona group ${\bf Bir}(\mathbf C^5)$ of birational transformations of $\mathbf C^5$. Our goal here is to make this representation explicit.
\sk 

Recall that $  W_{D_5}=N/H_{D_5}$
where $N=N_{ {\rm Spin}_{10}(\mathbf C)}\big(H_{D_5}\big)$ stands for the normalizer of $H_{D_5}$ in the spin group. From the proof of Theorem 2.2 in \cite{Skorobogatov}, we know that the action of an element $w\in W_{D_5}$ on $\boldsymbol{\mathcal Y}_5$ is induced by the one of a lift $\hat w$ in the normalizer. Recall that $w$ acts by permutations and transitively on the set of lines $\mathcal L$ and acts also naturally 
on the set of weights $\mathfrak W^+$. Moreover, since the bijection 
\eqref{Eq:bijection-mu} is $W_{D_5}$-equivariant (as an easy verification shows), one deduces 
a natural (but naive) lift for $w$, namely the linear map $\tilde w$ on $S^+$ defined requiring that it acts 
by permuting the element of the weight basis $(v_\varpi)_{\varpi \in \mathfrak W^+}$  
of $S^+$ (see \eqref{Eq:S+-direct-sum}) according to the weight, i.e. one has $\tilde w(v_\varpi)=v_{w(\varpi)}$ for any weight $\varpi$. However it is not the case that this lift belongs to the normalizer of the Cartan torus. Indeed, it is by no mean canonical, it actually depends on the chosen weight basis $(v_\varpi)_{\varpi\in \mathfrak W^+}$. Since each $v_\varpi$ is only well-defined up to multiplication by a non-zero scalar,  we look for 
a lift $\hat w$ acting as $v_\varpi\mapsto \lambda_{w}^\varpi \cdot v_{w(\varpi)}$ 
with $\lambda_{w}^\varpi\in \mathbf C^*$ for any $\varpi$. 
\sk

Let us apply this strategy to the first generator $s_1$ of $W_{D_5}$. 
 Let us first consider $\tilde s_i$, which stands for the `dumb lift' of $s_1$  mentioned above. 
 Given $\eta=\widehat W(X)$ for 
   $$
X=
 \left[\begin{array}{ccccc}
0 & x_{12}& {x_{13}}&  {x_{14}} &   \, {x_{15}} 
\\
-x_{12} & 0 &  {x_{23}} & {x_{24}} & {x_{25}}
\\
- {x_{13}} & -  {x_{13}}  & 0 &   x_{34}&   {x_{35}}  
\\
- {x_{14}}  &-  {x_{14}} & -  x_{34}& 0 &     x_{45}
\\
- {x_{15}}  & -  {x_{15}}&   -    {x_{35}} & -   x_{45}& 0 
\end{array}\right]\in {\rm Asym}_5(\mathbf C)\, , 
$$
generic, 
 one easily deduces from 
 \eqref{Eq:echanges-si} the weight coordinates  of $\tilde L_1(\eta)=(\tilde s_i\circ  W)(X)$ : one has
\begin{align*}
 \scalebox{0.9}{
$ \eta= 
 \Big( 1, x_{12}\, ,  \,  x_{13}\, , \, x_{14}\, , \, x_{15}\, , \, 
   x_{23}\, , \, 
 x_{24}\, , \, 
 x_{25}  \, , \, 
 x_{34}\, , \, 
 x_{35}\, , \, 
 x_{45}
 \, , \,   
 {\rm Pf}(X_{\hat 1})
 \, , \,
  {\rm Pf}(X_{\hat 2})
   \, , \,
  {\rm Pf}(X_{\hat 3})
   \, , \,
  {\rm Pf}(X_{\hat 4}) \, , \,
  {\rm Pf}(X_{\hat 5})\Big)$}\\
  \mbox{and }\, 
 \scalebox{0.9}{  $ \tilde L_1(\eta)= 
 \Big( 1, x_{12}\, ,   \,  
 \textcolor{blue}{x_{23}}\, , \, 
 \textcolor{blue}{x_{24}}\, , \, \textcolor{blue}{x_{25}}\, , \, \textcolor{blue}{x_{13}}\, , \, \textcolor{blue}{x_{14}}\, , \, \textcolor{blue}{x_{15}}\, , \, 
 x_{34}\, , \, x_{35}\, , \, x_{45}
 \, , \,  \textcolor{blue}{ {\rm Pf}(X_{\hat 2})}
 \, , \,
 \textcolor{blue}{ {\rm Pf}(X_{\hat 1})}
   \, , \,
  {\rm Pf}(X_{\hat 3})
   \, , \,
  {\rm Pf}(X_{\hat 4}) \, , \,
  {\rm Pf}(X_{\hat 5}) \Big)$}
 \end{align*}
 (the coordinates which have been permuted in $ \tilde s_1(\eta)$ are in blue). The transformation $\eta\mapsto  \tilde L_1(\eta)$ is induced by an invertible  linear map $\tilde L_1$ of $S^+$ (which permutes the coordinates in blue) which however 
  is not an element of the image of the spin group in ${\rm GL}(S^+)$ since it can be verified that $\tilde L_1(\mathbb S_5)\not \subset \mathbb S_5$.  Then let   
   $\overline{L}_1\in {\rm GL}(S^+)$ be obtained from $\tilde L_1$ by post-composing it with a diagonal matrix: there are non-zero complex numbers 
   $\lambda_k$ for $k=0,\ldots,5$
    and  $\lambda_{ij}$ for $(i,j)\in 
   [\hspace{-0.05cm}[5   ]\hspace{-0.05cm}]^2_<$ such that 
   $\overline{L}_1$ is entirely characterized by the fact that one has 
\begin{align*}
 &\scalebox{0.9}{  $ \overline L_1(\eta)= 
 \Big( \lambda_{0}, \lambda_{12} \, x_{12}\, ,   \,  
\lambda_{13} \, \textcolor{blue}{x_{23}}\, , \, 
 \lambda_{14} \,\textcolor{blue}{x_{24}}\, , \, 
\lambda_{15} \,  \textcolor{blue}{x_{25}}\, , \, \lambda_{23} \,\textcolor{blue}{x_{13}}\, , \, 
\lambda_{24} \,
\textcolor{blue}{x_{14}}\, , \, 
\lambda_{25} \, \textcolor{blue}{x_{15}}\, , \, 
 \lambda_{34} \,x_{34}\, , \, 
 \lambda_{35} \, x_{35}\, , \, 
 \lambda_{45} \,x_{45}
 \, , \,$} \\
 &  \hspace{6cm} \scalebox{0.9}{\, $
\lambda_1\,  \textcolor{blue}{ {\rm Pf}\big(X_{\hat 2}\big)}
 \, , \,
\lambda_2\,\textcolor{blue}{ {\rm Pf}\big(X_{\hat 1}\big)}
   \, , \,
\lambda_3\,  {\rm Pf}\big(X_{\hat 3}\big)
   \, , \,
\lambda_4\,  {\rm Pf}\big(X_{\hat 4}\big) \, , \,
\lambda_5\,  {\rm Pf}\big(X_{\hat 5}\big) \Big)$}
 \end{align*}
for $\eta=\widehat W(X)$ with $X$ generic.  Since post-composing by an element of the Cartan torus is irrelevant, there is no loss in generality by assuming that $\lambda_k=1$ for $k=1,\ldots,5$.  That $\overline{L}_1$ belongs to the normalizer of $H_{D_5}$ in the spin group implies in particular that $\overline{L}_1$ stabilizes $\mathbb S_5$.  Setting  $\Lambda=\{\,  \lambda_0\, \}\cup \big\{ \, \lambda_{ij}\,\lvert \, 1\leq i<j\leq 5\, \big\}$  and 
 $$ 
\overline X= \frac{1}{\lambda_0}
 \left[\begin{array}{ccccc}
0 & \lambda_{12} \,x_{12}& \lambda_{13} \, \textcolor{blue}{x_{23}}&  \lambda_{14}  \, \textcolor{blue}{x_{24}} &   \lambda_{15} \, \textcolor{blue}{x_{25}} 
\\
-\lambda_{12} \,x_{12} & 0 &    \lambda_{23}\, \textcolor{blue}{x_{13}} &  \lambda_{24}  \,\textcolor{blue}{x_{14}} &  \lambda_{25}\,  \textcolor{blue}{x_{15}}
\\
- \lambda_{13} \,\textcolor{blue}{x_{23}} & -  \lambda_{23} \, \textcolor{blue}{x_{13}}  & 0 &   \lambda_{34} \, x_{34}&    \lambda_{35} \, {x_{35}}  
\\
- \lambda_{14} \,  \textcolor{blue}{x_{24}}  &-  \lambda_{24} \,  \textcolor{blue}{x_{14}} & -  \lambda_{34} \, x_{34}& 0 &    \lambda_{45} \, x_{45}
\\
-  \lambda_{15} \, \textcolor{blue}{x_{25}}  & -  \lambda_{25} \,\textcolor{blue}{x_{15}}&   -   \lambda_{35} \, {x_{35}} & -  \lambda_{45} \, x_{45}& 0 
\end{array}\right]
$$
 the condition $\overline{L}_1(\mathbb S_5)= \mathbb S_5$ is equivalent to the fact that 
 the following pfaffian relations are satisfied for any generic (hence for any) matrix $X\in {\rm Asym}_5(\mathbf C)$: 
 \begin{equation}
 {\rm Pf}\big( \overline X_{\hat 1}\big)={\rm Pf}\big(  X_{\hat 2}\big)\, , \quad 
  {\rm Pf}\big(  \overline X_{\hat 2}\big)={\rm Pf}\big(  X_{\hat 1}\big)
 \qquad \mbox{ and }\qquad  
  {\rm Pf}\big(  \overline X_{\hat s}\big)={\rm Pf}\big(  X_{\hat s}\big)
  \quad \mbox{for}\quad s=3,4,5\, . 
 \end{equation}
These relations are polynomial identities in the indeterminates $x_{ij}$ with $1\leq i<j\leq 5$  
whose coefficients are rational expressions in the elements of $\Lambda$. Assuming that all these coefficients are zero corresponds to a system of polynomial equations in the $\lambda$'s which is not difficult to solve. One obtains that  the matrix $\overline X$ is necessarily the following one: 
  $$
i\, \left[\begin{array}{ccccc}
0 & x_{12}&  {x_{23}}&     {x_{24}} &     {x_{25}} 
\\
-x_{12} & 0 &     {x_{13}} &     {x_{14}} &    {x_{15}}
\\
-  {x_{23}} & -   {x_{13}}  & 0 & - x_{34}&   - {x_{35}}  
\\
-   {x_{24}}  &-    {x_{14}} & x_{34}& 0 &  - x_{45}
\\
-    {x_{25}}  & -{x_{15}}&   {x_{35}} & x_{45}& 0 
\end{array}\right]\, . 
$$
The Cremona map $\sigma_1$  induced by  $s_1$  is given by $\mathcal P_5\big( \overline{Y}\big)$ where $\mathcal P_5$ and $Y$ are given in Proposition \ref{Prop:maps-P5-Y}.  More explicitly,  for $y=(y_{13},y_{14},y_{24},y_{25},y_{35})$, one has 
$$
\sigma_1(y)=
 \mathcal P_5\Big( \overline{Y(y)}\Big)
 = 
 \mathcal P_5\left( \,{
\overline{
\scalebox{0.7}{$
 \left[\begin{array}{ccccc}
0 & 1& {y_{13}}& {y_{14}} & 1
\\
-1 & 0 & 1 & {y_{24}} &  {y_{25}} 
\\
-y_{13} & -1 & 0 & 1& {y_{35}}  
\\
- {y_{14}}  &- {y_{24}}  & -1& 0 &  1
\\
- 1  & -{y_{25}} & -{y_{35}} & -1& 0 
\end{array}\right]$}}}\,
\right) 
 = 
 \mathcal P_5\left( {
\scalebox{0.7}{$i \,\left[\begin{array}{ccccc}
0 & 1& 1& {y_{24}} & {y_{25}} 
\\
-1 & 0 & {y_{13}} & {y_{14}} &  1
\\
-1 & -{y_{13}}  & 0 & -1& -{y_{35}}  
\\
- {y_{24}}  &- {y_{14}}  & -1& 0 &  -1
\\
- {y_{25}}  & -1& {y_{35}} & 1& 0 
\end{array}\right]\,
$}}
\right) \, .
$$
One  obtains eventually that $\sigma_1$ is the following involutive Cremona transformation: 
\begin{equation}
\label{Eq:Formula-Sigma-1}
\sigma_1: y \longmapsto \left(\, \frac{1}{ {y_{25}}}
\, , \, 
 - {y_{24}}  {y_{13}}
 \, , \,  
- {y_{14}}  {y_{25}}
\, , \, 
 \frac{1}{ {y_{13}}}
 \, , \, 
-\frac{ {y_{35}}}{ {y_{13}}  {y_{25}}}
\, \right)\, .
\end{equation}

Proceeding in a similar way for each of the four other generators $s_2,s_3,s_4$ and $s_5$ of $W_{D_5}$, one obtains the following explicit formulas for  the Cremona transformations $\sigma_i$ induced by them:
%
%
\begin{align}
\label{Eq:Formula-Sigma-k}
\sigma_2: y \longmapsto &  \left(\, 
\frac{1}{ {y_{24}}} 
\, , \, 
-\frac{ {y_{14}}}{ {y_{24}}  {y_{13}}}
 \, , \, 
\frac{1}{ {y_{13}}}
 \, , \,  - {y_{35}}  {y_{24}}
  \, , \, 
- {y_{13}}  {y_{25}}
\, \right)
\nonumber 
\\
\sigma_3: y \longmapsto &  \left(\, 
- {y_{14}}  {y_{35}}
 \, , \,
  - {y_{24}}  {y_{13}}
   \, , \,
\frac{1}{ {y_{35}}}
 \, , \,
-\frac{ {y_{25}}}{ {y_{35}}  {y_{24}}}
 \, , \, 
\frac{1}{ {y_{24}}}
\, \right)
\\
\sigma_4: y \longmapsto &  \left(\, 
-\frac{ {y_{13}}}{ {y_{35}}  {y_{14}}}
 \, , \,
\frac{1}{ {y_{35}}}
 \, , \,
 - {y_{14}}  {y_{25}}
  \, , \,
- {y_{24}}  {y_{35}}, \frac{1}{ {y_{14}}}
\, \right)
\nonumber
\\
\sigma_5: y \longmapsto &  \left(\, 
 -P_2
\, , \,  
 {y_{14}}
\frac{P_1}{P_3}
\, , \,  
-  \frac{ {y_{24}}}{P_3}
 \, , \,  
-\frac{ {y_{25}}}{P_1}
\, , \,   {y_{35}} 
\frac{P_3}{P_1}
\, \right)
\nonumber
\end{align}
(where in the formula for $\sigma_5$, the $P_i$'s stand for the sub-pfaffians defined in 
\eqref{Eq:Pi}).

\begin{prop} 
The map $s_i\longmapsto \sigma_i$ for $i=1,\ldots,5$ gives rise to an embedding of groups 
\begin{equation}
\label{Eq:WD5->Bir(C5)}
W_{D_5}=\big\langle\, s_1,\ldots,s_5\, \rangle \longrightarrow {\rm Bir}\big(  \mathbf C^5\big) 
\end{equation}
which corresponds to Skorobogatov's isomorphism $W_{D_5}\simeq {\rm Aut}\big( 
\boldsymbol{\mathcal Y}_5
\big)$ up to the birational identification \eqref{Eq:Theta}.
\end{prop} 

\begin{proof} Actually, this follows from \cite[Theorem 2.2]{Skorobogatov} since  the formulas above for the $\sigma_i$'s have been obtained by following, in the specific case under scrutiny and working in a explicit manner, the proof given by 
Skorobogatov  in his paper. 

But that the proposition holds true can also be verified as follows: 
to ensure that \eqref{Eq:WD5->Bir(C5)} indeed induces a morphism of group, it suffices to verify that the Cremona transformations $\sigma_1,\ldots,\sigma_5$
satisfy all the relations \eqref{Eq:relations-W(D5)} satisfied by the generators 
$s_i$ of $W_{D_5}$. Using the explicit expressions 
for the $\sigma_i$'s given above, this is something straightforward to check.  Using a computer algebra system, there is no difficulty to  show that the $\sigma_i$'s generate a finite subgroup of 
${\rm Bir}(\mathbf C^5)$ with $1920$ elements. Since this is precisely the order of $W_{D_5}$, it follows that the morphism of groups \eqref{Eq:WD5->Bir(C5)} is injective. 
\end{proof}


We will denote by $\boldsymbol{\mathscr W}_{\hspace{-0.1cm}D_5}$ the subgroup of 
$ {\rm Bir}\big(  \mathbf C^5\big) $  generated by  the $\sigma_i$'s: 
\begin{equation}
\label{Eq:Bir-W-D5}
W_{D_5}\simeq \boldsymbol{\mathscr W}_{\hspace{-0.1cm}D_5}=\big\langle 
\sigma_1,\ldots,\sigma_5
\big\rangle  \subset {\rm Bir}\big(  \mathbf C^5\big) \, . 
\end{equation}

\begin{rem}
The birational realization 
$\boldsymbol{\mathscr W}_{\hspace{-0.1cm}D_5}$
of the Weyl group $W_{D_5}$  seems to be new.
\end{rem}

Because $W_{D_5}$ acts by permuations (and transitively) on the set of weights $\mathfrak W^+$ of $S^+$, it acts also (in exactly the same way) on the set of coordinate hyperplanes $\big\{
\, H_{\mathfrak w}\, \big\}_{\mathfrak w \in \mathfrak W^+} $ hence on the set of weight divisors 
$\big\{
\, D_{\mathfrak w}\, \big\}_{\mathfrak w \in \mathfrak W^+} \subset {\bf Pic}_{\mathbf Z}(\boldsymbol{\mathcal Y})$ as well.  For $w\in W_{D_5}$, denote by $\varphi_w$ the corresponding automorphism of $\boldsymbol{\mathcal Y}$. Then for any weight $\mathfrak w$ and any $w\in W_{D_5}$, one has $\varphi_w(D_{\mathfrak w})= D_{w(\mathfrak w)}$ from which it follows that $\varphi_w$ induces an automorphism of $\boldsymbol{\mathcal Y}_5^*$. 
 Combined with Proposition \ref{Prop:Map-Theta}, this gives us 
\begin{prop}
\label{Prop:WD5->BIRC5}
For any $i=1,\ldots,5$, the Cremona transformation $\sigma_i$ is defined on $\boldsymbol{Y}_5^* =\mathbf C^5\setminus Z_5$ and gives rise to an automorphism of the pair $(\mathbf C^5,Z_5)$. Consequently, the image $\boldsymbol{\mathscr W}_{\hspace{-0.1cm}D_5}$ of the group embedding \eqref{Eq:WD5->Bir(C5)} is a subgroup of 
${\rm Bir}\big(  \mathbf C^5\big)\cap {\rm Aut}(\mathbf C^5\setminus Z_5)$. 
\end{prop}

\section{\bf The Gelfand-MacPherson web $\boldsymbol{{\mathcal W}^{GM}_{ \hspace{-0.05cm}{\mathcal Y}_5}}$ and its abelian relations}
\label{S:Intro}
This is the main section of the paper, in which we study the web $\boldsymbol{\mathcal W}^{GM}_{  \hspace{-0.05cm} \boldsymbol{\mathcal Y}_5}$, in particular its abelian relations, which we make explicit and determine their invariance properties relative to the action of the Weyl group $W_{D_5}$. 

\subsection{\bf The web $\boldsymbol{{\mathcal W}^{GM}_{ \hspace{-0.05cm}{\mathcal Y}_5}}$ in coordinates}
\label{SS:WGMS5-in-coordinates}
We use here some results of \cite[\S4.5]{PirioAFST} to which we refer the reader for more details.

The moment polytope of the spinor tenfold, denoted by $\Delta_{D_5}$, is the $5$-demihypercube, realized in $\mathfrak h_{D_5}\simeq \mathbf R^5$ as the convex envelope of $\mathfrak W^+$. 
It has $10+2^4=26$ facets (= faces of codimension 1): 10 are $4$-demihypercubes, the 16 other being 4-simplices. The former are facets of type $D_4$, the latter of type $A_4$ 
 (see \cite[Fig.\,1]{PirioAFST}).   Seeing $\Delta_{D_5}$ as the convex envelope of lines in the Picard lattice of a given smooth del Pezzo surface $X={\rm dP}_4$, and seeing the facets as determined by their set of vertices, the facets of type $A_4$ are in correspondence with the description 
 of $X$ as a blow-up of $\mathbf P^2$ in 5 points, the five corresponding exceptional divisors on ${\rm dP}_4$ being the vertices of the corresponding facet. As for the facets which are demihypercubes, they correspond to the conic fibrations on $X$, the vertices of each such facet being the irreducible components (lines) of the non-irreducible fibers of the fibration. 
 
 Working in $\mathbf R^5$ by means of the isomorphism \eqref{Eq:psi}, then first 
 one has that $\Psi(\mathfrak W^+)$ is the set of 5-uplets $\frac{1}{2}\underline{\epsilon}$ 
 with $\underline{\epsilon}= (\epsilon_i)_{i=1}^5\in \{\pm 1\}^5$ even in the sense that $\epsilon_1\cdots \epsilon_5=1$. Then 
 it can be verified that the facets of type $A_4$ are given by intersecting $\Delta_{D_5}$ with the affine hyperplanes cut out by $\sum_{i=1}^5 \varepsilon_i x_i=3/2$ for all  $\underline{\varepsilon}=(\varepsilon_i )_{i=1}^5 \in \{\pm 1\}^5$ which are odd (that is such that $\varepsilon_1\cdots \varepsilon_5=-1$).  As explained in \cite{PirioAFST}, the $A_4$-facets of $\Delta_{D_5}$ are `{\it web-irrelevant'} and have not to be considered further.  
 
 The 10 demihypercubical facets are the intersections $\Delta_{5,i}^\epsilon=\Delta_{D_5}\cap \{ \, x_i= \epsilon/2\, \}$ for $i=1,\ldots,5$ and $\epsilon\in \{\pm 1\}$.  These 10 facets are `{\it web-relevant'} and from the material of  \cite[\S4.5]{PirioAFST} (in particular Figure 1 therein), we can see that the face map associated 
 to the facet $\Delta_{5,i}^\epsilon$ is a dominant rational map $w:\Pi_i^\epsilon : \mathbb S_5\dashrightarrow \mathbb S_4$. Here the target space $\mathbb S_4$ is the {\it `spinor 6-fold'}: it is the homogeneous space of type $(D_4,\omega_4)$. 
  Identifying $\mathbf P^7$ with the projectivization of $\mathbf C\oplus {\rm Asym}_4(\mathbf C) \oplus \mathbf C\simeq \mathbf C^8$, the spinor 6-fold  admits an affine  parametrization {\it \`a la Wick} given by 
\begin{align*}
 W_4: {\rm Asym}_4(\mathbf C)& \hookrightarrow \, \, \mathbb S_4\subset \mathbf P^7\\ 
 M& \longmapsto \big[ \, 1 :  M  :  {\rm Pf}(M)\, 
  \big]\, ,
\end{align*}
    from which it follows that 
$\mathbb S_4$ is isomorphic to the smooth hyperquadric $\mathbb Q^6\subset \mathbf P^7$.    
 
 
As for the expressions of the face maps $\Psi_i^\epsilon$ read in the antisymmetric matricial charts associated to Wick's parametrizations of the two corresponding spinor manifolds, there is a simple (and nice) formula when $\epsilon=+1$: for $i\in [\hspace{-0.05cm}[5]\hspace{-0.05cm}]$, one has 
\begin{align*}
\widetilde \Phi_i^+ = W_4^{-1}\circ \psi_i^+\circ W_5 \, : \, 
{\rm Asym}_5(\mathbf C)& \dashrightarrow {\rm Asym}_4(\mathbf C)\\
A& \longmapsto A_{\hat \imath}
\end{align*}
 where $A_{\hat \imath}$ stands for the $4\times 4$ antisymmetric matrix obtained from $A$ by deleting its $i$-th line and its $i$-th column.  One can give formulas for the maps 
 $\widetilde \Phi_i^-$ read in Wick's charts (see \cite[Prop.\,4.17]{PirioAFST}), but none as nice as the one above for the $ \widetilde \Phi_i^+$.\footnote{The reason behind the dichotomy for the formulas of the face maps $\widetilde \Phi_i^\pm$ read in matricial charts is that the target spaces of the 
 $\widetilde \Phi_i^+$'s are naturally identified with $\mathbb S_4$ embedded in $\mathbf P(S_4^+)$ whereas the images of the $\widetilde \Phi_i^-$'s naturally live in the projectivization $\mathbf P(S_4^-)$ of the other half-spin representation of ${\rm Spin}_8(\mathbf C)$. See \cite[\S4.5.2]{PirioAFST} for more details.}
 
 From the explicit formulas in Wick's charts for the face maps, one easily gets some formulas for some birational models of the $H_{D_5}$-equivariant quotient $\pi_i^\pm : \boldsymbol{\mathcal Y}_5\dashrightarrow \boldsymbol{\mathcal Y}_4$ of the face maps $\Pi_i^\pm$. Let us denote by $\boldsymbol{W}^{GM}_{\hspace{-0.05cm}\boldsymbol{Y}_5}$ the pull-back of $\boldsymbol{\mathcal W}^{GM}_{\hspace{-0.05cm}\boldsymbol{\mathcal Y}_5}=\boldsymbol{\mathcal W}\big( \, \pi_i^\pm\, \big)_{ i\in [\hspace{-0.05cm}[5]\hspace{-0.05cm}]}$ under the birational map \eqref{Eq:Theta}: 
 $$
 \boldsymbol{W}^{GM}_{\hspace{-0.05cm}\boldsymbol{Y}_5}=\Theta^*\Big( 
 \boldsymbol{\mathcal W}^{GM}_{\hspace{0.05cm}\boldsymbol{\mathcal Y}_5}
 \Big)\, . 
 $$
 Explicit first integrals for  $\boldsymbol{W}^{GM}_{\hspace{-0.05cm}\boldsymbol{Y}_5}
$ have been given in \cite{PirioAFST}. 
 For $i=1,\ldots,5$, let $\boldsymbol{F}_i$ (resp.\,$\boldsymbol{F}_{i+5}$) be the foliation on $\boldsymbol{Y}_5$ induced by $\pi_i^+\circ \Theta$ (resp.\, by $\pi_i^-\circ \Theta$).

\begin{prop}  
The following rational functions are first integrals of the foliations $\boldsymbol{F}_i$: 
\begin{align}
\label{Eq:Formules-Ui}
U_1
=&\, \bigg(\, {y_{25}}\, , \,  {y_{24}} {y_{35}}
\, \bigg)
&&
U_6
=
 \bigg(\, 
\frac{{y_{14}} {y_{35}} -{y_{13}} -1}{{y_{14}} \left({y_{13}} {y_{25}} -{y_{35}} -1\right) }
 \, , \, 
\frac{ {y_{13}}\left({y_{14}} {y_{25}} -{y_{24}} -1\right)}{{y_{14}} \left({y_{13}} {y_{25}} -{y_{35}} -1\right) }
 \,
\bigg) 
\nonumber
\\   \sk 
 U_2
=&\,
\bigg(\, \frac{1}{y_{13}}\, , \, 
\frac{{y_{14}} {y_{35}}}{{y_{13}}}\, \bigg)
&&
 U_7
=
\bigg(\, 
\frac{{y_{24}} {y_{35}} -{y_{25}} -1}{{y_{24}}\left({y_{13}} {y_{25}} -{y_{35}} -1\right) }
 \, , \, 
\frac{{y_{14}} {y_{25}} -{y_{24}} -1}{{y_{24}}\left({y_{13}} {y_{25}} -{y_{35}} -1\right) }
 \, 
\bigg)
\nonumber \\ \sk 
 U_3
=&\,
  \bigg( 
  \, {y_{24}}
 \, , \,  
   {y_{14}} {y_{25}}
   \, \bigg)
   &&
 U_8
=\bigg( 
\frac{{y_{13}} \left({y_{24}} {y_{35}} -{y_{25}} -1\right)}{{y_{13}} {y_{25}} -{y_{35}} -1}
 \, , \, 
\frac{{y_{14}} {y_{35}} -{y_{13}} -1}{{y_{13}} {y_{25}} -{y_{35}} -1}
 \,
\bigg)
\\ \sk 
 U_4
=&\,
\bigg(\,  \frac{1}{y_{35}}
\, , \, 
\frac{{y_{13}} {y_{25}}}{{y_{35}}}
\,
\bigg)
&&
 U_9
= 
 \bigg( 
\, 
\frac{{y_{14}} \left({y_{24}} {y_{35}} -{y_{25}} -1\right)}{{y_{14}} {y_{25}} -{y_{24}} -1}
 \, , \,
\frac{{y_{24}}\left({y_{14}} {y_{35}} -{y_{13}} -1\right) }{{y_{14}} {y_{25}} -{y_{24}} -1}
 \, \bigg)
\nonumber \\ \sk 
 U_5
=&\,  \bigg(\, {y_{14}}
\, , \, 
 {y_{13}} {y_{24}}
 \, 
\bigg) 
&& 
 U_{10}
=
\bigg( \, 
\frac{{y_{24}} {y_{35}} -{y_{25}} -1}{{y_{35}}\left({y_{14}} {y_{25}} -{y_{24}} -1\right) }
 \, , \,
\frac{ {y_{25}} \left({y_{14}} {y_{35}} -{y_{13}} -1\right)}{{y_{35}}\left({y_{14}} {y_{25}} -{y_{24}} -1\right) }\,
\bigg) 
\, . \nonumber
\end{align}
\end{prop} 


The choice of the specific first integrals above is motivated by the fact that it will induce nice forms for the abelian relations of $\boldsymbol{W}^{GM}_{\hspace{-0.05cm}\boldsymbol{Y}_5}$ we will deal with. For describing the ARs of $\boldsymbol{W}^{GM}_{\hspace{-0.05cm}\boldsymbol{Y}_5}$, it is convenient to introduce a `third component' to the first integral $U_i=\big(U_{i,1} , U_{i,2}\big)$. One sets 
\begin{equation}
U_{i,3}=1+U_{i,1}
-U_{i,2} 
\end{equation}
for $i=1,\ldots,10$ and via straightforward computations, one gets: 
\begin{align}
\label{Eq:Formules-Ui3}
 U_{1,3}
=&
- {y_{24}}  {y_{35}} + {y_{25}} +1
&& 
 U_{6,3}=
\frac{ {y_{13}}  {y_{24}} - {y_{14}} -1}{\left( {y_{13}}  {y_{25}} - {y_{35}} -1\right)  {y_{14}}}
\nonumber
\\ \sk 
 U_{2,3}
=&\,
\frac{- {y_{14}}  {y_{35}} + {y_{13}} +1}{ {y_{13}}}
&& 
 U_{7,3}
= \frac{ {y_{25}} \left( {y_{13}}  {y_{24}} - {y_{14}} -1\right)}{  {y_{24}}\left( {y_{13}}  {y_{25}} - {y_{35}} -1\right)}
\nonumber
\\ \sk 
U_{3,3}
=&\,
- {y_{14}}  {y_{25}} + {y_{24}} +1
&&
U_{8,3}
=
\frac{ {y_{35}} \left( {y_{13}}  {y_{24}} - {y_{14}} -1\right)}{ {y_{13}}  {y_{25}} - {y_{35}} -1}
\\ \sk 
U_{4,3}
=&\,
\frac{- {y_{13}}  {y_{25}} + {y_{35}} +1}{ {y_{35}}}
&&
U_{9,3}
=
\frac{ {y_{13}}  {y_{24}} - {y_{14}} -1}{ {y_{14}}  {y_{25}} - {y_{24}} -1}
\nonumber
\\ \sk 
U_{5,3}
=&\, 
- {y_{13}}  {y_{24}} + {y_{14}} +1
&&
 U_{10,3}
=
\frac{ {y_{13}}  {y_{25}} - {y_{35}} -1}{ {y_{35}}\left( {y_{14}}  {y_{25}} - {y_{24}} -1\right) }
\, .
\nonumber
\end{align}

The most important property which the $U_{i,s}$'s for $i=1,\ldots,10$ and $s=1,2,3$  satisfy is given by the 
\begin{lem}
\label{L:u3}
Up to a sign, any $U_{i,s}$ can be written as the product of 
some factors of the form $\zeta_k^{\pm 1}$, where the $\zeta_k$'s are the affine polynomials on $\boldsymbol{Y}_5=\mathbf C^5$ defined in \eqref{Eq:zeta-i}. 
\end{lem}

The  five first functions $U_i$ define a subweb of 
$\boldsymbol{\mathcal W}^{GM}_{\hspace{-0.05cm}\boldsymbol{Y}_5}$ 
which will be of great interest for us. 
This web will be denoted by $\boldsymbol{W}^{GM}_{\hspace{-0.05cm}\boldsymbol{Y}_5}$, it is  defined by simple 
monomial first integrals: one has 
\begin{equation}
\label{Eq:W+Y5}
\boldsymbol{W}^{+}_{\hspace{-0.05cm}\boldsymbol{Y}_5}
= 
\scalebox{1}{$
\boldsymbol{\mathcal W}
\Bigg(\, \Big(\,  {y_{25}}\, , \,   {y_{24}}  {y_{35}}
\, \Big)
\, , \, 
\bigg(\,  \frac{1}{y_{13}}\, , \, 
\frac{ {y_{14}}  {y_{35}}}{ {y_{13}}}\, \bigg)
\, , \,
  \Big( 
  \,  {y_{24}}
 \, , \,  
    {y_{14}}  {y_{25}}
   \, \Big)
\, , \, 
\bigg(\,   \frac{1}{y_{35}}
\, , \, 
\frac{ {y_{13}}  {y_{25}}}{ {y_{35}}}
\,
\bigg)
\, , \, 
 \Big(\,  {y_{14}}
\, , \, 
  {y_{13}}  {y_{24}}
 \, 
\Big) \,\Bigg)\, .$}
\end{equation}
\begin{center}
$\star$
\end{center}
Essentially all the results to come have been obtained my means of explicit computations using the explicit expressions above for the first integrals of the webs we will consider. For this reason, all our results will be stated for the birational models 
$\boldsymbol{W}^{GM}_{\hspace{-0.05cm}\boldsymbol{Y}_5}$ and 
$\boldsymbol{W}^{+}_{\hspace{-0.05cm}\boldsymbol{Y}_5}$ but are of course valid for their geometric avatars $\boldsymbol{\mathcal W}^{GM}_{\hspace{-0.05cm}\boldsymbol{\mathcal Y}_5}$ and  
$\boldsymbol{\mathcal W}^{+}_{\hspace{-0.05cm}\boldsymbol{\mathcal Y}_5}=
\boldsymbol{\mathcal W}\big( \, \pi_i^+ \,\big)_{i\in [\hspace{-0.05cm}[5]\hspace{-0.05cm}]} $  as well. 

\subsubsection{\bf Ranks}
Determining the virtual ranks of a $c$-codimensional web
$\boldsymbol{\mathcal W}$ defined by rational first integrals depending on $n$  variables, amounts to computing the rank of certain vector spaces defined over 
 $\mathbf C(u_1,\ldots,u_n)$, where the $u_i$'s are to be regarded as the affine coordinates of the generic point of $ \mathbf{C}^n$.

As for the $k$-abelian relations of  $\boldsymbol{\mathcal{W}}$, for any $k = 0, \ldots, c$, they correspond to the solutions of a linear differential system (which is not difficult to make explicit), and which is of finite type when the corresponding virtual rank is finite. In such cases, a computational approach similar to the one described in~\cite[\S1.5]{ClusterWebs} can be employed to determine the $k$-rank of the web.
\sk 

The methods just outlined for computing the virtual or actual $k$-rank of a given web can be implemented in a computer algebra system. This provides computationally effective techniques for determining these invariants.\footnote{Maple worksheets for computing the $k$-ranks of webs in an arbitrary number of variables are available from the author upon request.} Using these, one easily computes all the ranks of the web $\boldsymbol{W}^{GM}_{\hspace{-0.05cm}\boldsymbol{Y}_5}$: 
%
%
\begin{prop}
\label{Prop:Virtual-Ranks}
{\bf 1.} One has 
\begin{align*}
\rho^{\bullet}_0\Big( 
\boldsymbol{W}^{GM}_{\hspace{-0.05cm}\boldsymbol{Y}_5}
\Big)
= &  \,   (15,15,10,1)
&&
\rho^{\bullet}_1\Big( 
\boldsymbol{W}^{GM}_{\hspace{-0.05cm}\boldsymbol{Y}_5}
\Big)
= ( 15, 20, 15, 2 )
&& \rho^{\bullet}_2\Big( 
\boldsymbol{W}^{GM}_{\hspace{-0.05cm}\boldsymbol{Y}_5}
\Big)
=  (5,5,1) 
\\
r_0\Big( 
\boldsymbol{W}^{GM}_{\hspace{-0.05cm}\boldsymbol{Y}_5}
\Big)
=  &   \, 25
&&
r_1\Big( 
\boldsymbol{W}^{GM}_{\hspace{-0.05cm}\boldsymbol{Y}_5}
\Big)
=   35 
&& r_2\Big( 
\boldsymbol{W}^{GM}_{\hspace{-0.05cm}\boldsymbol{Y}_5}
\Big)
=   11 \, .
\end{align*}

In particular,  Gelfand-MacPherson's web $\boldsymbol{W}^{GM}_{\hspace{-0.05cm}\boldsymbol{Y}_5}$ has AMP 2-rank. \sk 

\noindent {\bf 2.}  One has 
\begin{align*}
\rho^{\bullet}_0\Big( 
\boldsymbol{W}^{+}_{ \hspace{-0.05cm} {\mathcal Y}_5}
\Big)
= &  \,   (5)
&&
\rho^{\bullet}_1\Big( 
\boldsymbol{W}^{+}_{ \hspace{-0.05cm} {\mathcal Y}_5}
\Big)
= ( 5,1 )
&& \rho^{\bullet}_2\Big( 
\boldsymbol{W}^{+}_{ \hspace{-0.05cm}  {\mathcal Y}_5}
\Big)
=  (1) 
\\
r_0\Big( 
\boldsymbol{W}^{+}_{\hspace{-0.05cm}\boldsymbol{Y}_5}
\Big)
=  &   \, 5
&&
r_1\Big( 
\boldsymbol{W}^{+}_{\hspace{-0.05cm}\boldsymbol{Y}_5}
\Big)
=   6
&& r_2\Big( 
\boldsymbol{W}^{+}_{\hspace{-0.05cm}\boldsymbol{Y}_5}
\Big)
=   1\, .
\end{align*}

In particular,  the 5-web $\boldsymbol{W}^{+}_{\hspace{-0.05cm}\boldsymbol{Y}_5}$ has AMP k-rank for $k=0,1,2$. \sk 
\end{prop}

\subsubsection{\bf The master 2-abelian relation $\boldsymbol{{\bf HLOG}_{ \boldsymbol{Y}_5}}$ of $\boldsymbol{W}^{GM}_{\hspace{-0.05cm}\boldsymbol{Y}_5}$}
\label{SSS:The-master-2-AR}
Given three variables $u_1,u_2$ and $u_3$ with the last one expressing 
in terms of the first two 
by $u_3=1+u_1-u_2$, one considers the following 'logarithmic' 2-form on $\mathbf C^3$
\begin{equation}
\label{Eq:Omega-0}
\Omega=\ln  u_{1}\,
d \ln  u_{2}\wedge 
d \ln  u_{3}
-
\ln  u_{2}\,
d \ln  u_{1}\wedge 
d \ln  u_{3}
+\ln  u_{3}\, 
d \ln  u_{1}\wedge 
d \ln  u_{2}
\end{equation}
that is
\begin{equation}
\label{Eq:Omega-u1u2u3}
\Omega=
\ln  u_{1}\, \left( 
\frac{ \mathit{du_2} \wedge  \mathit{du_3}
}{\mathit{u_2} \mathit{u_3}}
\right) 
-\ln    u_{2}\,   \left( 
\frac{ \mathit{du_1} \wedge  \mathit{du_3}
}{\mathit{u_1} \mathit{u_3}}
\right) 
+
\ln    u_{3}\, \left( 
\frac{ \mathit{du_1} \wedge  \mathit{du_2}
}{\mathit{u_1} \mathit{u_2}}
\right) \, , 
\end{equation}
%
or even more explicitly, setting $x=u_1$ and $u_2=y$: 
\begin{align}
\Omega= &\, 
\left( 
-\frac{\ln  u_{1}}{u_{2} u_{3}}
+\frac{\ln u_{2}}{u_{1} u_{3}}+\frac{\ln  u_{3}}{u_{1} u_{2}}
\right) du_1\wedge du_2
\nonumber
\\
= 
&\,
\left( 
\frac{\ln  \left(1+x-y\right)}{xy}
+\frac{\ln  \left(y\right)}{x(1+x-y) }
-\frac{\ln  \big(x\big)}{y(1+x-y) }
\right) dx\wedge dy\, .
\label{Eq:Omega-xy}
\end{align}


Let $I$ be the hypercube formed by 5-tuples $(y_i)_{i=1}^5 \in ]0,1[^5$ and denote by 
$\mathcal U$ an arbitrary but fixed open domain containing it. 
Since the five pfaffian determinants $P_1,\ldots,P_5$ ({\it cf.}\,\eqref{Eq:Pi}) are positive on $I$, it follows that the same occurs for the $U_{i,s}$ for all $i=1,\ldots,10$ and $s=1,2,3$. 

One sets 
\begin{equation}
\label{Eq:epsilon}
\big(\epsilon_i\big)_{i=1}^{10}=\big(1, -1, 1, -1, 1, 1, -1, 1, -1, 1\,\big)\, .
\end{equation}

\begin{prop}
\label{Eq:AR-HLOG-Y5}
The following differential relation is identically satisfied on $\mathcal U$: 
$$
\boldsymbol{\big( {\bf HLOG}_{ {\bf Y}_5}\big)}
\hspace{4cm}
\sum_{i=1}^{10} 
 \epsilon_i\,U_{i}^*\Big( \Omega \Big)=0\, .
\hspace{7cm} {}^{} 
%
%
%
%
$$
Consequently,  the 10-tuple  $\big( \,\epsilon_i U_{i}^*(\Omega )\, \big)_{i=1}^{10}$ can be seen as a  2-abelian relation for $\boldsymbol{W}^{GM}_{{\bf Y}_5}$, again denoted by $ {\bf HLOG}_{ \boldsymbol{Y}_5}$. \sk 
\end{prop}
\begin{proof}
 By formal elementary computations, one can express the scalar components of the sum $\Theta=\sum_{i=1}^{10} 
  \epsilon_i\, U_{i}^*\big( \Omega \big)$ in the basis $dy_i\wedge dy_j$ with $i,j$ such that $1\leq i<j\leq 5$. For instance, the $dy_1\wedge dy_2$-component 
 $\Theta_{12}=\big( \partial_{y_1}\wedge \partial_{y_2} \big)\, \lrcorner \, \Theta$
 is  a sum of terms of the form $R(y) \ln M$ with $R(y)\in \mathbf C(y)$ and where $M$ is a monomial in the $\zeta_k$'s defined in \eqref{Eq:zeta-i}. Since these quantities are positive on $I$, using the functional equation of the logarithm, one can express  $\Theta_{12}$ as a linear combination in the $\ln \zeta_k$'s with coefficients in $\mathbf C(y)$. Straigtforward formal computations give that all these coefficients actually vanish, hence $\Theta_{12}\equiv 0$ on $I$.  Proceeding similarly for all the other components of $\Theta$, one gets that it vanishes identically on $I$ hence on the complex domain $\mathcal U$.
\end{proof}

From the presence of logarithms in the definition of $\Omega$, it follows that  $ {\bf HLOG}_{ \boldsymbol{Y}_5}$ is not a global AR for $\boldsymbol{W}^{GM}_{\boldsymbol{Y}_5}$ but a multivalued one, with additive monodromy.  Hence  one has to be a bit careful when wondering about the invariance property of `the abelian relation' $ {\bf HLOG}_{ \boldsymbol{Y}_5}$ under the 
birational action of $W_{D_5}$ on $\boldsymbol{Y}_5$.  There are two approaches for circumventing this non important technical issue. 
\sk

 The first one is to restrict to the reals and to deal with the real web  $\widetilde{\boldsymbol{W}}^{GM}_{\boldsymbol{Y}_5}$ on $\widetilde{\boldsymbol{Y}}_5^*=\mathbf R^5\setminus Z$. One considers the following real-analytic version of $\Omega$: 
 \begin{align*}
  \Omega^\omega= &\, 
 \ln  {\big| \,\mathit{u_1}\,  \big|} \left( 
\frac{ \mathit{du_2} \wedge  \mathit{du_3}
}{\mathit{u_2} \mathit{u_3}}
\right) 
-\ln  {\big | \,\mathit{u_2}\,  \big|} \left( 
\frac{ \mathit{du_1} \wedge  \mathit{du_3}
}{\mathit{u_1} \mathit{u_3}}
\right) 
+
\ln  {\big | \,  \mathit{u_3} \, \big |} \left( 
\frac{ \mathit{du_1} \wedge  \mathit{du_2}
}{\mathit{u_1} \mathit{u_2}}
\right) 
\\
=
& \, 
\left( \, 
\frac{\ln  \,\lvert \,1+x-y\,\lvert}{xy}
+\frac{\ln  \,\lvert \,y\,\lvert }{x(1+x-y) }
-\frac{\ln  \,\lvert \,x\, \lvert }{y(1+x-y) }
\, \right) dx\wedge dy\, . 
 \end{align*}
 
 The scalar components of $\sum_{i=1}^{10} 
 U_{i}^*\big( \,\Omega^\omega\,\big)$ are linear combinations
 of the quantities $\ln\lvert \zeta_k \lvert$'s with coefficients in $\mathbf R(y)$ which can easily be computed. One obtains the 
\begin{prop}
 The following differential relation 
$$
\boldsymbol{\Big( {\bf HLOG}_{ \boldsymbol{Y}_5}^\omega\Big)}
\hspace{3.8cm}
\sum_{i=1}^{10} 
  \epsilon_i\, U_{i}^*\Big( \,\Omega^\omega\,\Big)=0
\hspace{7cm} {}^{} 
$$
 is identically satisfied on $\boldsymbol{Y}_5$ 
 hence 
 $\Big(  \epsilon_i\, U_{i}^*\big(\Omega^\omega \big)\Big)_{i=1}^{10}$, 
which will be again denoted by $ {\bf HLOG}^\omega_{ \boldsymbol{Y}_5}$, can be considered as a global real-analytic 2-AR either for 
${\boldsymbol{W}}^{GM}_{\boldsymbol{Y}_5}$ or for its real version 
$\widetilde{\boldsymbol{W}}^{GM}_{\boldsymbol{Y}_5}$.
\end{prop}
Since it is global, it now does make sense to 
consider the invariance properties of 
  $ {\bf HLOG}^\omega_{ \boldsymbol{Y}_5}$
with respect to the action of $W_{D_5}
\simeq  \boldsymbol{\mathscr W}_{\hspace{-0.1cm}D_5}$ on $\boldsymbol{Y}_5$ or on $\widetilde{\boldsymbol{Y}}_5$. But this is also possible 
when sticking to the holomorphic setting, 
 following a similar approach to the one described in \cite[p.\,96]{PirioWM06}. 
Indeed, as it follows from the description of the ARs of 
$\boldsymbol{W}^{GM}_{\boldsymbol{Y}_5}$ (see further), 
 the local system $\boldsymbol{\mathcal A\mathcal R}^2=\boldsymbol{\mathcal A\mathcal R}^2\big( \boldsymbol{W}^{GM}_{\boldsymbol{Y}_5} \big)$ of 2-ARs of ${\boldsymbol{W}}^{GM}_{\boldsymbol{Y}_5}$ (on 
$\boldsymbol{Y}^*_5$)  
 admits a 1-step `weight filtration' 
$F^\bullet  \boldsymbol{\mathcal A\mathcal R}^2$, with 
 $F^1  \boldsymbol{\mathcal A\mathcal R}^2= \boldsymbol{\mathcal A\mathcal R}^2$ and where $F^0 \boldsymbol{\mathcal A\mathcal R}^2=
\boldsymbol{AR}_{Rat}^2\big( \boldsymbol{W}^{GM}_{\boldsymbol{Y}_5} \big)
$ is the vector space of rational ARs of $\boldsymbol{W}^{GM}_{\boldsymbol{Y}_5}$. Then ${\rm Gr}^1  \boldsymbol{\mathcal A\mathcal R}^2$ is of dimension 1 and the multivalued abelian relation $ {\bf HLOG}_{ \boldsymbol{Y}_5}$ gives rise to a generator of this space.  It is easy to see that the birational action of $W_{D_5}$ on 
$\boldsymbol{Y}_5$ induces a linear action on 
the associated graded space 
$${\rm Gr}^\bullet  \boldsymbol{\mathcal A\mathcal R}^2=\boldsymbol{AR}_{Rat}^2\big( \boldsymbol{W}^{GM}_{\boldsymbol{Y}_5} \big)
\oplus {\rm Gr}^1  \boldsymbol{\mathcal A\mathcal R}^2\, .$$ 

Let us now discuss how the Cremona transformations $\sigma_i$ acts on 
${\bf HLOG}_{ \boldsymbol{Y}_5}$

\begin{prop}
\label{Prop:Transformation-sigma-i}
{\bf 1.} As elements of $\boldsymbol{AR}^2\big( \boldsymbol{W}^{GM}_{\boldsymbol{Y}_5} \big)$, one has $\sigma_i^*
\big(  {\bf HLOG}^\omega_{ \boldsymbol{Y}_5}\big)=-  \, {\bf HLOG}^\omega_{ \boldsymbol{Y}_5}$
for any $i=1,\ldots,5$.  Hence 
$ {\bf HLOG}^\omega_{ \boldsymbol{Y}_5}$ spans a 1-dimensional non-trivial $W_{D_5}$-subrepresentation 
of $\boldsymbol{AR}^2\big( \boldsymbol{W}^{GM}_{\boldsymbol{Y}_5} \big)$ 
which necessarily is the signature representation. 
\sk 

\noindent 
{\bf 2.} There is a similar statement for the linear action of $W_{D_5}$ on 
${\rm Gr}^\bullet  \boldsymbol{\mathcal A\mathcal R}^2$: the complex line 
${\rm Gr}^1  \boldsymbol{\mathcal A\mathcal R}^2$ spanned by ${\bf HLOG}_{ \boldsymbol{Y}_5}$ is $W_{D_5}$-stable and is the signature representation. 
\end{prop}
\begin{proof}
For $i=1,\ldots,5$, one 
sets $ {\bf HLOG}_{i,+}^{\omega}=
\big({U_{i,1}}\big)^*\big( \Omega^\omega
\big)$  and $
{\bf HLOG}_{i,-}^{\omega}=
\big({U_{i,2}}\big)^*\big( \Omega^\omega
\big)
$. 
Then by direct elementary computations, using the explicit expressions of the $\sigma_i$'s given above, 
one easily establish the following transformation formulas: 
\begin{itemize}
\item for $k\in \{1,2,3,4\}$, one sets $\nu_k=(k,k+1)\in \mathfrak S_5$. Then 
for $ i=1,\ldots,5$  and 
$\epsilon=\pm$, 
one has: 
\begin{equation}
\label{Eq:Transfoformulas_1}
\sigma_k^*\big( {\bf HLOG}_{i,\epsilon}^{\omega} \big) =- 
{\bf HLOG}_{\nu_k(i),\epsilon}^{\omega} \, ; 
\end{equation}
\item let $\nu_5$ be the transposition exchanging $4$ and $5$ (i.e.\,$\nu_5=\nu_4$). Then 
for $j\in \{1,2,3\}$, $\ell\in \{4,5\}$ and $\epsilon=\pm$, 
one has:
\begin{equation}
\label{Eq:Transfoformulas_2}
\sigma_5^*\big( {\bf HLOG}_{j,\epsilon}^{\omega} \big)=-
{\bf HLOG}_{j,\epsilon}^{\omega} 
\qquad \mbox{ and } \qquad 
\sigma_5^*\big( {\bf HLOG}_{\ell,\epsilon}^{\omega} \big)=-
{\bf HLOG}_{\nu_5(\ell),-\epsilon}^{\omega} \, 
\end{equation}
\end{itemize}
It follows that $\langle {\bf HLOG}_{\boldsymbol{Y}_5}^{\omega}\rangle$ is $W_{D_5}$-stable. Since this 1-dimensional representation is not trivial, it has to be the signature. 

The proof of the second point of the proposition 
is essentially similar. 
\end{proof}

\begin{rem}
The components of ${\bf HLOG}_{ \boldsymbol{Y}_5}$ are linear combinations of logarithms of rational functions, multiplied by wedge products of total derivatives of terms of the same type.  Such differential forms have already appeared in earlier works on polylogarithms and the functional equations they satisfy. For instance, see our previous work \cite{PirioWM06} on the curvilinear webs defined by the $n + 3 $ forgetful maps on the moduli spaces $\mathcal M_{0,n+3}$ (in particular, refer to formulas (5.61) and (5.65) therein). See also \cite{KLL}\footnote{More precisely, see the first line of page 157 in \cite{KLL}} where polylogarithms are studied from the perspective of the theory of algebraic cycles and reciprocity laws.

It would be of interest to gain a deeper understanding of the nature of such differential forms and to clarify why they naturally emerge in the context of functional identities and abelian relations.
\end{rem}

\subsection{\bf The rational 2-abelian relations of $\boldsymbol{{\mathcal W}^{GM}_{ {Y}_5}}$}
\label{SS:2-AR-of-WGMY5}
In this subsection, we give an explicit description of 10 linearly independent 
rational 
2-abelian relations  of 
$\boldsymbol{\mathcal W}^{GM}_{ \boldsymbol{Y}_5}$
which  are obtained as residues of ${\bf HLOG}_{\boldsymbol{Y}_5}$. 
Then considering Proposition \ref{Prop:Virtual-Ranks}, it will follow that together with 
${\bf HLOG}_{\boldsymbol{Y}_5}$, 
these 10 rational ARs  form a basis of  
 $\boldsymbol{AR}^2\big( 
\boldsymbol{\mathcal W}^{GM}_{ \boldsymbol{Y}_5} \big)$.

\subsubsection{\bf Residues of ${\bf HLOG}_{ \boldsymbol{Y}_5}$}
\label{SS:Residues-HLOG-Y5}
Recall that $\mathbf C(y)$ stands for the field of rational functions in the variables $y_i$ for $i=1,\ldots,5$, with $y_1=y_{13}$, $y_2=y_{14}$, $y_3=y_{24}$, $y_4=y_{25}$ and $y_5=y_{35}$.   Given  $F$, an irreducible non-constant polynomial in the $y_i$'s, we denote by $\nu_F$ the associated valuation $\mathbf C(y)\setminus \{0\}\rightarrow \mathbf Z$.  Then for $V_1,V_2,V_3\in \mathbf C(y)$, we set 
$$
{\bf Res}_F \Big( \ln \big(V_1\big)\, dV_2\wedge dV_3\Big)= \nu_F(V_1)\,dV_2\wedge dV_3
\in \Omega^2_{ \mathbf C(y)}\, , 
$$
and we claim that this 
definition makes sense, that is that the RHS 
is 
independent of the determination of $ \ln \big(V_1\big)$ taken in the LHS (its is 
an easy exercice, left to the reader).  We call $\nu_F(V_1)\,dV_2\wedge dV_3$ {\it the residue of 
the logarithmic 2-form $\ln \big(V_1\big)\, dV_2\wedge dV_3$ along the hypersurface cut out by $F$}.\footnote{The residue ${\bf Res}_F \Big( \ln \big(V_1\big)\, dV_2\wedge dV_3\Big)$ can also be defined topologically, in terms of the monodromy of 
$\ln \big(V_1\big)\, dV_2\wedge dV_3$ along a small loop around $\{\,F=0\,\}$, see \cite[\S2.2.6]{PirioAFST} 
for more details on this approach.}

We aim to take residues of (the components of) ${\bf HLOG}_{ \boldsymbol{Y}_5}$ with respect to the 10 irreducible components of the divisor $Z\subset \boldsymbol{Y}_5$ defined in \eqref{Eq:Def-Z}.   For any $i=1,\ldots,10$, one denotes by ${\bf HLOG}_i$ the $i$-th component of ${\bf HLOG}_{\boldsymbol{Y}_i}$, namely 
\begin{equation}
{\bf HLOG}_{i}= \epsilon_i\, U_i^*\big( \Omega\big)= \epsilon_i\,\left(
\ln U_{i,1}  \left( 
\frac{dU_{i,2}\wedge dU_{i,3}}{ U_{i,2}U_{i,3}}
\right)-
\ln U_{i,2}
 \left( 
\frac{dU_{i,1}\wedge dU_{i,3}}{ U_{i,1}U_{i,3}}
\right)
+ \ln U_{i,3} \left( 
\frac{dU_{i,1}\wedge dU_{i,2}}{ U_{i,1}U_{i,2}}
\right)\right)\, .
\end{equation}

From the topological definition of the residue (cf. the footnote below), it follows that for any non-constant irreducible polynomial $F$, the 10-tuple of residues 
\begin{equation}
{\bf Res}_F\big( {\bf HLOG}_{\boldsymbol{Y}_5} \big) 
=\Big( \, 
{\bf Res}_F\big( {\bf HLOG}_{i} \big) 
\, 
\Big)_{i=1}^{10}
\end{equation}
with 
$$
{\bf Res}_F \Big( {\bf HLOG}_i\Big)= 
 \nu_F\big( U_{i,1}) \left( 
\frac{dU_{i,2}\wedge dU_{i,3}}{ U_{i,2}U_{i,3}}
\right)
-
 \nu_F\big( U_{i,2}) \left( 
\frac{dU_{i,1}\wedge dU_{i,3}}{ U_{i,1}U_{i,3}}
\right)
+ \nu_F\big( U_{i,3}) \left( 
\frac{dU_{i,1}\wedge dU_{i,2}}{ U_{i,1}U_{i,2}}
\right)
$$
for any $i$,  belongs to $\boldsymbol{AR}_{Rat}^2\big( \boldsymbol{\mathcal W}_{ \boldsymbol{Y}_5}^{GM}\big)$. Specializing $F$ by taking for it one of the ten polynomials $\zeta_k$ defined in \eqref{Eq:zeta-i}, one gets 10 rational abelian relations
$${\bf Res}_{i}= {\bf Res}_{\zeta_i}=
{\bf Res}_{\zeta_i}\Big( {\bf HLOG}_{ \boldsymbol{Y}_5}\Big)\, , \quad 
i=1,\ldots,10
$$
which can be easily computed from the explicit expressions \eqref{Eq:Formules-Ui} and \eqref{Eq:Formules-Ui3} for the $U_{i,s}$.  For instance, one  gets that ${\bf Res}_{y_1}={\bf Res}_{y_1}\big( {\bf HLOG}_{ \boldsymbol{Y}_5}\big)$ corresponds to the following differential relation
\begin{align}
\label{Eq:Res-y1}
 U_2^*\Bigg( 
\bigg( 
\frac{du_{2}\wedge du_{3}}{ u_{2}u_{3}}
\bigg)-  & \, \bigg( 
\frac{du_{3}\wedge du_{1}}{ u_{3}u_{1}}
\bigg)
-\bigg( 
\frac{du_{1}\wedge du_{2}}{ u_{1}u_{2}}
\bigg)
\Bigg) +U_4^*\left(  
\frac{du_{3}\wedge du_{1}}{ u_{3}u_{1}}
\right) 
\\
+ & \, 
U_5^*\left(  
\frac{du_{3}\wedge du_{1}}{ u_{3}u_{1}}
\right) +
U_6^*\left(  
\frac{du_{3}\wedge du_{1}}{ u_{3}u_{1}}
\right) +
U_8^*\left(  
\frac{du_{2}\wedge du_{3}}{ u_{2}u_{3}}
\right)=0 
\nonumber
\end{align}
which can be directly verified to be  identically satisfied. 

By straightforward computations, it can be verified that 
one has $\rho^2\big( \boldsymbol{\mathcal W}_5\big) \leq 1$ 
for any 5-subweb 
$\boldsymbol{\mathcal W}_5$ of 
$\boldsymbol{\mathcal W}_{\boldsymbol{Y}_5}^{GM}$.  By analogy with the terminology introduced in \cite{Damiano}, any 2-abelian relation of such a 5-subweb  
$\boldsymbol{\mathcal W}_5$ will be said to be {\it `combinatorial'} and we will denote the space spanned by the combinatorial ARs by $\boldsymbol{AR}_C\big(\boldsymbol{\mathcal W}_{\boldsymbol{Y}_5}^{GM}\big)$. From above, it follows that ${\bf Res}_{y_1}$ is combinatorial. 
\mk 

By straightforward computations, one makes the 10 residues 
 ${\bf Res}_{i}$ entirely explicit from which one first deduces that all these ARs are combinatorial and rational. Using the  explicit formulas obtained for the residues ${\bf Res}_{i}$ together with the ones for the automorphisms of webs 
$\sigma_k$'s given in \eqref{Eq:Formula-Sigma-k}, it is straightforward (however a bit laborious) to compute all the pull-backs $\sigma_k^*({\bf Res}_{i})$ for 
$k=1,\ldots,5$ and  $i=1,\ldots,10$. One obtains that any $\sigma_k^*\big( {\bf Res}_{i}\big)$ is a linear combination (with non-zero coefficients $\pm 1$) of the ten residues of ${\bf HLOG}_{ \boldsymbol{Y}_5}$. For instance, one has 
%
%
$$
\sigma_1^*\Big( {\bf Res}_{y_1}
 \Big)
=-{\bf Res}_{y_2}+{\bf Res}_{y_4}+{\bf Res}_{y_5}+{\bf Res}_{P_3}+{\bf Res}_{P_{5}}
$$

More generally, by means of straightforward explicit computations, we get the 
\begin{prop}
{\bf 1.} The residues  ${\bf Res}_{i}\big( {\bf HLOG}_{ \boldsymbol{Y}_5}\big)$ ($i=1,\ldots,10$) all 
have exactly 5 non-trivial components hence 
belong to $\boldsymbol{AR}^2_C\Big( \boldsymbol{\mathcal W}_{Y_5}^{GM}\Big)$. Moreover, they form a  basis of this space, which coincides with the space 
$\boldsymbol{AR}^2_{Rat}\Big( \boldsymbol{\mathcal W}_{Y_5}^{GM}\Big)$ of rational ARs of  $\boldsymbol{\mathcal W}_{\boldsymbol{Y}_5}^{GM}$.

\noindent {\bf 2.} The birational maps $\sigma_k$'s induce linear automorphisms of $\boldsymbol{AR}^2_C\Big( \boldsymbol{\mathcal W}_{Y_5}^{GM}\Big)$ making of this space a 
$W_{D_5}$-representation.  As such, it is irreducible and isomorphic to $V^{10}_{[11,111]}$. 
\end{prop}
\begin{proof} The proof goes by explicit computations.\footnote{Another less computational proof  could have be given, using the action of $W_{D_5}\simeq \boldsymbol{\mathscr W}_{\hspace{-0.05cm}D_5}$ on the set of combinatorial abelian relations.} 
Since all the residues ${\bf Res}_i$'s have five non-zero terms which are rational, 
 they are  combinatorial 2-abelian relations which span a subspace of $\boldsymbol{AR}\big(\boldsymbol{\mathcal W}_{\boldsymbol{Y}_5}^{GM}\big)$ in direct sum with the line spanned by ${\bf HLOG}_{ \boldsymbol{Y}_5}$. The ${\bf Res}_i$'s being linearly independent, it follows from 
Proposition \ref{Prop:Virtual-Ranks} that $\boldsymbol{AR}_C\big(\boldsymbol{\mathcal W}_{Y_5}^{GM}\big)$ has dimension 10 and admits $\boldsymbol{\mathfrak R}=\big( {\bf  Res}_{s}\big)_{s=1}^{10}$ as a basis. This proves the first point of the proposition. 
\sk

Computing all the pull-backs $\sigma_k^*\big( {\bf Res}_i\big)$, one first obtains that they all have exactly five non-zero terms hence are combinatorial 2-abelian relations hence are linear combinations of the $ {\bf Res}_i$'s. It is then straightforward to get the explicit forms for the matrices of the $\sigma_k^*\in {\rm GL}\big( \boldsymbol{AR}^2_C\big( \boldsymbol{\mathcal W}_{\boldsymbol{Y}_5}^{GM}\big)\big)$ expressed in the basis  $\boldsymbol{\mathfrak R}$. For instance, one gets that 
$$
{\rm Mat}_{\boldsymbol{\mathfrak R}}\big({\sigma_1}^*\big)=
\scalebox{0.8}{
$\left[\begin{array}{cccccccccc}
0 & 0 & 0 & 1 & 0 & 0 & 0 & 0 & 0 & 0 
\\
 -1 & 0 & -1 & 0 & 0 & 0 & 0 & 0 & 0 & 0 
\\
 0 & -1 & 0 & -1 & 0 & 0 & 0 & 0 & 0 & 0 
\\
 1 & 0 & 0 & 0 & 0 & 0 & 0 & 0 & 0 & 0 
\\
 1 & 0 & 0 & 1 & -1 & 0 & 0 & 0 & 0 & 0 
\\
 0 & 0 & 0 & 0 & 0 & -1 & 0 & 0 & 0 & 0 
\\
 0 & 0 & 0 & 1 & 0 & 0 & 0 & 0 & 0 & -1 
\\
 1 & 0 & 0 & 1 & 0 & 0 & 0 & -1 & 0 & 0 
\\
 0 & 0 & 0 & 0 & 0 & 0 & 0 & 0 & -1 & 0 
\\
 1 & 0 & 0 & 0 & 0 & 0 & -1 & 0 & 0 & 0 
\end{array}\right]
$}\,.
$$
%
%
%
Knowing explicitly the matrices ${\rm Mat}_{\boldsymbol{\mathfrak R}}\big({\sigma_k}^*\big)$ for $k=1,\ldots,5$ and proceeding completely similarly as  in the Appendix of \cite{Pirio2022}, one computes the character of the representation of $W_{D_5}$ on $\boldsymbol{AR}^2_C\big( \boldsymbol{\mathcal W}_{Y_5}^{GM}\big)$: one get that this character is 
$$\chi=\big(\,  10, -2, 2, -4, 2, 0, -2, 2, -2, 0, 1, -1, 1, -1, 1, 0, 0, 0\, \big)\, . $$ 
It corresponds to the first line of the characters table of the Weyl group of type $D_5$ given in 
\cite[Table 4]{Pirio2022}, which gives us the second point of the proposition. 
\end{proof}

\begin{rem} 
In \cite[\S4.2]{PirioAFST}, it has been established that one also has 
$${\bf HLogAR}^2_{\rm asym}\simeq  V^{10}_{[1^2,1^3]} $$
as $W_{D_5}$-representations,  where the left hand side ${\bf HLogAR}^2_{\rm asym}$ stands for  the space of symbolic antisymmetric weight 2 AR of a del Pezzo web $\boldsymbol{\mathcal W}_{ {\rm dP}_4}$ of a quartic del Pezzo surface ${\rm dP}_4$. As it will be explained further, this is no mere coincidence. 
\end{rem}

From the results above completed by some computational checks, we deduce the following
\begin{cor}
\label{Cor:2-RA-WGM}
{\bf 1.} 
One has $\rho^2\big( \boldsymbol{\mathcal W}_5\big)\leq 1$ for any 5-subweb $\boldsymbol{\mathcal W}_5$ of $\boldsymbol{\mathcal W}_{\boldsymbol{Y}_5}^{GM}$. Those 
for which the virtual 2-rank is 1 actually are of maximal 2-rank 1. These subwebs are exactly the 
$ \boldsymbol{\mathcal W}^{\underline{\epsilon}}= 
\boldsymbol{\mathcal W}\big( \mathcal F_1^{{\epsilon}_1},\ldots, \mathcal F_5^{{\epsilon}_5}\big)$,  for all 5-tuples $\underline{\epsilon}=(\epsilon_i)_{i=1}^5 \in \{\pm 1\}^5$ of even parity (i.e. such that $\epsilon_1\cdots \epsilon_5=+1${\rm )}.

For each even $\underline{\epsilon}$, the space of 2-ARs of $ \boldsymbol{\mathcal W}^{\underline{\epsilon}}$ is spanned by a peculiar abelian relation ${\bf AR}^{\underline{\epsilon}}$, uniquely defined up to sign, whose components with respect to  each first integral $U_j=(U_{j,s})_{s=1}^3$ of $ \boldsymbol{\mathcal W}^{\underline{\epsilon}}$ is a linear combination, with coefficients $\pm 1$, of the rational 2-forms $$
{}^{} \quad d\,{\rm Log}\,U_{j,a}\wedge d\,{\rm Log}\,U_{j,b}= 
\frac{dU_{j,a}\wedge dU_{j,b}}{U_{j,a}U_{j,b}}
\qquad 
\big(\,1\leq a<b\leq 3\,)\, .
$$

Moreover, 
$\boldsymbol{AR}^2_C\Big( \boldsymbol{\mathcal W}_{\boldsymbol{Y}_5}^{GM}\Big)=
\big\langle \,{\bf AR}^{\underline{\epsilon}} \, \big\lvert \, \,
\underline{\epsilon}\in \{\pm1\}^5 \, \mbox{ is even} \,
\, \big\rangle $ coincides with the space $\boldsymbol{AR}^2_{Rat}\Big( \boldsymbol{\mathcal W}_{\boldsymbol{Y}_5}^{GM}\Big)$
of rational 2-ARs of $\boldsymbol{\mathcal W}_{\boldsymbol{Y}_5}^{GM}$ and this space is 10-dimensional: $ \dim \boldsymbol{AR}^2_C\Big( \boldsymbol{\mathcal W}_{\boldsymbol{Y}_5}^{GM}\Big)=10$.
\mk \\
{\bf 2.}
One has 
\begin{equation}
\label{Eq:Decomp-AR2-WGM}
\boldsymbol{AR}^2\Big( \boldsymbol{\mathcal W}_{\boldsymbol{Y}_5}^{GM}\Big)
=\boldsymbol{AR}^2_C\Big( \boldsymbol{\mathcal W}_{\boldsymbol{Y}_5}^{GM}\Big)\oplus \Big\langle \,
 {\bf HLOG}_{ \boldsymbol{Y}_5}\,
\Big\rangle\, . 
\end{equation}
from which it follows that the 2-rank of 
$\boldsymbol{\mathcal W}_{\boldsymbol{Y}_5}^{GM}$ is 11, that is is  AMP. 
\sk \\
{\bf 3.} By residues/monodromy, the abelian relation 
$ {\bf HLOG}_{ \boldsymbol{Y}_5}$
  spans  the subspace $\boldsymbol{AR}^2_C\Big( \boldsymbol{\mathcal W}_{\boldsymbol{Y}_5}^{GM}\Big)$ 
  of combinatorial ARs, which coincides with that of rational 2-ARs of $\boldsymbol{\mathcal W}_{\boldsymbol{Y}_5}^{GM}$: one has 
  $$
  {\bf Res}\big( {\bf HLOG}_{ \boldsymbol{Y}_5} \big) 
  = 
  \boldsymbol{AR}^2_C\Big( \boldsymbol{\mathcal W}_{\boldsymbol{Y}_5}^{GM}\Big)=\boldsymbol{AR}^2_{Rat}\Big( \boldsymbol{\mathcal W}_{\boldsymbol{Y}_5}^{GM}\Big)
  \, .
  $$
\noindent {\bf 4.} The decomposition in direct sum \eqref{Eq:Decomp-AR2-WGM} actually is the decomposition of $\boldsymbol{AR}^2\Big( \boldsymbol{\mathcal W}_{\boldsymbol{Y}_5}^{GM}\Big)$ 
into irreducible $W_{D_5}$-representations. The two pieces $\boldsymbol{AR}^2_C\Big( \boldsymbol{\mathcal W}_{\boldsymbol{Y}_5}^{GM}\Big)$ 
and $\big\langle \,
 {\bf HLOG}_{ \boldsymbol{Y}_5}\,
\big\rangle$ are  isomorphic to $V_{[11,111]}^{10}$ and to the signature representation respectively. 
\mk \\
\end{cor}

\subsubsection{\bf A specific combinatorial 2-abelian relation of 
$\boldsymbol{\boldsymbol{\mathcal W}_{\boldsymbol{Y}_5}^{GM}}$}
For what is to come further, it is interesting to consider the case of the 5-subweb $\boldsymbol{W}(U_1,\ldots,U_5)$  of $\boldsymbol{\mathcal W}_{\boldsymbol{Y}_5}^{GM}$.  This web will be denoted by $\boldsymbol{W}^{GM,+}_{ \boldsymbol{Y}_5}$ and is defined by simple monomial first integrals: 
\begin{equation*}
\boldsymbol{W}^{GM,+}_{ \boldsymbol{Y}_5}
= 
\boldsymbol{\mathcal W}
\Bigg(\, \Big(\,  {y_{25}}\, , \,   {y_{24}}  {y_{35}}
\, \Big)
\, , \, 
\Big(\,  \frac{1}{y_{13}}\, , \, 
\frac{ {y_{14}}  {y_{35}}}{ {y_{13}}}\, \Big)
\, , \,
  \bigg( 
  \,  {y_{24}}
 \, , \,  
    {y_{14}}  {y_{25}}
   \, \bigg)
\, , \, 
\bigg(\,  \frac{1}{y_{35}}
\, , \, 
\frac{ {y_{13}}  {y_{25}}}{ {y_{35}}}
\,
\bigg)
\, , \, 
 \Big(\,  {y_{14}}
\, , \, 
  {y_{13}}  {y_{24}}
 \, 
\Big) \,\Bigg)\, .
\end{equation*}

 According to the fourth point of the above corollary,  
this web  has maximal 2-rank (equal to 1). We aim below to describe very explicitly a generator of the complex line $\boldsymbol{AR}^2\big(\boldsymbol{W}^{GM,+}_{ \boldsymbol{Y}_5})$.
\sk 
 
%
%

Setting
\begin{equation}
\label{Eq:eta}
\eta=
d\,{\rm Log}\,x \wedge d\, {\rm Log}\,y
= \frac{dx}{x}
\wedge 
\frac{dy}{y}=
\frac{dx\wedge dy}{xy}
\,, 
\end{equation}
one verifies without any difficulty that the following proposition is satisfied: 
Recall (from \eqref{Eq:epsilon}) that one has $\epsilon_1=\epsilon_3=\epsilon_5=1$ and 
 $\epsilon_2=\epsilon_4=-1$. 
\begin{prop}
\label{Prop:AR-2-eta}
{\bf 1.} The 5-tuple ${\bf AR}^2_\eta=\Big( \epsilon_i\,U_i^*\big( \eta\big)\Big)_{i=1}^5$ is a rational 2-abelian relation  for $\boldsymbol{W}^{+}_{ \boldsymbol{Y}_5}$, i.e. in the space $\Omega^2_{\mathbf C(y)}$ of rational 2-forms on $\mathbf C^5$, one has 
 $$
\sum_{i=1}^5   \epsilon_i\,U_i^*\left( \frac{dx\wedge dy}{xy}\right)=
\sum_{i=1}^5   \epsilon_i\, \frac{dU_{i,1}\wedge dU_{i,2}}{
U_{i,1}U_{i,2}
}
=
0\, . 
$$
%
Moreover, in terms of the residues of ${\bf HLOG}_{\boldsymbol{Y}_5}$, one has 
$
{\bf AR}^2_\eta
=\sum_{i=1}^5 
{\bf Res}_{P_{i}}
$.

\noindent{\bf 2.} 
Consequently, one has $\boldsymbol{AR}^2\Big( \boldsymbol{W}^{+}_{ \hspace{-0.0cm}\boldsymbol{Y}_5}\Big)=
\big\langle \, {\bf AR}^2_\eta
\, \big\rangle
$. 
\end{prop}
Since $\big( dx\wedge dy\big)/(xy)$ is closed, all the components $U_i^*\big( (dx\wedge dy)/(xy)\big)$ of 
$ {\bf AR}^2_\eta$ are closed as well  hence by Poincar\'e's lemma, the latter abelian relation admits a primitive, at least locally.  Since $\boldsymbol{AR}_C( \boldsymbol{\mathcal W}^{GM}_{ {Y}_5})$ is irreducible as a $W_{D_5}$-representation, this space is spanned by the orbit $W_{D_5}\cdot  {\bf AR}^2_\eta$ which is formed of closed  2-abelian relations. It follows that all the residues ${\bf Res}_i$ are closed (a fact which can also be verified directly) hence locally exact.

We will use this to  construct specific primitives of the residues ${\bf Res}_i$ which will span an interesting subspace of $\boldsymbol{AR}^1\big(\boldsymbol{\mathcal W}^{GM}_{ \boldsymbol{Y}_5}\big)$ (see \eqref{Eq:AR1-C} in Theorem \ref{Thm:1-AR-} below).

\subsubsection{\bf A more intrinsic and abstract approach}
All the results of the previous subsection have been obtained by direct computations. Here we say a few words about a more abstract approach to the 2-ARs of ${\boldsymbol{\mathcal W}_{\boldsymbol{Y}_5}^{GM}}$, one of the key ingredient of which being the action of $W_{D_5}$ on $\boldsymbol{Y}_5$.
\sk

Up to the birational identification $\Theta$, the divisors $\{\,\zeta_k=0\,\}$ in $\boldsymbol{Y}_5$ correspond to 10 of the weight divisors $D_{\mathfrak w}$ in 
$\boldsymbol{\mathcal Y}_5$ hence to some weights and lines in $\mathfrak W^+$ and $\boldsymbol{\mathcal L}_r$ respectively (see the discussion following 
\eqref{Eq:Theta}). These correspondences are given in the table below, where we use the following notations:  we denote by $\big(\boldsymbol{\frac{1}{2}}\big)^5=\big( \frac{1}{2},\frac{1}{2},\frac{1}{2},\frac{1}{2},\frac{1}{2}\big)$ the dominant weight of $\mathfrak W^+$. As elements of ${\rm Pic}_{\mathbf Z}\big({\rm dP}_4\big)$, 
we set 
$\boldsymbol{e}_{tot}=\sum_{k=1}^5  \boldsymbol{e}_k$ 
and $\boldsymbol{e}_{\hat \imath}=\boldsymbol{e}_{tot}-\boldsymbol{e}_i=\sum_{k\neq i} \boldsymbol{e}_k$ for $i\in [\hspace{-0.05cm}[5]\hspace{-0.05cm}]$. Finally, we denote by $(\boldsymbol{v}_k)_{k=1}^5$ the standard basis of $\mathbf R^5$, that is  
$\boldsymbol{v}_i=(\delta_{ij})_{j=1}^5$ for $i=1,\ldots,5$. 

\begin{table}[!h]
\begin{tabular}{|c|c|c|}
\hline
\, 
  {\bf Line} $\boldsymbol{\ell}$ & {\bf Weight}
$\boldsymbol{{\mathfrak w}_{\ell}}$  
& \begin{tabular}{c} \vspace{-0.35cm}\\
  {\bf Polynomial} $\boldsymbol{\zeta_k}$
\vspace{0.1cm}
\end{tabular}  
   \\ \hline 
  $\boldsymbol{e}_{i}$  & $\boldsymbol{v}_i-\big(\boldsymbol{\frac{1}{2}}\big)^5$ 
 &
 \begin{tabular}{c} \vspace{-0.25cm}\\
 ${P_i}$
\vspace{0.2cm}
\end{tabular}   
   \\ \hline
 $\boldsymbol{h}- 
\boldsymbol{e}_{i}-\boldsymbol{e}_{j}$ 
   &  
$  \big(\boldsymbol{\frac{1}{2}}\big)^5-\boldsymbol{v}_i-\boldsymbol{v}_j$ 
&
\begin{tabular}{c} 
\vspace{-0.25cm}\\
 $y_{ij}$
\vspace{0.2cm}
\end{tabular} 
    \\ \hline
\end{tabular}
\end{table}
\sk 

Similarly, the correspondances between the $W$-relevant facets of the weight polytope, the corresponding conic classes and the first integrals of $\boldsymbol{\mathcal W}^{GM}_{\boldsymbol{\mathcal Y}_5}$ are given in the second table just 
below.\footnote{We recall the notation for the $W$-relevant facets of 
 $\Delta_5$: one has 
$\Delta_{5,i}^\epsilon=\Delta_{D_5}\cap \{ \, x_i= \epsilon/2\, \}$ for 
$i\in [\hspace{-0.05cm}[5 ]\hspace{-0.05cm}]$
and $\epsilon= \pm 1$.}
\begin{table}[!h]
\begin{tabular}{|c|c|c|c|}
\hline
 {\bf index} $\boldsymbol{i}$ &  {\bf Facet} & 
\begin{tabular}{c} \vspace{-0.35cm}\\
  {\bf Conic class}
\vspace{0.1cm}
\end{tabular}  
 &  {\bf First integral} 
    \\ \hline 
\begin{tabular}{c} \vspace{-0.35cm}\\
$i\in \{1,\ldots,5\}$ 
\vspace{0.1cm}
\end{tabular}  
 & $ \Delta_{5,i}^+$ &  $2\boldsymbol{h}- 
\boldsymbol{e}_{\hat \imath}$   & $U_i$    \\ \hline 
\begin{tabular}{c} \vspace{-0.35cm}\\
$i\in \{6,\ldots,10\}$ 
\vspace{0.1cm}
\end{tabular}  
 & 
$ \Delta_{5,i}^-$
  &  $\boldsymbol{h}- 
\boldsymbol{e}_{i}$   & $U_{i+5}$    \\ \hline 
\end{tabular}
\end{table}

The variable $\zeta_1=y_1=y_{13}$ corresponds to the divisor associated to the line $
\boldsymbol{\ell}=\boldsymbol{h}-\boldsymbol{e}_1-\boldsymbol{e}_3$. The conic classes adjacent to it are $\boldsymbol{h}-\boldsymbol{e}_1$, $\boldsymbol{h}-\boldsymbol{e}_3$ and the three classes $2\boldsymbol{h}-\boldsymbol{e}_{\widehat k}$  for $k=2,4,5$.  The two former classes correspond to the foliations induced by $U_6$ and $U_8$, the three latter to the ones with $U_2,U_4$ and $U_5$ as first integrals.  This is in accordance with the explicit expression \eqref{Eq:Res-y1} for the residue of  ${\bf HLOG}_{\boldsymbol{Y}_5}$ along the divisor cut out by $y_{1}=0$.  In terms of Gelfand-MacPherson's web on 
${\boldsymbol{\mathcal Y}_5}$, this translates as the fact that 
\begin{equation}
\label{Eq:Res-along-Dl}
\scalebox{0.97}{\begin{tabular}{l}{\it `the residue of 
${\bf HLOG}_{\boldsymbol{\mathcal Y}_5}$ along the weight divisor 
associated to 
${\mathfrak w}_{\boldsymbol{\ell}}=(-\frac12,\frac12,-\frac12,\frac12,\frac12)$ is}\\
{\it   the subweb defined by the face maps $\pi_F: {\boldsymbol{\mathcal Y}_5}\rightarrow 
{\boldsymbol{\mathcal Y}_F}$ for all  facets $F$ of $\Delta_5$ adjacent to ${\mathfrak w}_{\boldsymbol{\ell}}$'.}
\end{tabular}}
\end{equation}

For a line $\boldsymbol{l}\in \boldsymbol{\mathcal L}$, 
let ${\boldsymbol{\mathcal W}_{\boldsymbol{\mathcal Y}_5, 
\boldsymbol{l}}^{GM}}$ be the subweb of 
${\boldsymbol{\mathcal W}_{\boldsymbol{\mathcal Y}_5}^{GM}}$
defined by the face maps $\pi_F: \boldsymbol{\mathcal Y}_5\rightarrow 
\boldsymbol{\mathcal Y}_F$ for all facets $F\subset \Delta_5$ adjacent to $\boldsymbol{l}$.  More formally, setting $\boldsymbol{\mathcal K}(\boldsymbol{l})$ for the subset of conic classes $\boldsymbol{\mathfrak c}\in \boldsymbol{\mathcal K}$ such that $\boldsymbol{\mathfrak c}-\boldsymbol{\ell}\in \boldsymbol{\mathcal L}$, one has 
$$
{\boldsymbol{\mathcal W}_{\boldsymbol{\mathcal Y}_5, 
\boldsymbol{l}}^{GM}}
=\boldsymbol{\mathcal W}\Big( \, \pi_{F_{\boldsymbol{\mathfrak c}}} \, \big\lvert 
\, \boldsymbol{\mathfrak c} \in \boldsymbol{\mathcal K}(\boldsymbol{l})\, \Big)\, . 
$$

For instance, the web 
$\boldsymbol{W}^{GM,+}_{ \boldsymbol{Y}_5}$ considered in the previous subsection coincides with 
${\boldsymbol{\mathcal W}_{\boldsymbol{\mathcal Y}_5, 
2\boldsymbol{h}-\boldsymbol{e}_{tot}}^{GM}}$.

\begin{prop}
{\bf 1.}  The sixteen 5-subwebs of ${\boldsymbol{\mathcal W}_{\boldsymbol{\mathcal Y}_5}^{GM}}$ with maximal 2-rank (equal to 1) are precisely the subwebs 
${\boldsymbol{\mathcal W}_{\boldsymbol{\mathcal Y}_5, \boldsymbol{\ell}}^{GM}}$'s for all 
lines $\boldsymbol{\ell}\in \boldsymbol{\mathcal L}$. \sk

\noindent 
{\bf 2.}  Moreover, for any line $\boldsymbol{\ell}$, the residue 
$ {\bf Res}_{{\boldsymbol{\ell}}} \big( 
{\bf HLOG}_{\boldsymbol{\mathcal Y}_5}
\big)$
of ${\bf HLOG}_{\boldsymbol{\mathcal Y}_5}$ along the divisor $D_{\boldsymbol{\ell}}$, 
denoted by ${\bf Res}_{\boldsymbol{\ell}}$,  
is a non-trivial AR for  ${\boldsymbol{\mathcal W}_{\boldsymbol{\mathcal Y}_5, \boldsymbol{\ell}}^{GM}}$. In other terms,  
one has 
$
\big\langle \, 
{\bf Res}_{\boldsymbol{\ell}}
\,  \big\rangle =\boldsymbol{AR}^2\Big( {\boldsymbol{\mathcal W}_{\boldsymbol{\mathcal Y}_5, \boldsymbol{\ell}}^{GM}} \Big) $. 
\sk

\noindent 
{\bf 3.}  For any $\boldsymbol{\ell}\in \boldsymbol{\mathcal L}$, the following linear  relation holds true in $\boldsymbol{AR}^2_C\big( {\boldsymbol{\mathcal W}_{\boldsymbol{\mathcal Y}_5}^{GM}} \big) $: 
\begin{equation*}
\boldsymbol{\big( {\mathcal R}{el}_{\ell}\big)} \hspace{4cm}  
{\bf Res}_{\boldsymbol{\ell}}=\sum_{
\boldsymbol{\mathfrak c}
\in \boldsymbol{\mathcal K}(\boldsymbol{\ell})
 } {\bf Res}_{\boldsymbol{\mathfrak c}-\boldsymbol{\ell}}\, . 
 \hspace{7cm}{}^{} 
\end{equation*}
Moreover, for any exceptional collection\footnote{An {\it `exceptional collection'} is a subset of 
$\boldsymbol{\mathcal L}$ formed by five pairwise disjoint lines.} $\boldsymbol{\mathcal E}\subset 
\boldsymbol{\mathcal L}$, 
the set $\big\{ \, 
\boldsymbol{ {\mathcal R}{el}_{\boldsymbol{l}}}
\,  \big\}_{ \boldsymbol{l} \in \boldsymbol{\mathcal E}}$ is a basis of the space of linear relations between the 16 residues of 
${\bf HLOG}_{\boldsymbol{\mathcal Y}_5}$. 
\sk

\noindent  
 {\bf 4.} For any exceptional collection $\boldsymbol{\mathcal E}\subset 
\boldsymbol{\mathcal L}$, there exists a unique line $\ell_{\boldsymbol{\mathcal E}}\in 
\boldsymbol{\mathcal L}$ which intersects all the elements of $\boldsymbol{\mathcal E}$. 
Setting $\overline{\boldsymbol{\mathcal E}}=\boldsymbol{\mathcal E}
\cup \{ \ell_{\boldsymbol{\mathcal E}} \}$, 
then 
the ten  
abelian relations 
 ${\bf Res}_{\boldsymbol{\ell}}$ for  $\boldsymbol{\ell}\in 
 \boldsymbol{\mathcal L}\setminus 
 \overline{\boldsymbol{\mathcal E}}$,  form a basis of the space of combinatorial 2-ARs of ${\boldsymbol{\mathcal W}_{\boldsymbol{\mathcal Y}_5}^{GM}}$, i.e.
 there is an isomorphism
$$
\boldsymbol{AR}^2_C\Big( \,{\boldsymbol{\mathcal W}_{\boldsymbol{\mathcal Y}_5}^{GM}} \, \Big) \simeq \bigoplus_{\boldsymbol{\ell}\in 
 \boldsymbol{\mathcal L}\setminus 
 \overline{\boldsymbol{\mathcal E}}}  
 \, \mathbf C\cdot  
{\bf Res}_{\boldsymbol{\ell}}
\, .
$$
\end{prop}
\begin{proof}
This follows from explicit computations and  from \eqref{Eq:Res-along-Dl} combined with the facts that $W_{D_5}={\rm Aut}(\boldsymbol{\mathcal Y}_5)$ 
acts as the signature on the complex line spanned by  ${\bf HLOG}_{\boldsymbol{\mathcal Y}_5}$ (cf. Proposition \ref{Prop:Transformation-sigma-i}) and 
acts transitively on the set of weight divisors in $\boldsymbol{\mathcal Y}_5$ as well as on the set of exceptional collections $\boldsymbol{\mathcal E}\subset 
\boldsymbol{\mathcal L}$ (see \cite[Corollary 26.8]{Manin}). 
\end{proof}

\subsection{\bf The  1-abelian relations of $\boldsymbol{{\mathcal W}^{GM}_{ \boldsymbol{Y}_5}}$}
\label{SS:1-AR-of-WGMY5}
%
%
%
%
%
We consider the following privileged primitive of the $2$-form $\eta= d{\rm Log}(x)\wedge d{\rm Log}(y) $ : 
\begin{equation}
\label{Eq:delta}
\delta = 
 \frac{1}{2}\bigg(\,  
{\rm Log}(x)\, d\,{\rm Log}(y)-{\rm Log}(y)\,d\,{\rm Log}(x)\, 
\bigg) 
=
 \frac{1}{2}\bigg(\,  
{\rm Log}(x)\, \frac{dy}{y}-{\rm Log}(y)\,\, \frac{dx}{x}\, .
\bigg) 
\end{equation}

Via elementary computations, one gets 
\begin{align*}
 2\, U_1^*\big( \delta\big) = &\, 
\ln  {(  {y_{25}} )} \, \frac{ {dy_{24}}}{ {y_{24}}}-\ln  {\big(  {y_{24}}  {y_{35}} \big)}\,  \frac{ {dy_{25}}}{ {y_{25}}}+
\ln  {(  {y_{25}} )}\,  
\frac{ {dy_{35}}}{ {y_{35}}}
\\ 
 2\,U_2^*\big( \delta\big) = &\, 
\ln  \bigg( 
\frac{ {y_{14}}  {y_{35}}}{ {y_{13}}}
\bigg)\, \, \frac{  {dy_{13}}}{ {y_{13}}}-
\ln  \, (  {y_{13}} )\, 
\frac{  {dy_{14}}}{ {y_{14}}}-
\ln \,  (  {y_{13}} )\, 
\frac{  {dy_{35}}}{ {y_{35}}}
\\ 
2\,U_3^*\big( \delta\big) = &\, 
 \ln \, (  {y_{24}} )\,
\frac{  {dy_{14}}}{ {y_{14}}}-
\ln \, \big(  {y_{14}}  {y_{25}} \big)
\, 
\frac{  {dy_{24}}}{ {y_{24}}}+
\ln \,(  {y_{24}} )\, 
\frac{  {dy_{25}}}{ {y_{25}}}
\\ 
 2\,U_4^*\big( \delta\big) = &\, 
- \ln  \,(  {y_{35}} )  \, 
\frac{  {dy_{13}}}{ {y_{13}}}-
\ln\, (  {y_{35}} )\, 
\frac{ {dy_{25}}}{ {y_{25}}}+
 \ln \left(\,  \frac{ {y_{13}}  {y_{25}}}{ {y_{35}}}\right)\, 
\frac{  {dy_{35}}}{ {y_{35}}}
\\ 
\mbox{and }\quad  
2\,  U_5^*\big( \delta\big) = &\, 
 \ln \, (  {y_{14}} )\, 
\frac{ {dy_{13}}}{ {y_{13}}}-
\ln \,  \big(  {y_{13}}  {y_{24}} \big) \, 
\frac{ {dy_{14}}}{ {y_{14}}}+
\ln  (  {y_{14}} )\, 
\frac{  {dy_{24}}}{ {y_{24}}}
\, .
\end{align*}

Recall that  $\epsilon_1=\epsilon_3=\epsilon_5=1$ and 
 $\epsilon_2=\epsilon_4=-1$ (see \eqref{Eq:epsilon}).  The coefficient of the logarithmic differential ${\mathit{dy_{24}}}/{\mathit{y_{24}}}$ in 
the sum $\sum_{i=1}^5 \epsilon_i \, U_i^*\big(\,  
\delta \big)$ 
is the sum of those of the terms 
$ U_k^*\big( \delta\big) $ for $k=1,3, 5$, namely it is 
$$
\ln  \,({ \mathit{y_{25}} )} + \Big( - 
\ln  \big(  \mathit{y_{14}} \mathit{y_{25}} \big)\, 
\Big) +\ln \, ( \mathit{y_{14}} )\, , 
$$
a quantity which vanishes identically on any complex domain in $Y_5$ containing $(\mathbf R_{>0}\big)^5$. 
The same phenomenon occurs for all the logarithmic differential $dy_i/y_i$ with $i=1,\ldots,5$, which proves the 
\begin{prop}
\label{Prop:AR-Omega1}
 With  $\epsilon_1=\epsilon_3=\epsilon_5=1$ and 
 $\epsilon_2=\epsilon_4=-1$, 
  one has identically 
\begin{equation}
\label{Eq:AR-Omega1}
\sum_{i=1}^5 \epsilon_i \, U_i^*\big(\,  
\delta \, \big)
=\frac{1}{2}
\sum_{i=1}^5 \epsilon_i \,  U_i^*\bigg(\,  
{\rm Log}(x)\, \frac{dy}{y}-{\rm Log}(y)\,\, \frac{dx}{x}\, 
\bigg) =0\, 
\end{equation}
 hence 
 ${\bf AR}_\delta^1=\Big(  \epsilon_i \,U_i^*\big( \delta\big)\Big)_{i=1}^5$ is a 1-AR of weight 1 for $\boldsymbol{W}^{+}_{ \boldsymbol{Y}_5}$.  Moreover, one has 
 $d^1 \, \big({\bf AR}^1_\delta\big)= {\bf AR}^2_\eta$. 
\end{prop}

For any $\ell\in \boldsymbol{\mathcal L}$, the 2-abelian relation ${\bf AR}_\ell$ is exact and is the total derivative of an 1-abelian relation  
${\bf AR}_\ell^1$ of weight 1, which is equivalent to \eqref{Eq:AR-Omega1}.  We denote by 
$\boldsymbol{AR}_C^1\big(\boldsymbol{\mathcal W}^{GM}_{ \boldsymbol{\mathcal Y}_5}\big)$ (or just $\boldsymbol{AR}_C^1$ for short) the subspace of 
$W_1\Big( \boldsymbol{AR}^2\big(\boldsymbol{\mathcal W}^{GM}_{ \boldsymbol{\mathcal Y}_5}\big)\Big)$ spanned by these ARs, a notation which is justified by the fact that $d^1$ sends 
$\boldsymbol{AR}_C^1\big(\boldsymbol{\mathcal W}^{GM}_{ \boldsymbol{\mathcal Y}_5}\big)$ onto 
$\boldsymbol{AR}_C^2\big(\boldsymbol{\mathcal W}^{GM}_{ \boldsymbol{\mathcal Y}_5}\big)$, a fact which in particular implies that $\dim \boldsymbol{AR}_C^1\geq 10$. 
From  
$r_1\Big( 
\boldsymbol{\mathcal W}^{GM}_{ \boldsymbol{Y}_5}
\Big)=35$ and essentially using explicit computations, one gets the following result which gives a quite 
detailed description of the structure of the space of 1-abelian relations of $\boldsymbol{\mathcal W}^{GM}_{ \boldsymbol{\mathcal Y}_5}$. 
%

\begin{thm}
\label{Thm:1-AR-}
{\bf 1.} 
One has $r^1\big( \boldsymbol{\mathcal W}_5\big)\leq 6$ for any 5-subweb $\boldsymbol{\mathcal W}_5$ of $\boldsymbol{\mathcal W}_{\boldsymbol{Y}_5}^{GM}$. Those 
for which the 1-rank is 6 are the subwebs  
$ \boldsymbol{\mathcal W}^{\underline{\epsilon}}= 
\boldsymbol{\mathcal W}\big( \mathcal F_1^{{\epsilon}_1},\ldots, \mathcal F_5^{{\epsilon}_5}\big)$  for any 5-tuple $\underline{\epsilon}=(\epsilon_i)_{i=1}^5 \in \{\pm 1\}^5$.
\mk \\
{\bf 2.} When $\underline{\epsilon}$ is odd, 
the space of  1-abelian relations of $\boldsymbol{\mathcal W}^{\underline{\epsilon}}$ 
has a basis formed of ARs corresponding to 
 identities of the form  $\sum_{i \in I^{\underline{\epsilon}}} \sum_{s=1}^3
c_{i,s} d\,{\rm Log} U_{i,s}=0$ for some  coefficients $c_{i,s}\in \{ -1,0,1\}$. In particular, the 1-ARs  of this web all are of weight 0: 
$$\boldsymbol{AR}^1\big( \boldsymbol{\mathcal W}^{\underline{\epsilon}}\big)=W_0 \Big(\boldsymbol{AR}^1\big( \boldsymbol{\mathcal W}^{\underline{\epsilon}}\big)\Big)\, .$$
{\bf 3.} When $\underline{\epsilon}$ is even, 
the situation is different since 
not all the 
ARs of $\boldsymbol{\mathcal W}^{\underline{\epsilon}}$ are of weight 0. Indeed: 
\begin{itemize}
\item the space  1-abelian relations of $\boldsymbol{\mathcal W}^{\underline{\epsilon}}$ of  weight 0  is 
of dimension 5 and has  a basis formed of ARs 
of the same form as in {\bf 2};
\sk 
\item  but there exists an 1-AR of weight 1, denoted by ${\bf AR}_{\underline{\epsilon}}^1$ whose each component is a linear combination with coefficients in $\{\, 0,\pm 1\,\}$ of the 1-forms 
with logarithmic coefficients
$$
{}^{} \quad \frac{1}{2}\Big( {\rm Log}\,U_{j,a}\,d\,{\rm Log}\,U_{j,b}
- {\rm Log}\,U_{j,b}\,d\,{\rm Log}\,U_{j,a}\Big)
\qquad 
\big(\,1\leq a<b\leq 3\,)\, .
$$
and which is such that $d {\bf AR}_{\underline{\epsilon}}^1={\bf Res}_{\underline{\epsilon}}$. 
\sk 
\end{itemize}
{\bf 4.} The space of 1-abelian relations 
of weight 0 of  $\boldsymbol{\mathcal W}_{\boldsymbol{Y}_5}^{GM}$ is of dimension 20 and 
has a basis of ARs as in {\bf 2},  that is which correspond to identities of the form $\sum_{i=1}^{10}\sum_{s=1}^3 c_{i,s} d {\rm Log}U_{i,s}=0$ for some coefficients $c_{i,s}\in \{-1,0,1\}$.    
 Moreover, one has 
 $W_0\Big(\boldsymbol{AR}^1\big( \boldsymbol{\mathcal W}_{\boldsymbol{Y}_5}^{GM}\big)\Big)=
 \boldsymbol{AR}^1_{Rat}\big( \boldsymbol{\mathcal W}_{\boldsymbol{Y}_5}^{GM}\big)$ and this space is a subspace of the space ${\rm Ker}\big(d^1\big)$ of closed 1-ARs of $ \boldsymbol{\mathcal W}_{\boldsymbol{Y}_5}^{GM}$. 
\mk \\
{\bf 5.}  The 1-ARs ${\bf AR}_{\underline{\epsilon}}^1$ for ${\underline{\epsilon}}$ even span a space of dimension 10 
on which the restriction of the derivative \eqref{Eq:derivative-d1}  induces an isomorphism onto $\boldsymbol{AR}^2_C\Big( \boldsymbol{\mathcal W}_{\boldsymbol{Y}_5}^{GM}\Big)$: setting 
\begin{equation}
\label{Eq:AR1-C}
\boldsymbol{AR}^1_C=\boldsymbol{AR}^1_C\Big( \boldsymbol{\mathcal W}_{\boldsymbol{Y}_5}^{GM}\Big)
=\big\langle \,{\bf AR}^1_{\underline{\epsilon}} \, \big\lvert \, \,
\underline{\epsilon}\in \{\pm1\}^5 \, \mbox{ is even} \,
\, \big\rangle \, ,
\end{equation} 
this isomorphism is given by
$d^1: \boldsymbol{AR}^1_C\Big( \boldsymbol{\mathcal W}_{\boldsymbol{Y}_5}^{GM}\Big)
 \stackrel{\sim }{\longrightarrow}
\boldsymbol{AR}^2_C\Big( \boldsymbol{\mathcal W}_{\boldsymbol{Y}_5}^{GM}\Big)\, , \, \, 
{\bf AR}_{\underline{\epsilon}}^1  \longmapsto {\bf Res}_{\underline{\epsilon}}$.
\mk \\
{\bf 6.} For any $i,j,k$ such that $1\leq i<j<k\leq 5$, the 6-subweb $\boldsymbol{\mathcal W}_{ijki^*j^*k^*}$ of $\boldsymbol{\mathcal W}_{\boldsymbol{Y}_5}^{GM}$ defined by the first integrals $U_l$ for $l\in \{i,j,k,i+5,j+5,k+5\}$ carries a complete and irreducible 1-AR of weight 1, denoted by 
${\bf AR}^1_{ijk}$ (uniquely defined up to sign), the $l$-th component of which is a linear combination, with coefficients in $\{-1,0,+1\}$, of the exact 1-forms
$$
{}^{} \quad  
d\Big(  {\rm Log}\,U_{l,a}\,{\rm Log}\,U_{l,b}\Big) =
 {\rm Log}\,U_{l,a}\,d\,{\rm Log}\,U_{l,b}
+ {\rm Log}\,U_{l,b}\,d\,{\rm Log}\,U_{l,a}
\qquad 
\big(\,1\leq a<b\leq 3\,)\, .
$$
The ten 1-abelian relations ${\bf AR}^1_{ijk}$ span a subspace of $W_1\big(
\boldsymbol{AR}^1_C\big( \boldsymbol{\mathcal W}_{\boldsymbol{Y}_5}^{GM}\big)\big) \cap {\rm Im}\big(d^0\big)$ of dimension 5, which will be denoted by $
\boldsymbol{AR}^1_{Sym}\big( \boldsymbol{\mathcal W}_{\boldsymbol{Y}_5}^{GM}\big)$ or just  $\boldsymbol{AR}^1_{Sym}$ for short. 
\mk \\
{\bf 7.}
There is a decomposition in direct sum: 
\begin{equation}
\label{Eq:Decomp-AR1-WGM}
\boldsymbol{AR}^1\Big( \boldsymbol{\mathcal W}_{\boldsymbol{Y}_5}^{GM}\Big)=\overunderbraces{&&\br{3}{
W_1\big(
\boldsymbol{AR}^1\big( \boldsymbol{\mathcal W}_{\boldsymbol{Y}_5}^{GM}\big)\big)
}}%
{&W_0\Big(
\boldsymbol{AR}^1\big( \boldsymbol{\mathcal W}_{\boldsymbol{Y}_5}^{GM}\big)\Big)^{{\red{20}}} \oplus &
\boldsymbol{AR}^1_{Sym}\Big( \boldsymbol{\mathcal W}_{\boldsymbol{Y}_5}^{GM}\Big)^{{\red{5}}}
&\oplus &
\boldsymbol{AR}^1_C\Big( \boldsymbol{\mathcal W}_{\boldsymbol{Y}_5}^{GM}\Big)^{{\red{10}}}
}%
{&\br{2}{ {\rm Im}\big(d^0\big)\,=\, {\rm Ker}\big(d^1\big)}&
}
\, . 
\end{equation}
where the upper exponents in red stand for the corresponding dimensions.  It follows that the 1-rank of 
$\boldsymbol{\mathcal W}_{\boldsymbol{Y}_5}^{GM}$ is 35. 
\mk \\
\noindent {\bf 8.} The decomposition in direct sum \eqref{Eq:Decomp-AR1-WGM} actually is the decomposition of $\boldsymbol{AR}^1\Big( \boldsymbol{\mathcal W}_{\boldsymbol{Y}_5}^{GM}\Big)$ 
into irreducible $W_{D_5}$-representations, with  the following isomorphisms:
\begin{equation}
\label{Eq:1AR-Representations}
W_0\Big(
\boldsymbol{AR}^1\Big)\simeq V_{[2,21]}^{{\red{20}}}\, , \quad 
\boldsymbol{AR}^1_{Sym}\simeq V_{[-,221]}^{{\red{5}}}
\quad 
\mbox{ and }
\quad 
\boldsymbol{AR}^1_C\simeq \boldsymbol{AR}^2_C\simeq V_{[11,111]}^{{\red{10}}}\,.
\end{equation}
\end{thm}
\begin{proof} 
The proof essentially goes by explicit computations that we do not reproduce here.\footnote{Short Maple worksheets are available from the author upon request.}

For the eighth point, we proceeded as follows:  thanks to the explicit expressions  \eqref{Eq:Formula-Sigma-1} and \eqref{Eq:Formula-Sigma-k} for the generators $\sigma_i$ of 
$\boldsymbol{\mathscr W}_{\hspace{-0.1cm}D_5}$ (see \eqref{Eq:Bir-W-D5}), and working with an arbitrary but fixed basis of any one of the spaces 
$W_0\big(\boldsymbol{AR}^1\big)$,  
$\boldsymbol{AR}^1_{Sym}$ or 
$\boldsymbol{AR}^1_C$, noted here by $S$, there is no difficulty first to verify that the pull-back maps $\sigma_i^*$ give rise to automorphisms of $S$, second to get explicit matrices of these automorphisms with respect to the chosen basis. 
 Then proceeding as in the Appendix of \cite{Pirio2022}, there is no difficulty to compute the character of the representation 
$W_{D_5}\simeq \boldsymbol{\mathscr W}_{\hspace{-0.1cm}D_5} \hookrightarrow {\rm GL}(S)$. For instance, one obtains that the 
character of the action of $W_{D_5}$ on $W_0\Big( \boldsymbol{AR}^1\big( \boldsymbol{\mathcal W}_{\boldsymbol{Y}_5}^{GM}\big)\Big)$ is 
$$ \chi_{W_0( \boldsymbol{AR}^1)}= \big( \, 20, -4, 4, 2, -2, 2, -2, 0, 0, 0, -1, 1, -1, -1, 1, 0, 0, 0\,\big)\, .
$$ 
Looking at the character table of $W_{D_5}$ (cf.\,Table 4 in \cite{Pirio2022} for instance), the first isomorphism in \eqref{Eq:1AR-Representations} follows. 
 The two other cases are handled in the same way. 
\end{proof}

\subsection{\bf The  0-abelian relations of $\boldsymbol{{\mathcal W}^{GM}_{ {Y}_5}}$}
\label{SS:0-AR-of-WGMY5}
%
%
One has 
$r_0\big( 
\boldsymbol{\mathcal W}^{GM}_{ \boldsymbol{Y}_5}
\big)=25$ and because the differential $d^0$ induces an isomorphism 
$\boldsymbol{AR}^0=\boldsymbol{AR}^0\big( 
\boldsymbol{\mathcal W}^{GM}_{ \boldsymbol{Y}_5}
\big) \simeq {\rm Im}\big( d^0\big)= W_0\big( 
\boldsymbol{AR}^1\big)
\oplus \boldsymbol{AR}^1_{Sym}
$, the structure of $\boldsymbol{AR}^0$ essentially has already been given in the previous subsection. 

For $i \in [\hspace{-0.05cm}[5]\hspace{-0.05cm}]$, one sets $i^*=i+5\in \{6,\ldots,10\}$ and  $l,m,n\in [\hspace{-0.05cm}[10]\hspace{-0.05cm}] $ pairwise distinct, 
one denotes by $\boldsymbol{\mathcal W}_{lmn}$ the 3-subweb of 
$\boldsymbol{\mathcal W}^{GM}_{ \boldsymbol{Y}_5}$ defined by the first integrals $U_l$, $U_m$ and $U_n$: $\boldsymbol{\mathcal W}_{lmn}=\boldsymbol{\mathcal W}(U_l,U_m,U_n)$. By direct explicit computations, we get the following result which just adds a few details to what can be deduced from 
Theorem \ref{Thm:1-AR-}. 

\begin{prop}
{\bf 1.} One has a  direct sum 
$$
\boldsymbol{AR}^0\Big( 
\boldsymbol{\mathcal W}^{GM}_{ \boldsymbol{Y}_5}
\Big)=W_1\Big(\boldsymbol{AR}^0\Big)^{{\red{20}}}
\oplus 
W_2\Big(\boldsymbol{AR}^0\Big)^{{\red{5}}}
$$ 
with the differential $d^0$ inducing  isomorphisms  of $W_{D_5}$-representations 
$$
W_1\big(\boldsymbol{AR}^0\big)\simeq W_0\big(\boldsymbol{AR}^1\big)\simeq 
V^{{\red{20}}}_{[2,21]}
\qquad \mbox{and} \qquad 
W_2\big(\boldsymbol{AR}^0\big)\simeq \boldsymbol{AR}^1_{Sym}
\simeq 
V^{{\red{5}}}_{[-,221]}\, . 
$$

\noindent 
{\bf 2.} The 3-subwebs of $\boldsymbol{\mathcal W}^{GM}_{ \boldsymbol{Y}_5}
$ have 0-rank less than or equal to 1. Those with rank 1 are the 3-subwebs 
$\boldsymbol{\mathcal W}_{ijk}$, 
$\boldsymbol{\mathcal W}_{ijk^*}$, 
$\boldsymbol{\mathcal W}_{ij^*k^*}$, $\boldsymbol{\mathcal W}_{i^*j^*k^*}$, and are therefore 80 in number. Their 0-ARs all have weight 1 and they span the whole space $W_1\big(\boldsymbol{AR}^0\big)$ which is of dimension 20. 

\noindent 
{\bf 3.} For any $i,j,k$ such that $1\leq i<j<k\leq 5$, the 
1-AR ${\bf AR}^1_{ijk}$ of 
the 6-subweb $\boldsymbol{\mathcal W}_{ijki^*j^*k^*}$ of $\boldsymbol{\mathcal W}_{\boldsymbol{Y}_5}^{GM}$ is exact, and is the derivative of a 0-AR denoted by 
${\bf AR}^0_{ijk}$, which corresponds to  the (unique up to sign) functional identity of the form
$\sum_l \sum_{1\leq a<b\leq 3} c_{l}^{a,b} \,{\rm Log}\,U_{l,a} \, {\rm Log}\,U_{l,b}=0$ where the first sum is for $l$ ranging in $\{i,j,k,i^*,j^*,k^*\}$, and where all the coefficients $c_{l}^{a,b}$ belong to $\{-1,0,1\}$. 

 The ten 0-abelian relations ${\bf AR}^0_{ijk}$ span  the whole space 
$W_2\big( \boldsymbol{AR}^0\big)$ 
of  0-abelian relations of weight 2 of  $\boldsymbol{\mathcal W}_{\boldsymbol{Y}_5}^{GM}$, which is of dimension 5.   
Moreover, the space ${\bf Res}\big( W_2\big( \boldsymbol{AR}^0\big)\big)$ spanned 
by the residues of the  weight 2 abelian relations ${\bf AR}^0_{ijk}$ coincides with 
$W_1\big( \boldsymbol{AR}^0\big)$. One has:
$$
 {\bf Res}\Big( W_2\big( \boldsymbol{AR}^0\big)\Big)
 =\bigg\langle\,
{\bf Res}_{\zeta_\ell} \Big( {\bf AR}^1_{ijk}
\Big) \,\, \Big\lvert  \,
\scalebox{0.9}{\begin{tabular}{l}
$\ell=1,\ldots,10$\\
$1\leq i<j<k\leq 5$
\end{tabular}}
 \bigg\rangle = W_1\Big( \boldsymbol{AR}^0\Big)\, .
$$

\end{prop}
All the 0-ARs of $\boldsymbol{\mathcal W}_{\boldsymbol{Y}_5}^{GM}$ can be made explicit.  For instance,  the following weight 2 functional identity 
\begin{align*}
\Bigg(
\ln  \left(\frac{U_{1,1}}{U_{1,2}}\right) \ln  \big(U_{1,3}\big)\Bigg) 
\,+\,&
\Bigg(  
\ln  \left({U_{2,1}}U_{2,2}\right) \ln  \left(U_{2,3}\right)
\Bigg)
\,+\,\bigg( 
\ln  \left(\frac{U_{3,1}}{U_{3,2}}\right) \ln  \left(U_{3,3}\right)
\bigg)  \\
+\,\Bigg( 
\ln  \left(\frac{U_{6,1}}{U_{6,2}}\right) \ln  \left(U_{6,3}\right)
 \Bigg)
\,+\,&\Bigg( 
\ln  \left(U_{7,1}U_{7,2}\right) \ln  \left(U_{7,3}\right)
\Bigg)
\,+\,\Bigg(
\ln  \left(\frac{U_{8,1}}{U_{8,2}}\right) \ln  \left(U_{8,3}\right)
\Bigg) =0
\end{align*}
corresponds to the weight 2 abelian relation ${\bf AR}^0_{123}$ of the web 
$\boldsymbol{\mathcal W}_{123}=\boldsymbol{\mathcal W}(U_1,U_2,U_3,U_6,U_7,U_8)$.

\subsection{\bf The various links between the spaces of ARs of $\boldsymbol{{\mathcal W}^{GM}_{ {Y}_5}}$ summarized in a table}
\label{SS:pictural}
%
%
We find interesting/enlightening as well as convenient to gather all our findings about the abelian relations of $\boldsymbol{\mathcal W}^{GM}_{\boldsymbol{Y}_5}$ in Table \ref{Table:ARs} below. In it: 
\begin{itemize}
\item[$-$]  the columns are labelled by the weight of the ARs, the lines by their degree; 
\sk
\item[$-$] the diagonal black arrows are isomorphisms of $W_{D_5}$-representations induced by the differentials $d^k: \boldsymbol{AR}^{k}
\rightarrow \boldsymbol{AR}^{k+1}$ (with $k=0,1$);
\sk
\item[$-$] a dashed red arrow $\boldsymbol{T}\,{\red{\dashleftarrow}}\,\boldsymbol{S}$ means that the whole target space $\boldsymbol{T}$ is 
spanned by  the residues of the elements of the source space $\boldsymbol{S}$;
\sk
\item[$-$] the upper exponents in red stand for the dimension of each space;
\sk
\item[$-$] ${\bf HLOG}$ stands for the complex line spanned by the `master 2-abelian relation' ${\bf HLOG}_{\boldsymbol{Y}_5}$.
\sk
\end{itemize}


\begin{table}[h!]
\centering
\scalebox{0.37}{
\includegraphics{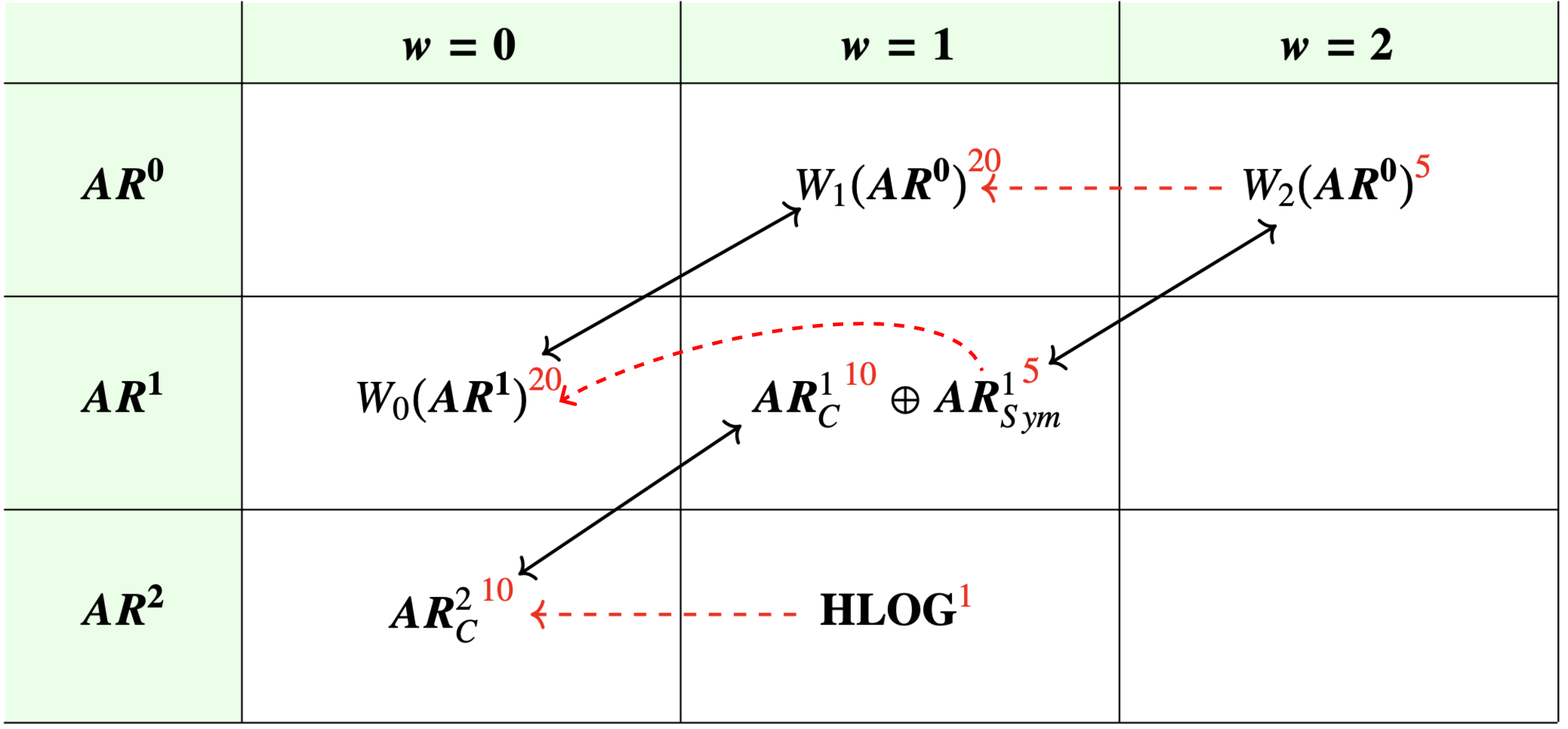}
}
\vspace{0.2cm}
\caption{The subspaces of the space of abelian relations of $\boldsymbol{\mathcal W}_{\boldsymbol{Y}_5}^{GM}$ and the many relations between them.}
\label{Table:ARs}
\end{table}
\vspace{-0.2cm}
It may be useful to gather the structures as $W(D_5)$-representations of the subspaces appearing 
in this table as well. This is given in the following Table 2.

\begin{table}[h!]
\centering
\scalebox{0.41}{\includegraphics{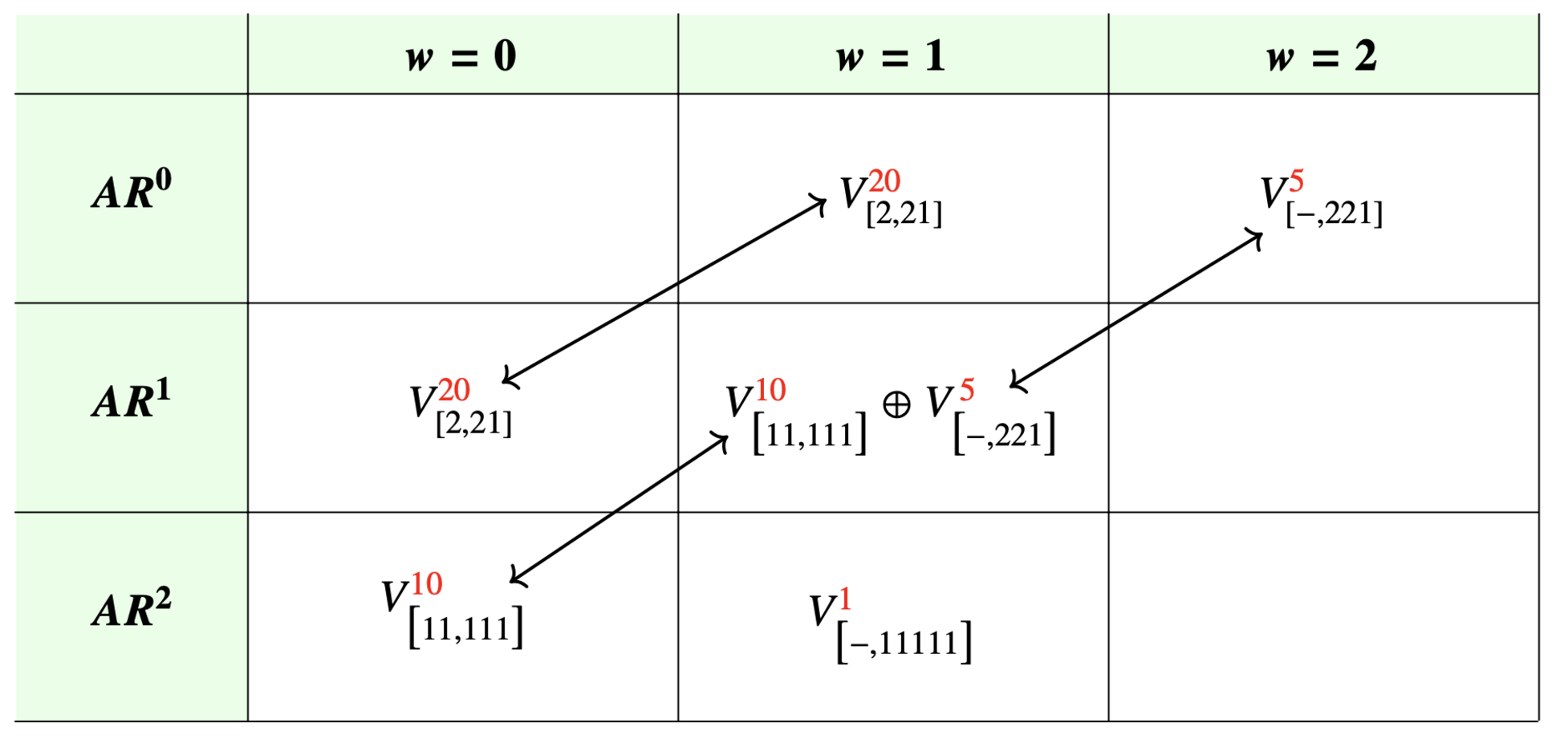}}
\vspace{0.2cm}
\caption{The subspaces of 
$\boldsymbol{AR}\big( \boldsymbol{\mathcal W}_{\boldsymbol{Y}_5}^{GM}\big)$ as $W_{D_5}$-representations.}
\label{Table:ARs}
\end{table}

\section{\bf A cluster view on  $\boldsymbol{\boldsymbol{\mathcal W}^{GM}_{\hspace{-0.1cm} \boldsymbol{\mathbb S}_5}}$}
\label{SS:Cluster-View}
${}^{}$
A particularly interesting feature of Abel's identity  of the dilogarithm is that it has a  `cluster nature', in the sense that it can be written in an equivalent nice form in terms of some cluster variables (see \cite{ClusterWebs} for details).  Indeed,  if $X_\ell$ (with $\ell \in \mathbf Z/5\mathbf Z$)  stands for the $\mathscr X$-cluster variables of type $A_2$ characterized by the 5-cyclic recurrence relations $X_{\ell-1}X_{\ell+1}=1+X_\ell$ for any 
$\ell$, then setting $X_1=u_1$ and $X_2=u_2$,  one has $X_3=\frac{1+u_2}{u_1}$, $X_4=\frac{1+u_1+u_2}{u_1u_2}$, $X_5=\frac{1+u_1}{u_2}$ and $X_{\ell+5}=X_\ell$ for every $\ell\in \mathbf Z$.\footnote{It seems that, in a slightly different form, the cyclic recurrence of the $X_\ell$'s was already known to Gauss, see \href{https://www.math.uni-bielefeld.de/~sek/cluster/pentagramma/}{here}.}
Moreover it is well-known that 
$\boldsymbol{ \big({\mathcal{A}b}\big)}$ is equivalent to the following functional identity
\begin{equation}
\label{Eq:R-A2}
\sum_{i=1}^5 \mathsf{R}\big(X_\ell\big)=\frac{{}^{} \pi^2}{2}\, 
\end{equation}
which is satisfied for all $X_1,X_2>0$ by the `cluster dilogarithm' $\mathsf{R}$.\footnote{This is the function defined by  $\mathsf{R}(u)=\frac{1}{2} \int_{0}^u 
\big( {\ln(1+t)}/{t}$ $-  {\ln(t)}/{(1+t)}\,
\big) dt   $
for any $u>0$ (see \cite[\S2.2.2.1]{ClusterWebs}). In terms of the classical bilogarithm, it is expressed as ${\mathsf R}(u)$ $= -{\bf L}{\rm i}_{2} (-u)-\frac{1}{2} {\rm Log}(u){\rm Log}(1+u)$ for  any $u\in \mathbf R_{>0}$.}
In addition to furnishing a very nice formal way to write this identity, the cluster perspective on it offers a more conceptual way to interpret it, since 
\eqref{Eq:R-A2} (hence Abel's identity $\boldsymbol{ \big({\mathcal{A}b}\big)}$)
 can be seen as the manifestation of a certain property of an important object associated to the $\mathscr X$-cluster algebra of type $A_2$, namely the associated scattering diagram $S\hspace{-0.05cm}C_{A_2}$.  Then that \eqref{Eq:R-A2} holds true follows from the 
fact that  this scattering diagram is {\it `consistent'} (see \cite[Theorem 3.6]{Nakanishi}). 
\sk 

If one believes that ${\bf HLOG}_{\boldsymbol{\mathcal Y}_5}$ is a natural generalization of the five terms identity of the dilogarithm, it is natural to wonder whether the former differential identity may be explained in terms of a certain property of a putative scattering diagram $S\hspace{-0.05cm}C_{\boldsymbol{\mathcal Y}_5}$ or not. Moreover, because ${\bf HLOG}_{\boldsymbol{\mathcal Y}_5}$ is given by the vanishing of a sum with finitely many terms, a naive expectation would be that $S\hspace{-0.05cm}C_{\boldsymbol{\mathcal Y}_5}$ be of finite type. Thus if one is fool and dreamy enough to want explain ${\bf HLOG}_{\boldsymbol{\mathcal Y}_5}$ by means of a scattering diagram, a first step would be to find a cluster-like structure associated to the spinor 10-fold $\mathbb S_5$, which firstly is of finite type and secondly, is well suited to the web $\boldsymbol{\boldsymbol{\mathcal W}^{GM}_{\hspace{-0.1cm} \boldsymbol{Y}_5}}$ under consideration.  There are many works on some cluster structures on some  pieces  (more precisely, some open domains of some peculiar subvarieties) of generalized flag manifolds, but except in very few known cases, these cluster structures are not of finite type. For the case under consideration, we are not aware of any classical cluster structure on (some dense open-subset of) $\mathbb S_5$ which is of finite type, which  may temper our dream of explaining ${\bf HLOG}_{\boldsymbol{\mathcal Y}_5}$ via a finite scattering diagram. However, some recent works by Ducat and Daisey-Ducat suggest that allowing to consider more general cluster-like structures than the classical ones might be the suitable path to reach our goal. 
\sk

In \cite{Ducat}, for the case of $\boldsymbol{\mathcal G}_5=\mathbb S_5$, and in \cite{DaiseyDucat} for the case of the Cayley plane $\boldsymbol{\mathcal G}_6=\mathbb O\mathbf P^2$, Ducat then Ducat together with Daisey describe a cluster-like structure of finite type on the coordinate ring $R_r$  of a certain Zariski open domain in $\boldsymbol{\mathcal G}_r$ for $r=5,6$. These cluster-like structures are known as LPAs\footnote{LPA is the acronym for `Laurent Phenomenon Algebra'.}, a generalization of the classical notion of cluster algebra introduced by Lam and Pylyavskyy in \cite{LamPylyavskyy}.  As the name suggests, the main feature of such an  algebra is that any `cluster' variable obtained by means of a finite sequence of `generalized mutations' from the initial cluster variables is a Laurent polynomial in the latter.  For $r=4$, one has $\boldsymbol{\mathcal G}_4=G_2(\mathbf C^5)$
and this grassmannian carries a classical $\mathscr A$-cluster structure of finite type $A_2$ which, after quotienting by the action of the rank 4 Cartan torus $H_4$ of ${\rm SL}_5(\mathbf C)$, gives rise to the $\mathscr X$-cluster structure of type $A_2$ on $\boldsymbol{\mathcal Y}_4^*=\boldsymbol{\mathcal G}_4^*/\hspace{-0.05cm}/ H_4= 
G_2(\mathbf C^5)^*/\hspace{-0.05cm}/ H_4\simeq \mathcal M_{0,5}$. 
\sk

In \cite{Ducat}, Ducat endows the coordinate ring of the complement $\mathbb S_5^\circ$ of a divisor in the spinor 10-fold with the structure of a finite LPA. 
Unfortunately, the theory of LPAs has not been developed that much so far. The notion of mutation for these algebras is a generalization of the binomial $\mathscr A$-mutation of the classical theory of cluster algebras but as of the time of writing, an equivalent notion for LPAs of that of $\mathscr X$-mutation has not be worked out yet. For that reason, we do not have a cluster description of Gelfand-MacPherson web on 
$\boldsymbol{\mathcal Y}_5^*$ as nice as the one of 
$\boldsymbol{\mathcal W}^{GM}_{ \boldsymbol{Y}_4}$ in terms of the $\mathscr X$-cluster variable of type $A_2$, which can be written in a concise mathematical form as 
$$
\boldsymbol{\mathcal W}^{GM}_{ \boldsymbol{Y}_4}\simeq 
\boldsymbol{\mathscr X \hspace{-0.45cm} \mathscr X\hspace{-0.1cm}\mathcal W}_{A_2}=\boldsymbol{\mathcal W}\big( X_i\big)_{i=1}^5
=\boldsymbol{\mathcal W}\bigg(\,u_1
\, , \, u_2
\, , \, \frac{1+u_2}{u_1}
\, , \,
\frac{1+u_1+u_2}{u_1u_2}
\, , \,
\frac{1+u_1}{u_2}
 \,\bigg)\,. 
$$

Because ther is no generalization to the case of LPAs of the notion of $\mathcal X$-mutation of the classical theory of cluster algebras yet, we are going to deal  with Gelfand-MacPherson's  web  of the spinor tenfold $\mathbb S_{5}$
instead. In short, we prove that $\boldsymbol{\mathcal W}^{GM}_{ \boldsymbol{\mathbb S}_5}$ is cluster with respect to the finite LPA structure on $\mathbb S_{5}$ constructed by Ducat in \cite{Ducat}. 

\begin{prop}
\label{Prop:Cluster}
{\bf 1.} Wick's parametrization \eqref{Eq:map-W-proj} of $\mathbb S_{5}$ is cluster, {\it i.e.} the Wick coordinates $x_{ij}\circ W^{-1}:  \mathbb S_5\simeq {\rm Asym}_5(\mathbf C)\dashrightarrow \mathbf C$ (for $i,j$ such that $1\leq i<j\leq 5$) naturally identify with some of the cluster variables of Ducat's LPA of $\mathbb S_{5}$. \sk

\noindent 
{\bf 2.}  The ten face maps $\Pi_F: \mathbb S_5\dashrightarrow  \mathbb S_{4,F}$ are cluster in the sense that their components, when written in the initial cluster variables 
at the source 
 and some Wick's coordinates at the target, are cluster variables, possibly up to sign and up to multiplication by a monomial in the frozen variables.
\end{prop}
In the lines below, we explain how we proceeded to get this result. 


\subsection{\bf Ducat's LPA on $\boldsymbol{\mathbb S_5}$}
Ducat's LPA on $\mathcal A_5$ has rank 3 and 18 cluster variables in total, among which 8 are frozen. These latter are denoted by  $a_1,\ldots,a_8$. Denoting by $x_1,x_2,x_3$  the initial unfrozen cluster variables, all the other ones can be obtained using the following SageMath script written by Daisey and which can be run online at \href{https://oliverdaisey.github.io/code/}{his webpage}: 
\begin{lstlisting}
sage: var("x1,x2,x3")
sage: coeffs = [var("a%d" %i) for i in range(1, 9)] # coefficients
sage: F1 = a5*x2 + a8*x3 + a2*a3
sage: F2 = a6*x1*x3 + a3*a4*x1 + a8*a1*x3 + a1*a2*a3
sage: F3 = a4*x1 + a7*x2 + a1*a2
sage: S = LPASeed({x1:F1, x2: F2, x3:F3})
sage: show(S.variable_class())
\end{lstlisting}

In addition to the three initial cluster variables $x_1,x_2,x_3$, the other unfrozen cluster variables are the following: 
\begin{align*}
& \frac{a_2  a_3 +  a_5  x_2 +  a_8  x_3}{ x_1}
 \, , \, 
   \frac{ a_1  a_2  a_3 +  a_1  a_8  x_3 +  a_3  a_4  x_1 +  a_6  x_1  x_3}{ x_2} 
  \, , \, 
 \frac{a_1  a_2 +  a_4  x_1 +  a_7  x_2}{ x_3}   \, , \, \\
  &
   \frac{ a_1  a_2  a_3 +  a_1  a_5  x_2 +  a_1  a_8  x_3 +  a_3  a_4  x_1 +  a_6  x_1  x_3}{ x_1  x_2}
  \, , \,  
    \frac{ a_1  a_2  a_3 +  a_1  a_8  x_3 +  a_3  a_4  x_1 +  a_3  a_7  x_2 +  a_6  x_1  x_3}{ x_2  x_3} 
    \, , \, \\
    &  
           \frac{ a_1  a_2^2  a_3 +  a_1  a_2  a_5  x_2 +  a_1  a_2  a_8  x_3 +  a_2  a_3  a_4  x_1 +  a_2  a_3  a_7  x_2 +  a_4  a_5  x_1  x_2 +  a_5  a_7  x_2^2 +  a_7  a_8  x_2  x_3}{ x_3  x_1}              \, , \,
 \\
& 
\frac{ a_1  a_2^2  a_3 +  a_1  a_2  a_5  x_2 +  a_1  a_2  a_8  x_3 +  a_2  a_3  a_4  x_1 +  a_2  a_3  a_7  x_2 +  a_2  a_6  x_1  x_3 +  a_4  a_5  x_1  x_2 +  a_5  a_7  x_2^2 +  a_7  a_8  x_2  x_3}{ x_1  x_2  x_3}\, .
\end{align*}

These cluster variables are Laurent polynomials in the initial unfrozen cluster variables $x_1,x_2$ and $x_3$, with coefficients in the monoid $\mathbb N_{>0}[a_1,\ldots,a_8]$. The non initial unfrozen cluster variables are characterized by their denominators which are monomials in the $x_i$'s. Hence one can denote these cluster variables by $X_i$, $X_{ij}$ and $X_{123}$ for $i,j\in \{1,2,3\}$ distinct, where $X_i$ (resp. $X_{ij}$ or $X_{123}$) stands for the cluster variables with monomial denominator $x_i$ (resp. $x_{i}x_j$ or $x_1x_2x_3$), that is: 
$$  X_1 = \frac{\mathit{a_2} \mathit{a_3} +\mathit{a_5} \mathit{x_2} +\mathit{a_8} \mathit{x_3}}{\mathit{x_1}} \, , \quad 
X_2 = 
\frac{\mathit{a_1} \mathit{a_2} \mathit{a_3} +\cdots+\mathit{a_6} \mathit{x_1} \mathit{x_3}}{\mathit{x_2}}
\, , \ldots 
 \,  , \, 
X_{123} =  
\frac{\mathit{a_1} \,\mathit{a_2}^{2} \mathit{a_3} +\cdots+\mathit{a_7} \mathit{a_8} \mathit{x_2} \mathit{x_3}}{\mathit{x_1} \mathit{x_2} \mathit{x_3}}
\, . 
$$

\subsection{\bf The initial cluster on $\boldsymbol{\mathbb S_5}$ and Wick's coordinates}
The initial cluster of Ducat's LPA structure on  (the affine cone $\widehat{\boldsymbol{\mathbb S}_5}$ over) $\boldsymbol{\mathbb S_5}$ is formed by the 
 triple $(x_1,x_2,x_3)$ together with the frozen variables $a_1,\ldots,a_8$.  These cluster variables are the restrictions to $\widehat{\boldsymbol{\mathbb S}_5}$ of some of the weight coordinates on the spinor representation $S^+_5$. So each of the initial cluster coordinates is associated to a well-defined weight of the spinor representation.  The same holds for the Wick coordinates we use here for  parametrizing birationnally $\boldsymbol{\mathbb S_5}$ and it is by comparing the weights of all these coordinates that we are going to make explicit the way they are related. 
 \mk

 In \cite{DaiseyDucat}, it is by means of a figure (namely Figure 3.(a)) that is indicated 
 what are the weights of  the 16 coordinates  $x_i,a_i$ for $i=1,\ldots,8$ considered by Ducat. 
 For the Wick coordinates, the analogous correspondence is given in  \cite[Table\,4.1]{PirioAFST}.
Since \cite[Figure 3.(a)]{DaiseyDucat} pictures the images of the weights by `the' Coxeter projection, it suffices to have an explicit expression for the latter in order to get how Ducat's cluster coordinates and Wick's ones are related. This is what we do below. 
  
  Recall that a useful description of the weights of $S_5^+$ is in terms of the sixteen lines of an arbitrary (but fixed) smooth del Pezzo surface ${\rm dP}_4$ described as the blow-up of the projective plane in five points.  Then ${\rm Pic}_{\mathbf Z}({\rm dP}_4)$ is freely generated by the classes 
  $\boldsymbol{e}_{1},\ldots,\boldsymbol{e}_{5}$  of the five exceptional divisors of the blow up plus the class $\boldsymbol{h}$ of the preimage under the blow-up map of a generic line in $\mathbf P^2$.  With these notations,  we use the following labels for (the classes  of) the lines $ \boldsymbol{\ell}_i$ ($i=1,\ldots,16$) of the considered del Pezzo surface: 
\begin{align*}
\Big( \boldsymbol{\ell}_i \Big)_{i=1}^{16}=\bigg(\, & \, \boldsymbol{e}_{1} \, , \, \boldsymbol{e}_{2}\, , \, \boldsymbol{e}_{3}\, , \,\boldsymbol{e}_{4}\, , \, \boldsymbol{e}_{5}\, , \, \boldsymbol{h} -\boldsymbol{e}_{1}-\boldsymbol{e}_{2}\, , \, \boldsymbol{h} -\boldsymbol{e}_{1}-\boldsymbol{e}_{3}\, ,\,  
\boldsymbol{h} -\boldsymbol{e}_{1}-\boldsymbol{e}_{4}\, , 
\\
&  \hspace{0.5cm} \boldsymbol{h} -\boldsymbol{e}_{1}-\boldsymbol{e}_{5}\, , \, \boldsymbol{h} -\boldsymbol{e}_{2}-\boldsymbol{e}_{3}\, , \, \boldsymbol{h} -\boldsymbol{e}_{2}-\boldsymbol{e}_{4}\, , \,  \boldsymbol{h} -\boldsymbol{e}_{2}-\boldsymbol{e}_{5}\, ,
\\
& \, \boldsymbol{h} -\boldsymbol{e}_{3}-\boldsymbol{e}_{4}\, , \, \boldsymbol{h} -\boldsymbol{e}_{3}-\boldsymbol{e}_{5}
\, , \,
 \boldsymbol{h} -\boldsymbol{e}_{4}-\boldsymbol{e}_{5}
 \, , \, 
2 \boldsymbol{h} -\boldsymbol{e}_{1}-\boldsymbol{e}_{2}-\boldsymbol{e}_{3}-\boldsymbol{e}_{4}-\boldsymbol{e}_{5}\, \bigg) \, .
\end{align*}
The canonical class is 
$
\boldsymbol{\kappa}=-3 \boldsymbol{h}+\sum_{i=1}^5 \boldsymbol{e}_{i}$ and as a basis for its orthogonal $\boldsymbol{\kappa}^\perp$, we take the one formed by the following classes: 
$$
\scalebox{0.9}{
\begin{tabular}{c}
$\boldsymbol{f}_{1}=\frac{1}{2}\Big( \, {\boldsymbol{h}}+\boldsymbol{e}_{1}-\boldsymbol{e}_{2}-{\boldsymbol{e}_{3}}-{\boldsymbol{e}_{4}}-{\boldsymbol{e}_{5}}\,\Big)
\quad 
\boldsymbol{f}_{2}=
\frac{1}{2}\Big( \, {\boldsymbol{h}}-{\boldsymbol{e}_{1}}+{\boldsymbol{e}_{2}}-{\boldsymbol{e}_{3}}-{\boldsymbol{e}_{4}}-{\boldsymbol{e}_{5}}\Big)
\quad 
\boldsymbol{f}_{3}=
\frac{1}{2} \Big(\,  {\boldsymbol{h}}
-{\boldsymbol{e}_{1}}-{\boldsymbol{e}_{2}}+
{\boldsymbol{e}_{3}}-{\boldsymbol{e}_{4}}-{\boldsymbol{e}_{5}}\, \Big)$\vspace{0.25cm}\\
$\boldsymbol{f}_{4}= \frac{1}{2} \Big(\, 
\, {\boldsymbol{h}}-{\boldsymbol{e}_{1}}-{\boldsymbol{e}_{2}}-{\boldsymbol{e}_{3}}+
{\boldsymbol{e}_{4}}-{\boldsymbol{e}_{5}}\, \Big)
\quad \mbox{ and } \quad 
\boldsymbol{f}_{5}=\frac{1}{2} \Big(\,  {\boldsymbol{h}}
-{\boldsymbol{e}_{1}}-{\boldsymbol{e}_{2}}-{\boldsymbol{e}_{3}}
-{\boldsymbol{e}_{4}}
+{\boldsymbol{e}_{5}}
\,\Big)$\, .
\end{tabular}}
$$

For $i=1,\ldots,16$, we denote by $\omega_i
=\omega_{\boldsymbol{\ell}_i}
\in  \mathbf R^5 $  the vector of coordinates of  $\boldsymbol{\ell}_i-\big(\boldsymbol{\kappa},\boldsymbol{\ell}_i\big)\boldsymbol{\kappa}\in \boldsymbol{\kappa}^\perp $ expressed in the basis $\boldsymbol{f}=(\boldsymbol{f}_i\big)_{i=1}^5$. These sixteen 5-tuples are the weights of the spinor representation we are working with:
\begin{align}
\label{Eq:Weights-R5}
 \nonumber
\omega_{{1}}=&\, \left({\frac{1}{2}}, -{\frac{1}{2}}, -{\frac{1}{2}}, -{\frac{1}{
2}}, -{\frac{1}{2}}\right)  && 
\omega_{{2}}=\left(-{\frac{1}{2}}, {\frac{1}{2}}, -{
\frac{1}{2}}, -{\frac{1}{2}}, -{\frac{1}{2}}\right) 
&& \omega_{{3}}=\left(-{\frac{1}{
2}}, -{\frac{1}{2}}, {\frac{1}{2}}, -{\frac{1}{2}}, -{\frac{1}{2}}
\right) 
\\
 \nonumber
\omega_{{4}}=&\,  \left(-{\frac{1}{2}}, -{\frac{1}{2}}, -{\frac{1}{2}}, {\frac{
1}{2}}, -{\frac{1}{2}}\right)
&& 
\omega_{{5}}=
 \left(-{\frac{1}{2}}, -{\frac{1}{2}}, -
{\frac{1}{2}}, -{\frac{1}{2}}, {\frac{1}{2}}\right),
&&
\omega_{{6}}=
 \left(-{\frac{1}{
2}}, -{\frac{1}{2}}, {\frac{1}{2}}, {\frac{1}{2}}, {\frac{1}{2}}\right)
\\
\omega_{{7}}=&\, 
 \left(-{\frac{1}{2}}, {\frac{1}{2}}, -{\frac{1}{2}}, {\frac{1}{2}}
, {\frac{1}{2}}\right)
&&
\omega_{{8}}=
 \left(-{\frac{1}{2}}, {\frac{1}{2}}, {\frac{1}{2}}, -{\frac{1}{2}}, {\frac{1}{2}}\right)
&&
\omega_{{9}}=
 \left(-{\frac{1}{2}}, {\frac{1}{2}}, {\frac{1}{2}}, {\frac{1}{2}}, -{\frac{1}{2}}\right)
 \\
 \nonumber
\omega_{{10}}=&\, 
\left({\frac{1}{2}}, -{\frac{1}{2}}, -{\frac{1}{2}}, {\frac{1}{2}}, {
\frac{1}{2}}\right)
&&
\omega_{{11}}=
 \left({\frac{1}{2}}, -{\frac{1}{2}}, {\frac{1}{2}}, -{\frac{1}{2}}, {\frac{1}{2}}\right)
&&
\omega_{{12}}= 
 \left({\frac{1}{2}}, -{\frac{1}{2}}, {\frac{1}{2}}, {\frac{1}{2}}, -{\frac{1}{2}}\right)
  \\
\omega_{{13}}=&\,  \left({
\frac{1}{2}}, {\frac{1}{2}}, -{\frac{1}{2}}, -{\frac{1}{2}}, {\frac{1}{2}}\right)
&& \omega_{{14}}=
 \left({\frac{1}{2}}, {\frac{1}{2}}, -{\frac{1}{2}}, {
\frac{1}{2}}, -{\frac{1}{2}}\right)
&& \omega_{{15}}=
 \left({\frac{1}{2}}, {\frac{1}{2}
}, {\frac{1}{2}}, -{\frac{1}{2}}, -{\frac{1}{2}}\right)
 \nonumber
\end{align}
\vspace{-0.2cm}
\begin{align*}
{}^{} \hspace{-1.1cm} \mbox{and } \quad \omega_{{16}}= 
 \left({\frac{1}{2}}, {\frac{1}{2}}, {\frac{1}{2}}, {\frac{1}{2}}, {\frac{1}{2}}\right)\, . 
 \nonumber
\end{align*}

As generators of the Weyl group, we take the involutions  $S_i$ for $i=1,\ldots,4$, which exchange  $\boldsymbol{f}_i $ with $\boldsymbol{f}_{i+1} $ and let the others $\boldsymbol{f}_k $'s fixed, and $S_5$ is the one such that   $S_5(\boldsymbol{f}_i)=\boldsymbol{f}_i$ for $i=1,2,3$ and $S_5(\boldsymbol{f}_j)=-\boldsymbol{f}_k$ for $\{j,k\}=\{4,5\}$. The corresponding matrices in the basis $\boldsymbol{f}$ are denoted the same: 
\begin{align*}
S_{\hspace{-0.04cm}1}=\scalebox{0.6}{$\left[\begin{array}{ccccc}
0 & 1 & 0 & 0 & 0 
\\
 1 & 0 & 0 & 0 & 0 
\\
 0 & 0 & 1 & 0 & 0 
\\
 0 & 0 & 0 & 1 & 0 
\\
 0 & 0 & 0 & 0 & 1 
\end{array}\right]$} 
\quad 
S_{\hspace{-0.04cm}2}=\scalebox{0.6}{$\left[\begin{array}{ccccc}
1 & 0 & 0 & 0 & 0 
\\
 0 & 0 & 1 & 0 & 0 
\\
 0 & 1 & 0 & 0 & 0 
\\
 0 & 0 & 0 & 1 & 0 
\\
 0 & 0 & 0 & 0 & 1 
\end{array}\right]$}
\quad 
S_{\hspace{-0.04cm}3}=\scalebox{0.6}{$\left[\begin{array}{ccccc}
1 & 0 & 0 & 0 & 0 
\\
 0 & 1 & 0 & 0 & 0 
\\
 0 & 0 & 0 & 1 & 0 
\\
 0 & 0 & 1 & 0 & 0 
\\
 0 & 0 & 0 & 0 & 1 
\end{array}\right]$} 
\quad 
S_{\hspace{-0.04cm}4}= \scalebox{0.6}{$\left[\begin{array}{ccccc}
1 & 0 & 0 & 0 & 0 
\\
 0 & 1 & 0 & 0 & 0 
\\
 0 & 0 & 1 & 0 & 0 
\\
 0 & 0 & 0 & 0 & 1 
\\
 0 & 0 & 0 & 1 & 0 
\end{array}\right]$}
\quad  
S_{\hspace{-0.04cm}5}=\scalebox{0.6}{$\left[\begin{array}{ccccc}
1 & 0 & 0 & 0 & 0 
\\
 0 & 1 & 0 & 0 & 0 
\\
 0 & 0 & 1 & 0 & 0 
\\
 0 & 0 & 0 & 0 & -1 
\\
 0 & 0 & 0 & -1 & 0 
\end{array}\right]$}\,.
\end{align*}

As a Coxeter element, we take 
$$C=S_{\hspace{-0.04cm}1}S_{\hspace{-0.04cm}2}S_{\hspace{-0.04cm}3}S_{\hspace{-0.04cm}4}S_{\hspace{-0.04cm}5}=\scalebox{0.7}{$\left[\begin{array}{ccccc}
0 & 0 & 0 & -1 & 0 
\\
 1 & 0 & 0 & 0 & 0 
\\
 0 & 1 & 0 & 0 & 0 
\\
 0 & 0 & 1 & 0 & 0 
\\
 0 & 0 & 0 & 0 & -1 
\end{array}\right]$}
$$
whose order is verified to be the Coxeter number $h=h_{D_5}$ of $W(D_5)$, namely $h=8$.  Let $P_C$ be the Coxeter plane  and $\Pi_C: \mathbf R^5\rightarrow P_C$ be the corresponding Coxeter projection. 
By definition, $P_C$ is the 2-plane in $\mathbf R^5$ which is fixed by $C$ and such that $C\lvert _{P_C}$ is a rotation of angle $2\pi/h=\pi/4$.  Over the complex, the Coxeter element $C$ admits $\lambda=e^{\frac{2i\pi}{h}}$ as an eigenvalue, with associated eigenspace of dimension 1 and spanned by the eigenvector 
$\nu=\big( \, {(-1+i)}/{\sqrt{2}}\, , \, i \, , \, 
 {(1+i)}/{\sqrt{2}}\, , \, 1 \, , \, 0 \, \big)$. 
 Then $P_C$ is the plane spanned by the real and the imaginary parts of $v$. After normalization, 
 we get that the following two vectors form a basis of $P_C$ which is 
 orthonormal with respect to the standard Euclidean structure on $\mathbf R^5$: 
$$ \nu_{\rm re}=\left(\,
-\frac{1}{2} \, , \,  0  \, , \, \frac{1}{2}  \, , \, \frac{1}{\sqrt{2}}  \, , \, 0 \, 
\right)\, , \, 
\nu_{\rm im}= 
 \left(\,
\frac{1}{2} \, , \,   \frac{1}{\sqrt{2}}  \, , \, \frac{1}{2}  \, , \, 0  \, , \, 0 \, 
\right)
$$ 
Then with respect to the bases $\boldsymbol{f}$ and 
$\big(\nu_{\rm re} , \nu_{\rm im}\, \big)$, Coxeter projection 
is written  $x\mapsto \big( 
\langle x,  \nu_r\rangle \, , \, \langle x,  \nu_{im}\rangle 
\big)$ and one can compute all points $\Pi_C(\omega_i)\in \mathbf R^{2}$ for $i=1,\ldots,16$. One obtains the sixteen black dots in Figure \ref{16-Coxeter-Projection-of-Weights} with the corresponding label indicated in green inside each. Then (up to an irrelevant rotation), one can identify the black dots of our figure with the vertices of \cite[Figure 3.(a)]{DaiseyDucat}. These latter being labeled by Ducat's notation $x_1,\ldots,x_8,a_1,\ldots,a_8$ for the weight variables, one can associate one of our weights $\omega_i$ to each of them. On the other hand, we indicated 
what are the weights of Wick's coordinates in \cite[Table 4.1]{PirioAFST}.  We thus deduce the following relations between Ducat's labelling of the weight variables and the components of Wick's parametrization: there exist complex constants $ \nu_i$'s (for $i=1,\ldots,10$) such that 
the following relations (as rational functions) hold true: 
 \begin{align}
\frac{x_1}{a_5} = & \, \nu_1 \, x_{1 4},\quad 
 \frac{x_2}{a_5} = \nu_2\, P_{2},  \quad 
 \frac{x_ 3}{a_5} = \nu_3 \,x_{1 3}, \quad 
 \frac{a_1}{a_5} = \nu_4 { P}_{5}, \quad 
  \frac{a_2}{a_5} = \nu_5  \, x_{15},\,  
  \nonumber 
 \\ 
  \frac{a_3}{a_5} =& \,   \nu_6 \, x_{34} , \quad  
  \frac{a_4}{a_5} = \nu_7  \,{P}_{4},\quad 
  \frac{a_6}{a_5} = \nu_8 \,
{P}_{1},\quad 
  \frac{a_7}{a_5} = \nu_9 \, x_{12},\quad 
  \frac{a_8}{a_5} = \nu_{10} \, x_{45}\, . 
  \nonumber 
\end{align} 

Taking  
$\big(\nu_i\big)_{i=1}^{10}=\big( 
\,  1\, , \,  -1\, , \,  1\, , \,  1\, , \,  1\, , \,  1\, , \, -1\, , \,  -1\, , \,  1\, , \,  1
\big)$ (that is $\nu_i=-1$ for $i\in \{2,7,8\}$ and $\nu_i=1$ otherwise), 
the relations above can be solved nicely.  We obtain that 
$A=(x_{ij})_{i,j=1}^5\in {\rm Asym}_5(\mathbf C)$ is such that $W(A)\in \mathbb S_5$ has 
Ducat's coordinates $x_1,\ldots,x_8,a_1,\ldots,a_8$ if and only if the following equality is satisfied: 
\begin{equation}
\label{Eq:A=Truc}
 \begin{pmatrix}
0 &x_{12}  & 
{x_{13}}  & 
{x_{14}}  & {x_{15}}  
\\
 -x_{12}  & 0 & {x_{23}} & 
x_{24}  & 
x_{25}\\
 -{x_{13}}  &  -{x_{23}}  & 0 & {x_{34}}  & {x_{35}}
 \\
 -{x_{14}}   & -{x_{24}}  & -{x_{34}}  & 0 & {x_{45}}  
\\
 -{x_{15}}   & -{x_{25}} & -{x_{35}}  & - {x_{45}}    & 0 
 \end{pmatrix}
=
\frac{1}{a_5} \begin{pmatrix}
0 & {a_7}  & 
{x_3}  & 
{x_1}  & {a_2}  
\\
 -{a_7}  & 0 & {X_{12}} & 
 {X_{23}}  & 
 {X_{123}}\\
 -{x_3}  & -{X_{12}} & 0 & {a_3}  & 
 {X_1}
 \\
 -{x_1}  & -{X_{23}}  & -{a_3}  & 0 & {a_8}  
\\
 -{a_2}  & -{X_{123}} & -{X_1}  & -{a_8}  & 0 
 \end{pmatrix} \, .
\end{equation}

This proves the first point of Proposition \ref{Prop:Cluster}. 
\mk 

With 
\eqref{Eq:A=Truc} at hand, proving the second point of Proposition  \ref{Prop:Cluster} becomes straightforward using some  formulas obtained in our previous article. 
In \cite{PirioAFST}, we gave explicit expressions for the 10 face maps expressed in Wick coordinates.  The weight polytope $\Delta_{D_5}\subset \mathbf R^5$ is the convex enveloppe

\newpage 
${}^{}$
\vspace{6.2cm}

\begin{center}
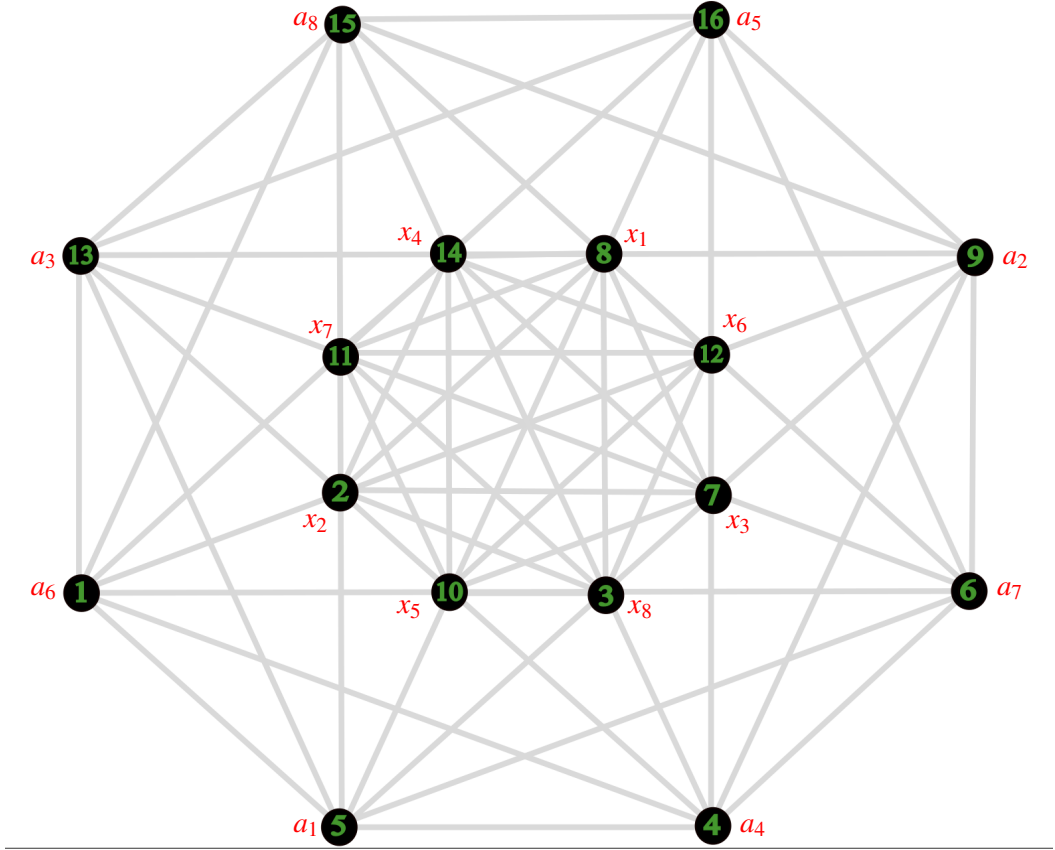
\begin{figure}[h!]
\begingroup%
  \makeatletter%
  \providecommand\color[2][]{%
    \errmessage{(Inkscape) Color is used for the text in Inkscape, but the package 'color.sty' is not loaded}%
    \renewcommand\color[2][]{}%
  }%
  \providecommand\transparent[1]{%
    \errmessage{(Inkscape) Transparency is used (non-zero) for the text in Inkscape, but the package 'transparent.sty' is not loaded}%
    \renewcommand\transparent[1]{}%
  }%
  \providecommand\rotatebox[2]{#2}%
  \newcommand*\fsize{\dimexpr\f@size pt\relax}%
  \newcommand*\lineheight[1]{\fontsize{\fsize}{#1\fsize}\selectfont}%
  \ifx\svgwidth\undefined%
    \setlength{\unitlength}{225bp}%
    \ifx\svgscale\undefined%
      \relax%
    \else%
      \setlength{\unitlength}{\unitlength * \real{\svgscale}}%
    \fi%
  \else%
    \setlength{\unitlength}{\svgwidth}%
  \fi%
  \global\let\svgwidth\undefined%
  \global\let\svgscale\undefined%
  \makeatother%
  \begin{picture}(1,0.5)%
    \lineheight{1}%
    \setlength\tabcolsep{0pt}%
    \put(-0.4,0){\includegraphics[scale=0.5]{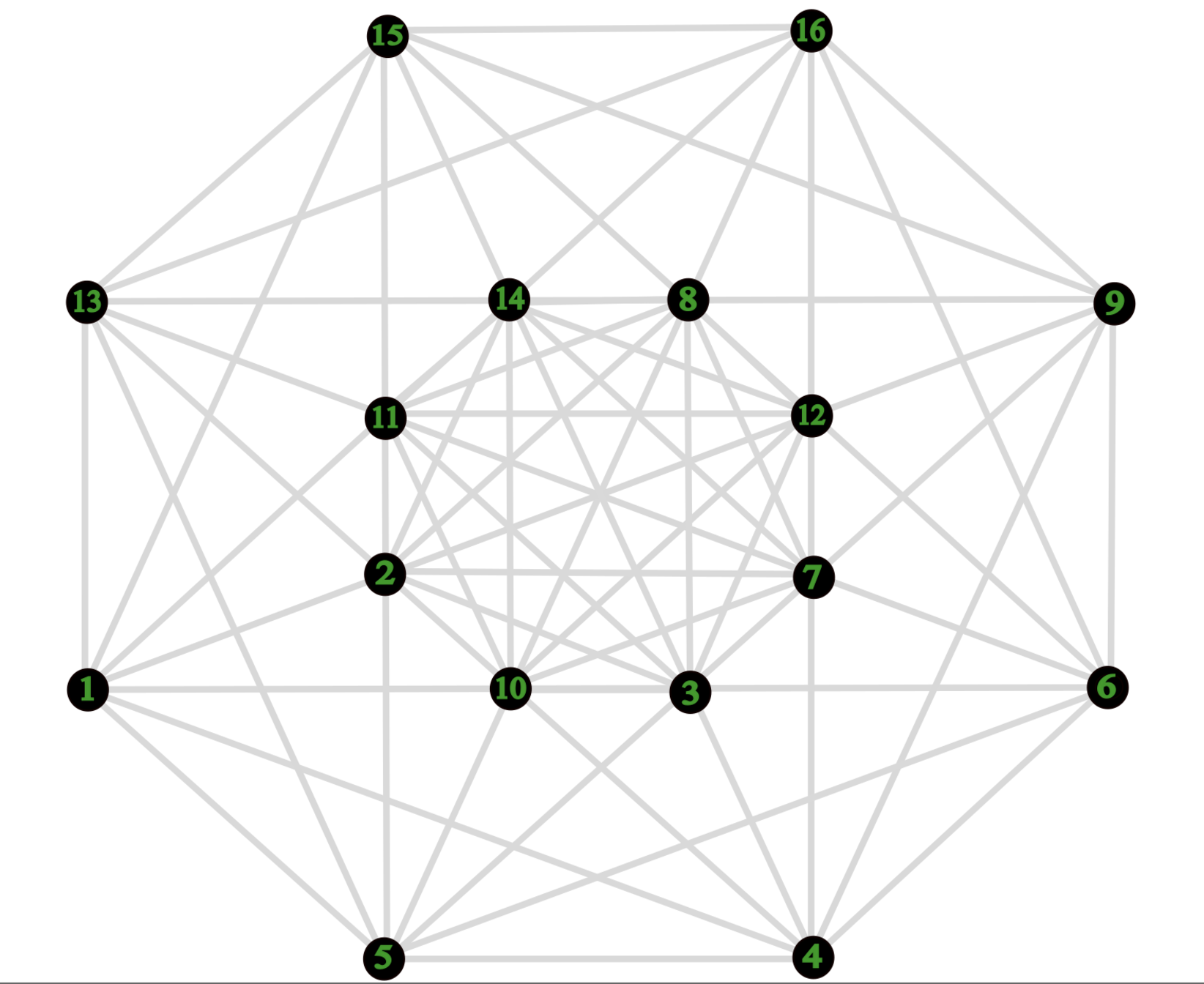}}%
    \put(0.6349934,1.017){\color[rgb]{0.07843137,0.62352941,0}\makebox(0,0)[lt]{\lineheight{1.25}\smash{\begin{tabular}[t]{l}
    ${\red{x_1}}$
    \end{tabular}}}}
    \put(0.097,0.53938042){\color[rgb]{0.07843137,0.62352941,0}\makebox(0,0)[lt]{\lineheight{1.25}\smash{\begin{tabular}[t]{l}
     ${\red{x_2}}$
    \end{tabular}}}}
    \put(0.082,1.38){\color[rgb]{0.07843137,0.62352941,0}\makebox(0,0)[lt]{\lineheight{1.25}\smash{\begin{tabular}[t]{l}
     ${\red{a_8}}$
    \end{tabular}}}}
    \put(-0.3577675,0.4257964){\color[rgb]{0.07843137,0.62352941,0}\makebox(0,0)[lt]{\lineheight{1.25}\smash{\begin{tabular}[t]{l}
     ${\red{a_6}}$ 
    \end{tabular}}}}
    \put(0.25580761,1.02){\color[rgb]{0.07843137,0.62352941,0}\makebox(0,0)[lt]{\lineheight{1.25}\smash{\begin{tabular}[t]{l}
     ${\red{x_4}}$
    \end{tabular}}}}    
    \put(-0.3577675,0.9808701){\color[rgb]{0.07843137,0.62352941,0}\makebox(0,0)[lt]{\lineheight{1.25}\smash{\begin{tabular}[t]{l}
      ${\red{a_3}}$
    \end{tabular}}}}
    \put(0.253549,0.3947964){\color[rgb]{0.07843137,0.62352941,0}\makebox(0,0)[lt]{\lineheight{1.25}\smash{\begin{tabular}[t]{l}
     ${\red{x_{5}}}$
    \end{tabular}}}}
    \put(0.08208947,0.02836255){\color[rgb]{0.07843137,0.62352941,0}\makebox(0,0)[lt]{\lineheight{1.25}\smash{\begin{tabular}[t]{l}
    ${\red{a_1}}$
    \end{tabular}}}} 
    \put(0.107,0.8638042){\color[rgb]{0.07843137,0.62352941,0}\makebox(0,0)[lt]{\lineheight{1.25}\smash{\begin{tabular}[t]{l}
     ${\red{x_7}}$
    \end{tabular}}}}
    \put(0.6409934,0.3947964){\color[rgb]{0.07843137,0.62352941,0}\makebox(0,0)[lt]{\lineheight{1.25}\smash{\begin{tabular}[t]{l}
${\red{x_{8}}}$
    \end{tabular}}}}
    \put(0.83,0.02836255){\color[rgb]{0.07843137,0.62352941,0}\makebox(0,0)[lt]{\lineheight{1.25}\smash{\begin{tabular}[t]{l}
    ${\red{a_4}}$
    \end{tabular}}}}
    \put(0.825,1.38) 
    {\color[rgb]{0.07843137,0.62352941,0}\makebox(0,0)[lt]{\lineheight{1.25}\smash{\begin{tabular}[t]{l}${\red{a_5}}$\end{tabular}}}}
    \put(1.2702667,0.9808701){\color[rgb]{0.07843137,0.62352941,0}\makebox(0,0)[lt]{\lineheight{1.25}\smash{\begin{tabular}[t]{l}
    ${\red{a_2}}$
    \end{tabular}}}}
    \put(1.2602667,0.4257964){\color[rgb]{0.07843137,0.62352941,0}\makebox(0,0)[lt]{\lineheight{1.25}\smash{\begin{tabular}[t]{l}
    ${\red{a_7}}$
    \end{tabular}}}}
    \put(0.805,0.53738042){\color[rgb]{0.07843137,0.62352941,0}\makebox(0,0)[lt]{\lineheight{1.25}\smash{\begin{tabular}[t]{l}
     ${\red{x_3}}$
    \\\end{tabular}}}}
         \put(0.8,0.8738042){\color[rgb]{0.07843137,0.62352941,0}\makebox(0,0)[lt]{\lineheight{1.25}\smash{\begin{tabular}[t]{l}
     ${\red{x_6}}$
    \\\end{tabular}}}}
  \end{picture}%
\endgroup%
\mk 
\caption{Projection of Gosset's graph\protect \footnotemark on the Coxeter plane and correspondance between
the weights (with their label used in this paper in green) 
   and the spinor coordinates of \cite[Figure 3.(a)]{DaiseyDucat} (in red).}
   \label{16-Coxeter-Projection-of-Weights}
\end{figure}
   \footnotetext{Gosset's graph is the 1-sekeleton of Gosset's polytope $1_{21}$, which here coincides with the weight polytope $\Delta_{D_5}$.}
\end{center}

\vspace{-0.5cm}
\noindent 
of the 16 weights \eqref{Eq:Weights-R5}. It has 10 $\mathcal W$-relevant 
 facets, five `positive' facets $\Delta_{D_5,i}^+=\Delta_{D_5}\cap \big\{\,x_i=+1/2\, \big\}$ and five `negative' ones $\Delta_{D_5,i}^-=\Delta_{D_5}\cap \big\{\,x_i=-1/2\, \big\}$ (with $i=1,\ldots,5$). 

The face maps $\mathbb S_5\dashrightarrow \mathbb S_4$ associated with 
the five positive facets are written very simply in Wick coordinates since they correspond to the maps ${\rm Asym}_5(\mathbf C)\ni  A
\mapsto  A_{\hat \imath}\in  {\rm Asym}_4(\mathbf C)$ 
where $A_{\hat \imath}$ stands for the $4\times 4$ matrix obtained by removing from $A$ 
its $i$-th line and  its $i$-th column ({\it cf.}\,Proposition 4.7 in \cite{PirioAFST}). 
Viewed \eqref{Eq:A=Truc}, we immediately get the second point of Proposition  \ref{Prop:Cluster}  for the five face maps associated to the positive facets of the weight polytope. \sk 

The expressions in Wick coordinates of the five face maps associated to the negative facets of $\Delta_{D_5}$ are more involved.  Some birational models of these maps are the maps 
$A\mapsto A_{i,j}=\Psi_{i,j}(A)$ with $i\neq j$ where 
$\Psi_{i,j}:  {\rm Asym}_5(\mathbf C)
\dashrightarrow   {\rm Asym}_4(\mathbf C)$ 
is the rational map considered in (actually just above) 
Corollary 4.8 of \cite{PirioAFST}.\mk

Using \eqref{Eq:A=Truc} and the explicit formulas for the $\Psi_{i,j}$'s, it is straightforward to get the following expressions for some models of the five `negative face map' written in terms of the initial cluster coordinates: 
\begin{align*}
A_{1,5} =  &\, 
\frac{1}{a_2}  
\left[\begin{array}{cccc}
0 & - {a_4}  {a_5}  & - {a_5}  {X_3}  &  {a_7}  
\\
  {a_4}  {a_5}  & 0 & - {a_5}  {x_2}  &  {x_3}  
\\
  {a_5}  {X_3}  &  {a_5}  {x_2}  & 0 &  {x_1}  
\\
 - {a_7}  & - {x_3}  & - {x_1}  & 0 
\end{array}\right]
\hspace{1.1cm} 
 A_{2,1} =    
\frac{1}{a_{7}} 
\left[\begin{array}{cccc}
0 & X_{12}  & X_{23}  & X_{123}
\\
-X_{12}   & 0 & - {a_1}  {a_5}  & a_4a_5 
\\
-X_{23}  &  {a_1}  {a_5}  & 0 & a_5 {X_{3}}  
\\
 -X_{123}  & -a_4a_5   & -a_5 {X_{3}}    & 0 
\end{array}\right] \\
  A_{3,4} =\, &   \frac{1}{a_{3}} 
\left[\begin{array}{cccc}
0 &  {a_1}  {a_5}  &  {x_3}  & - {a_5}  {x_2}  
\\
 - {a_1}  {a_5}  & 0 &  {X_{12}}  & - {a_6}  {a_5}  
\\
 - {x_3}  & - {X_{12}}  & 0 & - {X_1}  
\\
  {a_5}  {x_2}  &  {a_6}  {a_5}  &  {X_1}  & 0 
\end{array}\right]
\hspace{1cm}
A_{4,1} =\,   \frac{1}{x_{1}} 
\left[\begin{array}{cccc}
0 &  {X_{23}}  &  {a_3}  &  {a_8}  
\\
 - {X_{23}}  & 0 & - {a_1}  {a_5}  &  {a_5}  {X_3}  
\\
 - {a_3}  &  {a_1}  {a_5}  & 0 &  {a_5}  {x_2}  
\\
 - {a_8}  & - {a_5}  {X_3}  & - {a_5}  {x_2}  & 0 
\end{array}\right]\\
& \hspace{1cm} \mbox{and } \quad  A_{5,4} =  \, \frac{1}{a_{8}} 
 \left[\begin{array}{cccc}
0 &  {a_5}  {X_3}  &  {a_5}  {x_2}  & - {a_2}  
\\
 - {a_5}  {X_3}  & 0 &  {a_6}  {a_5}  & - {X_{123}}  
\\
 - {a_5}  {x_2}  & - {a_6}  {a_5}  & 0 & - {X_1}  
\\
  {a_2}  &  {X_{123}}  &  {X_1}  & 0 
\end{array}\right] \, .
\end{align*}
Up to an irrelevant minus sign for some entries, the coefficients of these matrices are all either monomials in the frozen cluster variables or products of a cluster variable with such a monomial, which completes the proof of Proposition \ref{Prop:Cluster}.
\vspace{-0.4cm}
\begin{center}
$\star$
\end{center}

We finish this subsection by two remarks. \sk 

The first is that, if our results above are only about the cluster nature of Gelfand-MacPherson's web of the spinor tenfold $\mathbb S_5$, we are convinced that they must have a kind of $\mathcal X$-counterparts for the web $\boldsymbol{\mathcal W}^{GM}_{ \boldsymbol{\mathcal Y}_5}$ which is truly the one we are interested in. The lack of a theory of `coefficients' and above all of their mutations is the reason why we didn't work with $\boldsymbol{\mathcal W}^{GM}_{ \boldsymbol{\mathcal Y}_5}$.  We hope to come back to this in the future.\mk 

Another remark concerning the above material is that while the  approach used to prove Proposition \ref{Prop:Cluster} ultimately relies on some explicit calculations, it  nevertheless points to a geometric way of constructing the LPA cluster structure on $\mathbb S_5$, namely by means of the face maps associated with the facets of the weight polytope. In an ongoing research project, we have verified that this geometric approach generalizes to several other homogeneous spaces. We hope to continue our researches in this direction in the near future as well.

\section{\bf The web $\boldsymbol{W^{+}_{\hspace{-0.1cm} \boldsymbol{Y}_5}}$, Bol's web, and their abelian relations}
\label{SS:WGM-+-Bol's-Web}
In this subsection, we study a certain 5-subweb $\boldsymbol{W^{+}_{\hspace{-0.1cm} \boldsymbol{Y}_5}}$ of $\boldsymbol{\mathcal W}_{\boldsymbol{Y}_5}^{GM}$ which carries interesting ARs which are in a 1-1 correspondance with those of Bol's web. In particular, one recovers Abel's 5-term identity from an 1-AR with logarithmic coefficients and monomial arguments, which is new as far we know.  The latter abelian relation will be used further in section \S\ref{SS:Recovering} for recovering the del Pezzo hypertrilogarithmic identities ${\bf HLog}_{{\rm dP}_4}^3$ from ${\bf HLOG}_{\boldsymbol{Y}_5}$. 
\bk

%
%
%
%

\subsection{\bf The web $\boldsymbol{W^{+}_{\hspace{-0.1cm} \boldsymbol{Y}_5}}$ and its abelian relations}

In this subsection, we aim to study the relations between the web 
\begin{equation*}
\boldsymbol{W}^{+}_{ {Y}_5}
= 
\boldsymbol{\mathcal W}
\Bigg(\, \Big(\,    {y_{25}}\, , \,     {y_{24}}    {y_{35}}
\, \Big)
\, , \, 
\Big(\,    \frac{1}{y_{13}}\, , \, 
\frac{   {y_{14}}    {y_{35}}}{   {y_{13}}}\, \Big)
\, , \,
  \bigg( 
  \,    {y_{24}}
 \, , \,  
      {y_{14}}    {y_{25}}
   \, \bigg)
\, , \, 
\bigg(\,     \frac{1}{y_{35}}
\, , \, 
\frac{   {y_{13}}    {y_{25}}}{   {y_{35}}}
\,
\bigg)
\, , \, 
 \Big(\,    {y_{14}}
\, , \, 
    {y_{13}}    {y_{24}}
 \, 
\Big) \,\Bigg)
\end{equation*}
and its abelian relations to those of Bol's web.

The 1-AR ${\bf AR}_\delta^1$ of 
Proposition \ref{Prop:AR-Omega1}, which corresponds to the 
differential 
identity 
$$
0=
\sum_{i=1}^5 \epsilon_i\, U_i^*\big(\,  
\delta \, \big)
 =\frac{1}{2}
\sum_{i=1}^5 \epsilon_i\, U_i^*\Big(\,  
{\rm Log}(x)\, {dy}/{y}-{\rm Log}(y)\,\, {dx}/{x}\, 
\Big)$$ spans the subspace of weight 1 abelian relations of $\boldsymbol{W}^{GM,+}_{ \boldsymbol{Y}_5}$: one has
$$
W_1\Big( \boldsymbol{AR}\big( \boldsymbol{W}^{GM,+}_{ \boldsymbol{Y}_5}
\big) 
\Big) =\big\langle 
{\bf AR}_\delta^1
\big\rangle \, .
$$

Determining the 1-ARs of weight 0 of $\boldsymbol{W}^{+}_{ {Y}_5}$  amouts to find the scalar constants $c_{i,s}$ (with $i=1,\ldots,5$ and $s=1,2$) such that the differential identity $
\sum_{i=1}^5 \Big(c_{i,1} \, {dU_{i,1}}/{U_{i,1}}+c_{i,2} \, {dU_{i,2}}/{U_{i,2}} \Big) =0
$ is satisfied, hence reduces to elementary linear algebra. We find that any 3-subweb of $\boldsymbol{W}^{GM,+}_{ \boldsymbol{Y}_5}$ carries such an 1-AR (which moreover is unique up to multiplication by a non-zero scalar) which will be said `combinatorial'. These ARs span a space 
$\boldsymbol{AR}_C\big( \boldsymbol{W}^{GM,+}_{ \boldsymbol{Y}_5}\big)$ which is of dimension 5. From Proposition \ref{Prop:Virtual-Ranks}{.2}, it follows that 
we have the following decomposition in direct sum: 
\begin{equation}
\label{Eq:Decomp-AR1-W+}
\boldsymbol{AR}^1\Big( \boldsymbol{W}_{\boldsymbol{Y}_5}^{+}\Big)=
{\boldsymbol{AR}^1_{C}\Big( \boldsymbol{W}_{\boldsymbol{Y}_5}^{+}\Big)^{{\red{5}}}}
\oplus \,
\big\langle 
{\bf AR}_\delta^1
\big\rangle^{{\red{1}}}
\end{equation}
with 
$\boldsymbol{AR}^1_{C}\Big( \boldsymbol{W}_{\boldsymbol{Y}_5}^{+}\Big)=
{W_0\Big( \boldsymbol{AR}^1\big(\boldsymbol{W}_{\boldsymbol{Y}_5}^{+}\big)\Big)}
$ and 
$\langle 
{\bf AR}_\delta^1
\big\rangle= W_1\Big( \boldsymbol{AR}^1\big(\boldsymbol{W}_{\boldsymbol{Y}_5}^{+}\big)\Big)$.

Let $\mathscr W'$ be the subgroup of the Weyl group $\boldsymbol{\mathscr W}_{\hspace{-0.1cm}D_5}$ 
spanned by the birational involutions $\sigma_i$ for $i=1,\ldots,4$ (see \eqref{Eq:Formula-Sigma-k}). It is a group isomorphic to the Weyl group of type $A_4$ (hence to the symmetric group $\mathfrak S_5$) which lets the 5-web $\boldsymbol{W}_{\boldsymbol{Y}_5}^{+}$ invariant.  
Moreover, it can be verified that 
$\mathscr W'$ acts by linear automorphisms on the space of 1-ARs of $\boldsymbol{W}_{\boldsymbol{Y}_5}^{+}$ making of it a representation of $\mathfrak S_5$. 
 By straightforward computations, we prove the 
 \begin{prop}
 The decomposition in direct sum \eqref{Eq:Decomp-AR1-W+} actually is the decomposition of $\boldsymbol{AR}\big( \boldsymbol{W}_{\boldsymbol{Y}_5}^{+}
 \big)$ into irreducible representations of $\mathfrak S_5$. 
 Moreover, we have the following isomorphisms of 
$\mathfrak S_5$-representations:
 \begin{equation}
 \label{Eq:1-ARs-W+}
 \boldsymbol{AR}^1_{C}\Big( \boldsymbol{W}_{\boldsymbol{Y}_5}^{+}\Big)
 \simeq V_{[221]}^{{\red{5}}}
 \qquad \mbox{ and } \qquad 
 \langle 
{\bf AR}_\delta^1
\big\rangle\simeq  V_{[1^5]}^{{\red{1}}}\, . 
 \end{equation}
 \end{prop}
Our aim is now to relate $\boldsymbol{W^{+}_{\hspace{-0.1cm} \boldsymbol{Y}_5}}$ and its ARs to Bol's web and more specifically  to its main abelian relation, the dilogarithmic AR given by Abel's relation. 

\subsection{\bf The web $\boldsymbol{W^{+}_{\hspace{-0.1cm} \boldsymbol{Y}_5}}$ and Bol's web}
In \S4.5.2 of 
\cite{PirioAFST},  
 we gave several birational models of Serganova-Skorobogatov's embedding 
$f_{S\hspace{-0.03cm}S}: X_r \hookrightarrow \boldsymbol{\mathcal Y}_r$.  Let $(a,b)\in \mathbf C^2$ be such that $X_r$ identifies with the blow-up of the plane at the vertices of the standard quadrilateral in $\mathbf P^2$ plus the point $[a:b:1]$. In 
\cite{PirioAFST} (around Corollary 4.12 more precisely), we 
 obtained that for generic complex constants $s_1,s_2,s_3\in \mathbf C$ and setting $s_{ij}=s_i-s_j$ for any $i,j$ and 
 $ \boldsymbol{\rm y}=\big(y_{1,3},y_{1,4}, y_{2,4}, y_{2,5}, y_{3,5}\big)$, the following rational map 
\begin{align*}
(x,y)& \longmapsto  \boldsymbol{\rm y}= 
\left(\, 
\frac{\left(x -1\right)  {s_{13}}}{\left(a -1\right)  {s_3}}
\, , \, \frac{x \left(y -1\right) \left(a -b \right)  {s_1}  {s_{23}}}{\left(x -y \right) a \left(b -1\right)  {s_3}  {s_{12}}}
\, , \, 
\frac{y \left(a -b \right)  {s_2}}{\left(x -y \right) b  {s_{12}}}
\, , \, 
\frac{\left(b -1\right)  {s_3}}{\left(y -1\right)  {s_{23}}}
\, , \, 
\frac{\left(x -y \right) \left(b -1\right)  {s_{12}}}{\left(y -1\right) \left(a -b \right)  {s_{23}}}
\right)
\end{align*}
is a birational model for $f_{S\hspace{-0.03cm}S}$.  If one is no longer interested by taking track of the parameters $a$ and $b$, by eliminating some parameters, one can consider the rational map 
%
%
%
%
\begin{align*}
(x,y)& \longmapsto  \boldsymbol{\rm y}= 
\left(\, 
g_{13}\,
\left(x -1\right) 
\, , \, 
g_{14}\,
\frac{ x \left(y -1\right)}{\left(x -y \right) }
\, , \, 
\frac{g_{24}\,y }{\left(x -y \right) }
\, , \, 
\frac{g_{25}}{\left(y -1\right) }
\, , \, 
g_{35}\,
\frac{\left(x -y \right) }{\left(y -1\right) }
\right)
\end{align*}
where the $g_{ij}$'s are non zero complex parameters required to satisfy the following polynomial relation:
\begin{align*}
1=g_{1  3} g_{1  4} g_{2  4} g_{2  5} g_{3  5}&\   -g_{1  3} g_{1  4} g_{2  4}-g_{1  3} g_{1  4} g_{3  5}-g_{1  3} g_{2  5} g_{3  5}-g_{2  4} g_{2  5} g_{1  4} \\
&\  -g_{2  4} g_{2  5} g_{3  5}  +g_{1  3} g_{1  4}+g_{1  3} g_{3  5}+g_{1  4} g_{2  4}+g_{2  4} g_{2  5}+g_{2  5} g_{3  5}\,.
\end{align*}
This relation admits $g_{13}=g_{14}=g_{24}=g_{25}=g_{35}=1$ as a peculiar solution.  We will work with the corresponding rational embedding of $\mathbf C^2 $ into $\boldsymbol{Y}_5=\mathbf C^5$, which is given by 
\begin{equation} 
\label{Eq:map-D} 
{\mathcal D} : (x,y)\longmapsto \big(\,
 y_{13}, y_{14}, y_{24}, y_{25}, y_{35} \,\big)= 
\bigg(\,  x -1
\, , \, 
 \frac{x \left(y -1\right)}{x -y}
\, , \, 
 \frac{y}{x -y}
\, , \, 
\frac{1}{y -1}
 \, , \, 
 \frac{x -y}{y -1}\, \bigg)\, .
 \end{equation} 

One verifies that 
\begin{align*}
{\mathcal D}^*\Big(  \boldsymbol{ W}^{+}_{ \boldsymbol{Y}_5}
\Big) 
=\boldsymbol{\mathcal W} 
\Bigg( \, \bigg(\frac{1}{y -1} ,   \frac{y}{y -1}\bigg)
\, , & \, 
\bigg(
\frac{1}{x -1}, \frac{x}{x -1}\bigg)\, ,  \,\bigg(\frac{y}{x -y}, \frac{x}{x -y}\bigg)
\, 
\\ 
&\, 
  \bigg(\frac{y -1}{x -y}, \frac{x -1}{x -y}
\bigg)
\, , \, 
 \bigg(\frac{x \left(y -1\right)}{x -y}, 
\frac{y\left(x -1\right)}{x -y}\bigg)\, \Bigg)
\end{align*} 
from which it follows that $d\big({U_{i,1}}\circ {\mathcal D}\big)\wedge 
d\big({U_{i,2}}\circ {\mathcal D}\big)=0$ for $i=1,\ldots,5$.  This implies that the pull-back under $\mathcal D$ of the 2-codimensional web $\boldsymbol{W}^{+}_{ \boldsymbol{Y}_5}$ is only 1-codimensional. More precisely, setting $\varpi_i={U_{i,1}}\circ {\mathcal D}$ for $i=1,\ldots,5$, that is 
$$
\varpi_1= \frac{1}{y -1} 
\, , \quad 
\varpi_2=
\frac{1}{x -1} 
\, , \quad 
\varpi_3 = \frac{y}{x -y} 
\, , \quad 
\varpi_4 =  \frac{y -1}{x-y}
\, , \quad 
\varpi_5 = \frac{x(y -1)}{x -y}\, ,
$$
it comes that ${\mathcal D}^*\Big(  \boldsymbol{W}^{+}_{ {Y}_5}
\Big) $ is the following web: $$
\boldsymbol{B}=
\boldsymbol{\mathcal W}\Big(\, \varpi_i\, \Big)_{i=1}^5 = 
\boldsymbol{\mathcal W}\bigg(\, 
\frac{1}{y -1} 
\, ,\, 
\frac{1}{x -1} 
\, ,\, 
 \frac{y}{x -y} 
\, ,\, 
 \frac{y -1}{x-y}
\, ,\, 
 \frac{x(y -1)}{x -y}
\, 
\bigg) \, .
$$ 
By using the criterion of \cite[\S3.5.2]{PirioAFST}, 
this web is easily seen  to be a model of Bol's web and, as such, 
carries a version of Abel's 5-terms identity of the dilogarithm. We are going to prove that the latter can be obtained in a very simple way from the logarithmic identity \eqref{Eq:AR-Omega1}.

We identify all the target affines spaces of the first integrals $U_i$ to a same 
affine space with standard coordinates denoted by $u_1,u_2$. 
Then one verifies that for any $i=1,\ldots,5$, the image of $  U_i\circ \mathcal D$ is the  line cut out by 
$0= 1+u_1-u_2$
which is parametrized by $L: z\mapsto (z,z+1)$. One has 
$$
L^*(\delta)=\frac{1}{2}\left( 
{\rm Log}(z) \frac{dz}{1+z}-
{\rm Log}(z+1) \frac{dz}{z}
\right) = - d {\mathsf R}(z) 
$$
where ${\mathsf  R}$ stands for the `{\it cluster dilogarithm}' defined by 
$$
{}^{} \quad 
{\mathsf  R}(u)=\frac{1}{2} \int_{0}^u 
\left( 
 \frac{{\rm Log}(1+t)}{t}
-  \frac{{\rm Log}(t)}{1+t}\,
\right) dt  \qquad \mbox{ for } \, u\geq 0\, .
$$
Pulling-back the identity $0=\frac{1}{2} \sum_{i=1}^5 \epsilon_i U_i^*
\big( 
{\rm Log}(x) \frac{dy}{y}-
{\rm Log}(y) \frac{dx}{x}\big)$ under $\mathcal D$ is written 
$$
\omega_1^*\big( d{\mathsf  R}\big)
-\omega_2^*\big( d{\mathsf  R}\big)+\omega_3^*\big( d{\mathsf  R}\big)
-\omega_4^*\big( d{\mathsf  R}\big)+\omega_5^*\big( d{\mathsf  R}\big)=0 
$$
or equivalently, in integrated form
\begin{equation}
\label{Eq:5-terms-R-cluster}
{\mathsf  R} \left( \frac{1}{y -1}  \right) 
-{\mathsf  R} \left( \frac{1}{x -1}  \right) 
+{\mathsf  R} \left( \frac{y}{x -y}  \right) 
-{\mathsf  R} \left( \frac{y -1}{x-y} \right) 
+{\mathsf  R} \left( \frac{x(y -1)}{x -y} \right) = \frac{{}^{}\, \pi^2}{3}\, 
\end{equation}
a functional relation which is identically satisfied on any complex domain containing the `triangle' $\{ (x,y)\in \mathbf R^2\, \big\lvert \, 1<y<x\, \}$. It is a version of Abel's identity of the dilogarithm. We will denote this identity as well as the corresponding AR of $\boldsymbol{B}$ by $\boldsymbol{Ab}_{\mathrm R}$. 

More is true: it can be verified that given any 1-AR of a 3-subweb of  $\boldsymbol{W}^{+}_{ {Y}_5}$, its pull-back under $\mathcal D$ is an abelian relation of the corresponding 3-subweb of $\boldsymbol{B}$. With a few more  computational checks, we get the following
\begin{prop}
\label{Prop:W+-B}
 The pull-back under $\mathcal D$ gives rise to two linear isomorphisms 
$$
\big\langle {\bf AR}^1_\delta \big\rangle \simeq \big\langle \boldsymbol{Ab}_{\mathrm R} \big\rangle
\qquad 
\mbox{ and }  \qquad 
\boldsymbol{AR}^1_{C}\big( \boldsymbol{W}_{\boldsymbol{Y}_5}^{+}\big) 
\simeq \boldsymbol{AR}_{C}\big( \boldsymbol{B}\big) 
$$
which actually turn out to be isomorphisms of $\mathfrak S_5$-representations.
\end{prop}
The last statement of this proposition requires a few lines of explanation.  The point is that 
the group $\mathscr W'=\langle \sigma_1,\ldots,\sigma_4\rangle \subset {\bf Bir}\big(\boldsymbol{Y}_5\big)$ lets invariant the image of the map \eqref{Eq:map-D} hence gives rise to a birational action on the source space of $\mathcal D$. More precisely, one verifies that setting 
$$\tilde \sigma_1(x,y)=(x/y,1/y)\, \quad 
\tilde \sigma_2(x,y)=(y,x)\, \quad 
\tilde \sigma_3(x,y)=(1-x,1-y)\, \quad 
\tilde \sigma_4(x,y)=(1/x,1/y)$$
then,  as rational maps from 
$\mathbf C^2 \dashrightarrow \mathbf C^5=\boldsymbol{Y}_5$, 
 one has  $\mathcal D\circ \tilde \sigma_i= \sigma_i \circ \mathcal D$
 for $i=1,\ldots,5$.  The birational action $\mathscr W'\rightarrow  {\bf Bir}\big(\mathbf C^2\big)$ gives rise to a linear  action of $\mathfrak S_5$ on $\boldsymbol{AR}\big( \boldsymbol{B}
 \big)$. This gives a precise and rigorous meaning to the last assertion of Proposition \ref{Prop:W+-B}.
 
 \begin{rem}
{\bf 1}.  $\boldsymbol{B}$ is isomorphic to the web $\boldsymbol{\mathcal W}_{\hspace{-0.05cm}\boldsymbol{\mathcal M}_{0,5}}$ 
defined on the moduli space $\boldsymbol{\mathcal M}_{0,5}$ by the five forgetful maps $\boldsymbol{\mathcal M}_{0,5}\rightarrow \boldsymbol{\mathcal M}_{0,4}$. Up to this isomorphism, the birational action of 
$\mathscr W'\simeq \mathfrak S_5$ on $\mathbf C^2$ corresponds to the one of 
${\rm Aut}(\boldsymbol{\mathcal M}_{0,5})\simeq \mathfrak S_5$ on $\boldsymbol{\mathcal M}_{0,5}$.  That $\boldsymbol{AR}\big(
\boldsymbol{\mathcal W}_{\hspace{-0.05cm}\boldsymbol{\mathcal M}_{0,5}}
\big) $ is isomorphic to $V_{[221]}^{{\red{5}}}\oplus   V_{[1^5]}^{{\red{1}}}$ as a $\mathfrak S_5$-representation is due to 
Damiano (see \cite[Remark 3.1]{PirioAFST}{\rm )}.\mk\\
{\bf 2}.  We give above a new derivation of the five-term identity of the dilogarithm 
(in the form \eqref{Eq:5-terms-R-cluster}) from the differential relation in five variables  \eqref{Eq:AR-Omega1} which involves only logarithms and monomial functions. As far we know, this is new. It would be interesting to investigate whether the latter relation in five variables can be obtained formally from Abel's identity in two variables or not. 
\mk\\
{\bf 3}.  The identities 
\eqref{Eq:AR-Omega1} and 
\eqref{Eq:5-terms-R-cluster} admit real analytic versions which are globally satisfied. 
Setting $\delta^+=\frac{1}{2}\big( {\rm Log}\lvert x\lvert dy/y- {\rm Log}\lvert y\lvert dx/x\big)$, one verifies that $\sum_{i=1}\epsilon_i U_i^*\big( \delta^+\big)=0$ holds true identically on $\boldsymbol{Y}_5$. Similarly, if one defines the `positive cluster dilogarithm' by setting 
${\mathsf  R}^+(u)=\frac{1}{2} \int_{0}^u 
\left( 
 {
 {\rm Log}\lvert 1+t \lvert}/{t}
- 
{
 {\rm Log}\lvert t\lvert }/{(1+t)}\,
\right) dt 
$, then  for each connected component $O$ of the complement 
in $\mathbf R^2$
of the arrangement of lines cut out by $xy \left(y -1\right) \left(x -1\right) \left(x -y \right)=0$, there exists an integer $n_O\in \mathbf Z$ such that  
$\sum_{i=1}^5 \epsilon_i \, {\mathsf  R}^+(\varpi_i)\equiv   \frac{{}^{}\, \pi^2}{3} n_O$  on $O$.
 \end{rem}
%
%
%
%

\section{\bf Recovering the functional identities $\boldsymbol{{\bf HLog}_{{\rm dP}_4}}$'s 
 from $\boldsymbol{{\bf HLOG}_{{\mathcal Y}_5}}$}
\label{SS:Recovering}
In this section, we state and prove a precised version of the seventh point of Theorem \ref{THM:WdP5-WGMS5-similarities} which is about the fact that any weight 3 hyperlogarithmic functional identity ${\bf HLog}_{{\rm dP}_4}$ 
of an arbitrary but fixed smooth del Pezzo quartic surface ${\rm dP}_4$ 
can be recovered from 
the identity between 2-differential forms ${\bf HLOG}_{{\mathcal Y}_5}$. 
\mk 

We will first argue symbolically: we will associate a symbolic relation to 
${\bf HLOG}_{\boldsymbol{Y}_5}$ 
and will explain how to get from it the symbol corresponding to the hyperlogarithmic identity 
${\bf HLog}({\rm dP}_d)$. Then we will argue more concretely, in terms of abelian relations. 

\subsection{\bf Arguing symbolically}
\label{SS:Arguing-symbolically}
Let $\square$ stand for the quadrilateral cut out by $xyz(z+x-y)=0$ in the homogeneous coordinates $[x:y:z]$ on $\mathbf P^2$. Then the space ${\bf H}_{\square}= {\bf H}^0\big(\mathbf P^2, \Omega^1_{\mathbf P^2}( {\rm Log}\,\square )\big)$ is of dimension 3 and 
admits the following  logarithmic 1-forms as a basis: 
$$
\nu_1=d{\rm Log}(x)=\frac{dx}{x}\, , \qquad 
\nu_2=d{\rm Log}(y)=\frac{dy}{y} \quad
\mbox{ and } \quad  
\nu_3=d{\rm Log}(1+x-y)=\frac{dx-dy}{1+x-y} \, . 
$$

Recall the 2-form $\Omega$ defined in \eqref{Eq:Omega-u1u2u3} (see also \label{Eq:Omega-xy}). To it, we associate the symbol 
$$
\mathcal S(\Omega)= \nu_1\otimes \big( \nu_2\wedge \nu_3
\big) 
- 
\nu_2\otimes \big( \nu_1\wedge \nu_3
\big) 
+
\nu_3\otimes \big( \nu_1\wedge \nu_2
\big)  \in 
{\bf H}_{\square} \otimes \wedge^2 {\bf H}_{\square}\, . 
$$

Let $\iota_{\square}: {\bf H}_{\square}\hookrightarrow \Omega^1_{\mathbf C(\mathbf P^2)}$ be the natural embedding into the space of rational 1-forms on the projective plane.  For $p\in \mathbf P^2\setminus \square$, one has a well-defined integration map 
${\rm II}_p : {\bf H}_{\square}\rightarrow \mathcal O_{\mathbf P^2,p}$, $\nu\mapsto 
\big(\, {\rm II}_p : z\mapsto 
\int_{p}^z \nu\big)$ from which one gets a realization map 
$$\mathcal R_p= {\rm II}_p \otimes \wedge^2 \iota_{\square} : {\bf H}_{\square} \otimes \wedge^2 {\bf H}_{\square}\rightarrow \Omega^2_{\mathbf P^2,p}$$ which associates to 
a tensor $\nu_a\otimes ( \nu_b\wedge \nu_c)$ the germ of 2-form $
\big(\int_{p}^\bullet \nu_a\big)
\big(\nu_b\wedge \nu_c\big)$ at $p$. Then it is clear that 
the germ of $\Omega$ at $p$ is the image of the symbol $\mathcal S(\Omega)$ by the 
realization map $\mathcal R_p$.

For any $i\in [\hspace{-0.05cm}[ 10  ]\hspace{-0.05cm}]$, the pull-back $U_i^*\big( \Omega\big)$ has the same nature on $\boldsymbol{Y}_5$ that $\Omega$ has: it is a linear combination of 2-forms of the type ${\rm Log}(f) \big( d{\rm Log}(g) \wedge {\rm Log}(h)\big)$ hence one can associate a symbol to it. For $i=1,\ldots,10$, let $\eta_i$ be the logarithmic derivative of the 
 polynomial $\zeta_i$ defined in \eqref{Eq:zeta-i}: one has 
 $$
 \eta_1=\frac{dy_{13}}{y_{13}},\quad 
  \eta_2=\frac{dy_{14}}{y_{14}},\quad 
  \eta_3=\frac{dy_{24}}{y_{24}},\quad 
  \eta_4=\frac{dy_{25}}{y_{25}},\quad 
  \eta_5=\frac{dy_{35}}{y_{35}}
  \quad 
  \mbox{and} \quad 
  \eta_{k+5}=\frac{dP_k}{P_k}\quad 
 $$
 for $k=1,\ldots,5$. Let ${\bf H}_{\boldsymbol{Y}_5}$ be the vector space admitting the $\eta_i$'s as a basis:\footnote{For concreteness, one can take for $\boldsymbol{H}$ the subspace spanned by the $\eta_k$'s in $\Omega^2_{\mathbf C( \boldsymbol{Y}_5)}$. But we believe it is better to deal with $\boldsymbol{H}$ as an abstract vector space and not as a subspace of a bigger space.} 
 $$
 {\bf H}_{\boldsymbol{Y}_5}=\bigoplus_{i=1}^{10} \mathbf C\,\eta_i\, . 
 $$
 
 For any $i\in [\hspace{-0.05cm}[ 10  ]\hspace{-0.05cm}]$ and any $s=1,2,3$,  one can express the logarithmic derivative of 
 $U_{i,s}$ as a linear combination in the $\eta_k$'s, with non-zero coefficients equal to $\pm 1$. For instance, for $i=1$, one has 
 \begin{align*}
 d {\rm Log}\big(U_{1,1}\big)=   \eta_4\, , \qquad 
  d {\rm Log}\big(U_{1,2}\big)=   \eta_3+\eta_5
  \quad 
\mbox{and } \quad   d {\rm Log}\big(U_{1,3}\big)=  \eta_6\, .
 \end{align*}
Accordingly, the `symbol'  we associate to 
 $U_1^*(\Omega)$ is the element of ${\bf H}_{\boldsymbol{Y}_5}\otimes \wedge^2
 {\bf H}_{\boldsymbol{Y}_5}$ given by 
 $$
 \mathcal S\Big(  U_1^*\big(\Omega\big)\Big)=
 \eta_4\otimes \big( (\eta_3+\eta_5) \wedge \eta_6\big)
 - (\eta_3+\eta_5) \otimes \big( \eta_4  \wedge \eta_6\big)
 +\eta_6\otimes \big(  \eta_4 \wedge (\eta_3+\eta_5) \big)\, . 
 $$
 
The $\eta_i$'s being closed, each is locally integrable on $\boldsymbol{Y}_5$. For any $z\in \boldsymbol{Y}_5\setminus Z$, one thus has a realization map 
\begin{equation}
\label{Eq:realization}
\mathcal R_z : {\bf H}_{\boldsymbol{Y}_5}\otimes \wedge^2
 {\bf H}_{\boldsymbol{Y}_5} \longrightarrow \Omega^2_{\boldsymbol{Y}_5,z}
\end{equation}
 which associates to a tensor $\eta_a\otimes (\eta_b\wedge \eta_c)$ the germ of 2-form 
 $(\int_{z}^\bullet \eta_a)\eta_b\wedge \eta_c$ at $z$.
 
 The nine other symbols $\mathcal S\big(  U_i^*(\Omega)\big)$, $i=2,\ldots,10$ can be computed easily. Recall (see \eqref{Eq:epsilon}) that $\epsilon_i=-1$ for $i\in \{2,4,7,9\}$, and $\epsilon_i=1$ otherwise. 
 \begin{lem}  
 \label{Lemma:1}
 In ${\bf H}_{\boldsymbol{Y}_5}\otimes \wedge^2
 {\bf H}_{\boldsymbol{Y}_5}$, one has 
 $
 \sum_{i=1}^{10} \epsilon_i \,  \mathcal S\big(  U_i^*\big(\Omega\big)\big)=0
 $.
 \end{lem}
 \begin{proof}
 This follows from the fact that $ \sum_{i=1}^{10} \epsilon_i \,   U_i^*\big(\Omega\big)\equiv 0$ and that the realization map \eqref{Eq:realization} is injective. But a more elementary proof is given by expressing the sum $S=
 \sum_{i=1}^{10} \epsilon_i \,  \mathcal S\big(  U_i^*\big(\Omega\big)\big)$  in the basis 
 $\eta_i\otimes (\eta_j\wedge \eta_k)$ of ${\bf H}_{\boldsymbol{Y}_5}\otimes \wedge^2
 {\bf H}_{\boldsymbol{Y}_5}$, for $i\in [\hspace{-0.05cm}[ 10  ]\hspace{-0.05cm}]$ and 
 $j,k$ such that $1\leq j<k\leq 10$.  Then one obtains that $S=0$ by pure elementary linear algebra.
\end{proof} 
From the very definition of $\Omega$ (see the formulas at the beginning of \S\ref{SSS:The-master-2-AR}), it can be easily verified that the symbol of $\Omega$
not only is an element of  
${\bf H}_{\square}\otimes \wedge^2
 {\bf H}_{\square}$ but actually lies in 
$\wedge^3 {\bf H}_{\square}$, the latter being seen as a subspace of $
 {\bf H}_{\square}\otimes {\bf H}_{\square}\otimes {\bf H}_{\square}$. It follows that 
 for any $i\in [\hspace{-0.05cm}[ 10  ]\hspace{-0.05cm}]$,  $S\big(  U_i^*\big(\Omega\big)\big)$
  lies in $ \wedge^3
 {\bf H}_{\boldsymbol{Y}_5}$ from what we deduce immediately the 
 \begin{lem}  
 The symbolic identity 
  of Lemma  \ref{Lemma:1} actually lies in $ \wedge^3
 {\bf H}_{\boldsymbol{Y}_5}$. One has 
\begin{equation}
\label{Eq: Sigma-i(UiOmega)}
 \sum_{i=1}^{10} \epsilon_i \,  S\Big(  U_i^*\big(\Omega\big)\Big)=0
 \quad \mbox{in  } \, \wedge^3
 {\bf H}_{\boldsymbol{Y}_5}\, . 
 \end{equation}
 \end{lem}

\begin{rem} 
\noindent 
In \cite{CP}, that ${\bf HLog}\big({\rm dP}_4\big)$ is satisfied is proven by showing that its symbol $\mathcal S\big( {\bf HLog}\big({\rm dP}_4\big)\big)$   vanishes in the third wedge product of the space of global 1-forms on ${\rm dP}_4$ with poles of order at most 1 along the lines contained in \( dP4 \). Hence, the vanishing of $\mathcal S\big( {\bf HLog}\big({\rm dP}_4\big)\big)=0$  is formally analogous to the vanishing of 
$\sum_{i=1}^{10} \epsilon_i \,  S\big(  U_i^*\big(\Omega\big)\big)$. 
However, there is a significant distinction: $\mathcal S\big( {\bf HLog}\big({\rm dP}_4\big)\big)$   is a sum of 10 weight 3 tensors, all of which are integrable in the sense of Chen, meaning that the iterated integrals associated with each tensor are homotopy invariant (see \cite[\S6]{Brown}  for further details). This is not the case for the symbols 
 $ \epsilon_i \,  S\big(  U_i^*\big(\Omega\big)\big)$
 as some (if not all) of them are not integrable. Therefore, \eqref{Eq: Sigma-i(UiOmega)} represents an algebraic identity that does not have an analytic counterpart in functional form.
\end{rem} 

From the material presented in \S\ref{SS:dP4-SS-Embedding}, we get that the map 
\begin{equation}
\label{Eq:F-SS}
F_{S\hspace{-0.03cm}S} : (x,y)\longmapsto \Big( \, Y_{13}\, , \, Y_{14}\, , \, 
Y_{24}\, , \,
Y_{25}\, , \,
Y_{35}\, \Big)
\end{equation}
with 
\begin{align*}
 {Y_{13}} =  &\, 
\frac{y \left(a +1\right)\left(a y -b x -a +b +x -y \right) }{\left(x -y \right) \left(b -y \right)}
\hspace{1.2cm}  {Y_{14}} = 
\frac{a x \left(y-1 \right)  \left(b +1\right)}{\left(b +a \right) \left(x -y \right)}
\\  {Y_{24}} = &\, 
\frac{\left(b -y \right) \left(-1+x \right) b}{\left(b +a \right) \left(\left(1-x \right) b +x +\left(-1+y \right) a -y \right)}
\hspace{0.6cm} 
 {Y_{25}} = 
\frac{\left(a -x \right) \left(x -y \right)}{\left(b +1\right) x \left(\left(1-x \right) b +x +\left(-1+y \right) a -y \right)} \\
& \, \hspace{2.6cm}
 \mbox{and } \quad   {Y_{35}} = \frac{\left(a y -b x \right) \left(b +a \right)}{x \left(b +1\right) \left(b -y \right)}\, . 
\end{align*}
is a birational model of Serganova-Skorobogatov embedding $f_{S\hspace{-0.03cm}S}: {\rm dP}_4\hookrightarrow \boldsymbol{\mathcal Y}_5$. 

By direct computation, we get that $F_{S\hspace{-0.03cm}S}^*\big( \boldsymbol{\mathcal W}^{GM}_{ \boldsymbol{Y}_5}\big)$ coincides with the 10-web on $\mathbf P^2$, denoted by $ \boldsymbol{W}_{a,b}$, which is the direct image of del Pezzo's web $\boldsymbol{\mathcal W}_{ {\rm dP}_4 }$
by the blow-up map $\beta=\beta_{a,b} : {\rm dP}_4\rightarrow \mathbf P^2$
at 
the five points 
$[1:0:0]$, $[0:1:0]$, $[0:0:1]$, $[1:1:1]$ and $[a:b:1]$: 
one has 
\begin{equation} 
\label{Eq:Web-above}
 \boldsymbol{W}_{a,b}=F_{S\hspace{-0.03cm}S}^*\Big( \boldsymbol{\mathcal W}^{GM}_{ \boldsymbol{Y}_5}\Big)=\beta_*\Big( \boldsymbol{\mathcal W}_{ {\rm dP}_4 }\Big)\, .
\end{equation} 

 Actually, 
we have that for any $i=1,\ldots,10$, one has $F_{S\hspace{-0.03cm}S}^*\big( 
dU_{i,1}\wedge dU_{i,2}\big)=0$ hence the map $U_i\circ F_{S\hspace{-0.03cm}S}$ has rank 1 (at the generic point of $\mathbf C^2$) and 
 defines foliation $F_{S\hspace{-0.03cm}S}^*\big( \mathcal F_{U_i}\big)$ which coincides with one of the web \eqref{Eq:Web-above}. 
 More explicitly, setting $\mathcal F_f$ for the foliation defined by the level subsets of $f$, one has 
 $$
 F_{S\hspace{-0.03cm}S}^*\big( \mathcal F_{U_i}\big)
 = \,  
 \mathcal F_{ 
\phi_i}
 $$
 for $i=1,\ldots,10$, where the $\phi_i$'s are the following rational functions
 \begin{align*}
\phi_1= &\,  \frac{(x-1)(bx-ay)}{(x-a)(x-y)}   && \phi_6= y 
\\
\phi_2=&\,  \frac{(y-1)(bx-ay)}{(y-b)(x-y)} 
&& \phi_7= x 
\\
 \phi_3=&\,  \frac{(y-1)(x-a)}{(x-1)(y-b)}
&& \phi_8=\frac{x}{y}
\\
  \phi_4=&\,  \frac{y(x-a)}{x(y-b)}
&& \phi_9=\frac{x-1}{y-1}
\\
 \phi_5=&\,  \frac{y(x-1)}{x(y-1)}
&& \phi_{10}=\frac{b(x-a)}{a(y-b)}\,.
 \end{align*}
(compare with the last pages of \cite[\S4.5]{PirioAFST}).

The preimage by $F_{S\hspace{-0.03cm}S}$ of the divisor $Z$ of $\boldsymbol{Y}_5$ defined in \eqref{Eq:Def-Z} is the divisor 
$\widetilde Z$ defined as the union of the ten irreducible divisors 
 $\widetilde Z_i=\{ \, \chi_i=0\, \}\subset \mathbf C^2$ for $i=1,\ldots,10$, where the $\chi_i$'s are given by : 
\begin{align*}
\chi_1=&\, x && \chi_2= y   \hspace{0.5cm} \chi_3=x-1 \hspace{0.5cm} \chi_4=y -1
\hspace{0.6cm}
 \chi_5= a -x
\hspace{0.6cm}
\chi_6= y-b  
\hspace{0.6cm}
  \chi_7= x-y
\\
  \chi_8=&\,  a y -b x
&& \chi_9= a y -b x -a +b +x -y 
\hspace{1.3cm}
 \chi_{10}=
 ab(x-y)+(b-a)xy+ay-bx
\, . 
\end{align*}
The 
logarithmic derivatives $\lambda_i=d {\rm Log}( \chi_i)$ for $i\in [\hspace{-0.05cm}[ 10  ]\hspace{-0.05cm}]$ are linearly independant over $\mathbf C$ then span 
a subspace, denoted by  ${\bf H}_{a,b}$, of the space of rational 1-forms on $\mathbf C^2$.
 All the pull-backs $\widetilde \eta_k=F_{S\hspace{-0.03cm}S}^*(\eta_k)$ of the rational 1-forms $\eta_k$'s under the map \eqref{Eq:F-SS} are elements of ${\bf H}_{a,b}$ and it is straightforward to compute them explicitly but we will not even need that. Just for the record, we state the 
 \begin{lem}
 The pull-back under $F_{S\hspace{-0.03cm}S}$ induces a linear isomorphism 
$F_{S\hspace{-0.03cm}S}^* \, : \, 
 {\bf H}_{\boldsymbol{Y}_5} 
 \stackrel{\sim}{\longrightarrow}
{\bf H}_{a,b} $ 
which is defined over $\mathbf Z$ with respect to the bases $(\eta_k)_{k=1}^{10}$ and $(\lambda_i)_{i=1}^{10}$ of  ${\bf H}_{\boldsymbol{Y}_5} $ and 
$
{\bf H}_{a,b} $ respectively. 
 \end{lem}

Setting 
$$u_{i,s}=F_{S\hspace{-0.03cm}S}^*\big( U_{i,s}
 \big)=U_{i,s}\circ F_{S\hspace{-0.03cm}S}$$
  for $i\in [\hspace{-0.05cm}[ 10  ]\hspace{-0.05cm}]$ and $s=1,2,3$, taking the pull-back of 
 \eqref{Eq: Sigma-i(UiOmega)} under $F_{S\hspace{-0.03cm}S}$, one gets that the following relation holds true  in $\wedge^3 {\bf H}_{a,b}$: 
\begin{equation}
\label{Eq: HLOG-Hlog}
 \sum_{i=1}^{10} \epsilon_i \,  \bigg( 
\frac{d u_{i,1}}{ u_{i,1} }
\wedge 
\frac{d u_{i,2}}{ u_{i,2} }
\wedge 
\frac{d u_{i,3}}{ u_{i,3} }
  \bigg)=0\, . 
 \end{equation}


Since everything is explicit,  \eqref{Eq: HLOG-Hlog} can be made explicit as well. 
For instance, for $i=6$, one has 
$$
U_6\circ F_{S\hspace{-0.03cm}S}=
\left(\frac{b +a}{\left(y-1 \right) a}, 
\frac{\left(a +1\right) y}{\left(y-1 \right) a}, 
\frac{b -y}{\left(y-1 \right) a}\right)
$$
hence 
\begin{equation}
\label{Eq:du-i-s}
\frac{d u_{6,1}}{ u_{6,1} }= {-\frac{dy}{y-1}}\, , \qquad  \frac{d u_{6,2}}{ u_{6,2} }=
\frac{dy}{y}-\frac{dy}{y-1}\qquad  \mbox{and}
\qquad  
\frac{d u_{6,3}}{ u_{6,3} }=
{\frac{dy}{y-b}-\frac{dy}{y-1}}
\end{equation}
from what it follows that 
$$
\frac{d u_{6,1}}{ u_{6,1} }
\wedge 
\frac{d u_{6,2}}{ u_{6,2} }
\wedge 
\frac{d u_{6,3}}{ u_{6,3} }
=
 \frac{dy}{y} \wedge 
\frac{dy}{(y-1)}
\wedge  \frac{dy}{(y-b)}
\, . 
$$

The computations of all the others wedge products 
$$\varpi_i= {d \ln(u_{i,1}})
\wedge 
{d \ln(u_{i,2}})
\wedge 
{d \ln(u_{i,3}})\in \wedge^3 {\bf H}_{a,b}$$ are as straightforward as the one above.
 Setting 
 $$
r_1=r_6=b
\, ,  \quad 
r_2=r_7=a
\, ,  \quad 
r_3=r_8=\frac{a}{b}
\, ,  \quad 
r_4=r_9=\frac{a-1}{b-1}
 \quad \mbox{ and } \quad 
r_5=r_{10}=\frac{b(a-1)}{a(b-1)}\, , 
$$
one obtains that for any $i\in [\hspace{-0.05cm}[ 10  ]\hspace{-0.05cm}]$, one has 
$$
\varpi_i
= 
  \left(\frac{d\phi_i}{\phi_i} \wedge 
\frac{d\phi_i}{\Big(\phi_i-1\Big)}
\wedge  \frac{d\phi_i}{\Big( \,\phi_i-
{r_i}\, 
\Big)}\right)\, . 
%
%
$$

In other terms, for any $i$, $\varpi_i$ is equal  to the symbol
of the weight 3 iterated integral $\boldsymbol{AH}^3_i\big( \phi_i\big)$ on $\mathbf C^2$, where 
$\boldsymbol{AH}^3_i$ stands for the complete antisymmetric weight 3 hyperlogarithm on $\mathbf P^1$ with respect to the 4-tuple of points $(0,1,r_i,\infty)$ (see \cite[\S2.1]{CP}).
Injecting this in \eqref{Eq: HLOG-Hlog} gives 
$$
\sum_{i=1}^{10}\epsilon_i\,   \left(\frac{d\phi_i}{\phi_i} \wedge 
\frac{d\phi_i}{(\phi_i-1)}
\wedge  \frac{d\phi_i}{\big( \,\phi_i-
{r_i}\, 
\big)}\right)=0\, .
$$
As explained in \cite{CP}, this is equivalent to the fact that the weight 3 hyperlogarithmic identity 
$$
\sum_{i=1}^{10}\epsilon_i\,   \boldsymbol{AH}^3_i\big( \phi_i\big) =0\, , 
$$
that is ${\bf HLog}\big({\rm dP}_4\big)$,  holds true.

\subsection{\bf From ${\bf HLOG}_{ \boldsymbol{Y}_5 }$ to 
${\bf HLog}\big({\bf dP}_4\big)$, by manipulating abelian relations}
\label{SS:manipulating-ARs}
While we consider that the above symbolic derivation of 
of ${\bf HLog}\big({\rm dP}_4\big)$ from 
${\bf HLOG}_{ \boldsymbol{Y}_5 }$ is pretty much clear, we have a feeling that it is somehow  not very concrete. 
  In this subsection, we describe succinctly how to 
obtain ${\bf HLog}\big({\rm dP}_4\big)$ from ${\bf HLOG}_{ \boldsymbol{Y}_5 }$ by only working with genuine abelian relations.  To simplify the notation, we write 
${\bf HLOG}$ instead of ${\bf HLOG}_{ \boldsymbol{Y}_5 }$ below. Non justified arguments are easy to check (and this is left to the reader...). 
\mk 

Let $F$ be one of the polynomials $\zeta_k$ defined in \eqref{Eq:zeta-i}, then 
${\bf Res}_F={\bf Res}_F\big( {\bf HLOG}\big) $ is a rational 2-AR of $\boldsymbol{\mathcal W}_{\boldsymbol{Y}_5}^{GM}$ carried by the 5-subweb $\boldsymbol{\mathcal W}_{\boldsymbol{Y}_5,F}^{GM}=\boldsymbol{\mathcal W}\big( U_j\, \big\lvert \, j \in J_F\big)$ for a certain subset $J_F\subset  [\hspace{-0.05cm}[  10 ]\hspace{-0.05cm}]$ 
of cardinality 5 (see \S\ref{SS:Residues-HLOG-Y5}). Moreover, its non-trivial components, denoted by 
${\bf Res}_{F,j} $ for 
$j\in J_F$, 
are linear combinations with coefficients $0,1$ or $-1$ of the wedge products 
$\big(dU_{j,a}\wedge dU_{j,b}\big)/(U_{j,a} dU_{j,b})$ for $a,b$ such that $1\leq a<b \leq 3$ (see Corollary 
\ref{Cor:2-RA-WGM}).  Replacing each such term by its primitive 
$\frac12 \big( \ln(U_{j,a}) d\ln(U_{j,b})-\ln(U_{j,b}) d\ln(U_{j,a})\big)$, 
one obtains a 1-AR denoted by $\int  {\bf Res}_F\in \boldsymbol{AR}^1\big( \boldsymbol{\mathcal W}_{\boldsymbol{Y}_5}^{GM} \big) $, which is such that $d \hspace{-0.05cm}\int  {\bf Res}_F= {\bf Res}_F$. The pull-back of it under $F_{S\hspace{-0.03cm}S}$ 
is a  1-abelian relation for $F_{S\hspace{-0.03cm}S}^*\big( \boldsymbol{\mathcal W}_{\boldsymbol{Y}_5,F}\big)=\boldsymbol{\mathcal W}\big( \phi_j\, \big\lvert \, j \in J_F\big)$ which is closed hence admits a primitive, that we will denote by 
$$
\int 
\hspace{-0.09cm} 
F_{\hspace{-0.05cm} 
\hspace{-0.04cm}
S\hspace{-0.03cm}S}^*\bigg[ 
\hspace{-0.0cm}
\scalebox{1.2}{$\int$}  \hspace{-0.03cm} {\bf Res}_F\bigg]\,.
$$

As an explicit example, let us consider the case when  $F=\zeta_{10}=P_5=1+y_{14}-y_{13}y_{24}$.  This case is completely similar\footnote{This not a coincidence. Indeed, since the Weyl group $W_{D_5}\simeq {\rm Aut}( \boldsymbol{\mathcal Y}_5)$ acts transitively on the set of weights of the half-spin representation, it acts similarly on the set of weight divisors $D_w\subset \boldsymbol{Y}_5$ hence on the set of residues of 
${\bf HLOG}_ {\boldsymbol{\mathcal Y}_5}$ along the latter (possibly up to rescaling but this is not relevant).} to the one of the web $\boldsymbol{W}_{\boldsymbol{Y}_5}^+$ whose ARs have been considered in 
Proposition \ref{Prop:AR-2-eta} and Proposition \ref{Prop:AR-Omega1} above. 
When $F=P_5$ : 
\begin{itemize}
\item the polynomial $F=P_5$ is a factor of  $U_{i,s}$ for $(i,s)=(j,3)$ with 
$j\in J_F=\{5,6,7,8,9\}$. \sk
\item  moreover for any $j\in J_F$, $F$ appears in the numerator of $U_{j,3}$, with multiplicity 1. It follows that ${\bf Res}_{F}$ is the rational 2-abelian relation 
$0= 
\sum_{j \in J_F} \epsilon_j\,   U_j^*\big( \eta \big) 
$  (where $\eta$ is the rational 2-form defined in \eqref{Eq:eta}); 
\sk 
\item 
its primitive $\int {\bf Res}_{F}$ is the 1-abelian relation 
$0= 
\sum_{j \in J_F} \epsilon_j\,   U_j^*\big( \delta \big) 
$ where $\delta$ is the privileged primitive of $\eta$  defined in 
\eqref{Eq:delta}; 
\sk 
\item then it is not difficult to compute the pull-backs 
of the components 
$\int  \hspace{-0.0cm} {\bf Res}_{F,j}$ under 
$F_{\hspace{-0.03cm} 
\hspace{-0.04cm}
S\hspace{-0.03cm}S}$. One obtains that, 
for any $j\in J_F$, 
up to the addition of 1-forms of the type $d \ln( \phi_j-z_j)$ with $z_j\in \{0,1,r_j\}$, 
one has 
\begin{equation}
\label{Eq:FSS-Int-ResFj}
F_{\hspace{-0.05cm} 
\hspace{-0.04cm}
S\hspace{-0.03cm}S}^*\Big( 
\hspace{-0.0cm}
\scalebox{1}{$\int$}  \hspace{-0.0cm} {\bf Res}_{F,j}\Big)=
\frac{1}{2}\left(  \frac{\ln(\phi_j)}{\phi_j-1}- 
 \frac{\ln(\phi_j-1)}{\phi_j}
\right)d\phi_j = \phi_j^*\big( d\tilde R\big)\, .
\end{equation}
where $\tilde R: x\mapsto \frac12 \int_1^x \big( {\ln  \left(u \right)}/{(u -1)}-{\ln  \left(u -1\right)}/{u}\big)du$ is a version of Rogers' dilogarithm defined on $[1,+\infty[$.  After integrating the identity $0=\sum_{j\in J_F} \epsilon_j F_{\hspace{-0.05cm} 
\hspace{-0.04cm}
S\hspace{-0.03cm}S}^*\big( 
\hspace{-0.0cm}
\scalebox{1}{$\int$}  \hspace{-0.0cm} {\bf Res}_{F,j}\big)$, one obtains the following explicit form for the identity corresponding to the 0-AR $\int F_{\hspace{-0.05cm} 
\hspace{-0.04cm}
S\hspace{-0.03cm}S}^*\big( 
\hspace{-0.0cm}
\scalebox{1}{$\int$}  \hspace{-0.0cm} {\bf Res}_{F}\big)$: 
\begin{equation}
\label{Eq:Rtilde}
\tilde R \left( \frac{y(x-1)}{x(y-1)}\right)
+ \tilde R \left( y\right)
-\tilde R \left(x\right)
+\tilde R \left( \frac{x}{y}\right)
-\tilde R \left( \frac{x-1}{y-1}\right)=0\, ,
\end{equation}
a relation  satisfied by all $x,y\in [1,+\infty[$ such that $1< y <x$. 
\end{itemize} 
\bk 

Given $\xi\in \mathbf P^1\setminus \{0,1,\infty\}$, one sets $\Sigma_{\xi}=\{0,1,\xi,\infty\}$ and for $i\in  [\hspace{-0.05cm}[ 10  ]\hspace{-0.05cm}]$, 
one denotes by ${\bf H}_{i}$ the pull-back under $\phi_i$ of the space of holomorphic 1-form on $\mathbf P^1\setminus\Sigma_{r_i}$ with logarithmic poles at the points of $\Sigma_{r_i}$: 
$$
{\bf H}_{i}=\phi_i^* 
{\bf H}^0\Big( \mathbf P^1, \Omega_{ \mathbf P^1}\big( {\rm Log}\,\Sigma_{r_i}
 \big) \Big)\subset {\bf H}_{a,b}\, . 
 $$
This space admits ${d\phi_i}/{\phi_i}$, ${d\phi_i}/({\phi_i}-1)$ and 
${d\phi_i}/({\phi_i}-r_i)$ as a basis and for any  $\sigma, \sigma' \in  \{0,1,r_i\} $, one defines a  weight 2 antisymmetric symbol by setting
$$
{\mathfrak R}^i_{\sigma, \sigma'}= \bigg( \frac{d\phi_i}{\phi_i-\sigma}\bigg) \wedge 
\bigg( \frac{d\phi_i}{\phi_i- \sigma'}\bigg)
=\phi_i^*\bigg( \Big( 
 \frac{dz}{z-\sigma}\Big) \wedge 
\Big( \frac{dz}{z- \sigma'}\Big)
\bigg) 
\in \wedge^2 {\bf H}_{i} \subset \wedge^2 {\bf H}_{a,b}
\, .
$$ 
Clearly, each ${\mathfrak R}^i_{\sigma,\tilde \sigma}$ is the symbol of a dilogarithmic function composed with $\phi_i$. For instance, it follows from \eqref{Eq:FSS-Int-ResFj}
 that ${\mathfrak R}^6_{0,1}$ is the symbol of the term $\tilde R(y)$ appearing in the analytic expression \eqref{Eq:Rtilde} of 
$\int F_{\hspace{-0.05cm} 
\hspace{-0.08cm}
S\hspace{-0.03cm}S}^*\big( 
\hspace{-0.0cm}
\scalebox{1}{$\int$}  \hspace{-0.0cm} {\bf Res}_{F}\big)$. \mk

For any $i$, the determination of the weight 2 symbol of 
the $i$-th components of the abelian relations $\int F_{\hspace{-0.05cm} 
\hspace{-0.08cm}
S\hspace{-0.03cm}S}^*\big( 
\hspace{-0.0cm}
\scalebox{1}{$\int$}  \hspace{-0.0cm} {\bf Res}_{\zeta_k}\big)$  of $\boldsymbol{W}_{a,b}=\boldsymbol{\mathcal W}\big(\phi_1,\ldots,\phi_{10}\big)$ 
for any 
$\zeta_k$
can be obtained by means of straightforward formal computations from the decompositions of the logarithmic derivatives of the components of the first integrals $U_{i}$ of $\boldsymbol{\mathcal W}_{\boldsymbol{Y}_5}^{GM}$. 
As an example, below we give details in the case $i=6$. 
\sk

Setting $\Delta_i= \big( d\ln U_{i,s} \big)_{s=1}^3$ for $i=1,\ldots,10$, one has the following expressions  of the $ d\ln U_{i,s}$'s as linear combinations of the $\eta_k$'s: 
\begin{align*}
\Delta_1= \, &\big( \, \eta_{4}, \eta_{5}+\eta_{3}, \eta_{6}\, \big) 
\\
\Delta_2= \, &\big( \, -\eta_{1}, \eta_{5}-\eta_{1}+\eta_{2}, \eta_{7}-\eta_{1}\, \big) 
\\
\Delta_3= \,&\big( \, \eta_{3}, \eta_{4}+\eta_{2}, \eta_{8}\, \big) 
\\
\Delta_4= \,&\big( \, -\eta_{5}, -\eta_{5}+\eta_{4}+\eta_{1}, -\eta_{5}+\eta_{9}\, \big) 
\\
\Delta_5= \,&
\big( \, \eta_{2}, \eta_{3}+\eta_{1}, \eta_{10}\, \big) 
\\
\Delta_6= \,&
\big( \, \eta_{7}-\eta_{2}-\eta_{9}, \eta_{8}+\eta_{1}-\eta_{2}-\eta_{9}, \eta_{10}-\eta_{2}-\eta_{9}\, \big) 
\\
\Delta_7= \,&
\big( \, \eta_{6}-\eta_{3}-\eta_{9}, -\eta_{3}+\eta_{8}-\eta_{9}, -\eta_{3}+\eta_{10}+\eta_{4}-\eta_{9}\, \big) 
\\
\Delta_8= \,&
\big( \, \eta_{6}+\eta_{1}-\eta_{9}, \eta_{7}-\eta_{9}, \eta_{5}+\eta_{10}-\eta_{9}\, \big) 
\\
\Delta_9= \,&
\big( \, \eta_{6}-\eta_{8}+\eta_{2}, \eta_{7}+\eta_{3}-\eta_{8}, \eta_{10}-\eta_{8}\, \big) 
\\
\Delta_{10}= \, &
\big( \, \eta_{6}-\eta_{5}-\eta_{8}, \eta_{7}-\eta_{5}+\eta_{4}-\eta_{8}, -\eta_{5}-\eta_{8}+\eta_{9}\, \big) 
%
%
\end{align*}

The indices $k$ such that the $6$-th component of the residue 
$${\bf Res}_k
={\bf Res}_{\zeta_k}\big({\bf HLOG}_{\boldsymbol{Y}_5} \big) 
\in \boldsymbol{AR}^2\Big( 
\boldsymbol{\mathcal W}_{\boldsymbol{Y}_5}^{GM}
\Big)$$  is non-trivial are exactly the ones such that $\eta_k$ appears in the expression of $\Delta_6$ above. Namely, the set of such indices is $\mathfrak K_6=\{1,2,7,8,9,10\}$. For each $k\in \mathfrak K_6$, let $\kappa_k\in \{-1,0,1\}^3$ be the triplet obtained from $\Delta_6$ be deriving it formally with respect to $\eta_k$ (the latter hence being 
then considered as a symbolic indeterminate). One has 
\begin{align*}
\kappa_1= & \, (0, 1, 0)
&& \kappa_2=(-1, -1, -1) 
&& \kappa_7=(1, 0, 0) 
\\
\kappa_8= & \, (0, 1, 0)
&& \kappa_9=(-1, -1, -1) 
&& \kappa_{10}=(0, 0, 1)\, . 
\end{align*}
From the formula \eqref{Eq:Omega-0} for $\Omega$, it follows that 
for any $\kappa_k=(\alpha_k,\beta_k,\gamma_k)$, the $6$-th component 
${\bf Res}_{k,6}=$
of  the 
rational 2-abelian relation 
${\bf Res}_k$ is 
$${\bf Res}_{k,6}= \alpha_k \left( \frac{dU_{6,2}}{U_{6,2}}\wedge 
\frac{dU_{6,3}}{U_{6,3}}
\right) 
-\beta_k  \left(
\frac{dU_{6,1}}{U_{6,1}}\wedge 
\frac{dU_{6,3}}{U_{6,3}}\right) 
+\gamma_k \left( \frac{dU_{6,1}}{U_{6,1}}\wedge 
\frac{dU_{6,2}}{U_{6,2}}\right)  \, .$$
Each ${\bf Res}_{k,6}$ is a rational 2-form, but the expression above, up to a scaling factor (equal to 1/2) can also be seen as an element of $\wedge^2 {\bf H}_{\boldsymbol{Y}_5}$, which will be denoted by $\frac12 {\bf Res}_{k,6}^{symb}$. 
Using \eqref{Eq:du-i-s}, for any $k\in {\mathfrak K}_6$ one can compute the pull-back  $ 
\varrho_k^6=
F_{\hspace{-0.05cm} 
\hspace{-0.08cm}
S\hspace{-0.03cm}S}^*\big( \frac12  {\bf Res}_{k,6}^{symb}\big)$ which belongs to $\wedge^2 {\bf H}_6$. 
One has 
\begin{align*}
\varrho_k^6 = &\, \frac{\alpha_k}{2} \left(
\bigg(  \frac{dy}{y}- \frac{dy}{y-1}
\bigg)
\wedge 
\bigg(  \frac{dy}{y-b}- \frac{dy}{y-1}
\bigg)
\right) 
 -
 \frac{\beta_k}{2}  \left(
\bigg(  -\frac{dy}{y-1} \bigg)
\wedge 
\bigg(  \frac{dy}{y-b}- \frac{dy}{y-1}
\bigg)
\right) \\
&\, \hspace{1.5cm}+
 \frac{\gamma_k}{2}
 \left( 
\bigg(  -\frac{dy}{y-1} \bigg)
\wedge 
\bigg(  \frac{dy}{y}- \frac{dy}{y-1}
\bigg)
\right)  \, .
\end{align*}

Expressed in terms of the 
dilogarithmic symbols $\mathfrak R^6_{\sigma,\sigma'}$ defined above and considering that $r_6=b$,  one obtains 
\begin{align*}
\varrho_k^6 = &\, \frac{\alpha_k}{2} \left(
{\mathfrak R}^6_{0,b}
-{\mathfrak R}^6_{0,1}-{\mathfrak R}^6_{1,b}
\right) 
 +
 \frac{\beta_k}{2}  \left(
{\mathfrak R}^6_{1,b}
\right)+
 \frac{\gamma_k}{2}
 \left( 
{\mathfrak R}^6_{0,1}
\right)  \, .
\end{align*}
hence 
\begin{align}
\label{Eq:varrho}
\varrho^6_1= & \,{\mathfrak R}^6_{1,b}
&& \varrho^6_2=- {\mathfrak R}^6_{0,b}
&& \varrho^6_7=-{\mathfrak R}^6_{0,1}-{\mathfrak R}^6_{1,b}+{\mathfrak R}^6_{0,b} 
\\
\varrho^6_8= & \, {\mathfrak R}^6_{1,b}
&& \varrho^6_9=-{\mathfrak R}^6_{0,b}
&& \varrho^6_{10}={\mathfrak R}^6_{0,1}\, .
\nonumber
%
%
\end{align}
One can verify that for any $k\in {\mathfrak K}_6$,  $\varrho^6_k$ is also the symbol 
of the dilogarithmic function which is the 6-th component of the 0-AR 
$\int F_{\hspace{-0.05cm} 
\hspace{-0.08cm}
S\hspace{-0.03cm}S}^*\big( 
\hspace{-0.0cm}
\scalebox{1}{$\int$}  \hspace{-0.0cm} {\bf Res}_{\zeta_k}\big)$. 
Then using  \eqref{Eq:varrho}, straightforward computations give us that the 
symbol of the sixth component  of the sum 
\begin{equation}
\label{Eq:quasi-la-fin}
\sum_{k=1}^{10} F_{\hspace{-0.05cm} 
\hspace{-0.08cm}
S\hspace{-0.03cm}S}^*
\big(
\eta_k 
\big) 
\int F_{\hspace{-0.05cm} 
\hspace{-0.08cm}
S\hspace{-0.03cm}S}^*\big( 
\hspace{-0.0cm}
\scalebox{1}{$\int$}  \hspace{-0.0cm} {\bf Res}_{k}\big)
\end{equation}
is equal to 
$$
\Big(\frac{dy}{y-b}\Big)\otimes 
{\mathfrak R}_{0,1}
-
\Big(\frac{dy}{y-1}\Big)\otimes 
{\mathfrak R}_{0,b}
+
\Big(\frac{dy}{y}\Big)\otimes 
{\mathfrak R}_{1,b}\,.
$$
This is the symbol of the complete antisymmetric weight-3 hyperlogarithm ${AH}_3(\phi_6)$ with respect to the set $\Sigma_{r_6}=\Sigma_b = \{0,1,b,\infty\} \subset {\mathbf P}^1$ (cf. \cite[\S2.1.3]{CP}). Through similar and straightforward computations, one can verify that the same result holds for each component of equation \eqref{Eq:quasi-la-fin}. It follows that the 1-abelian relation
\eqref{Eq:quasi-la-fin} 
 of $\boldsymbol{W}_{a,b}$ corresponds to the derivative of the hypertrilogarithmic  functional abelian relation ${\bf Hlog}({\rm dP}_4)$. This leads to the justification of the formula:
\begin{equation}
\label{Eq:la-fin}
{{\bf Hlog}({\rm dP}_4)}=\sum_{k=1}^{10} \bigintsss 
\left[
F_{\hspace{-0.05cm} 
S\hspace{-0.03cm}S}^*\big(
\eta_k 
\big) 
\hspace{-0.08cm}
\bigintssss    F_{\hspace{-0.05cm} 
\hspace{-0.08cm}
S\hspace{-0.03cm}S}^*\left( 
\hspace{-0.0cm}
\int
  \hspace{-0.08cm} {\bf Res}_{k}\Big(
{\bf HLOG}_{\boldsymbol{Y}_5}
 \Big)\right)\right]
\end{equation}
which summarizes the preceding discussion on how the weight-3 hyperlogarithmic functional identity ${\bf Hlog}({\rm dP}_4)$ can be derived 
from the differential identity between differential 2-form with logarithmic coefficients ${\bf HLOG}_{\boldsymbol{Y}_5}$.

\section{\bf The 
$\boldsymbol{(r-3)}$-rank  and the
$\boldsymbol{(r-3)}$-abelian relations  
of $\boldsymbol{\boldsymbol{\mathcal W}^{GM}_{\hspace{-0.1cm} {\mathcal Y}_r}}$ for $\boldsymbol{r=4,\ldots,7}$}
\label{SS:r=6,7}
In this section, we extend some of the results obtained above to the cases $r = 6$ and $r = 7$. The current proofs are based on explicit computations carried out using a computer algebra system. These computational details are omitted here, but we intend to provide conceptual proofs in a future work.

 \subsection{}  
 \label{S:7-1}
 In addition to the notations introduced in the Introduction, we use also the following ones: 
\begin{itemize}
\item we denote by $\boldsymbol{D}_r^\bullet$ the 1-marked Dynkin diagram associated to the simple Lie group $G_r$, with one marked vertex corresponding to the maximal parabolic subgroup $P_r\subset G_r$ such that 
$\boldsymbol{\mathcal G}_r=G_r/P_r$. We denote by $\gamma_r$ the dimension of $\boldsymbol{\mathcal G}_r$. Then $\gamma_r-r$ is equal to the dimension of the quotient 
space $\boldsymbol{\mathcal Y}_r$ , denoted by $y_r$. One has 
$$
\gamma_r=\dim \boldsymbol{\mathcal G}_r\qquad \mbox{ and } 
\qquad y_r=\dim \boldsymbol{\mathcal Y}_r=\gamma_r-r\, ;
$$
\sk 
\item  $V_r$ stands for the corresponding minuscule representation (of dimension $\upsilon_r$ equal to 10, 16, 27 and 56 for $r=4,5,6,7$ respectively);
\sk 
\item $\Delta_r$ stands for the moment polytope of $\boldsymbol{\mathcal G}_r$, which coincides with the weight polytope in $\mathfrak h_r^\vee (\mathbf R)$ of the minuscule representation $V_r$.  The series of $\Delta_r$'s coincides with the one of Gosset's polytopes $(r-4)_{21}$. Each admits two kinds of facets, some $(r-1)$-simplices and some $(r-1)$-cross-polytopes (also known as `orthoplexes'). Only the latter are $\boldsymbol{\mathcal W}$-relevant. 
The first two of the series are also known by other names: 
$\Delta_4$ is the hypersymplex $\Delta_{2,5}$ and $\Delta_5$ is the 5-demihypercube. 
\sk 
\item  $\boldsymbol{\mathcal L}_r$ (resp.\,$\boldsymbol{\mathcal K}_r$) denotes the number of lines (resp.\,of conic fibrations) on a fixed smooth del Pezzo surface ${\rm dP}_{9-r}$; 
\sk 
\item we recall that  there are natural bijections between $\boldsymbol{\mathcal L}_r$ and the set of weights $\mathfrak W_r$ of $V_r$ which is also the set of 
vertices $V_0(\Delta_r)$ of the weight polytope. Hence one has 
$$\ell_r=\big\lvert \boldsymbol{\mathcal L}_r\big\lvert= \big\lvert V_0(\Delta_r)\big\lvert \,  =\dim V_r=\upsilon_r\, .$$ Similarly, the set $\boldsymbol{\mathcal K}_r$ of conic classes on ${\rm dP}_{9-r}$ identifies with the one of cross-polytope facets of $\Delta_r$. 
\end{itemize}

\begin{table}[!h]
\begin{tabular}{|c|c|c|c|c|c|c|c|c|}
\hline
 \begin{tabular}{c}  \\  \vspace{0.4cm} $\boldsymbol{r}$ 
\\  \end{tabular}
 & $\boldsymbol{D}_r^\bullet$  & 
 ${}^{}$
 \hspace{-0.15cm} 
 \begin{tabular}{l} 
 $\boldsymbol{\upsilon_r}=\dim V_r$\\
 ${}^{}$ \hspace{0.2cm} $= \big\lvert \, 
\boldsymbol{\mathcal L}_r\,  \big\lvert =\ell_r$
 \end{tabular}
  & $\boldsymbol{\mathcal G}_r
\subset \mathbf PV_r  $ & $\boldsymbol{\gamma_r}=\dim \boldsymbol{\mathcal G}_r$ & $\boldsymbol{y_r}= \dim  \boldsymbol{\mathcal Y}_r$ &  $\kappa_r=\big\lvert \boldsymbol{\mathcal K}_r\big\lvert$ \\
 \hline  \hline
 \begin{tabular}{c}  \\  \vspace{0.4cm} $4$ 
\\  \end{tabular} &  
 \begin{tabular}{c}\vspace{-0.2cm}\\
\scalebox{0.16}{\includegraphics{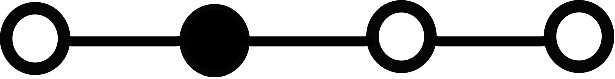}}
\end{tabular}
 & 10 & $G_2\big({\mathbf C}^5\big)\subset \mathbf P^{9}$  &   5    &   2   & 5 
 \\
  \hline  
 \begin{tabular}{c}  \\  \vspace{0.4cm} $5$ 
\\  \end{tabular} & 
 \begin{tabular}{c}\vspace{-0.2cm}\\
\scalebox{0.16}{\includegraphics{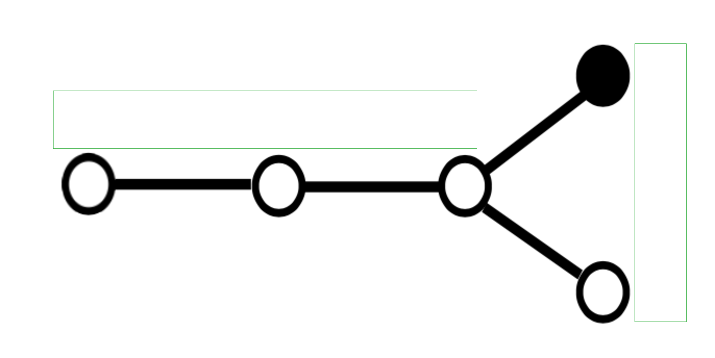}}
\end{tabular}
  & 16 & $\mathbb S_5\subset \mathbf P^{15}$  &   10    &   5  & 10 
 \\
   \hline  
 \begin{tabular}{c}  \\  \vspace{0.4cm} $6$ 
\\  \end{tabular} &  
 \begin{tabular}{c}\vspace{-0.2cm}\\
\scalebox{0.16}{\includegraphics{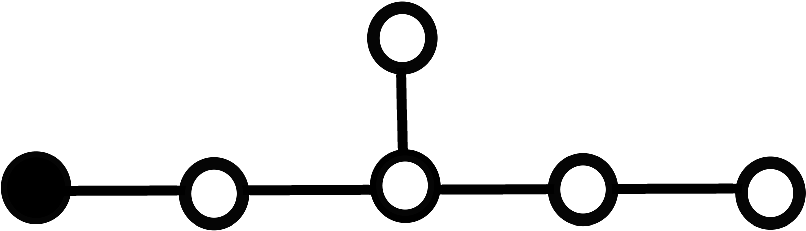}}
\end{tabular}
 & 27 & $\mathbb O_{\mathbf C}\mathbf P^2 \subset \mathbf P^{26}$  &   16    &   10   & 27
 \\
   \hline  
 \begin{tabular}{c}  \\  \vspace{0.4cm} $7$ 
\\  \end{tabular} &
 \begin{tabular}{c}\vspace{-0.2cm}\\
\scalebox{0.16}{\includegraphics{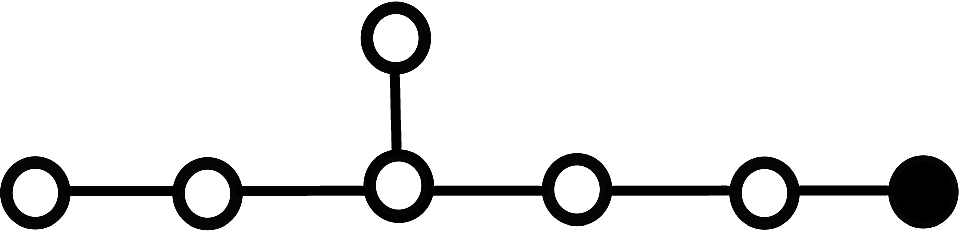}}
\end{tabular}
   & 56 & ${\bf Fr}\subset \mathbf P^{55}$  &   27     &  20   & 126 
 \\
 \hline 
 \end{tabular}
\bk  
 \label{Table:3}
\end{table}

Many features and 
numerical quantities of the objects considered above are given in the table 
above 
in which $\mathbf O\mathbf P^2 \subset \mathbf P^{26}$ stands for the Cayley plane over the complex octonions $\mathbf O$ and where ${\bf Fr}\subset \mathbf P^{55}$ denotes the 27-dimensional Freudenthal variety $E_7/P_6$. Both varieties are considered in their fundamental embeddings which can also respectively be described as the projectivizations 
of the exceptional rank 3 complex Jordan algebra $\mathbf J={\rm Herm}_3\big( \mathbf O\big)$  and of the space $Z_2(\mathbf J)$ of Zorn $2\times 2$ matrices  over $\mathbf J$. 
\mk 

The Gelfand-MacPherson web we are interested in lives on the quotient space 
$\boldsymbol{\mathcal Y}_{r}= \boldsymbol{\mathcal G}_r /\hspace{-0.06cm}/ H_r$. Given an orthoplex facet $F$ of $\Delta_r$, , there is a commutative diagram:
$$
\scalebox{1}{
  \xymatrix@R=0.2cm@C=1.5cm{ 
   \boldsymbol{\mathcal G}_r
\ar@{->}[dd]
   \ar@{-->}[r]^{ \hspace{0.1cm}\Pi_F} &  {\mathcal Q}_F
\ar@{->}[dd]
\\ & \\
\boldsymbol{\mathcal Y}_r  \ar@{-->}[r]^{ \hspace{0.1cm}\pi_F} 
& \boldsymbol{\mathcal Y}_F \, 
  }}
$$
in which the vertical maps are the quotient maps, and the horizontal maps are the corresponding \emph{face maps}.
It can be verified that ${\mathcal Q}_F$ is a 
 smooth hyperquadric in $\mathbf P^{2r-3}$ and that $\boldsymbol{\mathcal Y}_F$ embeds into $\mathbf P^{r-3}$ as a Zariski open subset. 
  With respect to this embedding, one has
 $\boldsymbol{\mathcal Y}_F^\times =\mathbf P^{r-3}\setminus \mathcal H_F$ where $\mathcal H_F$ is an arrangement of $r-1$ hyperplanes in general position.

We have that the corresponding Gelfand-MacPherson web considered here 
$$
\boldsymbol{\mathcal W}^{GM}_{\hspace{-0.1cm} \boldsymbol{\mathcal Y}_r}=
\boldsymbol{\mathcal W}\Big( \, \pi_F\, \big\lvert \, F\in \boldsymbol{\mathcal K}_r \Big)
$$
is a $\kappa_r$-web of codimension $r-3$ on $\boldsymbol{\mathcal Y}_{r}$which is a rational variety of dimension $y_r$ equal to 2,5,10 and 20 for $r=4,5,6$ and $7$ respectively.  Since the hyperplanes of $\mathcal H_r$ are in general position, for some homogeneous coordinates $u_0,u_1,\ldots,u_{r-3}$ on $\mathbf P^{r-3}$, one can assume that $\mathcal H_r$ is cut out by 
$$u_0u_1u_2\ldots u_{r-3}\big(u_0+u_1+\cdots+u_{r-3}\big)=0\, .$$ 

\subsubsection{\bf Explicit birational models of $\boldsymbol{{\mathcal W}^{GM}_{\hspace{-0.1cm} {\mathcal Y}_r}}$ for $\boldsymbol{r=6,7}$}
\label{SS:Explicit-birational-models}
The approach used above to study $\boldsymbol{\mathcal W}^{GM}_{\hspace{-0.1cm} \boldsymbol{\mathcal Y}_5}$ and its abelian relations relies on an explicit birational model of this web, obtained via a concrete birational parametrization of the spinor variety $\mathcal{S}_5$ embedded in the projectivization of the half-spin representation.  A key observation is that, for $r = 6,7$ as well, one can construct similar (birational) parametrizations of each corresponding homogeneous space $\boldsymbol{\mathcal G}_r$, considered in their respective minuscule embeddings. This allows the same strategy, based on explicit formal computations, assisted by a computer algebra system, to be effectively extended to these higher-rank cases. In this subsection, we succinctly describe the parametrizations  of $\boldsymbol{\mathcal G}_6=\mathbf O\mathbf P^2$ and  $\boldsymbol{\mathcal G}_7={\bf Fr}$ that we have worked with.  Our main reference for the material presented here is Yokota's remarkable book \cite{Yokota}. \sk

In order to obtain, in each case, explicit and well-behaved formulas (defined over ${\mathbf R}$, say) for the action of the Cartan torus, we considered a real rational parametrization of the split real form $\boldsymbol{\mathcal G}_r^{\mathbf R}$ of the complex homogeneous space $\boldsymbol{\mathcal G}_r$, embedded into the projectivization of the split real form $V^\mathbb{R}_r$ of the corresponding (complex) minuscule representation $V_r$.
To this end, we worked with the algebra ${\mathbb O}_s$ of split (real) octonions rather than with the usual (Hurwitz) octonion algebra $\mathbb{O}$. This choice is ultimately inconsequential, since tensoring $\mathbb{O}_s$ and $\mathbb{O}$ with $\mathbf{C}$ yields two isomorphic complex algebras.
In what follows, we will work with the complex algebra 
$\mathbf O=\mathbb O_s$ that we will call 
`the algebra of complex octonions'. \sk

Let us consider the following two invertible linear maps from $\mathbf R^8$ into itself. 
\begin{align*}
\phi : (x_i)_{i=0}^7 &\longmapsto \big(\, x_{0}-x_{7}, x_{1}-x_{6}, x_{2}-x_{5}, x_{3}-x_{4}, x_{3}+x_{4}, 
x_{2}+x_{5}, x_{1}+x_{6}, x_{0}+x_{7}\,\big) \\
\psi : 
(u_i)_{i=0}^7 & \longmapsto  \frac12
\big(\, {u_{0}}  + {u_{7}}  \, , \,  {u_{1}}  + {u_{6}}  
\, , \,   {u_2  }  + {u_{5}}  \, , \,  {u_{3}}  + {u_{4}}  \, , \, 
   {u_{4}}- {u_{3}}  \, , \,    {u_{5}}- {u_2  }  \, , \,
   {u_{6}}- {u_{1}}  \, , \,  {u_{7}}  - {u_{0}}  
 \, \big) 
\end{align*}
These maps are inverses of each other and denoting by 
$*$ the usual octonionic product on $\mathbb O=\mathbf R^8$, one defines $\mathbb O_s$ as $\mathbf R^8$ endowed with the product given by
\begin{equation}
\label{Eq:Special-Octonionic-Product}
u\cdot v=\phi\big(  \psi(u)*\psi(v) \big)
\end{equation}
 for any $u,v\in \mathbb O_s$. 
 Then the unit element ${\bf 1}$, the conjugate $\overline{u}$ and the norm $\lVert \, u\, \lVert^2$ of any element $u=(u_i)_{i=0}^7$ are given respectively by 
 \begin{align*}
{\bf 1}=&\, (1,0,0,0,0,0,0,1)\\
 \overline{u}=&\,
\big( \, u_{7} \, , \, -u_{1} \, , \, -u_{2} \, , \, -u_{3}\, , \, -u_{4} \, , \, -u_{5} \, , \, -u_{6} \, , \, u_{0}\, 
\big) 
\\ \mbox{ and } \quad 
\lVert\,u\,\lVert^2 =&\, 
u_{0} u_{7}+u_{1} u_{6}+u_{2} u_{5}+u_{3} u_{4}\, .
\end{align*}
(We thus have $u\cdot \overline{u} = \overline{u} \cdot u=
\lVert\,u\,\lVert^2 \,{\bf 1}$ for every $u$).

For $O$ one of the three algebras $\mathbb O$, $\mathbb O_s$ or $\mathbf O$, with definition field $\bf k=\mathbf R$ or $\mathbf C$, we define the space of hermitian $3\times 3$ matrices with coefficients in $O\simeq {\bf k}^{8}$ by setting 
$$
 {\rm Herm}_3(O)=\left\{\,
\, \begin{bmatrix}
s_1 &    v_3 & \overline{v_2} \\
\overline{v_3 } & s_2 & v_1 \\
v_2 & \overline{v_1} & s_3 
\end{bmatrix}\, \bigg\lvert \, \begin{tabular}{c}
$s_1,s_2,s_3\in {\bf k}$\\
$v_1,v_2,v_3\in O$
\end{tabular}\, 
\right\}\simeq  {\bf k}^3\oplus O^3\simeq 
{\bf k}^{27}\, , 
$$
where the symbols $\simeq$ stand for isomorphisms of {\bf k}-vector spaces.\footnote{Endowed with the symmetrization of the naive matricial product, ${\rm Herm}_3(O)$ becomes a rank 3 Jordan algebra over ${\bf k}$, which is simple if $O$ is non-degenerate (that is when $O\in \{ \mathbb O,\mathbf O\, \}$).  We will not really use these facts in this paper, 
they are mentioned only in passing.} In this definition, the diagonal coefficients $s_i\in{\bf k}$ have to be understood as scalar elements of $O$, that is modulo the embedding ${\bf k}\hookrightarrow O$, $s\longmapsto s\,{\bf 1}$, where ${\bf 1}$ is the unit element.  In what follows, we set $\mathbb J$, $\mathbb J_s$ and $\mathbf J$ for $ {\rm Herm}_3(O)$ when $O=\mathbb O$, $\mathbb O_s$ and $\mathbf O$ respectively. When $O$ is unspecified, we will denote $J$ for  ${\rm Herm}_3(O)$ to shorten the notations.  The algebra $J$ being power-associative and of rank 3, there exist homogeneous forms $T, Q,N\in {\rm Sym}\big(J^\vee\big)$, of degree 1,2 and 3 respectively, such that $X^3-T(X)X^2+Q(X)X-N(X){\bf 1}=0$ for any $X\in J$. The linear map $T$ is the natural trace and when  
$J= {\rm Herm}_3(O)$,  the cubic form $N$ is known as the {\it `determinant'} . As such, it is denoted by ${\rm det} : J\rightarrow {\bf k}$ and is explicitly given by 
$$
{\rm det}\left( \,
\begin{bmatrix}
s_1 &    v_3 & \overline{v_2} \\
\overline{v_3 } & s_2 & v_1 \\
v_2 & \overline{v_1} & s_3 
\end{bmatrix}
\, \right) = 
{T}\Big(\big[v_1,v_2,v_3 \big]\Big)- \sum_{i=1}^3 s_i \, \lVert\, v_i \, \lVert^2
+s_1s_2s_3
\, . 
$$
where 
$[\cdot,\cdot,\cdot] : J^3\rightarrow {\bf k}$ denotes the associated  {\it `triple product'} given by 
$[u,v,w]=(u v) w-(u w) v+(v w) u$ for $u,v,w\in J$.  Then one can define {\it `the adjoint $X^\#$'} of an element of $J$  by setting 
$$X^\#=X^2-T(X)\,X+Q(X)\,{\bf 1}$$
for any $X\in J$.  One has 
$X X^\#=X^\# X={\rm det}(X) \, {\bf 1}_J$ for any $X$.


To deal with the case $r=7$ and the $27$-dimensional Freudenthal variety ${\bf Fr}=E_6/P_6=\boldsymbol{\mathcal G}_7\subset \mathbf PV_7\simeq \mathbf P^{55}$, we need to go one step further and consider the {\it space 
 of Zorn $2\times 2$ matrices with coefficients in $J$}, defined by  
$$
\boldsymbol{\mathcal Z}_2(J)=\left\{
\, \, \begin{bmatrix}
{}^{}\, \zeta_1 &   Z_1 {}^{}\, \\
{}^{}\, Z_2 & \zeta_2 {}^{}\, 
\end{bmatrix}\,\, \bigg\lvert \, \begin{tabular}{c}
$\zeta_1,\zeta_2 \in {\bf k}$\\
$Z_1,Z_2\in J$
\end{tabular}\, 
\right\}\simeq  {\bf k}\oplus J\oplus J\oplus {\bf k}\simeq 
{\bf k}^{56}\, .
$$

Then we consider the following Veronese maps 
\begin{align*}
\nu_O^2\, : \, \, O\oplus O & \longrightarrow \,\,J  && 
\nu_J^3 \, : \, \,  J   \longrightarrow  \boldsymbol{\mathcal Z}_2(J)
\\
(u,v)& \longmapsto 
\scalebox{0.9}{$\begin{bmatrix}
{}^{}\, 1 &    u & \overline{v} \\
{}^{}\, \overline{u} & \lVert\,u \,\lVert^2 & 
\overline{v\cdot u}\, {}^{} 
 \\
v & v\cdot u &  \lVert\,v \,\lVert^2
\end{bmatrix}$}
&& {}^{} \hspace{0.8cm} X \longmapsto 
\scalebox{0.9}{$\begin{bmatrix}
1 &    X \\
X^\# & {\rm det}(X)
\end{bmatrix}$}\, .
\end{align*} 
These two Veronese maps are affine embeddings. 
The image of $\nu_O^2$ (resp.\,of $\nu_J^3$) lands into the algebraic subvariety of $J$ (resp.\,of $\boldsymbol{\mathcal Z}_2(J)$) cut out by the quadratic equation  $X^2=T(X)X$ (resp.\,the set of quadratic equations $Z_1Z_2=\zeta_1\zeta_2{\bf 1}$, 
$Z_1^\#-\zeta_1Z_2=Z_2^\#-\zeta_2 Z_1=0$).

\begin{prop}
Over $\mathbf C$, that is for $O=\mathbf O$ and $J=\mathbf J$, the Zariski closures of the images of the two Veronese maps above are the minuscule homogeneous spaces 
$\boldsymbol{\mathcal G}_6$ and $\boldsymbol{\mathcal G}_7$ in their minuscule embeddings respectively. In other terms, one has : 
$$ 
\overline{\nu_{\mathbf O}^2\big(\mathbf O\oplus \mathbf O\big)}=\mathbf  O\mathbf P^2\subset \mathbf P(\mathbf J)\simeq \mathbf P^{26}
\qquad 
\mbox{ and } 
\qquad 
\overline{\nu_{\mathbf J}^3\big(\mathbf J\big)}={\bf Fr} 
\subset 
\mathbf P\big(\boldsymbol{\mathcal Z}_2(\mathbf J)\big)\simeq
\mathbf P^{55}\, . 
$$
\end{prop}

The reason for our choice of the specific form\eqref{Eq:Special-Octonionic-Product} of the octonionic product lies in the fact that it yields particularly nice and fully explicit formulas-- moreover defined over \( \mathbb{R} \)-- for the action of the maximal torus of \( E_6 \) and \( E_7 \) on the associated minuscule representations \( V_6 \) and \( V_7 \).

Let us consider the following coordinates on the spaces of matrices introduced above: 
\begin{itemize}
\item  on $J= {\rm Herm}_3(\mathbb O)$,  $(\xi,v)$ 
with $\xi=(\xi_i)_{i=1}^3 \in {\bf k}^3$ and $v=(v_i)_{i=1}^3  \in O^3$ will stand for the coordinates of the hermitian matrix 
$
\bigg[ \hspace{-0.1cm} \scalebox{0.75}{
\begin{tabular}{rlc}
$\xi_1$ & $v_3$ & $\overline{v_2}$ \vspace{-0.1cm}\\
$\overline{v_3}$ & $s_2$ & ${v_1}$ \vspace{-0.1cm} \\
$v_2$ & $\overline{v_1}$ & $\xi_3$ 
\end{tabular}}
\hspace{-0.1cm}  \bigg]$; 
\item on $\mathcal Z_2(J)$, the coordinates of  the Zorn matrix 
$
\Big[ \hspace{-0.1cm} 
\scalebox{0.75}{
\begin{tabular}{cc}
$\zeta_1$ & $Z_1$   \vspace{-0.1cm}\\
$Z_2$ & $\zeta_2$ 
\end{tabular}}
\hspace{-0.1cm}  \Big]$
is the 4-tuple 
$(\zeta_1,Z_1,Z_2,\zeta_2)$  with 
$\zeta_i\in {\bf k}$ and $Z_i\in J$ for $i=1,2$. 
\end{itemize}

Starting from now, we deal with $O=\mathbb O_s$ and $J=\mathbb J_s={\rm Herm}_3\big(\mathbb O_s \big)$.

We start by considering the 
action of ${\bf H}^4_{>0}=\{\, \boldsymbol{h}=(h_i)_{i=1}^4 \in \big( \mathbf R_{>0}\big)^4\, \big\}$ on $\mathbb O_s$ given by
\begin{equation}
\label{Eq:Action-H-SO8}
\boldsymbol{h}\cdot u=\big(h_i)_{i=1}^4\cdot \big(u_j\big)_{j=1}^8= \Big(\, 
h_1 u_1\, , \, 
h_2 u_2\, , \, 
h_3 u_3\, , \, 
h_4 u_4\, , \, 
{h_4}^{-1} u_5\, , \, 
{h_3}^{-1} u_6\, , \, 
{h_2}^{-1} u_7\, , \, 
{h_1}^{-1} u_8\, 
\,\Big)
\end{equation}
for any $\boldsymbol{h}=(h_i)_{i=1}^4 \in {\bf H}^4_{>0}$ and any $u=(u_i)_{i=1}^8\mathbb O_s$. Clearly one has $\lVert \boldsymbol{h}\cdot u \lVert^2=\lVert  u \lVert^2$ hence the complexification of \eqref{Eq:Action-H-SO8} yields the action of the Cartan torus of ${\rm SO}_8(\mathbf C)$ on its standard representation. 
 
The triality on  ${\rm SO}_8(\mathbf C)$ gives us the two following order 3 automorphisms 
  $\varphi $ : $\boldsymbol{h}\mapsto \boldsymbol{h}^\varphi $ 
and $\psi=\varphi^2$: $\boldsymbol{h}\mapsto \boldsymbol{h}^\psi $
of ${\bf H}^4_{>0}$, defined by  (with $\boldsymbol{h}=(h_i)_{i=1}^4 \in {\bf H}^4_{>0}$): 
\begin{align*}
\boldsymbol{h}^\varphi=\varphi(\boldsymbol{h})=  & \, 
\left(\,
\frac{\sqrt{h_{3}}}{\sqrt{h_{1}\,h_{2}\,h_{4}}}
\, , \,  
\frac{\sqrt{h_{1}\,h_{2}\,h_{3}}}{\sqrt{h_{4}}}
\, , \,  
\frac{\sqrt{h_{2}\,h_{3}\,h_{4}}}{\sqrt{h_{1}}}
\, , \,  
\frac{\sqrt{h_{1}\,h_{3}\,h_{4}}}{\sqrt{h_{2}}}
\, 
\right)\qquad \\
\mbox{and }\quad 
\boldsymbol{h}^\psi=\psi(\boldsymbol{h}) = &\,  
\left(\, \frac{\sqrt{h_{2}\,h_{4}}}{\sqrt{h_{1}\,h_{3}}}
\, , \,  
 \frac{\sqrt{h_{2}\, h_{3}}}{\sqrt{h_{1}\, h_{4}}}
 \, , \,  
\sqrt{h_{1}\, h_{3}\, h_{2}\, h_{4}}
\, , \,  
\frac{\sqrt{h_{4}\, h_{3}}}{\sqrt{h_{1}\, h_{2}}}\, \right)\, . 
\end{align*}
\sk

The `positive torus' ${\bf H}^4_{>0}$ acts on $\mathbb J_s= {\rm Herm}_3(\mathbb O_s)$ by means of the following formula:  one sets 
\begin{equation}
\label{Eq:H4-action}
\boldsymbol{h}\cdot \big(\, \xi\, ,\, v\,\big)= 
\boldsymbol{h}\cdot \Big(\, 
\xi\, , \, 
\big(v_j\big)_{j=1}^3\,
\Big)= \Big(\, 
\xi\, , \, 
\big(\, \boldsymbol{h}\cdot v_1\, , \, 
\boldsymbol{h}^\phi\cdot v_2
\, , \, 
\boldsymbol{h}^\psi\cdot v_3\, 
\big)\,
\Big)
\end{equation}
for any ${\bf h}=(h_i)_{i=1}^4\in {\bf H}^4_{>0}$
and any $(\xi,v)\in \mathbf R^3\times \mathbb O_s^3$. 
We define another action of  ${\bf L}^2_{>0}= \{ \, \boldsymbol{\ell}=\big(l_1,l_2,l_3\big)\in \big( \mathbf R>0\big)^3 \, \big\lvert \,  l_1l_2l_3=1\, 
\big\}\simeq \big( \mathbf R_{>0}\big)^2$ on  $\mathbb J_s$ by setting 
\begin{equation}
\label{Eq:L2-action}
\boldsymbol{\ell}\cdot \big(\, \xi,v\,\big)= 
\Big(\, 
\big(l_i^2 \xi_i\big)_{i=1}^3  \, , \, 
\big(\, l_i^{-1} v_i\big)_{i=1}^3\,
\Big)\,.
\end{equation}

The actions \eqref{Eq:H4-action} and \eqref{Eq:L2-action} commute hence give rise to an action of the `positive torus'
$${\bf T}_{>0}^6= {\bf L}^2 \times {\bf H}^4 \simeq \big( \mathbf R_{>0}\big)^6$$
 on $\mathbb J_s$ given by
\begin{equation}
\label{Eq:H4L2-action}
\big(\boldsymbol{\ell}, \boldsymbol{h}\big) 
\bullet \begin{bmatrix}
\xi_1 &    v_3 & \overline{v_2}\,  {}^{} \\
 {}^{}\,  \overline{v_3 } & \xi_2 & v_1 \\
v_2 & \overline{v_1} & \xi_3 
\end{bmatrix}= 
\scalebox{1.1}{$
\begin{bmatrix}
(l_1)^2\,\xi_1 &    \boldsymbol{h}^\psi\cdot v_3 & \overline{\boldsymbol{h}^\varphi\cdot v_2}
\,\,  {}^{}\vspace{0.15cm} \\
{}^{}\, \,\overline{\boldsymbol{h}^\psi\cdot v_3 } & (l_2)^2\,\xi_2 & \boldsymbol{h}\cdot v_1 \vspace{0.15cm} \\
\boldsymbol{h}^\varphi\cdot v_2 & \overline{\boldsymbol{h}\cdot v_1} & (l_3)^2\,\xi_3 
\end{bmatrix}$}
\, . 
\end{equation}

We extend the preceding action to an action of the larger group 
$$ {\bf T}^7_{>0}= \big(\mathbf R_{>0}\big)\times {\bf L}^2 
\times 
{\bf H}^4 
\simeq \big( \mathbf R_{>0}\big)^7
$$
 on  $\boldsymbol{\mathcal Z}_2\big(\mathbb J_s\big)$ by setting 
\begin{equation}
\label{Eq:H(E7)-action}
\big( t\, ,\, (\boldsymbol{\ell}, \boldsymbol{h})\big) \cdot   \begin{bmatrix}
\zeta_1 &    Z_1 \\
 Z_2 & \zeta_2 
\end{bmatrix}
=\begin{bmatrix}
t\, \zeta_1 &   t^{1/3} \big(\boldsymbol{\ell}, \boldsymbol{h}\big) \bullet Z_1 
\,\,{}^{}
\vspace{0.1cm}\\
{}^{}\, \,  t^{-1/3}  \big(\boldsymbol{\ell}^{-1}, \boldsymbol{h}\big)
\bullet Z_2 & t^{-1}\zeta_2 
\end{bmatrix}
\end{equation}
for all $( t\, ,\, \big(\boldsymbol{\ell}, \boldsymbol{h}\big)\big)$ in the group and all Zorn matrices  
$\Big[
\hspace{-0.15cm}
\scalebox{0.78}{
\begin{tabular}{cc}
$\zeta_1$ & \hspace{-0.3cm} $Z_1$\\
$Z_2$ & \hspace{-0.3cm}  $\zeta_2$
\end{tabular}}\hspace{-0.12cm} \Big]\in \boldsymbol{\mathcal Z}_2\big(\mathbb J_s\big)$. 
\sk

We have explicitly described several isomorphisms of 
${\mathbf R}$-vector spaces above:  
$\mathbb O_s\simeq \mathbf R^8$, $\mathbb J_s={\rm Herm}_3\big( 
\mathbb O_s \big) \simeq \mathbf R^3\oplus \mathbb O_s^{3}$ and 
$\boldsymbol{\mathcal Z}_2\big(\mathbb J_s\big)\simeq \mathbf R\oplus 
\mathbb J_s\oplus \mathbb J_s \oplus \mathbf R$. From these, one gets 
well-defined isomorphisms of real vector spaces 
\begin{equation}
\label{Eq:R-ev-Isom}
\mathbb J_s\simeq \mathbf R^{27}
\qquad 
\mbox{ and }
 \qquad \boldsymbol{\mathcal Z}_2\big(\mathbb J_s\big)\simeq \mathbf R^{56}\,.
\end{equation}
We denote by 
$(\mathfrak E_i)_{i=1}^{27}$ and $(\mathfrak Z_j)_{j=1}^{56}$ 
the bases of 
$\mathbb J_s$ and 
$\boldsymbol{\mathcal Z}_2\big(\mathbb J_s\big)$ corresponding  to the canonical bases of 
$\mathbf R^{27}$  and $\mathbf R^{56}$ respectively. 

The genuine interest of working with the specific product \eqref{Eq:Special-Octonionic-Product} and the explicit actions 
 \eqref{Eq:H4L2-action} becomes apparent by considering the 
  content of the following 
\begin{prop}
{\bf 1.} The images in ${\rm GL}\big( \mathbf J\big)$  and 
${\rm GL}\big( \boldsymbol{\mathcal Z}_2\big(\mathbf J\big)\big)$ 
of the `positive tori' ${\bf T}_{>0}^6$ and 
${\bf T}_{>0}^7$ land into $E_6$ and $E_7$ respectively. Consequently, 
up isogenies, the actions 
\eqref{Eq:H4L2-action} 
and 
\eqref{Eq:H(E7)-action} are induced by the ones of some Cartan subtori of these two simple complex Lie groups.
\sk

\noindent 
{\bf 2.} The bases $(\mathfrak E_i)_{i=1}^{27}$ and $(\mathfrak Z_j)_{j=1}^{56}$  are bases of weight vectors for ${\bf T}_{>0}^6$ and 
${\bf T}_{>0}^7$ respectively.  

More precisely, for any $i\in \{1,\ldots,27\}$, 
there exists a weight $w_i=(w_{i,s})_{s=1}^6 \in \frac{1}{2}\mathbf Z^6$ such that for any 
$(\boldsymbol{\ell},\boldsymbol{h})\in {\bf T}_{>0}^6$ with $
\boldsymbol{\ell}=(\ell_1,\ell_2)\in \big( \mathbf R_{>0}\big)^2$ and $ 
\boldsymbol{h}=(h_i)_{i=1}^4\in \big( \mathbf R_{>0}\big)^4$, one has 
$$
(\boldsymbol{\ell},\boldsymbol{h})\bullet \mathfrak E_i= 
\omega_i(\boldsymbol{\ell},\boldsymbol{h} ) \,\mathfrak E_i=
\Big({\ell_1}^{\omega_{i,1}}
{\ell_2}^{\omega_{i,2}}
\prod_{k=1}^4 {h_k}^{\omega_{i,k+2}}\Big)
\, 
\mathfrak E_i \, . 
$$
And there is a similar statement for the action \eqref{Eq:H(E7)-action}, but with weights in $\frac{1}{2}\mathbf Z^7\cup \frac{1}{3}\mathbf Z^7$. 
\end{prop}

\begin{proof}
The first point follows from the explicit description of the Lie groups $E_6 $ and $E_7$ as subgroups of 
${\rm GL}\big( \mathbf J\big)$  and 
${\rm GL}\big( \boldsymbol{\mathcal Z}_2\big(\mathbf J\big)\big)$, respectively, as found in~\cite{Yokota}. For instance, direct formal computations show that
$
{\rm det}\Big((\boldsymbol{\ell},\boldsymbol{h} )\bullet Z\Big)={\rm det}\big( 
Z\big)$ for every $(\boldsymbol{\ell},\boldsymbol{h} )\in {\bf T}_{>0}^6$ and $Z\in \mathbf J$. Hence, the image of this torus in ${\rm GL}\big( \mathbf J\big)$actually lies in $E_6$, according to the definition of this group given in~\cite[\S3.1]{Yokota}.\footnote{Actually, it was by examining the descriptions of the real parts of the Lie algebras of the Cartan tori of the complex Lie groups $ E_6$ and $ E_7$, as given on pages 94 and 139 of \cite{Yokota}, that we derived the explicit formulas for the torus actions \eqref{Eq:H4L2-action} 
and 
\eqref{Eq:H(E7)-action}.}

Direct formal computations, made possible by the completely explicit nature of the actions 
\eqref{Eq:H4L2-action} 
and \eqref{Eq:H(E7)-action}, yield the second part of the proposition
\end{proof}

 One can easily determine the weights of the representation for each case.
For instance, when $r=6$, we obtain the following set of 27 pairwise distinct weights

\begin{align*}
\Big(\, 2, 0, 0, 0, 0, 0\,\Big) \, , \, & \Big(\,0, 2, 0, 0, 0, 0\,\Big) \, , \,
 \Big(-2, -2, 0, 0, 0, 0\,\Big)\, , \, 
 \Big(-1, 0, 1, 0, 0, 0\,\Big)\, , \, \\
  \Big(-1, 0, 0, 1, 0, 0\,\Big) \, , \,
  &
   \Big(-1, 0, 0, 0, 1, 0\,\Big)\, , \, 
   \Big(-1, 0, 0, 0, 0, 1\,\Big)\, , \,
    \Big(-1, 0, 0, 0, 0, -1\,\Big)\, , \, \\
     \Big(-1, 0, 0, 0, -1, 0\,\Big)\, , \,
     & \Big(-1, 0, 0, -1, 0, 0\,\Big)\, , \,
      \Big(-1, 0, -1, 0, 0, 0\,\Big)\, , \,
       \Big(\,0, -1, -{\frac{1}{2}}, -{\frac{1}
{2}}, {\frac{1}{2}}, -{\frac{1}{2}}\,\Big)\, , \, \\ 
 \left(0, -1, {\frac{1}{2}}, {\frac{1}{2}}, {\frac{1}{2}}, -{\frac{1}{2}}\right)\, , \,
 &
  \left(0, -1, -{\frac{1}{2}}, {\frac{1}{2}}, {\frac{1}{2}}, {\frac{1}{2}}\right)\, , \,
\left(0, -1, {\frac{1}{2}}, -{\frac{1}{2}}, {\frac{1}{2}}, {\frac{1}{2
}}\right)\, , \,
\left(0, -1, -{\frac{1}{2}}, {\frac{1}{2}}, -{\frac{1}{2}}
, -{\frac{1}{2}}\right)\, , \,\\
\left(0, -1, {\frac{1}{2}}, -{\frac{1}{2}}, -
{\frac{1}{2}}, -{\frac{1}{2}}\right)\, , \,&
 \left(0, -1, -{\frac{1}{2}}, -{
\frac{1}{2}}, -{\frac{1}{2}}, {\frac{1}{2}}\right)\, , \,
 \left(0, -1, {
\frac{1}{2}}, {\frac{1}{2}}, -{\frac{1}{2}}, {\frac{1}{2}}\right)\, , \,
\left(1, 1, -{\frac{1}{2}}, {\frac{1}{2}}, -{\frac{1}{2}}, {\frac{1}{2
}}\right)\, , \,
\\ 
 \left(1, 1, -{\frac{1}{2}}, {\frac{1}{2}}, {\frac{1}{2}}, -
{\frac{1}{2}}\right)
\, , \, 
&
 \left(1, 1, {\frac{1}{2}}, {\frac{1}{2}}, {\frac
{1}{2}}, {\frac{1}{2}}\right), \left(1, 1, -{\frac{1}{2}}, -{\frac{1}{
2}}, {\frac{1}{2}}, {\frac{1}{2}}\right)
\, , \,
 \left(1, 1, {\frac{1}{2}}, {
\frac{1}{2}}, -{\frac{1}{2}}, -{\frac{1}{2}}\right)
\, , \,
\\
 \left(1, 1, -{
\frac{1}{2}}, -{\frac{1}{2}}, -{\frac{1}{2}}, -{\frac{1}{2}}\right)%
\, , \,
&
\left(1, 1, {\frac{1}{2}}, -{\frac{1}{2}}, -{\frac{1}{2}}, {\frac{1}{2
}}\right)
\, , \,
 \left(1, 1, {\frac{1}{2}}, -{\frac{1}{2}}, {\frac{1}{2}}, -
{\frac{1}{2}}\right)\, .
\end{align*}

The weight polytope $\Delta_6$ is the convex enveloppe of these $27$ points of $\mathbf R^6$.  With the above list of weights at hand, it is not difficult to obtain a new explicit list, of the sets of vertices of the facets of 
$\Delta_6$ : for each facet $F$ of $\Delta_6$, let $I_F$ be the set of indices 
$i$ such that $w_i\in F$.  The $\boldsymbol{\mathcal W}$-relevant facets are 5-dimensional orthoplexes, with set of vertices of cardinality 10. 
 For any such $F$, we set $\mathbf C_F^{10}=\oplus_{i\in I_F} \mathbf C\,\mathfrak E_i$
  the linear projection $\Pi_F: \mathbf J\rightarrow  \mathbf C^{10}_F$ is written $\tilde \Pi_F: (x_i)_{i=1}^{27}\longmapsto (x_j)_{j\in I_F}$ in the linear coordinates associated to the basis $(\mathfrak E_i)_{i=1}^{27}$ of $\mathbf J$.  Then the 27 rational maps 
 $$
 \tilde \Pi_F \circ \nu_{\mathbf O}^2 : \, 
 {\mathbf O}\oplus {\mathbf O} \dashrightarrow \mathbf P(\mathbf C^{10}_F)\simeq \mathbf P^9_F
 $$
 are the first integrals of Gelfand-MacPherson web on $\mathbf O\mathbf P^2$, written in coordinates. 
 
 Given a web-relevant facet $F$ associated to a conic class $\boldsymbol{\mathfrak c}\in \boldsymbol{\mathcal K}_6$, 
 there is an involution  on $ I_F$ when this set of indices is seen as a subset of $\boldsymbol{\mathcal L}_6$, given by $\ell\mapsto \ell'=\boldsymbol{\mathfrak c}-\ell$. Let $J_F\subset I_F$ be a fixed  subset of cardinality 5 such that for any $\ell\in I_F$, either $\ell$ or $\ell'$ is in $J_F$. When seeing $I_F$ as a subset of $\{1,\ldots,27\}$, we will denote this involution by $i\mapsto i'_F$. To simplify the formulas, we will often omit 
the subscript $F$ in the lines below.  
 
  It can be verified that the closure in $\mathbf P^9$ of the image of  $\tilde \Pi_F \circ \nu_{\mathbf O}^2$ is a smooth hyperquadric $Q_F\subset \mathbf P^9$ 
  cut out by a quadric equation of the form 
  $$q_F=\sum_{j\in J_F} \epsilon_{F,j} \,x_{j} x_{j'}=0$$ with  
 $\epsilon_{F,j}\in \{\,\pm 1\,\}$ for any $j\in J_F$. 
The basis $(\mathfrak E_i)_{i\in I_F}$ actually is a weight basis for the action of   the Lie group ${\bf SO}(q_F)$ of type $D_5$ on $\mathbf C^{10}_F$, with the action of a Cartan torus $H_F\simeq \big( \mathbf C^*\big)^{J_F}$ of  ${\bf SO}(q_F)$  given by 
$$(\lambda_j)_{j\in J_F}\cdot \mathfrak E_i= \begin{cases}\, \, 
\lambda_i\, \mathfrak E_i \hspace{0.45cm} \mbox{ if } i\in J_F
\\
\lambda_{i'}^{-1}\, \mathfrak E_i  \hspace{0.3cm} \mbox{ if } i\not \in J_F\, .  
\end{cases}
 $$
 It can be verified that the action induced by the Cartan torus of $E_6$ obtained by taking the complexification of ${\bf T}^6_{>0}$ is linearly equivalent to the one described just above.  Let $Q_F^\star$ be $Q_F$ with the union of the ten hyperplane sections $Q_F\cap \{\,x_i=0\,\}$ removed. Then $H_F$ acts nicely with finite kernel on $Q_F^\star$, and gives rise 
 to a rational morphism $\tau_F:  Q_F^\star\longrightarrow \boldsymbol{\mathcal Y}_F^\star=Q_F^\star/H_F$ which can be made explicit:  fixing $j_0\in J_F$, we set $J_{F}^*=J_F\setminus \{j_0 \}$. 
 Then a birational model of $\tau_F$ is given by  
\begin{align} 
\label{Eq:Tilde-tau-F} 
 \tilde \tau_F \, :\,  \mathbf P^9_F & \, \dashrightarrow \mathbf P^3\\
\big[\,x_i\,\big]_{i\in I_F} & \longmapsto \left[ \frac{x_{j}x_{j'}}{x_{j_0}x_{j_0'}} \right]_{j\in J_{F}^*}\, .
\nonumber 
 \end{align} 
 
 It follows that,  possibly up to multiplying some of the monomials 
 ${x_{i}x_{i'}}/({x_{i_0}x_{i_0'}})$ by -1, one obtains a quotient morphism 
 $\tau_F: Q_F^\star\longrightarrow \boldsymbol{\mathcal Y}_F^\star$ with $\boldsymbol{\mathcal Y}_F^\star\subset \mathbf P^3$ equal to the complement of the hyperplane arrangement 
$\mathcal H_6\subset \mathbf P^3$ 
 cut 
 out by $u_0u_1u_2u_3(u_0+u_1+u_2+u_3)=0$ for some homogeneous coordinates $u_0,\ldots,u_3$ on $\mathbf P^3$. 
  
 The action of ${\bf T}_{>0}^6$ on $\mathbf O\oplus \mathbf O$ induced by the action \eqref{Eq:H(E7)-action} is given by 
 $$ 
 \big(\, a, b ,\boldsymbol{h}\big)\, (u,v)= \Big( \, ab \,(\boldsymbol{h}^\psi\cdot  u) \, , \, b^{-1} (\boldsymbol{h}^\varphi\cdot v)\, \Big) \, . 
 $$
 For $(u,v)\in \mathbf O\oplus \mathbf O$ generic, one can find $\big(\, a, b ,\boldsymbol{h}\big)\in\big( \mathbf C^*\big)^6$ such that 
$\big(\,  \tilde u\, , \, \tilde v\,\big)=
 \big(\, a, b , \boldsymbol{h}\big)\,(u,v)$ 
 is in `normal form', that is such that 
\begin{equation} 
\label{Eq:Normal-Form} 
 \tilde u=\big(\, \tilde u_1 , \tilde u_2 , \tilde u_3 ,  1, 1 , 1 ,1, 1\,  \big) \qquad \mbox{ and }
 \qquad \tilde v=\big(\, 1 \, , \, \tilde v_2 \, , \,  
 \tilde v_3 \, , \,  
 \tilde v_4 \, , \,  
 \tilde v_5 \, , \,  
 \tilde v_6 \, , \,  
 \tilde v_7 \, , \,  
  \tilde v_8\,  \big)\, .
 \end{equation} 
 Moreover, the element $\big(\, a , b ,\boldsymbol{h}\,\big)
 \in\big( \mathbf C^*\big)^6$ satisfying \eqref{Eq:Normal-Form} is essentially unique; that is, it is unique up to the action of a certain finite subgroup of $\big( \mathbf C^*\big)^6$, which can be explicitly determined without difficulty.
It follows that the ten algebraically independent variables $\tilde u_1 , \tilde u_2 , \tilde u_3$ and 
$\tilde v_2 , \ldots , \tilde v_8 $ define a rational chart on the quotient $\boldsymbol{\mathcal Y}_6$ of Cayley octonionic projective plane $\mathbf O\mathbf P^2$ by the action of the Cartan subgroup of $E_6$ we are working with. 
 Thus:  (1) $\boldsymbol{\mathcal Y}_6$ is rational;  (2) 
  $\boldsymbol{Y}_6={\rm Spec}\big( 
 \mathbf C[\tilde u_1 , \tilde u_2 , \tilde u_3, 
 \tilde v_2 , \ldots , \tilde v_8 ]
 \big)$ is a birational model of it; and (3) up to the birational identification 
$\boldsymbol{Y}_6\simeq  \boldsymbol{\mathcal Y}_6$,   the following map corresponds to a  rational section of the quotient map $\mathbf O\mathbf P^2\dashrightarrow \boldsymbol{\mathcal Y}_6$: 
\begin{align} 
\label{Eq:rational-section-r=6} 
\sigma_6\, :\, \mathbf C^6= \boldsymbol{Y}_6 & \, \dashrightarrow \mathbf O\mathbf P^2 \subset \mathbf P^{26}\\
 \Big(\,\big(\tilde u_i\big)_{i=1}^3 \, , \, 
\big(\tilde v_j\big)_{j=2}^8 \, \Big)
 & \longmapsto \nu_{\mathbf O}^2\big(\,\tilde u, 
 \tilde v\, \big)
 \nonumber 
\end{align}
where $\tilde u$ and $\tilde v$ stand for the two elements of $\mathbf O$ respectively defined by  the LHS of the two equalities in \eqref{Eq:Normal-Form}. 

Post-composing $\sigma_6$ with the linear projections $\tilde \Pi_F$ then with the rational model $\tilde \tau_F$ (defined in \eqref{Eq:Tilde-tau-F}) of the quotient maps $\tau_F : \boldsymbol{\mathcal Q}_F\dashrightarrow \mathbf P^3$ 
 for all web-relevant facets $F$ of $\Delta_6$, we obtain 27 rational maps 
 \begin{equation} 
 U_F = \tilde \tau_F \circ \tilde \Pi_F \circ \sigma_6 : \boldsymbol{Y}_6 \dashrightarrow \mathbf P^3
 \end{equation} 
 which define a web, noted by $\boldsymbol{W}_{\boldsymbol{Y}_6}^{GM}$, which is the birational model of $\boldsymbol{\mathcal W}_{\boldsymbol{\mathcal Y}_6}^{GM}$ on $\boldsymbol{Y}_6$ induced  by the birational identification  $\boldsymbol{Y}_6\simeq \boldsymbol{\mathcal Y}_6$. 
 
Thanks to our choice of the specific version
\eqref{Eq:Special-Octonionic-Product} of the (split) octonionic product, the torus actions 
\eqref{Eq:H4L2-action} and \eqref{Eq:H(E7)-action} 
 are defined over \( \mathbf{R} \), and so are the bases of weight vectors
 $(\mathfrak E_i)_{i=1}^{27}$ and $(\mathfrak Z_j)_{j=1}^{56}$
 induced by the linear isomorphisms \eqref{Eq:R-ev-Isom}. 
 This has the pleasant consequence that all the first integrals \( U_F \) are defined over \( \mathbf{R} \) (in fact, over \( \mathbf{Z} \)), which makes the effective computations we will perform from them -- in order to study the ranks and abelian relations (ARs) of the web $\boldsymbol{W}_{\boldsymbol{Y}_6}^{GM}$  -- significantly more efficient.
 \begin{center}
\vspace{-0.2cm}
$\star$\sk
\end{center}

 All the considerations and constructions above are not truly specific to the case \( r = 6 \); they admit direct analogues in the case \( r = 7 \), which makes it possible to state the following result uniformly for all \( r \in \{4, 5, 6, 7\} \):
\begin{prop} 
Let \( V_r \) be the minuscule representation in whose projectivization the homogeneous space \( \boldsymbol{\mathcal{G}}_r = G_r / P_r \), of dimension \( \gamma_r \), embeds.

\noindent 
{\bf 1.} One can explicitly describe:
\begin{itemize}
    \item[$-$] \vspace{-0.15cm} a  basis \( (\mathfrak e_i)_{1 \leq i \leq \upsilon_r} \) of \( V_r \), yielding an isomorphism \( V_r \simeq \bigoplus_{i=1}^{\upsilon_r} \mathbf{C} \,\mathfrak e_i \);\sk 
    
    \item[$-$] an affine space \( \mathbf{A}_r \) of dimension \( \gamma_r \);
    \sk 
    
    \item[$-$] an affine  embedding  \( \nu_r = (\nu_{r,i})_{1 \leq i \leq \upsilon_r} : 
    \mathbf{A}_r \to V_r \), 
\end{itemize}
  with the property that the image of the projectivization
  \(
  [\nu_r] : \mathbf{A}_r \to \mathbf{P}V_r
  \)
  has Zariski closure equal to \( \boldsymbol{\mathcal G}_r  \), that is, 
$$ \overline{[\nu_r]\big(\mathbf A_r\big)}=\boldsymbol{\mathcal G}_r
\subset 
\mathbf PV_r\simeq \mathbf P^{\upsilon_r-1}\, . $$

\noindent 
{\bf 2.} Moreover, one can describe just as explicitly an action of 
\( H_r=\big(\mathbf{C}^*\big)^r \) on \( V_r\oplus_{i=1}^{\upsilon_r} =\mathbf{C} \, \mathfrak e_i \) , satisfying the following properties:
\begin{itemize}
    \item[$-$]  Each $ \mathfrak e_i$ is a weight vector with respect to this action, associated to a weight 
   $\omega_i\in \mathbf Q^r$, which can be explicitly computed; 
    \sk 
    \item[$-$] 
    The \( \mathcal{W} \)-relevant facets of 
     \( \Delta_r = \mathrm{Conv}\big(\{ \omega_i \}_{i=1}^{ \upsilon_r}\big) \)
    are entirely determined by the set of their vertices:
          for each such facet \( F \), let $
    I_F \subset \{1, \ldots, \upsilon_r\}
    $ 
    be the subset of cardinality \( 2r - 2 \) consisting of the indices \( i \) such that 
    the weight 
    $\omega_i$ associated to 
    \( e_i \) is a vertex of \( F \);
 \sk 
  \item[$-$]  The set of $\kappa_r$ subsets $I_F$'s can be explicitly determined hence the linear projections 
  $$\tilde \Pi_F: V_r= \oplus_{i=1}^{\upsilon_r} \mathbf{C} \, \mathfrak e_i\rightarrow \oplus_{i \in I_F} \mathbf{C} \, \mathfrak e_i=:V_{r,F}$$ as well.  
 The projectivizations of the compositions $\tilde \Pi_F\circ \nu_r: {\bf A}_r \rightarrow V_F$ for all web-relevant facets $F$ of the weight polytope $\Delta_r$ are rational first integrals of the birational model $\nu_r^*\big( \boldsymbol{\mathcal W}^{GM}_{ \boldsymbol{\mathcal G}_r}\big)$ of Gelfand-MacPherson web on $\boldsymbol{\mathcal G}_r$.
\end{itemize}

\noindent 
{\bf 3.} For a certain subset of indices $\mathcal A_r\subset \{1,\ldots, \dim \, V_r\}$ 
of cardinality $\gamma_r$, the affine embedding $\nu_r$  post-composed with 
the natural projection $\oplus_{i=1}^{\upsilon_r} \mathbf{C} \,\mathfrak e_i\rightarrow \oplus_{a \in {\mathcal  A}_r } \mathbf{C} \,\mathfrak e_a$ is a linear isomorphism defined over $\mathbf R$ (actually, over $\mathbf Z$).  Consequently the $H_r$-action on $V_r$ induces an action of the same torus on $  \mathbf{A}_r\simeq 
\oplus_{a \in {\mathcal  A}_r } \mathbf{C} \,\mathfrak e_a$ which 
is a birational model of the action of the Cartan torus of
 $G_r$ on $\boldsymbol{\mathcal G}_r$.

 Moreover, 
 one can provide an explicit list of algebraically independent monomials 
 $M_{s}$ for $s=1,\ldots, {y}_r
 $,  which form an algebraic basis of the algebra of \( H_r \)-invariant rational functions on  ${\bf  A}_r$. 
 Setting 
 $\boldsymbol{Y}_r={\rm Spec}\Big( \mathbf C\big[ M_{1},\ldots,M_{{y}_r}\big]
 \Big) \simeq \mathbf C^{{y}_r}$,  
 we obtain an explicit rational map 
 $ {\mathcal P}_r : {\bf A}_r\rightarrow \boldsymbol{Y}_r$  which constitutes a birational model of the quotient map  $\tau_r \, : \, \boldsymbol{\mathcal G}_r \rightarrow 
\boldsymbol{\mathcal Y}_r =  
 \boldsymbol{\mathcal G}_r/H_r$.

  Finally, one constructs an explicit affine embedding 
  $\sigma_r :  \boldsymbol{Y}_r 
 \rightarrow {\bf A}_r$ which is a (rational) section of the quotient map 
 $ {\bf A}_r\rightarrow \boldsymbol{Y}_r ={\bf A}_r/H_r$, and such that the composed map 
 $ \Xi_r=\tau_r \circ \big[\nu_r\big] \circ \sigma_r  : 
 \boldsymbol{Y}_r 
 \dashrightarrow \boldsymbol{\mathcal Y}_r$ is birational. This yields a birational identification  $ \boldsymbol{Y}_r 
 \simeq \boldsymbol{\mathcal Y}_r$ and the construction fits into the following commutative diagram of rational maps:  
\begin{equation}
\label{Diag:gogogo}
\begin{tabular}{c}
\xymatrix@R=1cm@C=0.9cm{
{\bf A}_r   \ar@{->}[d]^{ \scalebox{1.01}{${\mathcal P}_r$}} \ar@{->}[r]^{ \scalebox{1.01}{$[\nu_r]$}}  &
\boldsymbol{\mathcal G}_r
\ar@{->}[d]^{\scalebox{1.06}{$\tau_r$}}
 \\
\ar@/^2pc/@{->}[u]^{ \scalebox{1.01}{$\textcolor{black}{\sigma_r}$}}
\boldsymbol{Y}_r 
  \ar@{->}[r]^{\scalebox{1.01}{$\Xi_r$}}& \, \,  \boldsymbol{\mathcal Y}_r\, .
 }
\end{tabular}\vspace{0.3cm}
\end{equation}
 
\noindent 
{\bf 4.} Given a $\mathcal W$-relevant facet $F$, let $\boldsymbol{\mathfrak c}\in 
\boldsymbol{\mathcal K}_r$
be the conic class associated to it. 
Identifying the basis $\{ \mathfrak e_i\,\}_{i=1}^{\upsilon_r}$ with $\boldsymbol{\mathcal L}_r$, the elements in the index set $I_F$ appear in {\it `associated pairs'}, with $\{i,j\}\subset I_F$ (with $i\neq j$) being such a pair  if and only if the sum of the two corresponding lines is equal to $\boldsymbol{\mathfrak c}$. For $i\in I_F$, we denote by $i'$ the unique other element of $I_F$ such that $\{i,i'\}$ is an associated pair. 
We then fix a subset \( J_F \subset I_F \) such that, setting $J_F'=\{j'\}_{j\in J_F}$, one has $I_F=J_F\sqcup J_F'$.

Then the image of $\boldsymbol{\mathcal G}_r$ by the projectivization $\big[ \tilde \Pi_F\big]: \mathbf PV_r\dashrightarrow 
\mathbf PV_{r,F}=\mathbf P^{2r-3}_F$ of the linear projection 
$ \tilde \Pi_F$ 
 is a smooth hyperquadric $\boldsymbol{\mathcal Q}_F$ cut out by an equation of the form 
$q_F=\sum_{i\in J_F} \varepsilon_{F,i}\,x_i x_{i'}=0$ in the coordinate system on 
$V_{r,F}$ dual to the base $(\mathfrak e_i)_{i\in I_F}$, with $\varepsilon_{F,i}\in \{\,\pm1\,\}$ for every $i\in J_F$. The torus action  on $V_{r,F}$ 
induced by that of $H_r$ on $V_{r}$ is equivalent to the  one defined by  
$$
\boldsymbol{\lambda}\cdot \mathfrak e_i= \lambda_i \mathfrak e_i \quad \mbox{if } i\in J_F \qquad \mbox{ and  } \qquad 
\boldsymbol{\lambda}\cdot \mathfrak e_i= \lambda_{i'}^{-1} \mathfrak e_i\quad \mbox{when } i\not \in J_F
$$
for any $\boldsymbol{\lambda}=(\lambda_s)_{s=1}^{r-1}
\in \big(\mathbf C^*\big)^{r-1}$ and any $i\in I_F$. 
We fix $j_0\in J_F$ and we set $J_F^*=J_F\setminus \{j_0\}$. 
It follows that, as a rational map, 
the quotient map of $\mathbf P^{2r-3}_F$ by the action of the torus of rank $r-1$ induced by that of $H_r$ is written
\begin{equation} 
\label{Eq:Tilde-tau-F-r}
 \tilde \tau_F \, :\,  \mathbf P^{2r-3}_F  \, \dashrightarrow \mathbf P^{r-3}\, , \,\, 
\big[\,x_i\,\big]_{i\in I_F}  \longmapsto \left[\, {x_{j}x_{j'}}/\big({x_{j_0}x_{j_0'}} \big)\,\right]_{j\in J_{F}^*}\, .
 \end{equation}
 in the coordinates $x_i$, $i\in I_F$. 
\sk

\noindent 
{\bf 5.} From the four points above,  it comes that  for any web-relevant facet $F$ of $\Delta_r$, the explicit rational map 
 \begin{equation} 
 \label{Eq:UF-r}
 U_F = \tilde \tau_F \circ \big[ \tilde \Pi_F \big] \circ 
 [  \nu_r ] \circ 
 \sigma_r : \boldsymbol{Y}_r  \dashrightarrow \mathbf P^{r-3}
\end{equation}
makes commutative the following diagram of rational maps 
\begin{equation}
\label{Diag:gogoga}
\begin{tabular}{c}
\xymatrix@R=1.4cm@C=2cm{
{\bf A}_r   \ar@{->}[d]^{ \scalebox{1.01}{${\mathcal P}_r$}} \ar@{->}[r]^{ \scalebox{1.01}{$[\nu_r]$}}  &
\boldsymbol{\mathcal G}_r
 \ar@{->}[r]^{ \scalebox{1.06}{$\Pi_F$} }   
\ar@{->}[d]^{\scalebox{1.06}{$\tau_r$}} & \boldsymbol{\mathcal Q}_F
\ar@{->}[d]^{\scalebox{1.06}{$\tau_F$}}
&\hspace{-3.2cm} \subset \, \mathbf P^{2r-3}_F 
 \\
\ar@/^2pc/@{->}[u]^{ \scalebox{1.01}{$\textcolor{black}{\sigma_r}$}}
\boldsymbol{Y}_r 
  \ar@{->}[r]^{\scalebox{1.06}{$\Xi_r$}}
  \ar@/_2pc/[rr]_{\scalebox{1.06}{$U_F$} }
  & \, \,  \boldsymbol{\mathcal Y}_r
 \ar@{->}[r]^{ \scalebox{0.9}{$\pi_F$} }  
  &  \mathbf P^{r-3}
 }
\end{tabular}\vspace{0.3cm}
\end{equation}
in which we have $\Pi_F=[\tilde \Pi_F ]\lvert_{ \boldsymbol{\mathcal G}_r}$  and 
$\tau_F=\tilde \tau_F\lvert_{\boldsymbol{\mathcal Q}_F} $.  It follows that 
the $U_F$'s  are rational first integrals of a $\kappa_r$-web of codimension $r-3$ on 
$\boldsymbol{Y}_r=\mathbf C^{y_r}$, that we will denote by 
$\boldsymbol{W}_{\boldsymbol{Y}_r}^{GM}$, which is the birational model of $\boldsymbol{\mathcal W}_{\boldsymbol{\mathcal Y}_r}^{GM}$ on 
$\boldsymbol{Y}_r$ induced  by the birational identification  $\boldsymbol{Y}_r\simeq \boldsymbol{\mathcal Y}_r$. One has:
$$
\boldsymbol{W}_{\boldsymbol{Y}_r}^{GM}=
\boldsymbol{\mathcal W}\left(\,U_F\,\, \Big\lvert\,\, 
\scalebox{0.8}{\begin{tabular}{c}
$\boldsymbol{\mathcal W}$-relevant \\
facet $F\subset \Delta_r$  \\
\end{tabular}}
\,
\right)=\Xi_r^*\left( 
\boldsymbol{\mathcal W}_{\boldsymbol{\mathcal Y}_r}^{GM}
\right) \, . 
$$

\noindent 
{\bf 6.} Given a web-relevant facet $F$, possibly up to multiplying certain components of the map $\tilde \tau_F$ by $-1$, then the image of $\boldsymbol{\mathcal Q}_F^\star$ by $\tilde \tau_F$ in $\mathbf P^{r-3}$ is the complement of the arrangement of $r-1$ hyperplanes $\mathcal H_r\subset \mathbf P^{r-3}$ cut out by the equation 
$$
u_0u_1\cdots u_{r-3}\big( u_0+u_1+\cdots+  u_{r-3}\big)=0\, . 
$$
in the homogeneous coordinates $u_0,\ldots,u_{r-3}$ corresponding to the explicit expression \eqref{Eq:Tilde-tau-F-r} of the map $\tilde \tau_F$.  Moreover, 
the face map $\pi_F: \boldsymbol{\mathcal Y}_r\dashrightarrow \mathbf P^{r-3}$ is defined on $\boldsymbol{\mathcal Y}_r^*$ and one has $\pi_F( \boldsymbol{\mathcal Y}_r^*)=\mathbf P^{r-3}\setminus \mathcal H_r$. 
Finally, on $\boldsymbol{Y}_r$, there exists an explicit arrangement of hypersurfaces $\boldsymbol{\sf H}_r\subset \boldsymbol{Y}_r$ on the complement 
$\boldsymbol{Y}_r^*=\boldsymbol{Y}_r\setminus \boldsymbol{\sf H}_r$ of which $\Xi_r$ is defined and induces an isomorphism $\Xi_r : \boldsymbol{Y}_r^*\simeq \boldsymbol{\mathcal Y}_r^*$. It follows that the rational map $U_F= \pi_F\circ \Xi_r$ is defined on $\boldsymbol{Y}_r^*$ with image 
$U_F\big(\boldsymbol{Y}_r^*\big)=\mathbf P^{r-3}\setminus \mathcal H_r$. 
\end{prop} 

The result above has been established through direct computations and case-by-case verifications. It would nonetheless be of interest to provide a conceptual and genuinely uniform proof valid for all values of \( r \).\sk

It can be verified that all the objects considered in the above result are defined over $\mathbf R$. In particular, this is true for the rational maps \eqref{Eq:UF-r} which actually can all be written as rational maps with non-zero coefficients $\pm 1$.

\subsection{\bf The virtual $\boldsymbol{(r-3)}$-ranks}
Having explicit rational first integrals of the web $\boldsymbol{W}^{GM}_{\hspace{-0.05cm}\boldsymbol{Y}_r}$ at our disposal, the determination of the virtual ranks of these webs becomes a matter of formal computation. After carrying out the necessary calculations, we obtain the following result for the virtual ranks of top-degree abelian relations:
\begin{prop}
\label{Prop:Virtual-Ranks-r=6,7}
%
%
%
%
%
%
One has 
\begin{align*}
\rho^{\bullet}_3\Big( 
\boldsymbol{W}^{GM}_{\hspace{-0.05cm}\boldsymbol{ Y}_6}
\Big)
=  &\,  (10,10,1) 
&& \hspace{-2.65cm}  \mbox{hence} \quad 
\rho_3\Big( 
\boldsymbol{W}^{GM}_{\hspace{-0.05cm}\boldsymbol{Y}_6}
\Big)
=  21\\
\mbox{ and } \quad 
\rho^{\bullet}_4\Big( 
\boldsymbol{W}^{GM}_{\hspace{-0.05cm}\boldsymbol{Y}_7}
\Big)
=  &\, (28,20,1) 
&& \hspace{-2.65cm} \mbox{hence} \quad 
\rho_4\Big( 
\boldsymbol{W}^{GM}_{\hspace{-0.05cm}\boldsymbol{Y}_7}
\Big)
=  49\, . 
\end{align*}
\end{prop}

\subsection{\bf The $\boldsymbol{(r-3)}$-abelian relations}
%
%
%
Most of the results concerning the (ranks and abelian relations of the) first two Gelfand-MacPherson webs $\boldsymbol{\mathcal W}^{GM}_{\hspace{-0.1cm} \boldsymbol{\mathcal Y}_r}$ appear to extend to the last two cases (i.e., for $r = 6$ and $r = 7$). In any case, this holds for abelian relations of maximal degree, as demonstrated below.

\subsubsection{\bf The master differential identity ${\bf HLOG}_{\boldsymbol{Y}_r}$}
In the affine coordinates $u_1,\ldots, u_{r-3}$ on $U_0=\{ u_0=1\}\simeq \mathbf C^{r-3}$, we set 
$L_i=u_i$ for $i=1,\ldots,r-3$ and $L_{r-2}=\sum_{i=0}^{r-3} u_i=1+\sum_{i=1}^{r-3} u_i$.  
Then one defines a $(r-3)$-differential form with logarithmic coefficients 
on $\mathbf P^{r-3}\setminus \mathcal H_r $ and with logarithmic poles along the components of $\mathcal H_r$ by setting 
\begin{align*}
\Omega_r = &\, 
\sum_{i=1}^{r-2} 
(-1)^{i-1}
{\ln}\big( L_i) \bigwedge_{\substack{k=1 \\ k\neq i}}^{r-2}
\frac{dL_k}{L_k} \\
= & \, \sum_{i=1}^{r-2} (-1)^{i-1} \ln\big(L_i\big)\, 
d \ln\big(L_1\big)\wedge \cdots \wedge 
\reallywidehat{d\ln\big(L_i\big)} \wedge \cdots \wedge 
d \ln\big(L_{r-2}\big)\, . 
\end{align*}
Up to a scaling factor, $\Omega_r$ can also be defined as the antisymmetrization  of $ \ln (L_{1})
\wedge_{k=2}^{r-2} 
d \ln(L_k)$ with respect to the $L_i$'s.
Then we have the following result: 
\begin{prop}
There exists $\big( \varepsilon_F\big)_{ F\in \boldsymbol{\mathcal K}_r}\in \{ \, \pm 1\, \}^{ \boldsymbol{\mathcal K}_r}$, unique up to sign, such that the following differential relation is satisfied: 
$$
\boldsymbol{\big( {\bf HLOG}_{ \boldsymbol{Y}_r}\big)}
\hspace{4cm}
\sum_{F\in \boldsymbol{\mathcal K}_r } 
 \varepsilon_F\, 
U_{F}^*\Big( \Omega_{r} \Big)=0\, .
\hspace{7cm} {}^{} 
$$
Consequently,  the $\kappa_r$-tuple  $\big( \varepsilon_F\, 
U_{F}^*( \Omega_{r} )\big)_{ F\in \boldsymbol{\mathcal K}_r }$ can be seen as a  $(r-3)$-abelian relation for $\boldsymbol{W}^{GM}_{\hspace{-0.1cm} \boldsymbol{Y}_r}$, again denoted by $ {\bf HLOG}_{ \boldsymbol{Y}_r}$. \sk 
\end{prop}
\begin{proof}
The current proof relies on direct computations performed in \textsc{Maple}.\footnote{Maple worksheets are available from the author upon request.}
%
\end{proof}

\begin{rem}
The $(r-3)$-form $\Omega_r$ is holomorphic and multivalued. It has a natural global univalued real-analytic version, namely
$
\Omega_r^\omega =
\sum_{i=1}^{r-2} 
(-1)^{i-1}
{\ln}\,\lvert L_i\vert \,d \ln\big(L_1\big)\wedge \cdots \wedge 
\reallywidehat{d\ln\big(L_i\big)}  \wedge \cdots \wedge 
d \ln\big(L_{r-2}\big)$. 
One can verify that the associated real-analytic identity 
$$
\boldsymbol{\big( {\bf HLOG}^\omega_{ \boldsymbol{\mathcal Y}_r}\big)}
\hspace{4cm}
\sum_{F\in \boldsymbol{\mathcal K}_r } 
 \varepsilon_F\, 
U_{F}^*\Big( \Omega^\omega_{r} \Big)=0\, .
\hspace{7cm} {}^{} 
$$
holds true identically on $\boldsymbol{Y}_r^+$. 
\end{rem}

\subsubsection{\bf Combinatorial $\boldsymbol{(r-3)}$-abelian relations}

It is more convenient to formulate the results of this subsection in terms of the web 
$\boldsymbol{\mathcal W}^{GM}_{ \hspace{-0.05cm}
\boldsymbol{\mathcal Y}_r}$ on  $\boldsymbol{\mathcal Y}_r$. Since the representation $V_r$, into whose projectivization $\boldsymbol{\mathcal G}_r=G_r/P_r$ embeds, is minuscule, it admits a natural system of linear coordinates $x_\ell$, indexed by $\boldsymbol{\mathcal L}_r\simeq \boldsymbol{\mathfrak W}_r$, each uniquely determined up to scaling.

It follows from \cite{Skorobogatov} that the material described in \S\ref{} generalizes to all cases $r = 4, \dots, 7$. In particular, for each $\ell\in \boldsymbol{\mathcal L}_r$, the image in $\boldsymbol{\mathcal Y}_r$ - by the quotient map - of the $H_r$-invariant hyperplane section 
$ \{x_\ell=0\} \cap \boldsymbol{\mathcal G}_r $ is an irreducible divisor $\boldsymbol{\mathcal D}_\ell$, along which the residue of the abelian relation $
{\bf HLOG}_{ \boldsymbol{\mathcal Y}_r}$ can be taken. We denote this residue by 
$$
{\bf AR}^{r-3}_\ell= 
{\bf Res}_{ \boldsymbol{\mathcal D}_{\ell} }\Big( 
{\bf HLOG}_{ \boldsymbol{\mathcal Y}_r }
\Big) \, .$$

We will use the following fact which characterizes the adjacency between the vertices and the orthoplex facets of $\Delta_r$ in terms of their respective labelling by the set of lines $
\boldsymbol{\mathcal L}_{r}$ and the one of conic classes $\boldsymbol{\mathcal K}_r$: 
\begin{lem}
The vertex of $\Delta_r$ corresponding to ${\ell}\in \boldsymbol{\mathcal L}_r$ is adjacent to the $D_{r-1}$-facet corresponding to ${\mathfrak c}\in \boldsymbol{\mathcal K}_r$ if and only if ${\mathfrak c}\in{\ell}+  \boldsymbol{\mathcal L}_r$, that is there exists ${\ell}'\in  \boldsymbol{\mathcal L}_r$ such that ${\mathfrak c}={\ell}+{\ell}'$. 
\end{lem}

For any $\ell \in \boldsymbol{\mathcal L}_r$, we define 
 $\boldsymbol{\mathcal K}_{r}(\ell)=\big\{ \, {\mathfrak c}  \in \boldsymbol{\mathcal K}\, \lvert \, 
 {\mathfrak c}-\ell \in \boldsymbol{\mathcal L}_{r}\,\big\}$. 
 This set can be identified with the collection of $(D_{r-1})$-facets of the polytope $\Delta_r$ that are adjacent to $\ell$, viewed as a vertex of $\Delta_r$.
 We denote by 
 $\boldsymbol{\mathcal W}^{GM}_{ \hspace{-0.05cm}
\boldsymbol{\mathcal Y}_r,\ell}$  the subweb of 
$\boldsymbol{\mathcal W}^{GM}_{ \hspace{-0.05cm}
\boldsymbol{\mathcal Y}_r}$
 corresponding to these facets; that is,
$$ 
\boldsymbol{\mathcal W}^{GM}_{ \hspace{-0.05cm}
\boldsymbol{\mathcal Y}_r,\ell}
=\boldsymbol{\mathcal W}\Big(\, \pi_F\, \big\lvert \,  
F \in 
\boldsymbol{\mathcal K}_{r}(\ell)\, 
\Big)\, .  
$$  Since the vertex figure of $\Delta_r$ at $\ell$ is isomorphic to $\Delta_{r-1}$, we obtain a natural bijection $\boldsymbol{\mathcal K}_r(\ell) \simeq \boldsymbol{\mathcal K}_{r-1}$. It follows that 
$ 
\boldsymbol{\mathcal W}^{GM}_{ \hspace{-0.05cm}
\boldsymbol{\mathcal Y}_r,\ell}$ is a $\kappa_{r-1}$-web.

With these notations in place, we obtain the following result by direct computation:
\begin{prop}
{\bf 1.} For any $\ell\in  \boldsymbol{\mathcal L}_r$,  the residue ${\bf AR}^{r-3}_\ell$
is a rational $(r-3)$-AR. Its non-trivial components are carried by the $\pi_F$ foliations for all $D_{r-1}$-facets $F$ of $\Delta_r$ adjacent to $\ell$. In other terms, ${\bf AR}^{r-3}_\ell$  is a complete $(r-3)$-AR of 
$ \boldsymbol{\mathcal W}^{GM}_{ \hspace{-0.05cm}
\boldsymbol{\mathcal Y}_r,\ell}$. 
\sk 

\noindent  {\bf 2.} Actually, ${\bf AR}^{r-3}_\ell$ spans the space of $(r-3)$-ARs of 
$ \boldsymbol{\mathcal W}^{GM}_{ \hspace{-0.05cm}
\boldsymbol{\mathcal Y}_r,\ell}$ which is $(r-3)$-rank 1: one has  
$$
\boldsymbol{AR}^{r-3}
\Big( 
\boldsymbol{\mathcal W}^{GM}_{ \hspace{-0.05cm}
\boldsymbol{\mathcal Y}_r,\ell}
\Big) = 
\Big\langle \,
{\bf AR}^{r-3}_\ell
\,\Big\rangle\, .  
$$

\noindent  {\bf 3.} For $F\in 
\boldsymbol{\mathcal K}_{r}(\ell)$, the $\pi_F$-component of ${\bf AR}^{r-3}_\ell$ 
is a linear combination with coefficients in $\{-1,0,1\}$
of the  following $r-2$ wedge products for $i=1,\ldots,r-2$:  
$$d \ln\big(\pi_{F,1}\big)\wedge \cdots \wedge 
\reallywidehat{d\ln\big(\pi_{F,i}\big) \wedge } \cdots \wedge 
d \ln\big(\pi_{F,r-2}\big)
\, .$$ 

\noindent  {\bf 4.} The ${\bf AR}^{r-3}_\ell$'s for $\ell \in 
\boldsymbol{\mathcal L}_{r}$ are called `combinatorial ARs'. They 
span a subspace of dimension 
$\rho_{r-3}( \boldsymbol{\mathcal W}^{GM}_{ \hspace{-0.05cm}
\boldsymbol{\mathcal Y}_r}\big)-1$ of 
$\boldsymbol{AR}^{r-3}\big( 
\boldsymbol{\mathcal W}^{GM}_{ \hspace{-0.05cm}
\boldsymbol{\mathcal Y}_r}\big)$ which we will denote by 
$
\boldsymbol{AR}^{r-3}_C\big( 
\boldsymbol{\mathcal W}^{GM}_{ \hspace{-0.05cm}
\boldsymbol{\mathcal Y}_r}\big)
$.
\end{prop}
This results shows that some properties of $\boldsymbol{\mathcal W}^{GM}_{ \hspace{-0.05cm}
\boldsymbol{\mathcal Y}_5}$ admit natural generalizations to the cases $r=5,6$. It is natural to wonder if it might be the case of some other properties, e.g. the first one of Corollary \ref{Cor:2-RA-WGM}. We believe that it is indeed the case and that for any $r\in \{4,\ldots,7\}$, the following holds true: 
\begin{quote}
{\it One has $\rho^{r-3}\big( \boldsymbol{\mathcal W}\big)\leq 1$ for any $\kappa_{r-1}$-subweb $\boldsymbol{\mathcal W}$ of $\boldsymbol{\mathcal W}_{\hspace{-0.1cm} \boldsymbol{\mathcal Y}_r}^{GM}$. Those 
for which the virtual $(r-3)$-rank is 1 actually are of maximal $(r-3)$-rank 1.  And these subwebs are exactly the $\ell_r$ subwebs 
$ \boldsymbol{\mathcal W}_{\hspace{-0.1cm} \boldsymbol{\mathcal Y}_r, \ell}^{GM}$ for all $\ell \in \boldsymbol{\mathcal L}_r$.} 
\end{quote}

\subsubsection{\bf The structure of $\boldsymbol{{AR}^{r-3}\Big( \boldsymbol{\mathcal W}^{GM}_{\hspace{-0.1cm} \boldsymbol{Y}_r}\Big)}$}
From the results above, we deduce the following
\begin{thm}
\label{THM:Main-r=4,5,6,7}
{\bf 1.} One has 
\begin{equation}
\label{Eq:Decomp-AR(r-3)-WGM-r}
\boldsymbol{AR}^{r-3}\Big( \boldsymbol{\mathcal W}^{GM}_{\hspace{-0.1cm} \boldsymbol{Y}_r}\Big)
=\boldsymbol{AR}^{r-3}_C\Big( \boldsymbol{\mathcal W}^{GM}_{\hspace{-0.1cm} \boldsymbol{Y}_r}\Big)\oplus \Big\langle \,
 {\bf HLOG}_{ \boldsymbol{Y}_r}\,
\Big\rangle\, . 
\end{equation}
Moreover, $\boldsymbol{AR}^{r-3}_C\big( \boldsymbol{\mathcal W}^{GM}_{\hspace{-0.1cm} \boldsymbol{Y}_r}\big)$ is of dimension $\rho_{r-3}\big(\boldsymbol{\mathcal W}^{GM}_{\hspace{-0.1cm} \boldsymbol{Y}_r}\big)-1$ hence the $(r-3)$-rank of 
$ \boldsymbol{\mathcal W}^{GM}_{\hspace{-0.1cm} \boldsymbol{Y}_r}$ is 
$\rho_{r-3}\big(\boldsymbol{\mathcal W}^{GM}_{\hspace{-0.1cm} \boldsymbol{Y}_r}\big)$, that is,  is  AMP. 
\sk \\
{\bf 2.} By residues/monodromy, the abelian relation 
$ {\bf HLOG}_{ \boldsymbol{Y}_r}$
  spans  the subspace $\boldsymbol{AR}^{r-3}_C\Big( \boldsymbol{\mathcal W}_{\hspace{-0.1cm} \boldsymbol{Y}_r}^{GM}\Big)$ 
  of combinatorial ARs, which coincides with that of rational $(r-3)$-ARs of $\boldsymbol{\mathcal W}_{\hspace{-0.1cm} \boldsymbol{Y}_r}^{GM}$: one has 
  $$
  {\bf Res}\Big( {\bf HLOG}_{\hspace{-0.03cm} \boldsymbol{Y}_r} \Big) 
  = 
  \boldsymbol{AR}^{r-3}_C\Big( \boldsymbol{\mathcal W}_{\hspace{-0.1cm} \boldsymbol{ Y}_r}^{GM}\Big)=\boldsymbol{AR}^{r-3}_{\rm Rat}\Big( \boldsymbol{\mathcal W}_{\hspace{-0.1cm} \boldsymbol{Y}_r}^{GM}\Big)
  \, .
  $$ 
\item[{\bf 3.}] For any smooth del Pezzo surface $X_r={\rm dP}_{d}$ with $d=9-r$, using Serganova-Skorobogatov embedding $f_{S\hspace{-0.05cm} S}: {\rm dP}_d \hookrightarrow \boldsymbol{\mathcal Y}_r$ (cf.\,\eqref{Eq:Embedding-SS} above), 
one can recover  in a natural way 
the weight $r-2$ hyperlogarithmic abelian relation ${\bf HLog}^{r-2}_{  \hspace{0.05cm} {\rm dP}_d}$ 
of $\boldsymbol{\mathcal W}_{  \hspace{-0.05cm} {\rm dP}_d}$ 
 from the $(r-3)$-AR 
${\bf HLOG}_{ \boldsymbol{\mathcal Y}_r } $. 
\mk
\end{thm}
We believe that the results known to hold for $r = 4, 5$ generalize straightforwardly to the cases $r = 6, 7$. For instance, we are convinced that the following statement remains valid for $r = 6, 7$, in the context of the natural linear action of the Weyl group $W_r$ on the space of $(r-3)$-abelian relations of the web $\boldsymbol{\mathcal W}_{\hspace{-0.1cm} \boldsymbol{Y}_r}^{GM}$
\begin{quote}
{\it The decomposition in direct sum \eqref{Eq:Decomp-AR(r-3)-WGM-r} is, in fact, the decomposition of $\boldsymbol{AR}^{r-3}\Big( \boldsymbol{\mathcal W}_{\hspace{-0.1cm} \boldsymbol{Y}_r}^{GM}\Big)$ 
into irreducible $W_r$-representations. The 1-dimensional subrepresentation 
spanned by 
 $ {\bf HLOG}_{ \boldsymbol{Y}_r}$ is  (isomorphic to) the signature representation.}
\end{quote}
By analogy with the cases $r = 4, 5$, one may also conjecture that $\boldsymbol{AR}^{r-3}_C\Big( \boldsymbol{\mathcal W}_{\hspace{-0.1cm} \boldsymbol{Y}_r}^{GM}\Big)$  is irreducible as a $W_r$-representation. We will revisit this  in a future work.

\section{\bf Perspectives and questions}
\label{SS:Perspectives-questions}
In this final section, we begin by discussing natural generalizations to 
$\boldsymbol{\mathcal W}_{\hspace{-0.1cm} \boldsymbol{\mathcal Y}_5}^{GM}$
 of the many remarkable properties enjoyed by Bol's web. In \S\ref{SS:HLOG-Yr-in-terms-scattering-diagram}, we offer some brief remarks on the possible interpretation of the differential identity ${\bf HLOG}_{\boldsymbol{\mathcal Y}_5}$ as the manifestation of a yet-to-be-determined property of a certain scattering diagram that might be associated to $\boldsymbol{\mathcal Y}_5$.

\subsection{\bf  Comparison between Bol's web $\boldsymbol{\mathcal W}_{\hspace{-0.1cm} \boldsymbol{\mathcal Y}_4}^{GM}$ and $\boldsymbol{\mathcal W}_{\hspace{-0.1cm} \boldsymbol{\mathcal Y}_5}^{GM}$}
\label{SS:Comp}
 The main objective of our previous paper \cite{PirioAFST} was to demonstrate that the web $\boldsymbol{\mathcal W}_{ {\rm dP}_4}$, formed by the ten pencils of conics on a smooth del Pezzo surface of degree four ${\rm dP}_4$, shares many remarkable properties with Bol's classical web $\boldsymbol{\mathcal B}\simeq \boldsymbol{\mathcal W}_{ {\rm dP}_5} \simeq \boldsymbol{\mathcal W}_{\hspace{-0.1cm} \boldsymbol{\mathcal Y}_4}^{GM}$.  Moreover, in  \cite[\S4.5]{PirioAFST}, we showed that any del Pezzo web $\boldsymbol{\mathcal W}_{ {\rm dP}_4}$ can be recovered from the Gelfand--MacPherson web $\boldsymbol{\mathcal W}_{\hspace{-0.1cm} \boldsymbol{\mathcal Y}_5}^{GM}$ as its pull-back under an embedding  ${\rm dP}_4\hookrightarrow \boldsymbol{\mathcal Y}_5$, first considered by Serganova and Skorobogatov in \cite{SerganovaSkorobogatov}. 
 
 One of the main results of the present paper is that, beyond this structural recovery, one can also retrieve the most significant abelian relation of 
$\boldsymbol{\mathcal W}_{ {\rm dP}_4}$ 
 -- namely, its hyperlogarithmic abelian relation of weight three, ${\bf HLog}_{\mathrm{dP}_4}$ -- from the most notable $(r - 3)$-abelian relation of the former web $\boldsymbol{\mathcal W}_{\hspace{-0.1cm} \boldsymbol{\mathcal Y}_5}^{GM}$, namely ${\bf HLOG}_{\boldsymbol{\mathcal Y}_5}$. This naturally raises the question (at least for the author) of whether the many remarkable properties satisfied by Bol's web might also admit natural analogues in the context of the Gelfand-MacPherson web $\boldsymbol{\mathcal W}_{\hspace{-0.1cm} \boldsymbol{\mathcal Y}_5}^{GM}$.
 \begin{center}
 $***$
 \end{center}
 
 Below, we review the list of remarkable properties of Bol's web as presented in \cite[\S1.1]{PirioAFST}, and for each of them, we provide a brief comment regarding its possible generalization to the web $\boldsymbol{\mathcal W}_{\hspace{-0.1cm} \boldsymbol{\mathcal Y}_5}^{GM}$. 
\sk

\paragraph{\bf 1. Geometric definition}
The del Pezzo webs are defined as the webs on del Pezzo surfaces whose foliations are the pencils of conics on these surfaces.
 It is plausible that a similar geometric interpretation may exist for the Gelfand--MacPherson web  $\boldsymbol{\mathcal W}_{\hspace{-0.1cm} \boldsymbol{\mathcal Y}_5}^{GM}$. Indeed, as established in  \cite{SerganovaSkorobogatov} (see the very end of the  proof of Theorem 6.1 therein), for any del Pezzo quartic surface ${\rm dP}_4$, the associated Serganova-Skorobogatov embedding $ F_{\hspace{-0.05cm}S\hspace{-0.03cm}S}: {\rm dP}_4\hookrightarrow \boldsymbol{\mathcal Y}_5$ induces an 
 isomorphism of Picard lattices $F_{\hspace{-0.05cm}S\hspace{-0.03cm}S}^* : {\rm Pic}_{\mathbf Z}\big( \boldsymbol{\mathcal Y}_5 \big)\simeq {\rm Pic}_{\mathbf Z}\big( {\rm dP}_4\big)$. Let $\boldsymbol{\mathfrak c}\in {\rm Pic}_{\mathbf Z}\big( {\rm dP}_4\big)$ be a conic class with associated facet $F\subset \Delta_5$ and corresponding face map $\pi_F: \boldsymbol{\mathcal Y}_5\dashrightarrow \mathbf P^2$. Then, in \cite{PirioAFST}, we proved that the pull-back
 of the linear system $\lvert \mathcal O_{\mathbf P^2}(1)\lvert$ 
  under the composition $\pi_F\circ F_{\hspace{-0.05cm}S\hspace{-0.03cm}S}
 : {\rm dP}_4 \dashrightarrow \mathbf P^2$  coincides precisely with the complete linear system 
  $\lvert \,\boldsymbol{\mathfrak c} \, \lvert$.
 This observation suggests that the 
  map 
  $\pi_F: \boldsymbol{\mathcal Y}_5 \dashrightarrow \mathbf P^2$ 
  may in fact be induced by the complete linear system $\lvert \,\boldsymbol{\mathfrak c} \, \lvert$ when $\boldsymbol{\mathfrak c}$ is now seen as an element of ${\rm Pic}_{\mathbf Z}\big( \boldsymbol{\mathcal Y}_5 \big)$. 
\sk

\paragraph{\bf 2. Linearizability}
Since it contains Bol's web as a subweb, any del Pezzo web $\boldsymbol{\mathcal W}_{ {\rm dP}_4}$ is non-linearizable. 
It is currently unclear to what extent the (non-)linearizability of the web 
$\boldsymbol{\mathcal W}_{\hspace{-0.1cm} \boldsymbol{\mathcal Y}_5}^{GM}$ 
 is relevant. We do believe that this web is not linearizable, but we also consider this fact to be non-essential.  The approach discussed in \cite{PirioLin} should allow to address this question.\sk

\paragraph{\bf 3.\,\& 4. Structure of the space of top degree ARs \& maximality of the rank}
The spaces of top degree abelian relations of $\boldsymbol{\mathcal B}\simeq 
\boldsymbol{\mathcal W}_{\hspace{-0.1cm} \boldsymbol{\mathcal Y}_4}^{GM}$ 
and $\boldsymbol{\mathcal W}_{\hspace{-0.1cm} \boldsymbol{\mathcal Y}_5}^{GM}$ share the same structure. Indeed, in both cases, there exists a subspace of so-called \emph{combinatorial} ARs (denoted by a capital subscript ${C}$), such that the following direct sum decompositions hold:
\begin{align}
\label{Eq:Decomp-Direct-Sums}
\boldsymbol{AR}
\Big( \boldsymbol{\mathcal W}_{\hspace{-0.1cm} \boldsymbol{\mathcal Y}_4}^{GM} \Big) = &\, 
\boldsymbol{AR}_C
\Big( \boldsymbol{\mathcal W}_{\hspace{-0.1cm} \boldsymbol{\mathcal Y}_4}^{GM} \Big)\oplus 
\big\langle \,
{\bf Ab} \,
\big\rangle  
\\ 
\boldsymbol{AR}^{(2)}
\Big( \boldsymbol{\mathcal W}_{\hspace{-0.1cm} \boldsymbol{\mathcal Y}_5}^{GM} \Big) = &\, 
\boldsymbol{AR}_C^{(2)}
\Big( \boldsymbol{\mathcal W}_{\hspace{-0.1cm} \boldsymbol{\mathcal Y}_5}^{GM} \Big)\oplus 
\Big\langle \, 
{\bf HLOG}_{\boldsymbol{\mathcal Y}_5}\,
\Big\rangle  \, . 
\nonumber
%
\end{align}
Moreover, as is well known, Bol's web has maximal rank 6. As shown above, the genuine 2-rank of the web $ \boldsymbol{\mathcal W}_{\hspace{-0.1cm} \boldsymbol{\mathcal Y}_5}^{GM} $ is equal to its virtual 2-rank. Therefore, both webs have AMP ranks. 
Regarding abelian relations and the rank structure of the webs, the similarity between $ \boldsymbol{\mathcal W}_{\hspace{-0.1cm} \boldsymbol{\mathcal Y}_r}^{GM}$  for $r = 4$ and $r = 5$ is clearly observable.
\sk

\paragraph{\bf 5. `Canonical algebraization'}
In \cite[\S3.5]{PirioAFST}, we explained how Bol's web can be canonically recovered from the space of its combinatorial abelian relations. It is an interesting question to ask whether a similar phenomenon holds for the web $\boldsymbol{\mathcal W}_{\hspace{-0.1cm} \boldsymbol{\mathcal Y}_5}^{GM}$. We believe that it might be the case. 

Let $U\Omega^2$ be the subspace 
of rational 2-forms on $\boldsymbol{Y}_5=\mathbf C^5$ spanned by the  30 wedge products $d\ln U_{i,a}\wedge d\ln U_{i,b}$ for all $i\in [\hspace{-0.05cm}[ 10 ]\hspace{-0.05cm}]$ and all $a,b$ such that $1\leq a<b\leq 3$. 
This is the space generated by all the components of all the combinatorial 2-abelian relations. It can be verified that $U\Omega^2$
  is a 20-dimensional vector subspace of  
 $\Omega^2_{\mathbf C( \boldsymbol{Y}_5)}$. We wonder whether the following statements hold:
\begin{itemize}
\item  {\it for any $y\in 
\boldsymbol{Y}_5^*$,  the evaluation map at this point,  $ev_y : U\Omega^2\rightarrow \Omega^2_{\boldsymbol{Y}_5,y}$, $\omega\mapsto \omega(y)$  is well-defined and surjective. Consequently, its kernel $K(y)={\rm Ker}(ev_y)$ is a vector subspace of $U\Omega^2$ of dimension 10;}
\sk 
\item   {\it we thus obtain a map $\boldsymbol{Y}_5^*\rightarrow G_{10}\big( 
U\Omega^2\big)\simeq G_{10}(\mathbf C^{20})$, $y\longmapsto K(y)$ which, composed with the Pl\"ucker embedding, gives rise to a morphism 
$$\mu : \boldsymbol{Y}_5^* \longrightarrow \mathbf P\Big(\wedge^{10} \mathbf C^{20}\Big)=\mathbf P^{ { 20 \choose 10}-1}\, .$$ 
Composed by the birational identification  $\boldsymbol{Y}_5\simeq \boldsymbol{\mathcal Y}_5$ (see \eqref{Eq:Theta}), this yields an embedding $\mu\circ \Theta^{-1}: \boldsymbol{\mathcal Y}_5 \hookrightarrow \mathbf P^{ { 20 \choose 10}-1}$ such that 
the closure of the image $\overline{\boldsymbol{\mathcal Y}}_5 = \overline{\mu(\boldsymbol{Y}_5^*)}$ 
provides a compactification of $\boldsymbol{\mathcal Y}_5$ with `nice properties'.}
\end{itemize}
A similar construction for the web $\boldsymbol{\mathcal W}_{\hspace{-0.1cm} \boldsymbol{\mathcal Y}_4}^{GM}$ on $\boldsymbol{\mathcal Y}_4^*\simeq \mathcal M_{0,5}$ gives rise to (the restriction to the open stratum of) the log canonical embedding $\overline{\mathcal M}_{0,5}\hookrightarrow \mathbf P^{9}$, that is the embedding induced
by the complete linear system associated to the ample log-canonical divisor $K_{\overline{\mathcal M}_{0,5}}  +\partial \overline{\mathcal M}_{0,5}$ (see \cite{KeelTevelev}). It is natural to ask whether an analogous result holds for the web $\boldsymbol{\mathcal W}_{\hspace{-0.1cm} \boldsymbol{\mathcal Y}_5}^{GM}$
on $\boldsymbol{\mathcal Y}_5^*$.

 In \cite{Corey}, the author studied the Chow quotient $\mathbb S_5/\hspace{-0.1cm}/ H$ of the spinor tenfold by the action of the Cartan torus $H\subset {\rm Spin}\big(\mathbf C^{10}\big)$.  He proved that, up to normalization, this Chow quotient is  smooth with a boundary having simple normal crossings, and described the log canonical model $\sf Y_5$ of $\mathbb S_5/\hspace{-0.1cm}/ H$.   It would be interesting to investigate possible relationships between the conjectural compactification $\overline{\boldsymbol{\mathcal Y}}_5$  discussed above and the proper varieties $\mathbb S_5/\hspace{-0.1cm}/ H$  and $\sf Y_5$  studied by Corey in his paper.
\mk 

\paragraph{\bf 6. Weyl group action}
Thanks to a result by Skorobogatov, for any $r=4,\ldots,7$, the Weyl group $W_{r}$ acts by automorphisms on $\boldsymbol{\mathcal Y}_r$ and preserves Gelfand--MacPherson web $\boldsymbol{\mathcal W}_{\hspace{-0.1cm} \boldsymbol{\mathcal Y}_r}^{GM}$.  From this, we deduce a linear action of $W_r$ on $\boldsymbol{AR}^{(r-3)}\big(  \boldsymbol{\mathcal W}_{\hspace{-0.1cm} \boldsymbol{\mathcal Y}_r}^{GM}\big) $. 
In both cases, the components 
$\big\langle 
\boldsymbol{\bf Ab} 
\big\rangle$ and  $\big\langle \,
{\bf HLOG}_{\boldsymbol{\mathcal Y}_5} \,
\big\rangle$
are isomorphic to the signature representation of $W_r$.
\mk

\paragraph{\bf 7. Hexagonality and characterization}
It has been known since Bol's paper \cite{Bol} that Bol's web, that is 
$\boldsymbol{\mathcal W}_{ {\rm dP}_5}\simeq 
\boldsymbol{\mathcal W}_{\hspace{-0.1cm} \boldsymbol{\mathcal Y}_4}^{GM}$, 
can be characterized -- up to local analytic equivalence -- as the unique planar hexagonal web that is not linearizable. It is natural, though perhaps somewhat naive, to ask whether a similar characterization holds for the Gelfand--MacPherson web $\boldsymbol{\mathcal W}_{\hspace{-0.1cm} \boldsymbol{\mathcal Y}_5}^{GM}$.

Let us discuss briefly what could be the analog of Bol's characterization for this 5-dimensional 10-web.
First, the formal similarity between the direct sum decompositions given in
\eqref{Eq:Decomp-Direct-Sums} suggests that the appropriate analogue, for $\boldsymbol{\mathcal W}_{\hspace{-0.1cm} \boldsymbol{\mathcal Y}_5}^{GM}$, of the notion of a hexagonal 3-subweb in planar web geometry might be the 5-subwebs of $\boldsymbol{\mathcal W}_{\hspace{-0.1cm} \boldsymbol{\mathcal Y}_5}^{GM}$ which carry a non-trivial irreducible and complete 2-abelian relation. These subwebs have been described above in Corollary \ref{Cor:2-RA-WGM}. In particular, it follows that not every 5-subweb of $\boldsymbol{\mathcal W}_{\hspace{-0.1cm} \boldsymbol{\mathcal Y}_5}^{GM}$ carries a complete and irreducible 2-abelian relation. However, this does not prevent us from envisioning a Bol-type characterization of this web based on its 5-subwebs. We now outline a few more specific remarks on what would need to be established in order to derive such a characterization:
\begin{itemize}
\item  {\it{Given a 10-web $\boldsymbol{\mathcal W}$ of codimension 2 on a 5-dimensional manifold $M$, 
one defines a function ${{R}}_{\boldsymbol{\mathcal W}}$ on the set 
of its 5-subwebs  by associating 
its 2-rank to any such subweb. 
The natural question is then whether the function 
 ${{R}}_{\boldsymbol{\mathcal W}_{\hspace{-0.1cm} \boldsymbol{\mathcal Y}_5}^{GM} }$ 
 characterizes 
   $\boldsymbol{\mathcal W}_{\hspace{-0.1cm} \boldsymbol{\mathcal Y}_5}^{GM}$ up to local 
   analytic equivalence.}}
   \mk 
\item {\it Before attempting to investigate this question, it may be necessary to first study more thoroughly the analytic classification of 5-webs of codimension 2 on 5-dimensional manifolds. In particular, one may ask:}
{\it \begin{itemize}
    \item  Is there a universal bound on the 2-rank of such a web? Could it be equal to 1?
    \item What can be said about the webs that achieve this maximal 2-rank?
    \item Can such webs be classified?
      Is there a unique normal form for them?
    \item Does the class of webs attaining the maximal 2-rank share, in some meaningful sense,  some `nice properties' with the class of
     hexagonal planar 3-webs?
\end{itemize}}
\end{itemize}

\paragraph{\bf 8. Description \`a la Gelfand-MacPherson}
That Bol's web $\boldsymbol{\mathcal B}\simeq \boldsymbol{\mathcal W}_{ {\rm dP}_5}$ can be described as a Gelfand-MacPherson web is a result established in \cite{GM}. This contrasts with  $\boldsymbol{\mathcal W}_{\hspace{-0.1cm} \boldsymbol{\mathcal Y}_5}^{GM}$, which is of this type by definition.

While the fact that $\boldsymbol{\mathcal W}_{\hspace{-0.1cm} \boldsymbol{\mathcal Y}_5}^{GM}$ is of Gelfand--MacPherson type is immediate from its very definition, extending this perspective to its master 2-abelian relation ${\bf HLOG}_{\boldsymbol{\mathcal Y}_5}$ is a much more subtle and interesting question, which deserves some further discussion. In their foundational work \cite{GM}, Gelfand and MacPherson not only provide a geometric construction of Bol's web (namely as the web  $\boldsymbol{\mathcal W}_{\hspace{-0.1cm} \boldsymbol{\mathcal Y}_4}^{GM}$), but also demonstrate how, within the same geometric framework, one can construct the dilogarithm function and derive its 5-term functional identity. 
Remarkably, this identity arises in their theory as a rather direct consequence of Stokes' theorem, applied to the integration along the fibers of the action of the Cartan torus of ${\rm SL}(\mathbf{R}^5)$ on the Grassmannian $G_2^{or}(\mathbf R^5)$ of oriented $2$-planes in $\mathbf{R}^5$. The integrand is an invariant differential $4$-form representing the first Pontrjagin class of the tautological bundle on $G_2^{or}(\mathbf R^5)$.

It appears natural to us to seek a similar geometric construction for the master 2-abelian relation. A question in this direction was already raised in \cite[\S5.9]{PirioAFST}, but for the weight 3 hyperlogarithmic abelian relation 
${\bf HLog}_{ {\rm dP}_4}$ associated with a smooth del Pezzo quartic surface  ${\rm dP}_4$.  In light of the results of the present paper -- particularly the fact that any identity ${\bf HLog}_{ {\rm dP}_4}$ can be obtained from 
${\bf HLOG}_{\boldsymbol{\mathcal Y}_5}$ -- we now believe that the following question should be preferred over Question 5.7 in \cite{PirioAFST}:

\noindent{\bf Question.--} {\it Can the differential identity 
${\bf HLOG}_{\boldsymbol{\mathcal Y}_5}$ be
obtained \`a la Gelfand-MacPherson by means of an invariant
differential form $\Omega(P)$ on a real form $S_5$ of the spinor tenfold $\mathbb S_5$, representing a certain characteristic class 
$P\in {\bf H}^6(  S_5,\mathbf R)$?} 
\sk

This is one of the questions suggested by our work that we find the most appealing and we plan to work on it in the near future. 
\mk

\paragraph{\bf 9. Modularity}
 The web $\boldsymbol{\mathcal W}_{\hspace{-0.1cm} \boldsymbol{\mathcal Y}_4}^{GM}$  can also be interpreted as the web defined on $\mathcal{M}_{0,5}$ by the five forgetful maps $\varphi_i : \mathcal M_{0,5}\rightarrow \mathcal M_{0,4}\simeq \mathbf P^1\setminus \{0,1,\infty\}$. 
 
Since it is defined on a moduli space via morphisms with a modular interpretation, it is reasonable to describe $\boldsymbol{\mathcal W}_{\hspace{-0.1cm} \boldsymbol{\mathcal Y}_4}^{GM}$ as a `modular web'. 
 In \cite[\S4.6]{PirioAFST}, we proved that that any del Pezzo web 
$\boldsymbol{\mathcal W}_{ {\rm dP}_4}$ can also be qualified as `modular web'.
This naturally leads to the following question: 

\noindent{\bf Question.} {\it Does Gelfand-MacPherson's web $\boldsymbol{\mathcal W}_{\hspace{-0.1cm} \boldsymbol{\mathcal Y}_5}^{GM}$ 
admit an interpretation as a modular web? For instance, is it naturally defined on a subvariety of some moduli space of projective configurations?} 
\sk 

At this stage, we do not have any clear insight into a possible answer to this question.
\mk

\paragraph{\bf 10. Cluster nature}
An interesting feature of the web $\boldsymbol{\mathcal W}_{ {\rm dP}_5}=
\boldsymbol{\mathcal W}_{\hspace{-0.1cm} \boldsymbol{\mathcal Y}_4}^{GM}$ is that it is a \emph{`cluster web'}. In \cite[\S4.7]{PirioAFST}, we showed that this property also holds for any del Pezzo web $\boldsymbol{\mathcal W}_{ {\rm dP}_4}$.  
This raised the question whether this holds  for $\boldsymbol{\mathcal W}_{\hspace{-0.1cm} \boldsymbol{\mathcal Y}_4}^{GM}$ as well. This is precisely what we established in \S\ref{SS:Cluster-View} above, provided one admits cluster-like structures that are more general than classical cluster structures.


\subsection{\bf An interpretation of ${\bf HLOG}_{\boldsymbol{\mathcal Y}_4}$ in terms of a scattering diagram?}
\label{SS:HLOG-Yr-in-terms-scattering-diagram}
The author has long been fascinated by the functional equations of polylogarithms, among which Abel's five-term identity $\big(\boldsymbol{{\mathcal A}b}\big)$ for the dilogarithm stands out as particularly central and archetypal. In his view, there are two especially meaningful explanations for why $\big(\boldsymbol{{\mathcal A}b}\big)$ holds: on the one hand, the cohomological-analytic and geometric approach developed by Gelfand and MacPherson discussed just above; on the other, the interpretation of this functional identity as a consequence of the consistency of the scattering diagram associated with the $\mathcal X$-cluster variety of type $A_2$.

Although Gelfand--MacPherson's perspective appears to us as the most elegant -- indeed, it is, in the author's opinion, one of the most beautiful mathematical constructions  he has ever encountered -- the scattering diagram approach seems to offer a significantly broader scope. 
One particularly  appealing consequence of the latter theory\footnote{In the case of scattering diagrams arising from cluster algebras, this phenomenon has been thoroughly investigated by 
 {T. Nakanishi}. For further details, the reader is referred to Nakanishi's book and article \cite{NakanishiBook} and \cite{Nakanishi}.} is that any closed loop intersecting transversely a (possibly countable) collection of walls in a consistent scattering diagram gives rise to a dilogarithmic identity -- which becomes a formal identity involving infinitely many terms if the loop meets infinitely many walls.
\sk

The fact that the spinor tenfold $\mathbb S_5$, and more generally all the spaces $\boldsymbol{\mathcal G}_r$ for $r = 4, \dots, 7$\footnote{For $r=4$ this is well-known; for $r=5$ and $r=6$, it has been established in \cite{Ducat} and \cite{DaiseyDucat}, respectively. Only the case $r=7$ remains conjectural at the time of writing.}, carry a cluster-like structure of finite type invites us to dream of an interpretation of all the identities 
${\bf HLOG}_{\boldsymbol{\mathcal Y}_r}$  in terms of certain properties of scattering diagrams associated with the spaces $\boldsymbol{\mathcal Y}_r$.

In \cite{Ducat}, Ducat constructed a finite type LPA structure on $\mathbb S_5$ and  described an associated scattering diagram. 
It is natural to ask whether these objects admit versions defined on the torus quotient $\boldsymbol{\mathcal Y}_5$. Addressing this question would require a theory of coefficients and their mutations for LPA algebras, which, to our knowledge, has not yet been developed.



\vfill 
{\small  ${}^{}$ \hspace{-0.6cm} {Luc Pirio} ({\tt luc.pirio@uvsq.fr})\\
 LMV -- UVSQ Universit\'e Paris-Saclay \& CNRS  (UMR 8100)\\

\end{document}